\documentclass[twoside]{scrartcl}
\usepackage{scrlayer-scrpage}

\usepackage[utf8]{inputenc}
\usepackage[T1]{fontenc}
\usepackage[english]{babel}
\usepackage[unicode=true,plainpages=false]{hyperref}

\usepackage[left=2.5cm, right=3cm, top=2cm, bottom=3.5cm]{geometry}

\usepackage[shortlabels]{enumitem}

\usepackage{amsmath}
\usepackage{amssymb}
\usepackage{amsthm}
\usepackage{mathrsfs}
\usepackage{yhmath}
\usepackage{wasysym}
\usepackage{mathtools}

\usepackage[capitalise]{cleveref}

\usepackage{caption}
\usepackage{subcaption}

\usepackage{graphicx}
\usepackage{pgf,tikz}
\usepackage{tikz-cd}
\usepackage{rotating}

\usepackage{lmodern} 
\usepackage{tabularx} 
\usepackage[draft=false,babel,kerning=true,spacing=true]{microtype} 

\usepackage{todonotes}

\usepackage{xcolor}

\newcommand{\cycl}{\mathrm{cycl}}
\newcommand{\FF}{\mathrm{FF}}

\newcommand{\Dhsa}[1]{\D_{\hat\solid}^a(\ri^+_{#1})}
\newcommand{\DYwz}[1]{\D_{\hat\solid}(\mathcal{Y}_{[0,\infty),{#1}})} 
\newcommand{\DY}[1]{\D_{\hat\solid}(\mathcal{Y}_{(0,\infty),{#1}})}
\newcommand{\RH}{\mathrm{RH}}
\newcommand{\calC}{\mathcal{C}}

\allowdisplaybreaks

\makeatletter
\newcommand*{\relrelbarsep}{.386ex}
\newcommand*{\relrelbar}{%
  \mathrel{%
    \mathpalette\@relrelbar\relrelbarsep
  }%
}
\newcommand*{\@relrelbar}[2]{%
  \raise#2\hbox to 0pt{$\m@th#1\relbar$\hss}%
  \lower#2\hbox{$\m@th#1\relbar$}%
}
\providecommand*{\rightrightarrowsfill@}{%
  \arrowfill@\relrelbar\relrelbar\rightrightarrows
}
\providecommand*{\leftleftarrowsfill@}{%
  \arrowfill@\leftleftarrows\relrelbar\relrelbar
}
\providecommand*{\xrightrightarrows}[2][]{%
  \ext@arrow 0359\rightrightarrowsfill@{#1}{#2}%
}
\providecommand*{\xleftleftarrows}[2][]{%
  \ext@arrow 3095\leftleftarrowsfill@{#1}{#2}%
}
\makeatother

\DeclarePairedDelimiter{\abs}{\lvert}{\rvert}

\newcommand{\noloc}{\nobreak\mskip6mu plus1mu\mathpunct{}\nonscript\mkern-\thinmuskip{:}\mskip2mu\relax} 

\newcommand{\Sp}{\mathrm{Sp}}

\renewcommand{\Pr}{\mathrm{Pr}}
\newcommand{\DFF}{\mathcal{D}_\FF}

\newcommand{\isom}{\cong}

\newcommand{\xto}[1]{\mathbin{\xrightarrow{#1}}}

\newcommand{\isoto}{\xto\sim}

\DeclareMathOperator{\Hom}{Hom}

\DeclareMathOperator{\Fun}{Fun}

\DeclareMathOperator*{\colim}{colim}

\renewcommand{\emptyset}{\varnothing}

\renewcommand{\subset}{\subseteq}

\newcommand{\comp}{\mathbin{\circ}}
\DeclareMathOperator{\id}{id}

\newcommand{\pr}{\mathrm{pr}}

\newcommand{\injto}{\mathrel{\hookrightarrow}}
\newcommand{\surjto}{\mathrel{\twoheadrightarrow}}

\newcommand{\Z}{{\mathbb{Z}}}
\newcommand{\N}{\mathbb{N}}
\newcommand{\Fld}{\mathbb{F}}
\newcommand{\Q}{\mathbb{Q}}
\newcommand{\R}{{\mathbb{R}}}

\newcommand{\F}{\mathbb{F}}

\newcommand{\divides}{\mathbin{|}}

\newcommand{\GL}{\mathrm{GL}}

\DeclareMathOperator{\End}{End}

\DeclareMathOperator{\coker}{coker}

\newcommand{\tensor}{\otimes}

\DeclareMathOperator{\Sym}{Sym}
\DeclareMathOperator{\Aut}{Aut}

\DeclareMathOperator{\Spec}{Spec}

\newcommand{\Spf}{\mathrm{Spf}\,}

\newcommand{\A}{\mathbf{A}}

\DeclareMathOperator{\Cat}{\mathsf{Cat}}
\DeclareMathOperator{\Nuc}{Nuc}

\DeclareMathOperator{\IHom}{\underline{\Hom}}

\newcommand{\an}{{\mathrm{an}}} 
\newcommand{\et}{{\mathrm{et}}}

\newcommand{\proet}{{\mathrm{proet}}} 
\newcommand{\qproet}{{\mathrm{qproet}}} 
\newcommand{\Qproet}{{\mathrm{Qproet}}} 

\newcommand{\ri}{\mathcal O} 
\newcommand{\mm}{\mathfrak m} 

\DeclareMathOperator{\Spa}{Spa}
\DeclareMathOperator{\Spd}{Spd}

\newcommand{\opp}{{\mathrm{op}}} 

\newcommand{\solid}{{\scalebox{0.5}{$\square$}}}
\DeclareMathOperator{\Mod}{Mod}

\newcommand{\Perf}{\mathrm{Perf}}
\newcommand{\Perfd}{\mathrm{Perfd}}

\newcommand{\D}{\mathcal D} 

\DeclareMathOperator{\Comm}{Comm} 

\DeclareMathOperator{\AnSpec}{AnSpec}

\DeclareMathOperator{\Ani}{Ani}

\newcommand{\oc}{{\mathrm{oc}}} 

\DeclareMathOperator{\fib}{fib}
\DeclareMathOperator{\cofib}{cofib}

\DeclareMathOperator{\Corr}{Corr}
\DeclareMathOperator{\vStacks}{vStack}

\newcommand{\dimtrg}{\mathrm{dim.trg}} 
\newcommand{\nuc}{{\mathrm{nuc}}} 

\DeclareMathOperator{\Bun}{Bun}

\DeclareMathOperator{\CAlg}{CAlg} 

\newcommand{\perfDisc}{\widetilde{\mathbb D}}
\newcommand{\perfPunctDisc}{\widetilde{\mathbb D}^\times}

\newcommand{\perfBall}[1][]{\widetilde{\mathbb B}^{#1}}

\DeclareMathOperator{\DSuave}{SD}

\newcommand{\lax}{\mathrm{lax}}

\theoremstyle{plain}
\newtheorem{theorem}{Theorem}[subsection]
\newtheorem{theorem*}{Theorem}
\newtheorem{proposition}[theorem]{Proposition}
\newtheorem{proposition*}[theorem*]{Proposition}
\newtheorem{corollary}[theorem]{Corollary}
\newtheorem{lemma}[theorem]{Lemma}

\theoremstyle{definition}
\newtheorem{definition}[theorem]{Definition}
\newtheorem{definition*}[theorem*]{Definition}
\newtheorem{example}[theorem]{Example}
\newtheorem{examples}[theorem]{Examples}
\newtheorem{remark}[theorem]{Remark}
\newtheorem{remarks}[theorem]{Remarks}

\newtheorem{hypothesis*}[theorem*]{Hypothesis}

\numberwithin{equation}{theorem}
\numberwithin{figure}{section}
\numberwithin{table}{section}

\newlist{thmenum}{enumerate}{1}
\setlist[thmenum]{label=(\roman*), ref=\thetheorem.(\roman*)}
\crefalias{thmenumi}{theorem}

\newlist{propenum}{enumerate}{1}
\setlist[propenum]{label=(\roman*), ref=\theproposition.(\roman*)}
\crefalias{propenumi}{proposition}

\newlist{corenum}{enumerate}{1}
\setlist[corenum]{label=(\roman*), ref=\thecorollary.(\roman*)}
\crefalias{corenumi}{corollary}

\newlist{lemenum}{enumerate}{1}
\setlist[lemenum]{label=(\roman*), ref=\thelemma.(\roman*)}
\crefalias{lemenumi}{lemma}

\newlist{examplesenum}{enumerate}{1}
\setlist[examplesenum]{label=(\alph*), ref=\theexamples.(\alph*)}
\crefalias{examplesenumi}{example}

\newlist{remarksenum}{enumerate}{1}
\setlist[remarksenum]{label=(\roman*), ref=\theremarks.(\roman*)}
\crefalias{remarksenumi}{remark}

\newlist{defenum}{enumerate}{1}
\setlist[defenum]{label=(\alph*), ref=\thedefinition.(\alph*)}
\crefalias{defenumi}{definition}

\title{A $6$-functor formalism for solid quasi-coherent sheaves on the Fargues--Fontaine curve}
\author{Johannes Anschütz \and Arthur-C\'esar Le Bras  \and Lucas Mann}
\date{\today}

\begin{document}

\maketitle

\begin{abstract}
   We develop a $6$-functor formalism $\D_{[0,\infty)}(-)$ with $\Z_p$-linear coefficients on small v-stacks, and discuss consequences for duality and finiteness for pro-\'etale cohomology of rigid-analytic varieties of general pro-\'etale $\Q_p$-local systems as well as first examples motivated by a potential $p$-adic analog of Fargues--Scholze's geometrization program of the local Langlands correspondence.
\end{abstract}

\tableofcontents

\section{Introduction}
\label{sec:introduction}

\subsection{Pro-\'etale cohomology of rigid-analytic varieties: finiteness and duality}
Let $p$ be a prime and let $C$ be a complete algebraically closed non-archimedean extension of $\Q_p$. Let $X$ be a smooth rigid space over $C$. One main object of interest of this paper is the pro-\'etale cohomology of $X$ with coefficients the ``constant'' sheaf $\Q_p$\footnote{This sheaf sends $S \in X_\proet$ to $C(|S|,\Q_p)$. It would be more appropriate to denote it $\underline{\Q_p}$ rather than $\Q_p$ , but we will not do so, to keep the notation easy on the eye.} or more generally a finite rank pro-\'etale $\Q_p$-local system $\mathbb{L}$ on $X$. When $X$ is also assumed to be proper and of pure dimension $d$, the cohomology groups $H^i(X_\proet, \Q_p)$, $i\geq 0$, are finite dimensional $\Q_p$-vector spaces and one has for all $i\geq 0$ a perfect pairing
$$
H^i(X_\proet, \Q_p) \otimes_{\Q_p} H^{2d-i}(X_\proet, \Q_p) \to \Q_p(-d).
$$
Indeed, in this case, 
$$
R\Gamma(X_\proet, \Q_p) = \left(\varprojlim_n R\Gamma(X_\proet,\Z/p^n) \right)[1/p]
$$
and so the statement follows from the results of Gabber-Zavyalov and Mann (\cite{zavyalov2025mod}, \cite{mann-mod-p-6-functors}). The same arguments apply when $\Q_p$ is replaced by any pro-\'etale $\Q_p$-local system $\mathbb{L}$ of the form $\mathbb{L}=\mathbb{L}_0 \otimes_{\Z_p} \Q_p$, with $\mathbb{L}_0$ a finite rank pro-\'etale $\Z_p$-local system.

However, these nice properties break down completely when one removes the assumption that $X$ is proper or the assumption that the pro-\'etale $\Q_p$-local system $\mathbb{L}$ admits a stable lattice, as illustrated by the following two examples.

\begin{example} \label{sec:pro-etale-cohomology-example-introduction-affine-line}
Let $X=\mathbb{A}_C^1$ be the analytic affine line over $\mathrm{Spa}(C)$. The cohomology groups $H^i(X_\proet,\Q_p)$ can be computed explicitly (\cite[Th\'eor\`eme 1.5]{le2018espaces}, \cite[Theorem 1]{colmez2020cohomology}). They are zero in degrees $i\geq 2$, and
$$
    H^0(X_\proet,\Q_p)=\Q_p, \qquad H^1(X_\proet,\Q_p)= \mathcal{O}(X)/C(-1) = \Omega^1_X(X)(-1).
$$
Since $X$ is Stein, one can also make a natural guess for what compactly supported pro-\'etale $\Q_p$-cohomology is (it will be explained later in this paper why this seemingly ad hoc defintion is the correct one, cf. \Cref{prop:description-compactly-supp-cohomology-via-excision} and \Cref{rslt:pushforward-commutes-with-sigma-nuclear-Qp-case} below). Set:
$$
    R\Gamma_c(X_\proet,\Q_p) := \colim_U \fib(R\Gamma(X_\proet,\Q_p) \to R\Gamma((X\backslash \overline{U})_\proet,\Q_p)),
$$
where the colimit runs over quasi-compact open subspaces $U$ of $X$. 
Moreover, one can again compute the cohomology groups $H_c^i(X_\proet,\Q_p)$ in this case (see \cite[Example A.2]{colmez2021p} and note that the $\Q_p$-case follows from the $\Z_p$-case, cf. beginning of \cite[\S A.2.1]{colmez2021p}). They vanish for $i\neq 2$ and
$$
    H_c^2(X_\proet,\Q_p) = \Q_p(-1) \oplus \mathcal{O}_{\mathbb{P}_C^1,\infty}/C(-1).
$$ 
This example shows that pro-\'etale $\Q_p$-cohomology is much bigger than pro-\'etale $\Z_p$-cohomology (which vanishes in positive degrees for the affine line). The shape of the cohomology groups with and without support also does not seem compatible with the existence of a perfect duality pairing between them (in one case, we see cohomology in two degrees, but only in one degree in the other case).
\end{example}

\begin{example}
  \label{sec:pro-etale-cohomology-example-introduction-lubin-tate}
Let $X=\mathbb{P}_C^1$ be the projective line over $\mathrm{Spa}(C)$. In this example (that we learnt from David Hansen) we will, to avoid confusion, ignore Tate twists and reserve the notation $(-)$ for twists by the usual ample line bundle on $X$. The Gross-Hopkins period map from the infinite-level Lubin-Tate space of a height $2$ and dimension $1$ connected $p$-divisible group $H$ over $\overline{\mathbb{F}}_p$ is a pro-\'etale $\mathrm{GL}_2(\Q_p)$-torsor over $X$, along which we can descend the rank $2$ trivial local system, equipped with the standard action of $\mathrm{GL}_2(\Q_p)$, to a rank $2$ pro-\'etale $\Q_p$-local system $\mathbb{L}$ on $X$. We sketch the computation of its cohomology. The local system $\mathbb{L}$ lives in a short exact sequence of abelian sheaves on $X_\proet$:
$$
0 \to \mathbb{L} \to \mathbb{B}^{\varphi^2=p} \to \widehat{\mathcal{O}}_X(1) \to 0,
$$
where $\mathbb{B}$ denotes the period sheaf from \cite[D\'efinition 8.2]{le2018espaces} and $\widehat{\mathcal{O}}_X$ the completed structure sheaf on $X_\proet$ (the existence of this exact sequence is essentially a rephrasement of the \'etale-crystalline comparison between the Tate module and the Dieudonn\'e crystal of the universal $p$-divisible group). The cohomology groups of the middle and right terms can be explicitly computed. For the middle term, an application of Scholze's primitive comparison theorem (\cite[Theorem 1.3]{rigid-p-adic-hodge}) implies that
$$
R\Gamma(X_\proet, \mathbb{B}^{\varphi^2=p}) = R\Gamma(X_\proet,\Q_p) \otimes_{\Q_p} B^{\varphi^2=p}
$$
while for the right term one can use Scholze's computation $R^i\nu_\ast \widehat{\mathcal{O}}_X = \Omega_{X/C}^i$, $i\geq 0$. Moreover, $\mathbb{L}$ has non non-zero global sections, since by \cite[Proposition 7.4]{de1995etale}, its monodromy group is $\mathrm{SL}_2(\Q_p)$. One deduces from these considerations that $H^i(X_\proet,\mathbb{L})$ vanishes in degrees $i\neq 1,2$ and that
$$
H^1(X_\proet,\mathbb{L}) \cong C^2/B^{\varphi^2=p} ~~; \quad H^2(X_\proet,\mathbb{L}) \cong B^{\varphi^2=p}.
$$
This example shows that the cohomology of a finite rank pro-\'etale $\Q_p$-local system can be highly infinite-dimensional over $\Q_p$, even for a smooth proper rigid variety. It also illustrates, like the previous example was already, the fact that (compactly supported) pro-\'etale cohomology is not invariant by change of the complete algebraically closed base field $C$, hence does not satisfy base change.
\end{example}

However, contemplating these at first confusing examples reveals some interesting features. While both of them deal with the cohomology of pro-\'etale local systems, the computations show that the cohomology groups appearing have some ``coherent'' flavor. This in particular leads to contributions which are most naturally seen as $C$-vector spaces rather than $\Q_p$-vector spaces. Thinking of these cohomology groups as a mixture of $\Q_p$-vector spaces and $C$-vector spaces restores some patterns. In the first example, one observes that the two copies of $\Q_p$ appearing match with each other, while the two other pieces could correspond to each other by $C$-linear duality, resembling Grothendieck-Serre duality. In the second example, some finiteness does hold, as long as one allows finite-dimensional $C$-vector spaces in addition to $\Q_p$-vector spaces: indeed, $B^{\varphi^2=p}$ is an extension of $C$ by $\Q_p^2$.

Such combinations of finite-dimensional $\Q_p$-vector spaces and $C$-vector spaces (informally speaking) are not unfamiliar in $p$-adic Hodge theory: they are the $C$-points of Banach-Colmez spaces. What makes the latter ubiquitous is their close relation (made precise in \cite{le2018espaces}) to the fundamental geometric object of $p$-adic Hodge theory, the Fargues--Fontaine curve: Banach-Colmez spaces are (relative) cohomology groups of coherent sheaves on the Fargues--Fontaine curve. The sheaf $\Q_p$ is the Banach-Colmez space associated to the structure sheaf on the Fargues--Fontaine curve, while the sheaf $\widehat{\mathcal{O}}$ is the Banach-Colmez space of the skyscraper sheaf at the closed point of the curve corresponding to the untilt $C$ of $C^\flat$. From these two, one cooks up other Banach-Colmez spaces and the computation from the second example, more specifically the exact sequence used in it, comes from a simple short exact sequence of coherent sheaves on the (relative) Fargues--Fontaine curve. The same can be said of the first example, where the cohomology can be computed by using another exact sequence:
$$
0 \to \Q_p(1) \to \mathbb{B}^{\varphi=p} \to \widehat{\mathcal{O}}_X \to 0
$$
(the ``fundamental exact sequence'' of $p$-adic Hodge theory), which also has a simple description in terms of the Fargues--Fontaine curve. 

The Fargues--Fontaine curve therefore appears in disguise in both situations. It also suggests an explanation for them. A Banach--Colmez space, infinite-dimensional as it may look, comes from a \textit{coherent} sheaf on the Fargues--Fontaine curve, i.e. a ``finite type'' object. Also, the (derived) linear dual of the structure sheaf is itself, while the (derived) linear dual of a skyscraper sheaf is itself put in cohomological degree $1$, displaying a degree shift similar to what happens in the first example\footnote{If one takes into account the Galois equivariant structure, a Tate twist by $-1$ also occurs in the latter case, explaining as well the Tate twists in the first example.}.   

Finally, the first example also shows that one needs to remember some topological structure on the cohomology groups: if one wants the $C$-linear duality mentioned above to function, one has to understand it in some suitable topological sense. 

This paper pushes all these observations to their natural end, leading to a very general framework in which these two examples do not appear as pathological. Roughly speaking, we upgrade everything  to solid quasi-coherent sheaves on relative Fargues--Fontaine curves and propose to formulate and prove all statements at this level. The previous examples hopefully provide justification for this change of perspective. It is also sensible from a more abstract point of view: according to Scholze's vision on global shtukas and cohomology (\cite{scholze_rio}), $p$-adic cohomology of a space $X$ should be viewed as quasi-coherent cohomology on a space ``$X\times \Spf(\Z_p)$'', where the product is taken over a (heuristic) deeper base. If $X=\Spa(R,R^+)$ is a perfectoid space of characteristic $p$, then a convincing candidate of the hypothetical space ``$X\times \Spf(\Z_p)$'' is the space
\[
    \mathcal{Y}_{[0,\infty),X} := \Spa(W(R^+))\setminus V([\pi]),
\]
where $W(-)$ denotes the $p$-typical Witt vectors and $\pi\in R$ a pseudo-uniformizer. Moreover, solid mathematics is the natural language to deal with quasi-coherent sheaves on analytic spaces and their six operations.
  
\begin{remark}
The phenomena discussed above are reminiscent of what happens for flat cohomology of $\mu_p$ on schemes over a perfect field $k$ of characteristic $p$, where one sees both $\mathbb{F}_p$-vector spaces and $k$-vector spaces in the cohomology. Milne \cite{milne1976duality} has shown (for surfaces) that Poincar\'e duality works well in this setup, if one upgrades this cohomology to a functor valued in the derived category of perfect unipotent group schemes rather than mere $\mathbb{F}_p$-vector spaces. This perspective has been revisited and generalized to syntomic cohomology of characteristic $p$ schemes or $p$-adic formal schemes in the recent work of Bhatt-Lurie \cite{bhatt-prismatic-F-gauges}, where they phrase everything as coherent duality on a certain formal stack $X^{\mathrm{syn}}$, called the syntomification of $X$, attached to the (formal) scheme $X$. Our approach is very parallel, replacing the syntomification by the relative Fargues--Fontaine curve and perfect unipotent group schemes by Banach-Colmez spaces.
\end{remark}

\begin{remark}
\label{rmk-intro:motivation-p-adic-langlands}
Besides understanding in a conceptual manner the properties of pro-\'etale $\Q_p$-cohomology of rigid spaces, another important motivation for this work was the desire to find a suitable category of coefficients for a $p$-adic version of Fargues-Scholze. See \Cref{rmk-intro-D_0-infty-useful} below.
\end{remark}

\subsection{Main results}
We now explain the main results proved in this paper. Scholze has introduced in \cite{etale-cohomology-of-diamonds} the category of small v-stacks as a geometric framework in $p$-adic geometry, and most notably he has developed a powerful 6-functor formalism on them for $\ell$-adic coefficients for a prime $\ell\neq p$, which is based on \'etale cohomology. (Recall that there is a functor from analytic adic spaces over $\Z_p$, in particular rigid spaces over a $p$-adic field, to small v-stacks (even locally spatial diamonds), preserving the \'etale site, cf. \cite[\S 15]{etale-cohomology-of-diamonds}.) We provide a $p$-adic analogue. In the following statement, we use the definition of 6-functor functor formalisms from \cite[Definition~3.2.1]{heyer-mann-6ff}.

\begin{theorem} \label{sec:introduction-1-main-theorem-introduction}
There exists a 6-functor formalism $S \mapsto \D_{[0,\infty)}(S)$ on the category $\vStacks$ of small v-stacks with the following properties:
\begin{thmenum}
    \item The implicit class of distinguished morphisms is given by the $!$-able maps in the sense of \Cref{def:shriekable-maps}.

    \item The underlying functor $\vStacks^\opp\to \Pr^L,\ S\mapsto \D_{[0,\infty)}(S)$ satisfies hypercomplete v-descent.

    \item \label{rslt:intro-6ff-value-on-nice-perfectoid} If $Y=S$ is a perfectoid space which admits a morphism of finite $\dimtrg$ to a totally disconnected perfectoid space, then $\D_{[0,\infty)}(S)\cong \D_{\hat\solid}(\mathcal{Y}_{[0,\infty),S})$ is the category of (modified in the sense of \Cref{sec:defin-d_hats--1-definition-of-d-hat-solid}) solid quasi-coherent sheaves on the stably uniform analytic adic space $\mathcal{Y}_{[0,\infty), S}$.

    \item Every étale map $j\colon U \to X$ of small v-stacks is cohomologically étale and hence $j_!$ is left adjoint to $j^*$. Every proper lpbc\footnote{See \cref{def:lpbc-maps} for the definition of lpbc maps. Every map of rigid varieties has this property.} map $f\colon Y \to X$ of small v-stacks is cohomologically proper and thus $f_! = f_*$.

    \item \label{rslt:intro-smooth-map-is-cohom-smooth} If $f\colon X'\to X$ is a smooth morphism of analytic adic spaces over $\Q_p$, then $f$\footnote{At least in this introduction, we won't distinguish in notation between a morphism of adic spaces and the associated morphism of diamonds.} is $\D_{[0,\infty)}$-smooth. If $f$ has pure dimension $d$ then its dualizing complex is given by $\omega_f = f^! 1 = 1(d)[2d]$, with $(-)(d)$ referring to a Tate twist.

    \item \label{rslt:intro-smoothness-of-Spd-Qp} The morphism $\Spd(\Q_p)/\varphi^\Z \to \Spd(\mathbb{F}_p)$ is cohomologically smooth and proper, with explicit dualizing complex.
\end{thmenum}
\end{theorem}

The 6-functor formalism $\D_{[0,\infty)}(-)$ is constructed in \cref{rslt:6ff-for-D-0-infty}, which also proves (i)--(iv). Part (v) is proved in \cref{rslt:smooth-maps-of-adic-spaces-are-cohom-smooth} and \cref{rslt:dualizing-complex-on-smooth-adic-space}. Part (vi) is proved in \cref{sec:case-texorpdfstr-smoothness-of-spd-q-p} and \cref{sec:case-texorpdfstr-smoothness-of-spd-q-p-mod-hi}.

\begin{remark}
  \label{sec:main-results-remark-on-qp-version}
  Similarly, there exists a six-functor formalism $S\mapsto \D_{(0,\infty)}(S)$ in which $\mathcal{Y}_{[0,\infty),S}$ gets replaced by $\mathcal{Y}_{(0,\infty),S}$.
  Its existence and its properties follow formally from the statements for $\D_{[0,\infty)}(-)$, cf.\ \ref{sec:6-funct-texorpdfstr-variant-without-zero}.
  Indeed, for a perfectoid space $S$, which admits a morphism of finite $\dimtrg$ to a totally disconnected perfectoid space, there exists an idempotent algebra $A_S\in \D_{[0,\infty]}(S)\cong \D_{\hat\solid}(\mathcal{Y}_{[0,\infty)),S})$ with open complement $\D_{(0,\infty)}(S)=\D_{\hat\solid}(\mathcal{Y}_{(0,\infty),S})$ by localizing on the stably uniform analytic adic space $\mathcal{Y}_{[0,\infty),S}$ to the open subspace $\mathcal{Y}_{(0,\infty),S}$, cf.\ \ref{rslt:open-immersion-induces-categorical-open-immersion-on-hatsolid} \ref{sec:6-funct-texorpdfstr-variant-without-zero}.
  As $S\mapsto A_S$ is compatible with pullbacks for morphisms $S'\to S$, the collection of $A_S$'s defines an idempotent algebra $A\in R:=\D_{[0,\infty)}(\Spd(\F_p))$.
  Let $R\to R':=R/\Mod_A(R)$ be the natural symmetric monoidal functor.
  Using Lurie's tensor product in $\Pr^L$ we get that $\D_{(0,\infty)}(-)\cong R'\otimes_R \D_{[0,\infty)}(-)$ as in \ref{sec:abstract-results-6}, and the formal results of \ref{sec:abstract-results-6} imply the desired existence and properties.
  We stress that it is \emph{not} true that the $6$-functor formalism $\D_{(0,\infty)}(-)$ is equivalent to the $6$-functor formalism $\D(\Q_p)\otimes_{\D(\Z_p)}\D_{[0,\infty)}(-)$ obtained by inverting $p$ on $\D_{[0,\infty)}(-)$, \ref{cor-exicsion-Qp-cohomology}.
\end{remark}

In words: the category of (modified) solid quasi-coherent sheaves on $\mathcal{Y}_{[0,\infty),S}$ considered on (a nice enough class of) perfectoid spaces satisfies strong descent properties, allowing to define it for any small v-stack, and comes equipped with a full $6$-functor formalism. In \cite{anschuetz_mann_descent_for_solid_on_perfectoids}, the first and third authors have established strong descent results for (modified) solid quasi-coherent sheaves on perfectoid spaces and we rely critically on these results. More precisely, rather formal arguments in \Cref{sec:abstract-results-6} and \Cref{sec:v-descent-d_ffs} reduce the assertion to the construction of a suitable $6$-functor formalism
\[
    S\mapsto \Mod_{\Z_p^\cycl}(\D_{[0,\infty)}(S)) \quad \underset{\textrm{if $S$ is nice}}{=} \D_{\hat\solid}(\mathcal{Y}_{[0,\infty),S}\times_{\Spa(\Z_p)}\Spa(\Z_p^\cycl))
\]
using \Cref{sec:v-descent-d_ffs-1-modules-under-o-e-infty}. Now, if $S$ is a perfectoid space of characteristic $p$, then $\mathcal{Y}_{[0,\infty),S}\times_{\Spa(\Z_p)}\Spa(\Z_p^\cycl)$ is a perfectoid space over $\Z_p$. As $\D_{\hat\solid}(-)$ is referring here to (a slight modification of) solid quasi-coherent sheaves, it suffices to construct a $6$-functor formalism for solid quasi-coherent sheaves on arbitrary perfectoid spaces or small v-stacks over $\mathrm{Spd}(\Z_p)$ (we call small v-stacks with a morphism to $\Spd(\Z_p)$ \textit{untilted small v-stacks}). In fact, we develop a stronger $6$-functor formalism for solid quasi-coherent $\ri^{+a}$-cohomology on untilted small v-stacks in \Cref{sec:6ff-for-Dhsa}, where $\mathcal{O}$ is the untilted structure sheaf. We note that our formalism extends the one in \cite{mann-mod-p-6-functors}, but in fact it relies on it as we commonly reduce statements for $\ri^{+a}$ to statements for $\ri^{+a}/\pi$.

\begin{remark}
As made visible by the statement of \Cref{sec:introduction-1-main-theorem-introduction}, we make heavy use of higher categories, higher algebra and condensed mathematics. This is unavoidable if one wants to formulate the results in their natural generality. We refer to the introduction of \cite{mann-mod-p-6-functors} for an introduction to this circle of ideas and the advantages and necessity to use this machinery for the kind of questions studied in this paper. 
\end{remark}

\begin{example}
  \label{sec:main-results-example-torus}

  Let $C$ be a non-archimedean algebraically closed extension of $\Q_p$, and $S:=\Spa(C,\mathcal{O}_C)$.
  For simplicity we assume that $C$ has finite transcendence degree over $\Q_p$, so that the difference of $\D_{\hat\solid}$ and $\D_\solid$ disappears in the following, cf.\, \Cref{sec:defin-d_hats--4-comparison-for-modified-modules}.
  We describe now the (compactly supported) cohomology of the torus $X=\Spa(C\langle T^{\pm 1}\rangle, \mathcal{O}_C\langle T^{\pm 1}\rangle)$, i.e., the values $\Gamma_c(X,\mathcal{O})$ and $\Gamma(X,\mathcal{O})$ of $\mathcal{O}$ for the $!$- and $\ast$-pushforward $\D_{[0,\infty)}(X)\to \D_{[0,\infty)}(S)=\D_{\solid}(\mathcal{Y}_{[0,\infty),S})$.
  Let $I\subseteq [0,\infty)$ be a compact interval and $\mathcal{Y}_{I,S}=\Spa(B_I,B^+_I)\subseteq \mathcal{Y}_{[0,\infty),S}$ the respective rational affinoid subspace.
  Let $\widetilde{X}:=\Spa(C\langle T^{\pm 1/p^\infty}\rangle ,\mathcal{O}_C\langle T^{\pm 1/p^\infty}\rangle)$ be the perfectoid $\Z_p(1)$-cover of $X$.
  After fixing $\varepsilon=(1,\zeta_p,\ldots)\in C^\flat$ with $\zeta_p\neq 1$, $q=[\varepsilon]$ and $U:=[T^\flat]$ with $T^\flat=(T,T^{1/p},\ldots)$, the $\Z_p(1)$-action on $\mathcal{Y}_{[0,\infty),\widetilde{X}^\flat}$ is determined by $(q,U)\mapsto qU$.
  Over $\mathcal{Y}_{I,S}$ we have $\mathcal{Y}_{I,\widetilde{X}^\flat}=\Spa(B_I\langle U^{\pm 1/p^\infty}\rangle, B_I^+\langle U^{\pm 1/p^\infty}\rangle)$
  and the restriction of $\Gamma(X,\mathcal{O})\in \D_{\solid}(\mathcal{Y}_{[0,\infty),S})$ to $\mathcal{Y}_{I,S}$ identifies with the complex (in cohomological degrees $0,1$)
  \[
    B_I\langle U^{\pm 1/p^\infty}\rangle \xto{\gamma-\mathrm{Id}} B_I\langle U^{\pm 1/p^\infty}\rangle,
  \]
  where $\gamma=\varepsilon\in \Z_p(1)$ is the fixed generator.
  Indeed, this uses that $\D_{[0,\infty)}(X)$ is given by objects in $\D_{\solid}(\mathcal{Y}_{[0,\infty),\widetilde{X}^\flat})$ with continuous equivariant $\Z_p(1)$-action as $X^\diamond\cong \widetilde{X}^\flat/\Z_p(1)$.
  Following the fundamental calculation in \cite[Lemma 5.5]{rigid-p-adic-hodge} (important in some form or another to all results in $p$-adic Hodge theory) this complex is quasi-isomorphic to
  \[
    \widehat{\bigoplus}_{i\in \Z}B_I\cdot U^i\xto{(q^i-1)_{i\in \Z}} \widehat{\bigoplus}_{i\in \Z} B_I\cdot U^i
  \]
  as the non-integral powers of $U$ do not contribute to cohomology.
  Let $f\colon Z:=\Spa(B_I\langle U\rangle, B^+_I\langle U\rangle)$ be the relative torus over $\mathcal{Y}_{I,S}$ with coordinate $U$ and $\Z_p(1)$-action through $(q,U)\mapsto qU$.
  We see from the above that
  \[
    \Gamma(X,\mathcal{O})_{|\mathcal{Y}_{I,S}}\cong \mathrm{fib}(f_\ast(\mathcal{O})\xto{\gamma-\mathrm{Id}} f_\ast(\mathcal{O})),
  \]
  is the fiber of $\gamma-\mathrm{Id}$ on the object $f_\ast(\mathcal{O})$ (which sits in degree $0$)
  and similarly one deduces
  \[
    \Gamma_c(X,\mathcal{O})_{|\mathcal{Y}_{I,S}}\cong \mathrm{fib}(f_!(\mathcal{O})\xto{\gamma-\mathrm{Id}} f_!(\mathcal{O})),
  \]
  is the fiber of $\gamma-\mathrm{Id}$ on the object $f_!(\mathcal{O})$ (which sits in cohomological degrees $1$ as we check below).
  However, $f\colon Z\to \mathcal{Y}_{I,S}$ is the base change of the morphism $g\colon \Spa(\Z[U^{\pm 1}],\Z[U^{\pm 1}])_\solid\to \Spa(\Z,\Z)_\solid$ of analytic stacks.
  Now, $g_!(\mathcal{O})\cong \prod\limits_{i\in \Z}\Z\cdot V^i[-1]$ (with $V^i$ dual to $U^i$).
  Base change along the morphism of analytic rings $(\Q_p,\Z_p)_\solid\to (B_I,B_I^+)_\solid$ yields
  \[
    f_!(\mathcal{O})\cong (B_I,B^+_I)_\solid \otimes_{(\Q_p,\Z_p)_\solid} (\prod\limits_{i\in \Z}\Z\cdot V^i)[1/p].
  \]
  Using that $g_!(\mathcal{O})$ is dual to $g_\ast(\mathcal{O})$, one checks concretely in this example that (ignoring Tate twists)
  \[
    \IHom_{\D_\solid(\mathcal{Y}_{[0,\infty),S})}(\Gamma_c(X,\mathcal{O}),\mathcal{O})\cong \Gamma(X,\mathcal{O})[2]
  \]
  as predicted by \ref{rslt:intro-smooth-map-is-cohom-smooth}.
\end{example}

\begin{remark}
  A similar computation as in \cref{sec:main-results-example-torus} can be used to compute the cohomology of the closed disk $\mathbb D^1_C$. Namely, as will be shown in upcoming joint work of the third author, there is a pushout of small v-stacks
  \begin{equation*}\begin{tikzcd}
    */\Z_p \arrow[r] \arrow[d] & \tilde{\mathbb D}^{1,\diamond}_C/\Z_p \arrow[d]\\
    * \arrow[r] & \mathbb D^{1,\diamond}_C
  \end{tikzcd}\end{equation*}
  where $\tilde{\mathbb D}^1_C$ denotes the perfectoid closed disk. This implies that the cohomology of $\mathbb D^{1,\diamond}_C$ is the pullback of the cohomologies of the other three spaces in the square. Now the cohomology of the stack quotient $\tilde{\mathbb D}^{1,\diamond}_C/\Z_p$ can be computed as in \cref{sec:main-results-example-torus} (without inverting the variable $U$) and the cohomology of $*/\Z_p$ is given by continuous group cohomology for the trivial $\Z_p$-action.
\end{remark}

As consequences of \Cref{sec:introduction-1-main-theorem-introduction} and formal properties of $6$-functor formalisms, we can derive the following consequences for pro-\'etale cohomology of rigid-analytic varieties.

\begin{theorem} \label{sec:introduction-1-consequences-for-rigid-analytic-varieties}
For $X \in \vStacks$, denote $\DFF(X,\Q_p)=\D_{(0,\infty)}(X/\varphi^\Z)$.
\begin{thmenum}
    \item For any small v-stack $X$, the category $\DFF(X,\Q_p)$ contains fully faithfully and compatibly with colimits and symmetric monoidal structures, the category of nuclear overconvergent $\Q_p$-sheaves on $X$, in the sense of \Cref{sec:an-overc-riem-3-v-descent-for-d-nuc} (in particular, it contains fully faithfully pro-\'etale $\Q_p$-local systems).  
    
    \item Let $f\colon X\to Y$ be a morphism of analytic adic spaces over $\Q_p$. If $f$ is smooth of dimension $d$, then
    $$
        f^! = f^\ast(d)[2d]\colon \DFF(Y,\Q_p)\to \DFF(X,\Q_p).
    $$
    
    \item If $f$ is as in the previous point and also assumed to be proper, then
    \[
        f_\ast\colon \DFF(X,\Q_p) \to \DFF(Y,\Q_p)
    \]
    sends dualizable objects to dualizable objects.
\end{thmenum}
\end{theorem}

Part (i) is proved in \cref{rslt:construction-of-RH-functors}, part (ii) is a reformulation of \cref{rslt:intro-smooth-map-is-cohom-smooth} and part (iii) is formal in any 6-functor formalism (cf. \cref{proper-pushforward-of-q-p-local-systems}).

\begin{remark}
\label{rmk-intro:Zp-coefficients}
An analogous result holds integrally, using $\DFF(-,\Z_p):=\D_{[0,\infty)}(-/\varphi^\Z)$ instead.
\end{remark}

Let $C$ be a complete algebraically closed non-archimedean extension of $\Q_p$. The first point of \Cref{sec:introduction-1-consequences-for-rigid-analytic-varieties} tells us that in particular we can compute the cohomology of a pro-\'etale $\Q_p$-local system $\mathbb{L}$ on a rigid variety $X$ over $C$ as the pushforward of the object corresponding to $\mathbb{L}$ in $\DFF(X,\Q_p)$ and point (iii) together with \cref{rslt:intro-6ff-value-on-nice-perfectoid} show that these cohomology groups are naturally the $C$-points of Banach-Colmez spaces if $f$ is proper and smooth.

Point (ii) of \Cref{sec:introduction-1-consequences-for-rigid-analytic-varieties} gives Poincaré duality for the pro-\'etale cohomology of $X$. For partially proper smooth rigid spaces over a geometric point\footnote{Although that should not be necessary, we will assume for this that the base complete algebraically closed field is the completed algebraic closure of a complete discretely valued non-archimedean extension of $\Q_p$ with perfect residue field.}, we make this rather abstract Poincar\'e duality result more explicit, at the level of the v-site, see \Cref{sec:misc}. Hence the two examples stated at the beginning of this introduction appear as a special case of our general results.

Using the cohomological smoothness and identification of the dualizing sheaf for $\Spd(\Q_p)/\varphi^\Z \to \Spd(\mathbb{F}_p)$ (by \cref{rslt:intro-smoothness-of-Spd-Qp}), one also gets ``arithmetic duality'' theorems for the pro-\'etale cohomology of smooth rigid spaces over a finite extension of $\Q_p$. 

\begin{remark}[Related work]
The fact that the cohomology groups of a (finite rank) pro-\'etale $\Q_p$-local system on a smooth proper rigid space over $C$ are Banach-Colmez spaces seems to have been known to experts in the field since \cite{kedlaya2016finiteness} (which states the wrong claim that they are even finite dimensional $\Q_p$-vector spaces), but we are not aware of a written reference.

Poincar\'e duality for the $\F_p$-cohomology of proper smooth rigid analytic varieties has been established independently by Zavyalov \cite{zavyalov2025mod}, building on previous work of Gabber, and the third author \cite{mann-mod-p-6-functors}. See also \cite{li2024relative} for another approach (in the setting of Zariski-constructible \'etale sheaves). Our approach in this paper is modeled on \cite{mann-mod-p-6-functors}, which not only proves duality in the proper smooth case, but provides a full $6$-functor formalism. 

Less has been known about duality for $\Z_p$- or $\Q_p$-coefficients. \cite{lan2023rham} proves duality for the cohomology of an \'etale $\Q_p$-local system with a stable lattice on Zariski open subspaces of proper smooth rigid varieties over a discretely valued field (these results are generalized in \cite{li2024relative}). More recently, Colmez, Gilles and Nizio{\l} have obtained and announced arithmetic and geometric duality theorems for the pro-\'etale $\Q_p$-cohomology of smooth Stein rigid spaces over $\mathbb{C}_p$ \cite{colmez2023arithmetic}, \cite{colmez2024duality}. The duality results we obtain for the pro-\'etale $\Q_p$-cohomology of these spaces at the level of the v-site in \Cref{sec:misc} are very parallel to their formulation (see \Cref{relation-colmez-gilles-niziol}).
\end{remark}

\begin{remark}
\label{rmk-intro-D_0-infty-useful}
From the applications to finiteness and duality results for pro-\'etale $\Q_p$-cohomology, one may be tempted to believe that the $6$-functor formalism $S \mapsto \DFF(S,\Q_p)$ (or its $\Z_p$-version) is the most important one. But we rather think of $S \mapsto \D_{[0,\infty)}(S)$ (from which one formally obtains $S \mapsto \D_{(0,\infty)}(S)$ and $S\mapsto \DFF(S,\Q_p)$) as the fundamental one. 
In contrast to \cite{mann-mod-p-6-functors} the $6$-functor formalism $S\mapsto \D_{[0,\infty)}(S)$ from \Cref{sec:introduction-1-main-theorem-introduction} does not require the choice of a pseudo-uniformizer. In particular, it makes sense to evaluate it on interesting small v-stacks like $\Spd(\F_p)$, $\mathrm{Div}^1$ or $\Bun_G$ for a reductive group $G$ over $\Q_p$. This is explored in \Cref{sec:examples}, where we analyze a few examples which show that this $6$-functor formalism might be useful in investigating $p$-adic analogues of Fargues' conjecture. Rather than developing the theory in general, we tried to highlight interesting analogies and differences with the $\ell$-adic case. In particular, it is a highly interesting question whether the spectral action from \cite{fargues-scholze-geometrization} (or variants of it) can also be constructed in this context. 
\end{remark}

\begin{remark}
After this paper was realeased, a different approach to the questions discussed above was outlined in \cite{anschutz2025analytic}. It is proved there (see \cite[\S 4.4]{anschutz2025analytic} for details) that the functor $A\mapsto \mathcal{D}^a(W(A^{\circ,\flat}))$, where $\mathcal{D}^a(W(A^{\circ,\flat}))$ stands for the category of \textit{light} solid almost $W(A^{\circ,\flat})$-modules, satisfies $!$-hyperdescent (and thus also $\ast$-hyperdescent) on separable totally disconnected perfectoid rings and thus extends to a $6$-functor formalism on the category of light arc-stacks (\cite[Definition 4.1.5]{anschutz2025analytic}). By base change, this gives rise to variants $\mathcal{D}_{\solid}^{a, \mathrm{arc-light}}(-)$, $\mathcal{D}_{[0,\infty)}^{\mathrm{arc-light}}(-)$ of the $6$-functor formalisms $\mathcal{D}_{\hat\solid}^a(-)$, $\mathcal{D}_{[0,\infty)}$ respectively. 

For applications to the study of pro-\'etale $\Z_p$- or $\Q_p$-cohomology of partially proper rigid spaces, the $6$-functor formalism $\mathcal{D}_{[0,\infty)}^{\mathrm{arc-light}}(-)$ could equally be used. Note that the construction of $\mathcal{D}_{[0,\infty)}^{\mathrm{arc-light}}(-)$ does not require any modification of the category of solid quasi-coherent sheaves; relatedly, the amount of descent results to be proved to construct it is much more reduced that what goes in the construction of $\mathcal{D}_{[0,\infty)}$; however it is defined on light arc-stacks rather than (small) v-stacks, i.e. it is limited to light and partially proper objects.  
\end{remark}

\paragraph{Plan of the paper.}
In \Cref{sec:recoll-d_hats-}, we recall the crucial $\D_{\hat\solid}(-)$- and $\Dhsa{(-)}$-formalisms, developed in \cite{anschuetz_mann_descent_for_solid_on_perfectoids}. Furthermore, we establish in \Cref{sec=6ff-for-ohat-cohomology} a full $6$-functor formalism for $\Dhsa{(-)}$ on v-stacks on all perfectoid spaces and briefly discuss cohomological smoothness and cohomological properness for $\Dhsa{(-)}$ in relation to the corresponding notions for $\D^a_\solid(\ri^+/\pi)$. This $6$-functor formalism for $\mathcal{O}^{+a}$-cohomology is only used in this paper as a tool to define and study the $6$-functor formalism $\D_{[0,\infty)}$, but it should be of independent interest. 

In \Cref{sec:the-6ff-formalism-D-0-infty}, we introduce the main players $\D_{[0,\infty)}(-)$ and $\D_{(0,\infty)}(-)$ (and their variants with Frobenius $\DFF(-,\Z_p)$ and $\DFF(-,\Q_p)$) and prove their hypercomplete v-descent. We prove \Cref{sec:introduction-1-main-theorem-introduction} by a combination of \Cref{sec:abstract-results-6} and \Cref{sec=6ff-for-ohat-cohomology}, and \Cref{sec:introduction-1-consequences-for-rigid-analytic-varieties} follows as a (rather) formal consequence. Regarding cohomological smoothness, in addition to smooth morphisms of analytic adic spaces over $\Q_p$ and $\Spd(\Q_p) \to \Spd(\F_p)$, mentioned in \Cref{sec:introduction-1-main-theorem-introduction}, we also discuss examples of classifying stacks of $p$-adic Lie groups and Banach-Colmez spaces.

The categories $\DFF(-,\Z_p)$ and $\DFF(-,\Q_p)$ are defined by v-descent of quasi-coherent sheaves on Fargues--Fontaine curves and may not look like what the naive guess for categories of $\Z_p$- or $\Q_p$-sheaves on small v-stacks could have been: one would at first perhaps rather be tempted to look at sheaves of $\Z_p$- or $\Q_p$-modules on the (small) pro-\'etale or (big) v-site. There is no good $6$-functor formalism for these naive guesses. Nevertheless, the categories we define are related to them. Roughly speaking, the relation comes from the fact that the relative Fargues--Fontaine curve for a perfectoid space lives over $S$ (as a diamond), the pushforward of the structure sheaf being the sheaf $\Q_p$, and that this gives pull and push functors in both directions. The actual constructions are more involved and we discuss them in more detail in \Cref{sec:an-overc-riem}.

Finally, in \Cref{sec:applications} we establish our desired applications, which come in two guises. First, we make more concrete, using the results of \Cref{sec:an-overc-riem}, the abstract finiteness and duality results coming out of the $6$-functor formalism from \Cref{sec:the-6ff-formalism-D-0-infty}, in the particular case where we consider a smooth partially proper (e.g., Stein) rigid space over a geometric point. Second, we discuss examples coming from Fargues-Scholze's geometrization of the Langlands correspondence. We do not try to set up a $p$-adic version of their work, but rather to highlight through simple examples new interesting features (and bugs) of our formalism for this kind of questions. Instead of reproducing our findings in the introduction, we invite the reader to consult \Cref{sec:examples}.

\paragraph{Acknowledgements.}
\label{sec:acknowledgements}

We thank Ko Aoki, Dustin Clausen, David Hansen, Wies{\l}awa Nizio{\l}, Juan Esteban Rodr\'iguez Camargo, Peter Scholze and Felix Zillinger for discussions during the preparation of this paper. We thank David Hansen, Wies{\l}awa Nizio{\l} and Felix Zillinger for their comments on a preliminary draft. We especially thank Wies{\l}awa Nizio{\l} for pointing out a gap in an earlier version of the paper. The results of this paper were first discussed among the authors after an interesting conference in Oberwolfach on non-archimedean geometry in Oberwolfach in early 2022, and the authors thank its organizers. This paper is part of the DFG-Sachbeihilfe 534205068. The first author thanks the IH\'{E}S for its hospitality during a stay where parts of this paper were written. Moreover, the third author was partially supported by ERC Consolidator Grant 770936:NewtonStrat.

The authors are very grateful to the anonymous referee for all the detailed feedback on the paper.

\paragraph{Notations and conventions.} \label{sec:notat-conv}

We will use the following notation.

\begin{itemize}
    \item $\omega_1$ denotes the first uncountable cardinal.
    
    \item If $X$ and $Y$ are topological spaces, $C(X,Y)$ denotes the set of continuous maps from $X$ to $Y$.
    
    \item $\Perf_{\F_p}$ is the category of perfectoid spaces over $\F_p$.
    
    \item Small v-stacks on $\Perf_{\F_p}$ together with a map to $\Spd(\Z_p)$ are usually identified with small v-stacks on $\mathrm{Perfd}$. To emphasize this, we call the latter ``untilted small v-stacks''. Similarly, we identify diamonds over $\Spd(\Z_p)$ with suitable v-sheaves on $\Perfd$ and if so call the latter ``untilted diamonds''. This terminology follows \cite[Definition 3.2.1]{mann-mod-p-6-functors} and \cite[Definition 4.13]{anschuetz_mann_descent_for_solid_on_perfectoids}. We denote by $\vStacks^\sharp$ the category of untilted small v-stacks. We equip untilted small v-stacks with the untilted structure sheaf, which sends an arbitrary perfectoid space $T$ over $\Z_p$ to $\mathcal{O}(T)$.

    \item We work fully derived, hence each functor $f^\ast, f_\ast, \IHom,\otimes,...$ is always assumed to be the derived one (unless mentioned otherwise). In the same spirit, in the context of $\D(\Z)$-linear categories, $\Hom$ will usually denote the $\D(\Z)$-enriched version (i.e. what is classically denoted ``$\operatorname{RHom}$'').
    
    \item We work completely in the $\infty$-categorical framework. In particular we refer to $\infty$-categories simply as ``categories'' and so on.

    \item We equip topological rings with their natural condensed structure, similarly for sheaves. For example $\Z_p$ denotes the condensed ring sending a profinite set $S$ to the ring of continuous maps $S \to \Z_p$. In the context of v-stacks, $\Z_p$ also denotes the v-sheaf sending a perfectoid space $S$ to continuous maps $\abs S \to \Z_p$ (this is denoted by $\hat{\Z}_p$ in \cite[Definition~8.1]{rigid-p-adic-hodge}).
\end{itemize}
For technical convenience we fix an implicit cut-off cardinal $\kappa$ (in the sense of \cite[Section 4]{etale-cohomology-of-diamonds}), and assume all our perfectoid spaces, and condensed sets to be $\kappa$-small. In particular, for a Huber pair $(A,A^+)$ its associated category $\D_\solid(A,A^+)$ (\cite[Theorem 3.28]{andreychev-condensed-huber-pairs}) is generated by a \textit{set} of compact objects.
Passing to the filtered colimit over all $\kappa$'s implies the statements in general (more precisely, one needs to see that the right adjoints in the $6$-functor formalisms that we construct stabilize for large cardinals, this can be done using arguments as in \cite[Lemma 3.2.5]{camargo2024analytic}, \cite[Lemma 3.2.19, Corollary 3.2.21, Theorem 3.6.12]{mann-mod-p-6-functors}).

\section{Supplements on \texorpdfstring{$\D_{\hat\solid}(-)$}{D\_hatsolid(-)}}
\label{sec:recoll-d_hats-}

In this section we collect necessary results from \cite{anschuetz_mann_descent_for_solid_on_perfectoids} on the category of (modified) solid quasi-coherent modules. Moreover, we supplement the theory by establishing several results which are necessary to develop the six functor formalism $S\mapsto \D_{[0,\infty)}(S)$ for solid quasi-coherent sheaves on the Fargues--Fontaine curve.
\subsection{Definition of \texorpdfstring{$\D_{\hat\solid}(-)$}{D\_hatsolid(-)} for adic rings}
\label{sec:defin-d_hats-}

Let $(A,A^+)_\solid$ be an adic analytic ring in the sense of \cite[Definition 2.1]{anschuetz_mann_descent_for_solid_on_perfectoids}, i.e., $A$ is a condensed, animated ring $A$, which is $I$-complete for some finitely generated ideal $I\subseteq \pi_0(A)(\ast)$ and the derived quotient\footnote{More precisely, when writing $A/I$ in this paper we choose implicitly a finite set $f_1,\ldots, f_r$ of generators of $I$ and set $A/I$ as the derived quotient $A/(f_1,\ldots, f_r)$, which does depend on this choice. However, the statements concerning $A/I$ in this paper do not depend on this choice, e.g., that $A/I$ is discrete.} $A/I$ is discrete, while the analytic ring structure is obtained by requiring that the elements in a subring $A^+\subseteq \pi_0(A)(\ast)$ are solid.

We note that $A^+$ is not required in this section to be integrally closed in $\pi_0(A)(\ast)$ or that it contains the topologically nilpotent elements (for the $I$-adic topology). Consequently, $(A,A^+)_\solid\cong (A,\widetilde{A^+})_\solid$ where $\widetilde{A^+}\subseteq \pi_0(A)(\ast)$ is the integral closure of $A^++A^{\circ\circ}$, where $A^{\circ\circ}\subseteq \pi_0(A)(\ast)$ denotes the ideal of topologically nilpotent elements.

In \cite[Definition 2.4]{anschuetz_mann_descent_for_solid_on_perfectoids} we introduced ``modified solid quasi-coherent sheaves'' for $(A,A^+)_\solid$, which recall now. In the following let $\D_\solid(A,A^+)$ denote the category of $(A,A^+)_\solid$-complete (derived) $A$-modules.

\begin{definition}[{\cite[Definition 2.4]{anschuetz_mann_descent_for_solid_on_perfectoids}}]
  \label{sec:defin-d_hats--3-definition-of-modified-solid-quasi-coherent-sheaves}
  We set $\D_{\hat\solid}(A,A^+):=\mathrm{Ind}(\mathcal{C})$ for $\mathcal{C}\subseteq \D_\solid(A,A^+)$ the full stable subcategory spanned by the $I$-adic completions $P^\wedge_I$ for $P\in \D_\solid(A,A^+)$ compact.
\end{definition}

The natural completion functor $\D_\solid(A,A^+)^\omega\to \mathcal{C}$, $P \mapsto P^\wedge_I$ on compact objects extends to a symmetric monoidal functor $\alpha^\ast\colon \D_{\solid}(A,A^+)\to \D_{\hat\solid}(A,A^+)$, which admits a colimit-preserving, conservative right adjoint $\alpha_\ast$ (\cite[Lemma 2.5]{anschuetz_mann_descent_for_solid_on_perfectoids}). By the same lemma, there exists a natural $t$-structure on $\D_{\hat\solid}(A,A^+)$ for which $\alpha_\ast$ is $t$-exact and commutes with truncations.

We will also need the important subcategory of nuclear objects in the symmetric monoidal categories $\D_{\hat\solid}(A,A^+)$ and $\D_{\solid}(A,A^+)$ in the sense of \cite[Lecture 8]{condensed-complex-geometry}.

\begin{proposition} \label{sec:defin-d_hats--4-comparison-for-modified-modules}
Let $(A,A^+)_\solid\to (B,B^+)_\solid$ be a morphism of adic analytic rings.
\begin{propenum}
    \item If $A$ is discrete or $A^+$ is a finitely generated $\Z$-algebra, then the functor $\alpha^\ast\colon \D_{\solid}(A,A^+)\to \D_{\hat\solid}(A,A^+)$ is an equivalence.

    \item \label{rslt:hat-solid-Nuc-same-as-solid-Nuc} The functors $\alpha^\ast,\alpha_\ast$ induce a symmetric monoidal equivalence
    \begin{align*}
      \D^\nuc_\solid(A,A^+) = \D^\nuc_{\hat\solid}(A,A^+)
    \end{align*}
    on nuclear objects. Moreover, $M\in \D_{\hat\solid}(A,A^+)$ is nuclear if $\alpha_\ast M$ is nuclear.

    \item \label{rslt:hat-solid-relative-tensor-in-PrL} If $B^+$ is a finitely generated $A^+$-algebra, then the natural functor
    \begin{align*}
      \D_\solid(B,B^+) \tensor_{\D_{\solid}(A,A^+)}\D_{\hat\solid}(A,A^+) \isoto \D_{\hat\solid}(B,B^+)
    \end{align*}
    is an equivalence.
\end{propenum}
\end{proposition}
\begin{proof}
  These statements are \cite[Lemma 2.9]{anschuetz_mann_descent_for_solid_on_perfectoids}, \cite[Proposition 2.17]{anschuetz_mann_descent_for_solid_on_perfectoids} and \cite[Proposition 2.21]{anschuetz_mann_descent_for_solid_on_perfectoids} except for the second assertion in (ii).
  Assume that $\alpha_\ast M$ is nuclear.
  It suffices to see that the natural morphism $\varphi\colon \alpha^\ast \alpha_\ast(M)\to M$ is an isomorphism (because $\alpha^\ast\alpha_\ast(M)$ is nuclear).
  As $\alpha_\ast$ is conservative, this can be checked for $\alpha_{\ast}(\varphi)\colon \alpha_\ast \alpha^\ast\alpha_\ast(M)\to \alpha_\ast(M)$ instead.
  By the triangle equality, and the fact that $\alpha_\ast\alpha^{\ast}\alpha_{\ast}(M)\cong \alpha_\ast(M)$ (this holds by (ii) and nuclearity of $\alpha_\ast(M)$) one sees that $\alpha_\ast(\varphi)$ is an isomorphism as desired. 
  \end{proof}

A serious advantage of $\D_{\hat\solid}$ over $\D_\solid$ is the adic descent theorem \cite[Theorem 2.30]{anschuetz_mann_descent_for_solid_on_perfectoids}, which is crucial for the strong descent results in \cite{anschuetz_mann_descent_for_solid_on_perfectoids}. Another advantage is the preservation of completeness under base change (\cite[Lemma 2.12.(iv)]{anschuetz_mann_descent_for_solid_on_perfectoids}). Although crucial for the technical backbone of this paper, we invite the reader to not take the difference between $\D_{\hat\solid}$ and $\D_\solid$ as very important for this paper.

\subsection{Definition of \texorpdfstring{$\D_{\hat\solid}(-)$}{D\_hatsolid(-)} for stably uniform analytic adic spaces} \label{sec:defin-d_hats--1}

The following definitions are taken from \cite[Definition 2.34, Definition 2.39]{anschuetz_mann_descent_for_solid_on_perfectoids}.
We note that we use the classical notion of a stably uniform analytic adic space here because in this case the $\D_{\hat\solid}$-formalism can be defined naturally (using that the ring of powerbounded elements is adic in the sense of \ref{sec:defin-d_hats-}).

\begin{definition}
  \label{sec:defin-d_hats--1-definition-of-d-hat-solid}
  Let $Z$ be a stably uniform analytic adic space.
  \begin{defenum}
  \item If $Z=\Spa(A,A^+)$ is affinoid, then we set
    \[
      \D_{\hat\solid}(Z):=\mathrm{Mod}_{A}(\D_{\hat\solid}(A^\circ,A^+)),
    \]
    where the right hand side refers to $A$-modules in the category from \Cref{sec:defin-d_hats--3-definition-of-modified-solid-quasi-coherent-sheaves}.
    
  \item In general, $\D_{\hat\solid}(Z)$ is defined as the limit in $\Cat$ of the categories $\D_{\hat\solid}(U)$ for $U\subseteq Z$ open affinoid.
    \item If $f\colon Z'\to Z$ is a morphism of stably uniform analytic adic spaces, then we denote by $f^\ast\colon \D_{\hat\solid}(Z)\to \D_{\hat\solid}(Z')$ the natural induced base change functor (\cite[Lemma 2.7]{anschuetz_mann_descent_for_solid_on_perfectoids}), and by $f_\ast\colon \D_{\hat\solid}(Z')\to \D_{\hat\solid}(Z)$ its right adjoint. 
  \end{defenum}
\end{definition}

By \cite[Theorem 2.38]{anschuetz_mann_descent_for_solid_on_perfectoids} the functor $Z\mapsto \D_{\hat\solid}(Z)$ satisfies descent for the analytic topology, and in particular \cref{sec:defin-d_hats--1-definition-of-d-hat-solid} is unambiguous.

\begin{remark} \label{rmk:nuclear-objects-on-stably-uniform-adic-space}
By \cref{rslt:hat-solid-Nuc-same-as-solid-Nuc}, $\D^\nuc_{\hat\solid}(Z) = \D^\nuc_\solid(Z)$ if $Z = \Spa(A,A^+)$ is an affinoid and stably uniform analytic adic space. By \cite[Theorem 5.42]{andreychev-condensed-huber-pairs} we can conclude that nuclear objects in $\D_{\hat\solid}(Z)$ satisfy analytic descent and hence globalize to the category
\begin{align*}
    \D^\nuc_{\hat\solid}(Z) \subset \D_{\hat\solid}(Z)
\end{align*}
on any stably uniform analytic adic space $Z$. As dualizable objects in $\D_{\hat\solid}(Z)$ are automatically nuclear, we can conclude that the category of dualizable objects in $\D_{\hat\solid}(Z)$ is equivalent to the category of perfect complexes on $Z$, i.e., perfect complexes over $A$ (\cite[Corollary 5.51.1]{andreychev-condensed-huber-pairs}).
\end{remark}

As fiber products of uniform analytic adic spaces are not necessarily stably uniform (\cite[Example 2.43]{anschuetz_mann_descent_for_solid_on_perfectoids}), we do not develop a full six functor formalism for $\D_{\hat\solid}(-)$ on them. In the following, we nevertheless discuss classes of ``proper maps'' and ``open immersions'' that will be used later.

\begin{lemma} \label{sec:defin-d_hats--1-characterization-of-proper-morphisms}
Let $f\colon Z'=\Spa(A',A^{\prime,+})\to Z=\Spa(A,A^+)$ be a morphism of stably uniform affinoid analytic adic spaces. Then the following are equivalent:
\begin{lemenum}
    \item $A^{\prime,+}$ is the completed integral closure of $A^++A^{\prime,\circ\circ}$ in $A'$,
    \item $f_\ast\colon \D_{\solid}(Z')\to \D_{\solid}(Z)$ satisfies the projection formula,
    \item $f$ is proper in the sense of \cite[Definition 18.1]{etale-cohomology-of-diamonds} (if $f$ is a morphism of adic spaces over $\Spa(\Z_p)$). 
\end{lemenum}
If these conditions are satisfied, then $f_\ast\colon \D_{\hat\solid}(Z')\to \D_{\hat\solid}(Z)$ satisfies the projection formula and for any morphism $g\colon W\to Z$ with $W$ stably uniform such that $W':=W\times_{Z}Z'$ is stably uniform with projections $f'\colon W'\to W, g'\colon W'\to Z'$, the base change morphism $g^\ast f_\ast\to f'_\ast g^{\prime,\ast}$ of functors $\D_{\hat\solid}(Z')\to \D_{\hat\solid}(W)$ is an isomorphism.
\end{lemma}
\begin{proof}
Note that by \cite[Proposition 13.16]{scholze-analytic-spaces} (i) is equivalent to the statement that the analytic ring $(A',A^{',+})_\solid$ has the induced analytic ring structure from $(A,A^+)_\solid$. With this replacement, the equivalence of (i) and (ii) is a general property of morphisms of analytic rings. We have that (i) implies (iii) using the valuative criterion \cite[Proposition 18.3]{etale-cohomology-of-diamonds}. Assume (iii). By \cite[Lemma 15.6]{etale-cohomology-of-diamonds} we have $|Z^{\prime,\diamond}|\cong |Z^\prime|$. Now (iii) implies that the natural morphism $Z^\prime\to \overline{Z'}^{/Z}$ is a homeomorphism. This implies (i) and hence finishes the proof of the equivalence of (i), (ii) and (iii).

Assume now that (i), (ii) hold. Then by \cref{rslt:hat-solid-relative-tensor-in-PrL} the natural functor $\D_{\hat\solid}(Z)\otimes_{\D_\solid(Z)}\D_\solid(Z')\to \D_{\hat\solid}(Z')$ is an equivalence, i.e., $\D_{\hat\solid}(Z') = \Mod_{A'}(\D_{\hat\solid}(Z))$. This easily implies the projection formula. Moreover, the proof of \cite[Lemma 3.2.5]{camargo2024analytic} shows that this satisfies base change for $\D_\solid$. Using \cref{rslt:hat-solid-relative-tensor-in-PrL} again, we can deduce base change for $\D_{\hat\solid}$. 
\end{proof}

\begin{lemma} \label{rslt:open-immersion-induces-categorical-open-immersion-on-hatsolid}
  Let $j\colon U\to Z$ be an open immersion of stably uniform analytic adic spaces.
  \begin{enumerate}
  \item The natural functor $\D_\solid(U)\otimes_{\D_\solid(Z)}\D_{\hat\solid}(Z)\to \D_{\hat\solid}(U)$ is an equivalence. 
  \item The functor $j^\ast$ is an open immersion in the category $\Sym$ of presentably symmetric monoidal categories (\cite[Definition 6.3]{condensed-complex-geometry}).
    In particular, $j^\ast$ admits a fully faithful left adjoint $j_!$ satisfying the projection formula.
  \item    If $f\colon W\to Z$ is a morphism with $W$ stably uniform, then for $V:=U\times_Z W$ with projections $j'\colon V\to W$, $f'\colon V\to U$ the natural morphism $j'_!f^{\prime,\ast}\to f^\ast j_!$ of functors $\D_{\hat\solid}(U)\to \D_{\hat\solid}(W)$ is an equivalence.
\item   If in addition $f$ satisfies the equivalent conditions of \Cref{sec:defin-d_hats--1-characterization-of-proper-morphisms}, then furthermore the natural morphism $j_!f^{\prime}_\ast\to f_\ast j_!^\prime$ of functors $\D_{\hat\solid}(V)\to \D_{\hat\solid}(Z)$  is an isomorphism.
  \end{enumerate}
\end{lemma}
\begin{proof}
  The first assertion is implied by \cite[Proposition 2.37]{anschuetz_mann_descent_for_solid_on_perfectoids} and analytic descent \cite[Theorem 2.38]{anschuetz_mann_descent_for_solid_on_perfectoids}.
  The second follows from \cite[Lemma 2.35]{anschuetz_mann_descent_for_solid_on_perfectoids}.
  The last assertions follow from the first by base change along $\D_\solid(Z)\to \D_{\hat\solid}(Z)$ and the proof of \cite[Lemma 3.2.5]{camargo2024analytic}. 
\end{proof}

\begin{remark}
  \label{sec:defin-d_hats--1-if-sousperfectoid-can-extend-to-etale}
  If $Z, W$ in \Cref{rslt:open-immersion-induces-categorical-open-immersion-on-hatsolid} are sousperfectoid, then the assertions 1., 3.\ and 4.\ hold true if $j\colon U\to Z$ is only assumed to be \'etale.
  In fact, the assertions are local on $Z$ and thus $j$ can be assumed to be a composition of an open immersion and of a finite \'etale map.
  The case of a finite \'etale map follows now from \Cref{sec:defin-d_hats--1-characterization-of-proper-morphisms}.
\end{remark}

\subsection{Boundedness conditions}
\label{sec:boundedness}

In this short paragraph, we recall a few definitions regarding boundedness that will be frequently used in the sequel. We recall from \Cref{sec:notat-conv} that we will usually view small v-stacks on $(\Perf_{\F_p})_{/\Spd(\Z_p)}$ as v-stacks on the site $\Perfd$ of all perfectoid spaces over $\Z_p$. We will call the latter ``untilted small v-stacks''.

\begin{remark} \label{sec:defin-d_hats--3-remark-geometric-properties-of-small-v-Stacks}
Given an untilted small v-stack $T$, we can forget the structure morphism to $\Spd(\Z_p)$ to define the tilt $T^\flat$ of $T$, which is a small v-stack on $\Perf_{\F_p}$. When referring to geometric properties of untilted small v-stacks (like properness, \'etaleness, and $p$-boundedness as in \cref{def:p-bounded-morphism}), these are understood to be properties of the tilt. We stress that in general $T$ and $T^\flat$ are equipped with different structure sheaves (e.g. the one on $T$ could be of characteristic $0$).
\end{remark}

In the next section, we will introduce a category $\Dhsa{T}$ for any untilted small v-stack $T$. Following \cite[Section 3.6]{mann-mod-p-6-functors} it would be natural to try to prove qcqs base change or the projection formula for proper, $p$-bounded morphisms in this context ($p$-bounded being defined below). We would however face a problem here: we don't have a good control of $\Dhsa{T}$ without $T$ being close to a spatial untilted diamond. Hence, it seems natural to ask morphisms to be representable in spatial diamonds. This however has the drawback that it is not known if compactifications of spatial diamonds are again spatial. A similar issue arises for $\ell$-adic étale sheaves and has been solved in that setting: Following \cite{mod-ell-stacky-6-functors} (and \cite[Proposition 5.6]{mann-nuclear-sheaves}) we will make use of so-called \emph{prespatial diamonds}. We recall that a qcqs untilted diamond $X$ is prespatial if there exists a spatial subdiamond $X_0\subseteq X$ such that $X_0(K,\ri_K)=X(K,\ri_K)$ for all perfectoid fields $K$ (\cite[Definition 3.1]{mod-ell-stacky-6-functors}), and that (among other properties) prespatial untilted diamonds are stable under fiber products (\cite[Lemma 3.4]{mod-ell-stacky-6-functors}).

Using the notion of prespatial diamonds, we introduced in \cite[Definition 3.7]{anschuetz_mann_descent_for_solid_on_perfectoids} the following class of morphisms of small v-stacks, providing a notion of ``cohomological boundedness''.

\begin{definition} \label{def:p-bounded-morphism}
A morphism $f\colon T'\to T$ of small v-stacks is called \emph{$p$-bounded} if
\begin{enumerate}[(i)]
    \item $f$ is locally separated and representable in prespatial diamonds,
    \item after pullback to any prespatial diamond, $f$ has finite $\dimtrg$,
    \item for each strictly totally disconnected perfectoid space $Z$ over $T$, there exists some $d\geq 0$ such that for each maximal point $z'$ in the prespatial diamond $T'\times_T Z$ the $p$-cohomological dimension of $z'$ is bounded by $d$.
\end{enumerate}
\end{definition}

As a warning we mention that in \cite[Definition 3.5.5.(ii)]{mann-mod-p-6-functors} a different (more complicated) notion of ``$p$-boundedness'' is introduced, which we will call $+$-boundedness in the following. The relation between both notions is clarified by the next theorem.

\begin{theorem} \label{sec:funct-da_h--2-+-boundedness-in-representable-case-is-equivalent-to-p-boundedness}
Let $f\colon T'\to T$ be a locally separated morphism of small v-stacks which is representable in prespatial diamonds. Then $f$ is $+$-bounded if and only if $f$ is $p$-bounded.
\end{theorem}
\begin{proof}
This is proven in \cite[Proposition 4.24]{anschuetz_mann_descent_for_solid_on_perfectoids}.
\end{proof}

Let us note that $p$-boundedness of a morphism is a condition which is easier to check in practice than $+$-boundedness. For example it is satisfied for morphisms of finite $\dimtrg$ between qcqs perfectoid spaces, and stable under quotients by profinite groups of finite $p$-cohomological dimension as long as the quotient is representable in prespatial diamonds. In particular, weakly perfectly finite type morphisms between perfectoid spaces are $p$-bounded, as are morphisms between untilted diamonds of rigid-analytic varieties. Therefore \Cref{sec:funct-da_h--2-+-boundedness-in-representable-case-is-equivalent-to-p-boundedness} simplifies parts of the discussion in \cite[Section 3.5]{mann-mod-p-6-functors}. 

\begin{remarks} \label{sec:funct-da_h--1-stability-of-p-boundedness}
\begin{remarksenum}
    \item By \cite[Lemma 3.13]{anschuetz_mann_descent_for_solid_on_perfectoids} $p$-bounded morphisms are stable under composition and base change.

    \item \label{rslt:qcqs-qproet-map-is-p-bounded} Every qcqs quasi-pro-\'etale morphism $f\colon T'\to T$ is $p$-bounded. Indeed, \cite[Proposition 9.6]{etale-cohomology-of-diamonds} implies that $f$ is representable in prespatial diamonds while $\dimtrg(f)=0$ and, if $T$ is strictly totally disconnected, then each maximal point of $T'$ is the adic spectrum of an algebraically closed perfectoid field.

    \item Assume that $T$ is a $p$-bounded prespatial untilted diamond, i.e., there exists some $d\geq 0$ such that for each $\mathcal{F}\in \D_\et(T,\F_p)$ which is concentrated in degree $0$ we have $\Hom_{\D_\et(T,\F_p)}(\F_p,\mathcal{F}) \in \D^{\leq d}$. If $f\colon T'\to T$ is a $p$-bounded morphism, then $T'$ is a $p$-bounded prespatial untilted diamond (\cite[Definition 3.11]{anschuetz_mann_descent_for_solid_on_perfectoids}).

    \item If $f\colon Y\to X$, $g\colon Z\to Y$ are morphisms of untilted small v-stacks, and $f$, $f\circ g$ are $p$-bounded, then $g$ is $p$-bounded. Indeed, quasi-compact injections are quasi-pro-\'etale by \cite[Lemma 7.19]{etale-cohomology-of-diamonds}, and thus $p$-bounded. Factoring $g$ over the graph of $g$ as $Z\to Z\times_{X}Y\to Y$ shows that $g$ is $p$-bounded as $Z\to Z\times_X Y$ is (and $p$-bounded morphisms are stable under composition and base change by (i)).

    \item Let $f\colon T'\to T$ be a locally separated morphism of untilted small v-stacks which is representable in prespatial diamonds and which has locally bounded dimension. Then by \Cref{sec:funct-da_h--2-+-boundedness-in-representable-case-is-equivalent-to-p-boundedness} and \cite[Lemma 3.5.10.(ii)]{mann-mod-p-6-functors} $p$-boundedness of $f$ can be checked v-locally on $T$. (One can also avoid the use of the (difficult) \Cref{sec:funct-da_h--2-+-boundedness-in-representable-case-is-equivalent-to-p-boundedness} and directly check condition (iii) in \Cref{def:p-bounded-morphism}.)
\end{remarksenum}
\end{remarks}

\section{A 6-functor formalism for \texorpdfstring{$\mathcal{O}^+$}{O\textasciicircum+}-cohomology on untilted small v-stacks} \label{sec=6ff-for-ohat-cohomology}

In this section we define the category $\D^a_{\hat\solid}(-)$ on untilted small v-stacks and establish a $6$-functor formalism for it. The discussion relies on \cite[Section 3.6]{mann-mod-p-6-functors} and \cite{heyer-mann-6ff}.

The results proved here provide a $6$-functor formalism for solid quasi-coherent $\mathcal{O}^{+a}$-cohomology (and thus also for $\mathcal{O}$-cohomology) on untilted small v-stacks. In terms of definitions, only the $\mathcal{O}$-linear version is used in the next section of the paper, which builds the theory on the Fargues--Fontaine curve. But, as already mentioned in the introduction, the proofs of many results require to work integrally, to be able to reduce modulo a pseudo-uniformizer and invoke the results of \cite{mann-mod-p-6-functors}. We develop the $\D^a_{\hat\solid}(-)$-formalism a bit further than what will be strictly necessary for the results of the next sections, since we believe it is also interesting in its own right. However, in the interest of brevity, we discuss essentially no examples; they will be discussed for the 6-functor formalism $\D_{[0,\infty)}(-)$.

\subsection{Definition of \texorpdfstring{$\D_{\hat\solid}^a(-)$}{D\textasciicircum a\_hatsolid(-)} for untilted small \texorpdfstring{$v$}{v}-stacks} \label{sec:defin-d_hats--2}

For perfectoid spaces, one can define a well-behaved almost category of modified solid $\mathcal{O}^+$-modules. From now on we assume that for a Huber pair $(A,A^+)$ the ring $A^+$ is open and integrally closed in $A$ (contrary to \Cref{sec:defin-d_hats-}). Moreover, we view $A^+$ as a condensed ring through its natural topology.

\begin{definition}
  \label{sec:defin-d_hats--3--definition-of-d-a-hat-solid-for-perfectoids}
  Let $Z=\Spa(A,A^+)$ be an affinoid perfectoid space. We set
  \[
    \D^a_{\hat\solid}(A^+):=\D^a(A^+)\otimes_{\D(A^+)} \D_{\hat\solid}(A^+,A^+),
  \]
  where $\D^a(A^+)$ is the almost category for the ideal $A^{\circ\circ}\subseteq A^+$ of topologically nilpotent elements, and $\D(A^+)$ the usual derived category of $A^+$-modules. Moreover, $\D_{\hat\solid}$ refers to \Cref{sec:defin-d_hats-}.
\end{definition}

This construction satisfies very strong descent results. We recall that the v-topology on the category $\Perfd$ of perfectoid spaces over $\Z_p$ is generated by disjoint unions and surjections of affinoid perfectoid spaces (see \cite[Definition 8.1.(iii)]{etale-cohomology-of-diamonds}).

\begin{theorem} \label{rslt:Dhsa-is-sheaf}
There exists a unique hypercomplete v-sheaf of categories
\begin{align*}
    \Perfd^\opp \to \Cat, \qquad Z \mapsto \Dhsa{Z},
\end{align*}
such that for each affinoid perfectoid space $Z=\Spa(A,A^+)$, whose tilt admits a morphism of finite $\dimtrg$ to a totally disconnected perfectoid space, we have $\Dhsa{Z} \cong \D^a_{\hat\solid}(A^+)$ compatibly with pullback.
\end{theorem}
\begin{proof}
This is the main result of \cite{anschuetz_mann_descent_for_solid_on_perfectoids}. More precisely, it is \cite[Theorem 1.1]{anschuetz_mann_descent_for_solid_on_perfectoids}.
\end{proof}

\begin{definition}
Let $f\colon T'\to T$ be a morphism of untilted small v-stacks.
\begin{defenum}
    \item \label{def:Dhsa} We set $\Dhsa{T}$ as the value of the sheaf of categories from \Cref{rslt:Dhsa-is-sheaf} on $T$. This is a stable presentably symmetric monoidal category (recall that we have fixed some cut-off cardinal $\kappa$ in \Cref{sec:notat-conv}). We denote the symmetric monoidal structure by $\tensor$.
    
    \item We set $f^\ast\colon \Dhsa{T} \to \Dhsa{T'}$ as the restriction functor for this sheaf.
    
    \item We set $f_\ast\colon \Dhsa{T'} \to \Dhsa{T}$ as the right adjoint to $f^\ast$, which exists by the adjoint functor theorem.
\end{defenum}
\end{definition}

\begin{remarks}
\begin{remarksenum}
    \item \label{rslt:Dsha-comparison-to-Mann-thesis} If $Z = \Spa(A,A^+)$ is affinoid perfectoid and $\pi \in A$ a pseudo-uniformizer, then the functor $\Spa(B,B^+) \mapsto B^+/\pi$ on affinoid perfectoid spaces over $Z$ defines a ring object $\ri^+_Z/\pi \in \Dhsa{Z}$ and by \cite[Remark 4.5]{anschuetz_mann_descent_for_solid_on_perfectoids} we have $\Mod_{\ri^+_Z/\pi}(\Dhsa{Z}) \cong \D^a_{\solid}(\ri^+_{Z}/\pi)$, where the latter is defined in \cite[Definition 3.1.3]{mann-mod-p-6-functors}. This property (together with the generation of $\Dhsa{Z}$ by complete objects if $Z$ is sufficiently nice, see \cref{rslt:Dhsa-generated-by-complete-objects} below) will allow us to reduce many questions to results in \cite{mann-mod-p-6-functors}.

    \item Let $T$ be an untilted small v-stack. Following \cite[Definition 3.3.1]{mann-mod-p-6-functors} we can define (left/right) bounded objects in $\D^a_{\hat\solid}(\mathcal{O}^+_T)$ as those which are (left/right) bounded when pulled back to $\D^a_{\hat\solid}(A^+)$ for any totally disconnected perfectoid space $Z=\Spa(A,A^+)$ over $T$. We note that thanks to \cite[Proposition 3.1.20]{mann-mod-p-6-functors}, \cite[Corollary 2.14]{anschuetz_mann_descent_for_solid_on_perfectoids}, \cite[Proposition 3.2.13.(ii)]{mann-mod-p-6-functors} the proofs of \cite[Lemma 3.3.2, Proposition 3.3.3]{mann-mod-p-6-functors} apply and show that (left/right) boundedness can be checked v-locally on $T$. We note that there does not seem to be a good t-structure on $\Dhsa{T}$ in general, as the pullback functors are not t-exact.
\end{remarksenum}
\end{remarks}

In the spirit of \cref{rslt:Dsha-comparison-to-Mann-thesis} we next introduce the notion of (adically) complete objects in $\Dhsa{T}$:

\begin{definition} \label{def:complete-object-in-Dhsa}
Let $T$ be an untilted small v-stack. An object $M\in \Dhsa{T}$ is called \textit{complete} if for any morphism $f\colon T'=\Spa(A,A^+)\to T$ with $T'$ a totally disconnected perfectoid space the pullback $f^\ast M\in \Dhsa{T'}=\D^a_{\hat\solid}(A^+)$ is $\pi$-adically complete for some (equivalently any) pseudo-uniformizer $\pi\in A$. 
\end{definition}

For properties of $\pi$-adic completeness in $\D^a_{\hat\solid}(A^+)$ we refer to \cite[Lemma 2.12]{anschuetz_mann_descent_for_solid_on_perfectoids}. We check that \Cref{def:complete-object-in-Dhsa} is unambiguous in the following sense:

\begin{lemma} \label{rslt:completeness-in-Dhsa-properties}
\begin{lemenum}
    \item \label{rslt:completeness-in-Dhsa-is-local} If $f\colon T'\to T$ is a morphism of small v-stacks and $M \in \Dhsa{T}$, then $f^\ast M$ is complete if $M$ is complete. The converse holds if $f$ is a v-cover.
    
    \item \label{rslt:completeness-in-Dhsa-given-by-adic-completeness} If $g\colon T\to Z=\Spa(A,A^+)$ is a morphism with $Z$ affinoid perfectoid and $M\in \Dhsa{T}$, then $M$ is complete if and only if for some (equivalently any) pseudo-uniformizer $\pi\in A$ the object $M$ is $\pi$-adically complete in the $\D^a_{\hat\solid}(A^+)$-linear category $\Dhsa{T}$, i.e., the inverse limit of the diagram $(\ldots \overset{\pi}{\to}M\overset{\pi}{\to} M)$ vanishes.
\end{lemenum}
\end{lemma}
\begin{proof}
The first part of (i) is clear by definition. Assume that $f$ is a v-cover and $f^\ast M$ is complete. Then we may reduce to the case that $T'=\Spa(B,B^+), T=\Spa(A,A^+)$ are totally disconnected. We note that \cite[Proposition 3.1.20]{mann-mod-p-6-functors} and \cite[Corollary 2.14]{anschuetz_mann_descent_for_solid_on_perfectoids} imply that the pullback $f^\ast$ commutes with $\pi$-adic completions for any pseudo-uniformizer $\pi\in A$ (this does not need that $f$ is a v-cover). Since $f^*$ is also conservative, we deduce that $M$ is $\pi$-adically complete if $f^\ast M$ is so.

Let us prove (ii). Let $f_\bullet\colon T_\bullet\to T$ be the v-hypercover by a disjoint union of totally disconnected perfectoid spaces. As each $f_{n,\ast}$ preserves $\pi$-adic completeness, and $M\cong \varprojlim_{\Delta} f_{n,\ast}f_n^{\ast} M$, we see that $M$ complete implies that $M$ is $\pi$-adically complete. Assume conversely that $M$ is $\pi$-adically complete. We already checked that pullback for morphisms between (disjoint unions of) totally disconnected perfectoid spaces preserves $\pi$-adic completions. This implies that if $\widehat{(-)}$ denotes $\pi$-adic completion, then $(\widehat{f_{n}^\ast M})$ defines a descent datum, i.e., a cartesian section, and that there is a natural morphism $M\to N:=\varprojlim_{\Delta} f_{n,\ast}(\widehat{f^\ast_nM})$ with $N$ $\pi$-adically complete in $\Dhsa{T}$. As $f_n^\ast N\cong \widehat{f^\ast_n M}$ by descent, it suffices to check that $M\to N$ is an isomorphism. By $\pi$-adic completeness of $M, N$ this can be checked mod $\pi$, and then after pullback along $f^\ast$, where it is clear. This finishes the proof. 
\end{proof}

\begin{corollary} \label{rslt:complete-objects-in-Dhsa-summary}
Complete objects in $\Dhsa{(-)}$ form a hypercomplete sub-v-sheaf and if the tilt of an affinoid perfectoid space $T = \Spa(A,A^+)$ admits a morphism of finite $\dimtrg$ to a totally disconnected perfectoid space, then for any pseudo-uniformizer $\pi$ on $T$ completeness in $\Dhsa{T} = \D^a_{\hat\solid}(A^+)$ agrees with $\pi$-adic completeness.
\end{corollary}
\begin{proof}
The claim about v-descent follows immediately from v-descent for $\Dhsa{(-)}$ and \cref{rslt:completeness-in-Dhsa-is-local}. The second claim is a special case of \cref{rslt:completeness-in-Dhsa-given-by-adic-completeness}.
\end{proof}

We remark that in contrast to $\pi$-adic completeness, completeness in $\Dhsa{T}$ in the sense of \cref{def:complete-object-in-Dhsa} is absolute, i.e. does not need the existence of a base. For example, it makes sense to speak about completeness in $\Dhsa{\Spd(\F_p)}$, which might be surprising at first glance. The following important result, taken from \cite{anschuetz_mann_descent_for_solid_on_perfectoids}, allows to make many reductions to the case of complete sheaves:

\begin{proposition} \label{rslt:Dhsa-generated-by-complete-objects-and-pullback-cocts}
Let $f\colon Y\to Z$ be a separated, $p$-bounded morphism of prespatial untilted diamonds. Assume that $Z = \Spa(A,A^+)$ is an affinoid perfectoid space whose tilt admits a morphism of finite $\dimtrg$ to a totally disconnected perfectoid space, and let $\pi\in A$ be a pseudo-uniformizer. Then the following hold true:
\begin{propenum}
    \item \label{rslt:Dhsa-generated-by-complete-objects} $\Dhsa{Y}$ is generated under colimits by right-bounded $\pi$-complete objects.

    \item $f_\ast\colon \Dhsa{Y} \to \Dhsa{Z}$ commutes with colimits.
\end{propenum}
\end{proposition}
\begin{proof}
This is proven in \cite[Proposition 4.33]{anschuetz_mann_descent_for_solid_on_perfectoids} (use \cref{rslt:completeness-in-Dhsa-given-by-adic-completeness} to identify completeness with $\pi$-adic completeness).
\end{proof}

With the basic functors $f^*$, $f_*$, $\tensor$ and a good notion of completeness at hand, we can now prove some basic compatibilities between all of these notions. The following result shows a very general base-change and colimit-preservation property for the pushforward (although it is slightly weaker than \cite[Proposition 3.5.15]{mann-mod-p-6-functors} as we have to assume representability in prespatial diamonds):

\begin{proposition} \label{rslt:Dsha-pushforward-colim-and-base-change}
Let $f\colon T^\prime\to T$ be a morphism of untilted small v-stacks. Assume that $f$ is $p$-bounded (and in particular qcqs).
\begin{propenum}
    \item \label{rslt:Dsha-pushforward-preserve-colim} The functor $f_\ast\colon \Dhsa{T'}\to \Dhsa{T}$ preserves colimits and right-bounded objects.

    \item \label{rslt:Dsha-pushfoward-base-change} Let 
    \[\begin{tikzcd}
        {W'} & {T'} \\
        W & T
        \arrow["{g'}", from=1-1, to=1-2]
        \arrow["g", from=2-1, to=2-2]
        \arrow["f", from=1-2, to=2-2]
        \arrow["{f'}", from=1-1, to=2-1]
    \end{tikzcd}\]
    be a cartesian diagram of untilted small v-stacks. Then the natural transformation
    \begin{align*}
        g^\ast f_\ast\to f'_\ast g^{\prime,\ast}\colon \Dhsa{T'}\to \Dhsa{W}
    \end{align*}
    is an isomorphism.
\end{propenum}
\end{proposition}
\begin{proof}
The assertion reduces by v-descent formally to the case that $T$ and $W$ are strictly totally disconnected. As $p$-bounded morphisms are assumed to be locally separated, we may assume that $T'$ is separated (using that it is qcqs). By \Cref{rslt:Dhsa-generated-by-complete-objects-and-pullback-cocts} we can conclude that all categories in question are generated by right-bounded, complete objects and that all functors in question commute with colimits and preserve completeness of right-bounded objects. Hence, we can reduce modulo a pseudo-uniformizer and  apply \cite[Proposition 3.5.15]{mann-mod-p-6-functors} to conclude (implicitly we used \Cref{sec:funct-da_h--2-+-boundedness-in-representable-case-is-equivalent-to-p-boundedness} to assure that $f$ is $+$-bounded).  
\end{proof}

We can also derive the following restricted version of the projection formula that holds for general (not necessarily proper) morphisms $f$. In that result, recall that for an affinoid perfectoid space $T = \Spa(A, A^+)$ we denote by $\Nuc(\D_{\hat\solid}(A^+))$ the category of nuclear objects in $\D_{\hat\solid}(A^+)$, which by \cref{rslt:hat-solid-Nuc-same-as-solid-Nuc} agrees with nuclear objects in $\D_\solid(A^+)$. By \cite[Lemma~2.18]{anschuetz_mann_descent_for_solid_on_perfectoids} these are exactly the $(A^+)_\solid$-modules that can be written as the colimit of bounded objects which are $\pi$-adically complete and discrete mod $\pi$ for a pseudo-uniformizer $\pi\in A$.

\begin{lemma} \label{sec:funct-da_h--2-restricted-projection-formula-for-nuclear-stuff}
Let $f\colon T'\to T$ be a $p$-bounded morphism of untilted small v-stacks. Assume that $T = \Spa(A,A^+)$ is an affinoid perfectoid space. Then for every $M \in \Nuc(\D_{\hat\solid}(A^+))$ and $N \in \Dhsa{T'}$ the natural morphism
\begin{align*}
    M^a\otimes f_\ast(N)\to f_\ast(f^\ast(M^a)\otimes N)
\end{align*}
is an isomorphism. 
\end{lemma}
\begin{proof}
Note that $f^\ast\colon \Dhsa{T}\to \Dhsa{T'}$ is (by symmetric monoidality) a $\D_{\hat\solid}(A^+)$-linear, colimit preserving functor of $\D_{\hat\solid}(A^+)$-module categories. By \cite[Lemma 3.32]{anschuetz_mann_descent_for_solid_on_perfectoids} and \cite[Lemma 2.18]{anschuetz_mann_descent_for_solid_on_perfectoids} the category $\mathrm{Nuc}(\D_{\hat\solid}(A^+))$ is a rigid symmetric monoidal category. Hence, the assertion follows from \cite[Lemma 3.33(ii)]{anschuetz_mann_descent_for_solid_on_perfectoids} because $f_\ast$ preserves colimits by \cref{rslt:Dsha-pushforward-preserve-colim}. 
\end{proof}

\subsection{Six functors for \texorpdfstring{$\D^a_{\hat\solid}(-)$}{D\textasciicircum a\_hatsolid(-)}} \label{sec:6ff-for-Dhsa}

In the following we will introduce the 6-functor formalism for $\Dhsa{(-)}$. Using the general construction techniques of 6-functor formalisms (see e.g. \cite[\S3]{heyer-mann-6ff}) we are reduced to verifying the base-change property and the projection formula for étale and proper maps. Let us start with the étale case:

\begin{lemma} \label{rslt:Dhsa-etale-lower-shriek-properties}
Let $j\colon U\to T$ be an \'etale morphism of untilted small v-stacks. Then the functor $j^\ast\colon \D^a_{\hat\solid}(\ri^+_T)\to \D^a_{\hat\solid}(\ri^+_U)$ admits a left adjoint
\[
  j_!\colon \D^a_{\hat\solid}(\ri^+_U)\to \D^a_{\hat\solid}(\ri^+_T),
\]
which satisfies the projection formula. Moreover, for every morphism $f\colon T'\to T$ of untilted small v-stacks with base change $f'\colon U':=T'\times_T U\to U$, $j'\colon U'\to T'$ the natural morphism
\[
  j_!'f^{\prime,\ast}\to f^\ast j_!
\]
of functors $\D^a_{\hat\solid}(\ri^+_{U})\to \D^a_{\hat\solid}(\ri^+_{T'})$ is an isomorphism. Finally, if $j$ is qcqs then $j_!$ preserves right-bounded complete objects.
\end{lemma}
\begin{proof}
Regarding base-change the formal arguments in the proof of \cite[Lemma 3.6.2]{mann-mod-p-6-functors} can be applied here as well and reduce the assertion to the case that $T, U, T'$ are strictly totally disconnected. In particular, $f, j$ are quasi-compact and separated. We note that if $j_!, j'_!$ exist, then necessarily base change holds. Namely, as in \cite[Lemma 3.6.2]{mann-mod-p-6-functors} it suffices to show that $j^\ast f_\ast\cong f^{\prime}_\ast j^{\prime, \ast}$ (via the natural map). Both sides commute however with colimits (as $f_\ast, f^\prime_\ast$ are just forgetful functors) and send compact objects to complete objects (using \cite[Corollary 2.14]{anschuetz_mann_descent_for_solid_on_perfectoids}). By \Cref{rslt:Dsha-comparison-to-Mann-thesis} we can therefore reduce to the base change claim in the proof of \cite[Lemma 3.6.2]{mann-mod-p-6-functors}.

Now we prove existence of $j_!$. For this it is sufficient to show that $j^\ast$ commutes with limits. We note that \cite[Corollary 2.14]{anschuetz_mann_descent_for_solid_on_perfectoids} implies that $j^\ast$ preserves completions and that inverse limits of complete objects are again complete. From \cite[Lemma 3.6.2]{mann-mod-p-6-functors} we can therefore conclude that $j^\ast$ commutes with inverse limits of complete objects. Let $M_i,i\in I$, be a diagram in $\D^a_{\hat\solid}(\ri^+_T)$, and let $M_i^\wedge$ be the completion of $M_i$ and $N_i:=\mathrm{fib}(M_i\to M_i^\wedge)$. Set $M:=\varprojlim_{i\in i}M_i$ and $N:=\varprojlim_{i\in I} N_i$. As we have argued above, $j^\ast(\varprojlim_{i\in I} M_i^\wedge)\cong \varprojlim_{i\in I}j^\ast(M_i^\wedge)$. Hence, it suffices to see that
\[
    j^\ast(N)\cong \varprojlim\limits_{i\in I} j^\ast(N_i).
\]
Note that we have a fiber sequence $N\to M\to M^{\wedge}\cong \varprojlim_{i\in I} M_i^\wedge$. Now, $N_i\in \Mod_{\mathcal{O}_T}(\D^a_{\hat\solid}(\ri^+_T)) = \D_{\hat\solid}(T)$ for each $i$ and the full subcategory $\D_{\hat\solid}(T)\subseteq \D^a_{\hat\solid}(\ri^+_T)$ of $\ri$-modules is stable under limits and colimits (because $\ri$ is idempotent). Hence, the assertion follows from \Cref{sec:defin-d_hats--1-if-sousperfectoid-can-extend-to-etale}.

To show that $j_!$ satisfies the projection formula we can first reduce to the case that $X,U$ are strictly totally disconnected perfectoid spaces (by the arguments in \cite[Lemma 3.6.4]{mann-mod-p-6-functors}). In fact, we may assume that $U=\Spa(B,B^+)\to X=\Spa(A,A^+)$ is a qcqs open immersion. Let $M\in \D^a_{\hat\solid}(B^+)$ and $N\in \D^a_{\hat\solid}(A^+)$. If $M$ (resp.\ $N$) is a $B$ (resp.\ $A$)-module, then both sides depend only on $M, N\otimes_{A^+}A$ (resp.\ $M\otimes_{B^+} B, N$) and the projection formula holds as $j^\ast\colon \mathrm{Mod}_A(\D^a_{\hat\solid}(A^+))\to \mathrm{Mod}_B(\D^a_{\hat\solid}(B^+))$ is a categorical open immersion (cf. \Cref{rslt:open-immersion-induces-categorical-open-immersion-on-hatsolid}). Hence, we may assume (by commuting colimits through both sides of the projection formula) that $M, N$ are $\pi$-torsion for some pseudo-uniformizer $\pi\in A$. In particular, both sides of the projection formula are complete in this case (as they are torsion), and hence it suffices to check the statement after $A^+/\pi\otimes_{A^+}(-)$ on both sides. Then the assertion follows from \cite[Lemma 3.6.4]{mann-mod-p-6-functors}.

For the last statement we can reduce to the case that $j\colon U = \Spa(B, B^+) \injto T = \Spa(A, A^+)$ is an open immersion of strictly totally disconnected perfectoid spaces. By \cite[Lemma 3.6.1]{mann-mod-p-6-functors} each quasi-compact open subset of a totally disconnected perfectoid space is a finite union of rational open subsets of the form $\{|f|\geq 1\}$ for some $f\in A^+$. Hence, we may assume that $U=\{|f|\geq 1\}$ is of this form. By \cite[Lemma 3.6.1]{mann-mod-p-6-functors} we can conclude that $\mathcal{O}^+_{T}/\pi(U) \cong A^+/\pi[1/f]$, and by \cite[Theorem 3.24]{etale-cohomology-of-diamonds} we see that $B^+/\pi$ is almost isomorphic to $(\mathcal{O}^+_T/\pi)(U)$. Altogether, we get that $A^+\langle 1/f\rangle\to B^+$ is an almost isomorphism.

We now claim that the pullback $\widetilde j^*\colon \D_{\hat\solid}(A^+) \to \D_{\hat\solid}(A^+\langle 1/f \rangle)$ is a categorical open immersion. Namely, for a pseudo-uniformizer $\pi\in A$ we have
\[
    (A^+\langle 1/f\rangle)_\solid\cong A^+_\solid\otimes_{\Z_p[[\pi]][T]_\solid} (\Z_p[[\pi]][T^{\pm 1}],\Z_p[[\pi]][T^{-1}])_\solid,
\]
where the morphism $\Z_p[[\pi]][T]_\solid \to A_\solid^+$ sends $T$ to $f$.
Indeed, using nuclearity of $A^+$ over $\Z_p[[\pi]][T]$ one can argue for $(A^+,\Z_p[[\pi]][T])_\solid$ instead of $A^+_\solid$ and then use that the base change $(-)\otimes_{\Z_p[[\pi]][T]_\solid} (\Z_p[[\pi]][T^{\pm 1}], \Z_p[[\pi]][T^{-1}])$ on modules commutes with limits, and in particular with $\pi$-adic completions.
In particular, $D_\solid(A^+)\to D_\solid(A^+\langle 1/f\rangle)$ is a categorical open immersion defined by the $\pi$-adically complete idempotent solid $A^+$-algebra $\Z_p[[\pi]][[T]]\otimes_{\Z_p[[\pi]][T]_\solid}A^+_\solid$. Using \cite[Lemma 2.35]{anschuetz_mann_descent_for_solid_on_perfectoids} and \cite[Proposition 2.21.(ii)]{anschuetz_mann_descent_for_solid_on_perfectoids} this implies the same for $\widetilde j^*$. In particular $\widetilde j^*$ admits a left adjoint $\widetilde j_!$ whose almostification agrees with $j_!\colon \D^a_{\hat\solid}(B^+) \to \D^a_{\hat\solid}(A^+)$. Hence $j_! B^{+a} = (\widetilde j_! A^+\langle 1/f \rangle)^a$, and since $\widetilde j_! A\langle 1/f\rangle^+$ is compact, we deduce that $j_! B^{+a}$ is (right-)bounded and complete. Thus, the projection formula and \cite[Lemma~2.12(iii)]{anschuetz_mann_descent_for_solid_on_perfectoids} imply that $j_!$ preserves right-bounded complete objects.
\end{proof}

Next we discuss proper $p$-bounded morphisms of untilted small v-stacks. We already get base-change from \cref{rslt:Dsha-pushfoward-base-change}, so it only remains to check the projection formula:

\begin{lemma} \label{rslt:Dhsa-projection-formula-for-proper-map}
Let $f\colon T'\to T$ be a proper $p$-bounded morphism of untilted small v-stacks. Then $f_\ast$ satisfies the projection formula, i.e., for $\mathcal{M}\in \D^a_{\hat\solid}(\ri^+_T)$ and $\mathcal{N}\in \D^a_{\hat\solid}(\ri^+_{T'})$ the natural morphism
\[
    \mathcal{M}\otimes f_\ast(\mathcal{N})\to f_\ast(f^\ast\mathcal{M}\otimes \mathcal{N})
\]
is an isomorphism.
\end{lemma}
\begin{proof}
By \cref{rslt:Dsha-pushfoward-base-change} we may assume that $T$ is totally disconnected. By \Cref{rslt:Dhsa-generated-by-complete-objects-and-pullback-cocts} we can conclude that $f_\ast$ commutes with colimits and that $\D^a_{\hat\solid}(\ri^+_{T'})$ is generated by complete objects. Therefore we can reduce to \cite[Lemma 3.6.5]{mann-mod-p-6-functors}.
\end{proof}

Now we are in the position to establish the 6-functor formalism for $\Dhsa{(-)}$. Following more or less \cite[Section 3.6]{mann-mod-p-6-functors} (and \cite{mann-nuclear-sheaves}) we do this in two steps: first, we handle morphisms between locally spatial diamonds, and then we pass to stacky maps between untilted small v-stacks.

A convenient class of morphisms of untilted small v-stacks which will admit $!$-functors is the following class of maps.
  
\begin{definition} \label{def:lpbc-maps}
Let $f\colon T'\to T$ be a morphism of small v-stacks. Then $f$ is called $lpbc$\footnote{=Locally P-Bounded and Compactifiable} if for all locally spatial diamonds $Y$ and maps $Y\to T$ the v-stack $Y' := T'\times_T Y$ is a locally spatial diamond and locally in the analytic topology on $Y'$ the map $Y' \to Y$ is $p$-bounded and compactifiable.
\end{definition}

\begin{remark} \label{sec:6-funct-da_h-2-qcqs-and-lpbc}
We note that qcqs lpbc morphisms are representable in \textit{spatial} diamonds. In particular not every $p$-bounded compactifiable morphism is lpbc. We furthermore note that following \Cref{sec:defin-d_hats--3-remark-geometric-properties-of-small-v-Stacks} a morphism of untilted small v-stacks is called lpbc if its tilt is.
\end{remark}

The lpbc maps from \Cref{def:lpbc-maps} are closely related to the bdcs maps from \cite[Definition 3.6.9]{mann-mod-p-6-functors} (recall that the notions of $p$-boundedness differ!). In fact, bdcs maps are lpbc (see \Cref{sec:six-funct-texorpdfst-bdcs-is-lpbc} below), and the definitions differ only for a technical reason: we required $p$-bounded maps to be qcqs and representable in prespatial diamonds, however asking the existence of an analytic cover by subspaces with a qcqs morphism to the target is an unnecessary restriction, e.g., using \cite[Corollary 11.28]{scholze2017etale} one sees that the natural map $\underline{\R}\to \underline{\R/\Z}$ is lpbc, but not bdcs. In the construction of the 6-functor formalism below, the class of lpbc maps falls out quite naturally as well. This explains why we deviate from \cite{mann-mod-p-6-functors} here.

Before we come to the construction of the 6-functor formalism, we establish some basic properties of lpbc maps. The following criteria are useful for checking that a morphism is lpbc:

\begin{lemma} \label{sec:exampl-da_h-criterion-to-check-lpbc}
Let $f\colon T'\to T$ be a morphism of untilted small v-stacks. Assume that $f$ is representable in (untilted) locally spatial diamonds, compactifiable and locally of bounded dimension (\cite[Definition~3.5.3]{mann-mod-p-6-functors}). Assume furthermore that after base change along any morphisms $Y\to T$ with $Y$ any (untilted) strictly totally disconnected perfectoid space $Y\to T$, the morphism $f$ is analytic locally $p$-bounded. Then $f$ is lpbc.
\end{lemma}
\begin{proof}
  Let $X\to T$ be a morphism from a spatial diamond $X$, and let $Y\to X$ be a quasi-pro-\'etale surjection from a strictly totally disconnected perfectoid space. Then $X':=X\times_T T'$ is a compactifiable locally spatial diamond. Let $X'=\bigcup_{i\in I} V_i$ be an open analytic cover by spatial diamonds. It suffices to see that for each $i\in I$ the morphism $V_i\to X$ is $p$-bounded. This follows from point (v) in \Cref{sec:funct-da_h--1-stability-of-p-boundedness} because the base change of $V_i$ to $Y$ is $p$-bounded.
\end{proof}

\begin{lemma} \label{sec:six-funct-texorpdfst-bdcs-is-lpbc}
  Let $f\colon T'\to T$ be a bdcs map of small v-stacks (in the sense of \cite[Definition 3.6.9]{mann-mod-p-6-functors}). Then $f$ is lpbc.
\end{lemma}
\begin{proof}
  The property of being lpbc is analytically local on the source (as follows directly from the definition). By definition of bdcs, we may therefore assume that $f$ is compactifiable (and in particular separated). Moreover, bdcs maps are representable in locally spatial diamonds, locally of bounded dimension and the class of bdcs maps is stable under base change. By \Cref{sec:exampl-da_h-criterion-to-check-lpbc}, we may therefore assume that $T$ is a strictly totally disconnected perfectoid space. Then $T'$ is a locally spatial diamond, and we may furthermore assume that $f$ is qcqs. Then \Cref{sec:funct-da_h--2-+-boundedness-in-representable-case-is-equivalent-to-p-boundedness} implies that $f$ is $p$-bounded, and hence $f$ is lpbc.
\end{proof}
  
We record the following stability properties of lpbc maps which are both useful in practice and necessary for setting up the 6-functor formalism.

\begin{lemma} \label{sec:6-funct-da_h-1-properties-of-lpbc-maps}
\begin{lemenum}
    \item lpbc maps are stable under composition and base change.
    \item The property of being lpbc is analytically local on source and target.
    \item Every étale morphism is lpbc.
    \item Let $f\colon Y\to X$, $g\colon Z\to Y$ be morphisms of untilted small v-stacks. If $f,f\circ g$ are lpbc, then $g$ is lpbc.
\end{lemenum}
\end{lemma}
\begin{proof}
Part (i) and (ii) are clear. Because quasi-compact open subdiamonds form a basis for the analytic topology of a locally spatial diamond, it suffices to show (iii) under the additional assumption that the étale morphism is quasi-compact. As étale morphisms are by definition locally separated, we may as well assume it to be separated. Now, each quasi-compact, separated \'etale morphism is $p$-bounded (\Cref{sec:funct-da_h--1-stability-of-p-boundedness}) and compactifiable (\cite[Proposition 22.3.(vi)]{etale-cohomology-of-diamonds}). For part (iv) we may reduce to the case $X,Y,Z$ are spatial diamonds. Then $g$ is $p$-bounded by \Cref{sec:funct-da_h--1-stability-of-p-boundedness} and compactifiable (\cite[Proposition 22.3.(viii)]{etale-cohomology-of-diamonds}).
\end{proof}

With the above preparations we can now construct the 6-functor formalism for $\Dhsa{(-)}$. As mentioned above, we can employ recent advances in the abstract theory of 6-functor formalisms to make the construction rather formal:

\begin{theorem} \label{rslt:6ff-for-Dhsa}
There exist a class $E$ of maps in the category $\vStacks^\sharp$ of untilted small v-stacks, and a 6-functor formalism $\Dhsa{(-)}$ on $(\vStacks^\sharp, E)$ with the following properties:
\begin{thmenum}
    \item When restricted to $\vStacks^{\sharp,\opp}$ the 6-functor formalism $\Dhsa{(-)}$ coincides with the functor $\Dhsa{(-)}\colon \vStacks^{\sharp,\opp} \to \Cat$ from \cref{rslt:Dhsa-is-sheaf}.

    \item The pair $(\vStacks^\sharp, E)$ is a geometric setup in the sense of \cite[Remark 2.1.2]{heyer-mann-6ff}, i.e. $E$ is stable under pullback and composition and if $f$ and $g$ are composable maps such that $f, f \comp g \in E$ then $g \in E$.

    \item $E$ contains all lpbc maps. Moreover, if $j\colon U \to X$ is étale then it is cohomologically étale and thus $j_!$ agrees with the functor from \cref{rslt:Dhsa-etale-lower-shriek-properties}. If $f\colon Y \to X$ is lpbc and proper then it is cohomologically proper and thus $f_! = f_*$.

    \item $E$ is $!$-local on source and target, i.e. a map $f\colon Y \to X$ lies in $E$ as soon as there is some universal $!$-cover for $Y$ or $X$ on which $f$ lies in $E$ (cf.\ \cite[Definition~3.4.6]{heyer-mann-6ff}).
    
    \item $E$ is $\ast$-local on the target in the following sense: if $f\colon Y\to X$ is a morphism of small v-stacks, such that for any spatial diamond $X'$ over $X$ which admits a morphism to a strictly totally disconnected perfectoid space the base change $Y\times_X X'\to X'$ lies in $E$, then $f$ in $E$.
\end{thmenum}
\end{theorem}
  
\begin{proof}
Let $\mathcal C$ be the category of prespatial diamonds, which admit a morphism to a strictly totally disconnected perfectoid space and let $bdc$ be the class of $p$-bounded compactifiable maps in $\mathcal C$. Let furthermore $I, P \subset bdc$ denote the classes of open immersions and proper maps, respectively. Then $I$ and $P$ forms a suitable decomposition of $bdc$ (in the sense of \cite[Definition~3.3.2]{heyer-mann-6ff}) and we can thus employ \cite[Proposition~3.3.3]{heyer-mann-6ff} to construct a 6-functor formalism $\Dhsa{(-)}$ on the geometric setup $(\mathcal C, bdc)$. Here condition (a) is satisfied by \cref{rslt:Dhsa-etale-lower-shriek-properties}, condition (b) is satisfied by \cref{rslt:Dhsa-projection-formula-for-proper-map,rslt:Dsha-pushfoward-base-change} and condition (c) is automatic by \cite[Corollary~3.3.5]{heyer-mann-6ff}.

We restrict the 6-functor formalism on $\mathcal C$ to the full subcategory spanned by the spatial diamonds. Since spatial diamonds (admitting a morphism to a strictly totally disconnected perfectoid space) form a basis of the v-site of small v-stacks, we can use \cite[Theorem~3.4.11]{heyer-mann-6ff} to extend $\Dhsa{(-)}$ to a 6-functor formalism on $(\vStacks^\sharp, E)$ satisfying (i), where $E$ is a collection of morphisms satisfying (ii) and (iv). Moreover, containment in $E$ can be checked after every pullback to a spatial diamond admitting a morphism to a strictly totally disconnected perfectoid space (thus (v) holds) and $E$ contains all compactifiable $p$-bounded maps of spatial diamonds (using \Cref{sec:funct-da_h--1-stability-of-p-boundedness} to get rid of the assumption that there exists a morphism to a strictly totally disconnected perfectoid space). Together with the fact that by construction every qcqs open immersion is cohomologically étale and hence open covers form universal $!$-covers (see \cite[Lemma~4.7.1]{heyer-mann-6ff}) we easily deduce that every lpbc map lies in $E$.

We now show that étale maps $j\colon U \to X$ are cohomologically étale, i.e. satisfy $j^! = j^*$. By \cite[Lemma~4.6.3.(ii)]{heyer-mann-6ff} étaleness is local on the target, so we may assume that $X$ is a strictly totally disconnected perfectoid space. By loc. cit. étaleness is also cohomologically étale local on the source, but locally on $U$, the map $j$ becomes a qcqs open immersion and is thus cohomologically étale by construction.

The only statement left to be proven is the claim that a proper, lpbc morphism $f\colon Y\to X$ is cohomologically proper and hence satisfies $f_! = f_\ast$. If $X$ is a spatial diamond which admits a morphism to a strictly totally disconnected perfectoid space, then the statement follows by construction of the $6$-functor formalism. To show the claim we apply now \cite[Lemma 4.6.4]{heyer-mann-6ff} twice. Assume first that $f$ is a monomorphism, i.e., that the diagonal $\Delta_f$ is an isomorphism. Then we need to see that the natural map $f_!\to f_\ast$ is an isomorphism of functors. Using \cref{rslt:Dsha-pushfoward-base-change} this can be checked after base change to a spatial diamond admitting a morphism to a strictly totally disconnected perfectoid space. But as remarked before this case is settled.
If $f$ is again a general proper, lpbc morphism, then $\Delta_f$ is a proper, lpbc monomorphism and hence cohomologically proper. Now the same argument as before proves that $f$ is cohomologically proper.
\end{proof}

We recall that a morphism $g\colon Z\to Y$ in $E$ is of $!$-descent if the natural functor $\D(Y)\to \varprojlim_n \D(Z^{\times {n+1}/Y})$ is an equivalence, where the implicit functors are given by $!$-pullback. We note that this limit is formed in $\mathrm{Pr}^R$, which is antiequivalent to $\Pr^L$ (\cite[Corollary 5.5.3.4]{lurie-higher-topos-theory}). Thus, the condition is equivalent to $\D(Y)$ being the colimit (in $\Pr^L$) of the $\D(Z^{\times {n+1} /Y})$ along $!$-pushforwards.
The condition for a morphism $f\colon Y\to X$ of untilted small v-stacks to be of universal $!$-descent may look restrictive as $!$-descent has to be checked for the pullback of $f$ to any untilted small v-stack over $X$. We will nevertheless show in \Cref{sec:6-functor-formalism} that many (non-lpbc) morphisms occuring in practice are $\D^a$-$!$-able in the sense below.

\begin{definition} \label{sec:6-funct-da_h-1-definition-p-shriekable-morphism}
We choose a class of morphisms $E$ as in \cref{rslt:6ff-for-Dhsa}. A morphism $f\colon T'\to T$ of untilted small v-stacks is called \emph{$\D^a$-$!$-able} if it lies in this class $E$.
\end{definition}

\begin{remark} \label{sec:six-funct-texorpdfst-remark-on-tameness}
The proof of \cite[Theorem 3.4.11]{heyer-mann-6ff} gives a rather concrete recipe for constructing a class $E$ satisfying the conditions in \Cref{rslt:6ff-for-Dhsa} by enlarging the class of compactifiable $p$-bounded morphisms of spatial diamonds (admitting a morphism to a strictly totally disconnected perfectoid space) to be $\ast$- and $!$-local in the sense in \Cref{rslt:6ff-for-Dhsa}. Moreover, the recipe gives a class which is tame in the sense that each morphism in $E$ (with target a spatial diamond, which admits a morphism to a strictly totally disconnected perfectoid space) is $!$-locally on the source given by a compactifiable $p$-bounded morphism of spatial diamonds.
\end{remark}

\begin{remark} \label{rmk:choice-of-E}
We note that a priori the definition of $\D^a$-$!$-able maps depends on the choice of $E$ in \cref{sec:6-funct-da_h-1-definition-p-shriekable-morphism}, but all results in this paper hold for any such choice.
\end{remark}

\subsection{Cohomological smoothness and cohomological properness}
\label{sec:cohom-smooth-cohom}

Having established the full 6-functor formalism for $\D^a_{\hat\solid}(-)$ in \Cref{rslt:6ff-for-Dhsa} we now aim to provide discussions of $f$-suave and $f$-prim objects (as introduced in \cite[Definition 1.3.5]{heyer-mann-6ff}).

We recall that \cite[Definition 3.2.2]{mann-mod-p-6-functors} has introduced the notion of a pseudo-uniformizer $\pi$ on an untilted small v-stack (and thus also on small $v$-stacks).
If $\pi$ is fixed on some base untilted small v-stack $T$, then from \Cref{rslt:6ff-for-Dhsa} one obtains the $6$-functor formalism
\[
  T'\mapsto \mathrm{Mod}_{\ri^{+a}/\pi}(\Dhsa{T'}), 
\]
which by \Cref{rslt:Dsha-comparison-to-Mann-thesis} agrees with the one in \cite[Theorem 3.6.12]{mann-mod-p-6-functors} resp.\ \cite[Proposition~2.6]{mod-p-stacky-6-functors}, up to restricting the class of $p$-fine maps to the $\D^a$-$!$-able ones. Given $M\in \Dhsa{T'}$ we will write $M/\pi\in \D^a_{\solid}(\ri^+_{T'}/\pi)$ for its reduction. 

\begin{proposition} \label{sec:f-smooth-f-complete-f-smooth-objects}
Let $f\colon T'\to T$ be a $\D^a$-$!$-able morphism of untilted small v-stacks, let $\pi$ be a pseudo-uniformizer on $T$. Let $M\in \Dhsa{T'}$ be an object. 
\begin{propenum}
    \item If $M$ is $f$-suave, then $M$ is complete and $M/\pi$ is $f$-suave in the $6$-functor formalism on $\D^a_{\solid}(\ri^+_{(-)}/\pi)$.
    \item Assume that $M$ is complete, of finite Tor-amplitude and that $M/\pi$ is bounded and $f$-suave in $\D^a_{\solid}(\ri^+_{T'}/\pi)$. Then $M$ is $f$-suave.
\end{propenum}
\end{proposition}
\begin{proof}
Assume that $M$ is $f$-suave. Then $M$ is complete as $M = \DSuave_f(\DSuave_f(M))$ is Verdier self dual by \cite[Lemma~4.4.4(i)]{heyer-mann-6ff} (here $\DSuave_f(-) = \IHom(-,f^!\ri^{+a})$ by \cite[Lemma~4.4.5]{heyer-mann-6ff}). The $f$-suaveness of $M/\pi$ is formally implied by $f$-suaveness of $M$, e.g. by \cite[Corollary~4.4.13]{heyer-mann-6ff}. This finishes the proof of (i).

To prove (ii), assume that $M$ is complete and that $M/\pi$ is bounded, discrete and $f$-suave for $\D^a_{\solid}(\ri^+_{(-)}/\pi)$. Then $\DSuave_f(M)$ is complete (as $\ri^{+a}$ is). By \cite[Lemma~4.4.5]{heyer-mann-6ff} it suffices to see that the natural morphism
\[
    \alpha\colon p_1^\ast M \tensor p_2^\ast \DSuave_f(M)\to \IHom(p_2^\ast M, p_1^! M)
\]
is an isomorphism, where $p_1, p_2\colon T'\times_T T'\to T'$ are the two projections. By completeness of $M$ the object $\IHom(p_2^\ast M, p_1^! M)$ is complete. By \cref{rslt:completeness-in-Dhsa-is-local} $p_1^\ast M$, $p_2^\ast \DSuave_f(M)$ are complete. As $M/\pi$ is bounded and $M$ of finite Tor-amplitude, we can conclude that $M$ is bounded and that $p_1^\ast M \tensor p_2^\ast \DSuave_f(M)$ is complete (using \cite[Lemma 2.12]{anschuetz_mann_descent_for_solid_on_perfectoids}). Hence, checking that $\alpha$ is an isomorphism can be done modulo $\pi$, where it follows from $f$-suaveness of $M/\pi$ and \Cref{sec:suppl-6-funct-base-change-for-6-functor-formalisms}.
\end{proof}

As a consequence of \Cref{sec:f-smooth-f-complete-f-smooth-objects} the $f$-suaveness of $\mathcal{O}^{+a}$ can be checked modulo $\pi$. We can derive the following consequence for cohomological smoothness.

\begin{lemma} \label{sec:f-smooth-f-checking-cohom-smoothness-mod-pi}
Let $f\colon T'\to T$ be a $\D^a$-$!$-able morphism of untilted small v-stacks. Assume that $\pi$ is a pseudo-uniformizer on $T$. Then $f$ is cohomologically smooth for $\D^a_{\solid}(\ri^+_{(-)}/\pi)$ if and only if $f$ is cohomologically smooth for $\Dhsa{(-)}$. 
\end{lemma}
\begin{proof}
By \Cref{sec:f-smooth-f-complete-f-smooth-objects} $\mathcal{O}^{+a}$ is $f$-suave if and only if $\mathcal{O}^{+a}/\pi$ is. Thus it suffices to check that $f^! \ri^+$ is invertible in $\Dhsa{T'}$ if and only if $f^!(\ri^{+a}/\pi)$ is invertible in $\D^a_\solid(\ri^+_{T'}/\pi)$. The ``only if'' part is clear. Assume now that $f^!\mathcal{O}^{+a}/\pi$ is invertible in $\D^a_\solid(\ri^+_{T'}/\pi)$. This holds then for any choice of $\pi$, and hence we may assume that $\pi|p$. By the existence of a $\varphi$-module structure on $f^!\mathcal{O}^{+a}/\pi$ and \cite[Theorem 3.9.23]{mann-mod-p-6-functors} we may conclude that $f^!\mathcal{O}^{+a}/\pi$ is bounded and discrete. From here we can conclude by completeness of $f^!\mathcal{O}^{+a}$ that $M:=f^!\mathcal{O}^{+a}$ is invertible. Indeed, $M \tensor \IHom(M,\ri^{+a})$ is complete by the same argument as in \Cref{sec:f-smooth-f-complete-f-smooth-objects} (note that $M$ and $\IHom(M,\ri^{+a})$ are bounded as this can be checked mod $\pi$ where $M$ is invertible) and the natural morphism $M\otimes \IHom(M,\ri^{+a})\to \ri^{+a}$ is an isomorphism mod $\pi$. This finishes the proof.
\end{proof}

\begin{remarks} 
\begin{remarksenum}
    \item It follows from \cref{sec:f-smooth-f-checking-cohom-smoothness-mod-pi} and \cite[Theorem~3.10.17]{mann-mod-p-6-functors} that for every smooth map $f\colon Y \to X$ of rigid-analytic varieties over $\Q_p$ the associated map $f^\diamond$ on diamonds is cohomologically smooth for $\Dhsa{(-)}$. If $f$ is of pure dimension $d$ then its dualizing complex is given by
    \begin{align*}
        f^{\diamond!}(\mathcal O^+_X) = \mathcal O_Y^+(d)[2d].
    \end{align*}
    Namely, since $f^{\diamond!}$ commutes with limits, this statement can be proved after reduction modulo $p^n$ for all $n$, where it holds by the same argument as in \cite[Theorem~6.2.9]{zavyalov-poincare-duality}.

    In the case of rational coefficients, i.e., after inverting $p$, one can also argue as follows: By \cref{rslt:relate-sheaves-on-FF-curve-to-sheaves-at-divisor} the claim reduces to a similar claim for the 6-functor formalism $\D_{[0,\infty)}$ constructed below, where it can proved by the same technique as in \cite[Theorem~3.10.20]{mann-mod-p-6-functors} (see \cref{rslt:dualizing-complex-on-smooth-adic-space}).

    \item \label{sec:f-smooth-f-cohomological-etaleness} Another consequence of \cref{sec:f-smooth-f-checking-cohom-smoothness-mod-pi} is that cohomological étaleness for $\Dhsa{(-)}$ is equivalent to cohomological étaleness for $\D^a_\solid(\ri^+/\pi)$ (here cohomological étaleness refers to \cite[Definition~4.6.1(a)]{heyer-mann-6ff}).
\end{remarksenum}
\end{remarks}

Unfortunately, the situation is not as good for $f$-prim objects and cohomologically proper morphisms. Namely, let $f\colon T'\to T$ be a $\D^a$-$!$-able morphism of untilted small v-stacks. Then $M \in \Dhsa{T'}$ is $f$-prim if and only if the natural morphism
\[
    f_!(M\otimes p_{2,\ast}\IHom(p_1^\ast M,\Delta_!(\ri^{+a}_{T'})))\to f_\ast(\IHom(M,M))
\]
is an isomorphism after applying $\Gamma(T,-)$ (see \cite[Lemma~4.4.6]{heyer-mann-6ff}). However, tracing completeness through such a statement is more difficult because the $!$-pushforwards $f_!$ and $\Delta_!$ are involved. In fact, preservations of completeness under $!$-pushforward seems to be quite seldom. We do have the following positive results:

\begin{lemma} \label{sec:f-smooth-f-shriek-pushforward-preserves-completeness}
Let $f\colon T'\to T$ be a $\D^a$-$!$-able morphism of untilted small v-stacks. Let $\pi$ be a pseudo-uniformizer on $T$. Let $M\in \Dhsa{T'}$ be right-bounded, complete and discrete modulo $\pi$. Assume that $f$ is $p$-bounded and compactifiable. Then $f_! M$ is complete.
\end{lemma}
\begin{proof}
We may assume that $T$ is a strictly totally disconnected perfectoid space. Factoring $f$ over its canonical compactification reduces to the case that $f$ is a qcqs open immersion (but $T$ no longer a strictly totally disconnected perfectoid space), because for any morphism of untilted small v-stacks $\ast$-pushforward over a perfectoid space preserves completeness (using \cref{rslt:completeness-in-Dhsa-given-by-adic-completeness} to identify completeness with adic completeness here). If $f$ is a qcqs open immersion then the claim follows from the last part of \cref{rslt:Dhsa-etale-lower-shriek-properties}.
\end{proof}

\begin{corollary} \label{sec:f-smooth-f-checking-cohomological-properness}
A $p$-bounded and compactifiable morphism is cohomologically proper for the $\Dhsa{(-)}$-formalism if and only if it is for the $\D^a_{\solid}(\ri^+/\pi)$-formalism. 
\end{corollary}
\begin{proof}
This follows from \Cref{sec:f-smooth-f-shriek-pushforward-preserves-completeness} and \cite[Lemma~4.4.6]{heyer-mann-6ff} by arguing as in \cref{sec:f-smooth-f-complete-f-smooth-objects}.
\end{proof}
 
It would be interesting to extend \Cref{sec:f-smooth-f-shriek-pushforward-preserves-completeness} to stacky maps in order to analyze $f$-primness on classifying stacks.

\section{The 6-functor formalism \texorpdfstring{$\D_{[0,\infty)}(-)$}{D\_[0,∞)(-)}} \label{sec:the-6ff-formalism-D-0-infty}

In this section we will define the categories $\D_{[0,\infty)}(S)$  and $\D_{(0,\infty)}(S)$, for a small v-stack $S$ over $\F_p$.  We also establish a full $6$-functor formalism (in the sense of \cite[Definition A.5.7]{mann-mod-p-6-functors}) for them. 

\subsection{Fargues--Fontaine curves}
\label{sec:farg-font-curv}

We recall the relevant Fargues--Fontaine curves for this paper.

\begin{definition} \label{sec:d_ffs-e-perfectoid-definition-of-fargues-fontaine-curves}
Let $S$ be a perfectoid space over $\F_p$.
\begin{defenum}
    \item We let $ \mathcal{Y}_{[0,\infty),S}$ be the analytic adic space over $\Spa(\Z_p)$ constructed in \cite[II.1.1]{fargues-scholze-geometrization} for $S$ (and $\Q_p$). We let $\varphi\colon \mathcal{Y}_{[0,\infty),S}\to \mathcal{Y}_{[0,\infty),S}$ be the Frobenius on $\mathcal{Y}_{[0,\infty),S}$.
    
    \item We set $ \mathcal{Y}_{(0,\infty),S}:= \mathcal{Y}_{[0,\infty),S} \times_{\Spa(\Z_p)}\Spa(\Q_p)$, and $\FF_S:=\mathcal{Y}_{(0,\infty),S}/\varphi^\Z$.
\end{defenum}
\end{definition}

More explicitly, if $S=\Spa(R,R^+)$ is an affinoid perfectoid space and $\varpi\in R$ a pseudo-uniformizer, then we have
\[
    \mathcal{Y}_{[0,\infty),S}=\Spa(W(R^+))\setminus V([\varpi])
\]
where the Witt vectors are endowed with the $(p,[\varpi])$-adic topology, the morphism $\varphi$ is induced by the Frobenius of $R^+$, and
\[
  \mathcal{Y}_{(0,\infty),S}=\Spa(W(R^+))\setminus V([\varpi]p).
\]

In the following we recall several basic properties of the relative Fargues--Fontaine curve, as developed in \cite{scholzegeometrization}.

\begin{lemma} \label{sec:d_ffs-e-perfectoid-properties-of-fargues-fontaine-curves}
Let $S$ be a perfectoid space over $\F_p$.
\begin{lemenum}
  \item The action of $\varphi$ on $\mathcal{Y}_{(0,\infty),S}$ is properly discontinuous. In particular, the morphism $\mathcal{Y}_{(0,\infty),S}\to \FF_S$ admits local sections.

  \item \label{rslt:open-immersion-on-S-induces-open-immersion-on-Y-S} If $S'\to S$ is an open immersion, then $\mathcal{Y}_{[0,\infty),S'}\to \mathcal{Y}_{[0,\infty),S}, \mathcal{Y}_{(0,\infty),S'}\to \mathcal{Y}_{(0,\infty),S}, \FF_{S'}\to \FF_S$ are open immersions.

  \item If $S'\to S$ is proper, then $\mathcal{Y}_{[0,\infty),S'}\to \mathcal{Y}_{[0,\infty),S}, \mathcal{Y}_{(0,\infty),S'}\to \mathcal{Y}_{(0,\infty),S}, \FF_{S'}\to \FF_S$ induce proper morphisms on the associated diamonds.
\end{lemenum}
\end{lemma}
\begin{proof}
  The first statement is \cite[Proposition II.1.16]{fargues-scholze-geometrization}. The second statement is implied by \cite[Proposition II.1.3]{fargues-scholze-geometrization}. The third statement follows from the formula $\mathcal{Y}_{[0,\infty),S}^\diamond \cong S\times_{\Spd(\F_p)} \Spd(\Z_p)$ (\cite[Proposition II.1.2]{fargues-scholze-geometrization}) because proper morphisms of small v-stacks are stable under base change.
\end{proof}

We recall that the first assertion of \Cref{sec:d_ffs-e-perfectoid-properties-of-fargues-fontaine-curves} is proven in the case that $S=\Spa(R,R^+)$ affinoid perfectoid by constructing a continuous morphism
\[
  \kappa_{\varpi}\colon |\mathcal{Y}_{[0,\infty),S}|\to [0,\infty)
\]
of topological spaces, which intertwines $\varphi$ with multiplication by $p$ and sends the Cartier divisor $(p-[\varpi])$ for a chosen pseudo-uniformizer $\varpi\in R$ to $1$. If $I\subseteq [0,\infty)$ is a compact interval ($I\neq [0,0]$), then the interior $\mathcal{Y}_{I,S}$ of $\kappa_{\varpi}^{-1}(I)$ is open affinoid (\cite[Proposition II.1.16]{fargues-scholze-geometrization}). More concretely, assuming $I=[0,r]$ with $r=\frac{a}{b}$, $a,b\in \mathbb{N}_{>0}$, then
\[
  \mathcal{Y}_{I,S}=\Spa(A_I,A^+_I)
\]
with $A_I=A^+_I[1/[\varpi]]$ and $A^+_I$ the $[\varpi]$-adic completion of $W(R^+)[\frac{p}{[\varpi^{1/a}]^b}]$. An important property of the (relative) Fargues--Fontaine curve is that it is close to being perfectoid. More precisely, it has the following canonical cover by a perfectoid space:

\begin{lemma} \label{sec:farg-font-curv-1-perfectoid-cover-by-perfectoid}
Let $S\in \Perf_{\F_p}$. Let $\Q_{p,\infty}$ be the completion of $\bigcup_{n\geq 0} \Q_p(p^{1/p^n})$ and $\Z_{p,\infty}$ its ring of integers. Then the spaces
\[
    \mathcal{Y}_{[0,\infty),S}\times_{\Spa(\Z_p)}\Spa(\Z_{p,\infty}), \quad \mathcal{Y}_{(0,\infty),S}\times_{\Spa(\Q_p)}\Spa(\Q_{p,\infty}), \quad \FF_S\times_{\Spa(\Q_p)}\Spa(\Q_{p,\infty})
\]
are perfectoid.
\end{lemma}
\begin{proof}
The first assertion follows from \cite[Proposition II.1.1]{fargues-scholze-geometrization} and the other assertions follow easily from that.
\end{proof}

\begin{remark} \label{sec:farg-font-curv-1-base-change-agrees-with-base-change-in-analytic-rings}
The fiber products in \Cref{sec:farg-font-curv-1-perfectoid-cover-by-perfectoid} agree with their derived version in the sense of solid mathematics, i.e. locally the occuring (derived) solid tensor products of solid rings is concentrated in degree $0$. Indeed, this follows from $p$-complete flatness of $\Z_{p,\infty}$ over $\Z_p$, the fact that the occuring derived $p$-completions are already $p$-adically separated in this case and the preservation of $p$-completeness for the solid tensor product (see \cite[Proposition 2.12.10]{mann-mod-p-6-functors}).
\end{remark}

In a similar vein, we can use \cref{sec:farg-font-curv-1-perfectoid-cover-by-perfectoid} to show the following compatibility of fiber products.

\begin{lemma} \label{sec:farg-font-curv-1-commutation-with-pullbacks}
The functor $S\mapsto \mathcal{Y}_{[0,\infty),S}$ from perfectoid spaces in characteristic $p$ to analytic adic spaces commutes with fiber products.
\end{lemma}
\begin{proof}
To simplify notation, we write $Z_S:=\mathcal{Y}_{[0,\infty),S}$ in this proof, and $(-)_\infty:=(-)\times_{\Spa(\Z_p)}\Spa(\Z_{p,\infty})$. Let $S_1\to S, S_2\to S$ be morphisms of perfectoid spaces in characteristic $p$. We need to see that the natural morphism $\Phi\colon Z_{S_1\times_{S} S_2}\to Z_{S_1}\times_{Z_S} Z_{S_2}$ is an isomorphism as analytic adic spaces. As $\Z_{p,\infty}$ is a descendable algebra in $\D_{\solid}(\Z_p)$ the base change functor $(-)_\infty$ is conservative on analytic adic spaces. Hence, it suffices to show that $\Phi_\infty$ is an isomorphism. Now, $\Phi_\infty$ is a morphism between the perfectoid spaces $(Z_{S_1}\times_{Z_S} Z_{S_2})_\infty\cong Z_{S_1,\infty}\times_{Z_{S,\infty}}Z_{S_2,\infty}$ and $(Z_{S_1\times_S S_2})_\infty$. Passing to the associated diamond $(-)^\diamond$ is a conservative functor on perfectoid spaces which commutes with fiber products, and $Z_{S_i,\infty}^\diamond\cong S_i\times \Spd(\Z_{p,\infty})$. Thus, it suffices to see that the natural morphism $(S_1\times_S S_2)\times \Spd(\Z_{p,\infty})\to (S_1\times \Spd(\Z_{p,\infty}))\times_{(S\times \Spd(\Z_{p,\infty}))}(S_2\times \Spd(\Z_{p,\infty}))$ is an isomorphism, which is clear.
\end{proof}

\subsection{v-descent for \texorpdfstring{$\D_{[0,\infty)}(-)$}{D\_[0,∞)(-)}} \label{sec:v-descent-d_ffs}

Let $S \in \Perf_{\F_p}$. By \Cref{sec:farg-font-curv-1-perfectoid-cover-by-perfectoid} the adic spaces $\mathcal{Y}_{[0,\infty),S}$, $\mathcal{Y}_{(0,\infty),S}$ and $\FF_S$ are sousperfectoid and hence stably uniform (see \cite[Proposition 6.3.4]{scholze-berkeley-lectures}). Thus, given $S\in \Perf_{\F_p}$ the categories
\[
    \D_{\hat\solid}(\mathcal{Y}_{[0,\infty),S}), \D_{\hat\solid}(\mathcal{Y}_{(0,\infty),S}), \D_{\hat\solid}(\FF_S)  
\]
are defined in \Cref{sec:defin-d_hats--1}. In the following we show that these categories satisfy v-descent. We start with the following preliminary observation. As in the previous subsection, we denote by $\Q_{p,\infty}$ the completion of $\bigcup_{n\geq 0} \Q_p(p^{1/p^n})$ and by $\Z_{p,\infty}$ its ring of integers. 

\begin{lemma} \label{sec:v-descent-d_ffs-1-modules-under-o-e-infty}
Let $S\in \Perf_{\F_p}$. Then the pushforward along $\mathcal{Y}_{[0,\infty),S} \times_{\Spa(\Z_p)} \Spa(\Z_{p,\infty})\to \mathcal{Y}_{[0,\infty),S}$ induces an equivalence
\[
   \D_{\hat\solid}(\mathcal{Y}_{[0,\infty),S} \times_{\Spa(\Z_p)} \Spa(\Z_{p,\infty})) \isoto \Mod_{\Z_{p,\infty}}(\D_{\hat\solid}(\mathcal{Y}_{[0,\infty),S})).
\]
Similarly, we have equivalences
\begin{align*}
    \D_{\hat\solid}(\mathcal{Y}_{(0,\infty),S} \times_{\Spa(\Q_p)} \Spa(\Q_{p,\infty})) &\isoto \Mod_{{\Q_{p,\infty}}}(\D_{\hat\solid}(\mathcal{Y}_{(0,\infty),S})),\\
    \D_{\hat\solid}(\FF_S \times_{\Spa(\Q_p)} \Spa(\Q_{p,\infty})) &\isoto \Mod_{{\Q_{p,\infty}}}(\D_{\hat\solid}(\FF_S)). 
\end{align*}
\end{lemma}
\begin{proof}
By \cref{rslt:open-immersion-on-S-induces-open-immersion-on-Y-S} and \cite[Theorem 2.38]{anschuetz_mann_descent_for_solid_on_perfectoids} both sides satisfy analytic descent in $S$. Hence, we may assume that $S=\Spa(R,R^+)$ is affinoid, and by a further analytic descent we can replace $\mathcal{Y}_{[0,\infty),S}$ by $\mathcal{Y}_{I,S}$ for some compact interval $I\subseteq [0,\infty)$ (for the radius function associated with some pseudo-uniformizer $\varpi\in R$). In particular, we may assume that $\mathcal{Y}_{I,S}=\Spa(A_I,A_I^+)$ is affinoid. Then $\Spa(\Z_{p,\infty})\times_{\Spa(\Z_p)}\mathcal{Y}_{I,S}$ is affinoid perfectoid (by the proof of \cite[II.1.1]{fargues-scholze-geometrization}), and 
\begin{align*}
    \Spa(\Z_{p,\infty})\times_{\Spa(\Z_p)}\mathcal{Y}_{I,S} = \Spa(A_{I,\infty},A_{I,\infty}^+)
\end{align*}
with $A_{I,\infty}\cong A_I\widehat{\otimes}_{\Z_p} \Z_{p,\infty}$ the completed tensor product (equivalently the solid tensor product) and $A_{I,\infty}^+$ the completed integral closure of $A_I^+$. By \cite[Proposition 2.21]{anschuetz_mann_descent_for_solid_on_perfectoids} we can now conclude the result. We note that the tensor product agrees here with the derived version (see \Cref{sec:farg-font-curv-1-base-change-agrees-with-base-change-in-analytic-rings}). 
\end{proof}

Using \cref{sec:v-descent-d_ffs-1-modules-under-o-e-infty} we can now deduce v-descent of the assignment $S\mapsto \D_{\hat\solid}(\mathcal{Y}_{[0,\infty),S})$ by reducing it to the v-descent for $S \mapsto \Dhsa{S}$:

\begin{theorem} \label{rslt:v-descent-for-D-0-infty}
There exists a unique hypercomplete v-sheaf $S\mapsto \D_{[0,\infty)}(S)$ of categories on $\Perf_{\F_p}$ such that for $S\in \Perf_{\F_p}$, which admits a morphism of finite $\dimtrg$ to a totally disconnected perfectoid space, we have $\D_{[0,\infty)}(S)\cong \D_{\hat\solid}(\mathcal{Y}_{[0,\infty),S})$ compatibly with pullback.
\end{theorem}
\begin{proof}
  Let $\mathcal{C}\subseteq \Perf$ be the full subcategory of perfectoid spaces admitting a morphism of finite $\dimtrg$ to a totally disconnected perfectoid space.
  We need to see that the functor $S\in \mathcal{C}\mapsto \D_{\hat\solid}(\mathcal{Y}_{[0,\infty),S})$ with values in $\D_\solid(\Z_p)$-linear presentable $\infty$-categories satisfies descent for $v$-hypercovers.
  As $\Z_{p,\infty}\in \D_\solid(\Z_p)$ is a descendable algebra (because $\Z_p\to \Z_{p,\infty}$ admits a splitting) and the functor $\D_{\solid}(\Z_{p,\infty})\otimes_{\D_{\solid}(\Z_p)}(-)$ on $\D_{\solid}(\Z_p)$-linear presentable $\infty$-categories preserves all limits, it suffices to prove the descent for the functor $S\in \mathcal{C}\mapsto \mathrm{Mod}_{\Z_{p,\infty}}(\D_{\hat\solid}(\mathcal{Y}_{[0,\infty),S}))$, which identifies with the functor $S\in \mathcal{C}\mapsto \D_{\hat\solid}(\mathcal{Y}_{[0,\infty),S]}\times_{\Spa(\Z_p)}\Spa(\Z_{p,\infty}))$ by \Cref{sec:v-descent-d_ffs-1-modules-under-o-e-infty}
  
  The results of \cite{anschuetz_mann_descent_for_solid_on_perfectoids} recalled in \Cref{sec=6ff-for-ohat-cohomology} give a hypercomplete v-sheaf $Z\mapsto \D^a_{\hat\solid}(\ri^+_Z)$ on the category $\Perfd$ of all perfectoid spaces over $\Z_p$. This formally implies that $Z\mapsto \D_{\hat\solid}(\mathcal{O}_Z) = \Mod_{\mathcal{O}_Z}(\D^a_{\hat\solid}(\ri^+_Z))$ is also a hypercomplete v-sheaf on the category $\Perfd$ of all perfectoid spaces over $\Z_p$. Composing with the functor
\[
    \Perf_{\F_p}\to \Perfd, \qquad S \mapsto \mathcal{Y}_{[0,\infty),S} \times_{\Spa(\Z_p)} \Spa(\Z_{p,\infty}),
\]
which preserves v-covers, yields a hypercomplete v-sheaf on $\Perf_{\F_p}$ sending $S$ to $$\D_{\hat\solid}(\mathcal{O}_{\mathcal{Y}_{[0,\infty),S} \times_{\Spa(\Z_p)} \Spa(\Z_{p,\infty})}).$$ If $S$ admits a morphism of finite $\dimtrg$ to a totally disconnected perfectoid space, then by \Cref{rslt:Dhsa-is-sheaf} (using analytic descent for $\D_{\hat\solid}$ and that $\Spd(\Z_{p,\infty})\times S\to S$ has finite $\dimtrg$)
\[
    \D_{\hat\solid}(\mathcal{O}_{\mathcal{Y}_{[0,\infty),S} \times_{\Spa(\Z_p)} \Spa(\Z_{p,\infty})})\cong \D_{\hat\solid}(\mathcal{Y}_{[0,\infty),S} \times_{\Spa(\Z_p)} \Spa(\Z_{p,\infty})).
  \]
  This yields the desired descent.
\end{proof}

\begin{definition} \label{def:D-0-infty}
\begin{defenum}
    \item Using \cref{rslt:v-descent-for-D-0-infty} we let 
    \begin{align*}
        \Perf_{\F_p}^\opp \to \Cat, \qquad S \mapsto \D_{[0,\infty)}(S)
    \end{align*}
    be the unique hypercomplete v-sheaf on the category of small v-stacks over $\Spd(\F_p)$ such that $\D_{[0,\infty)}(S)\cong \D_{\hat\solid}(\mathcal{Y}_{[0,\infty),S})$, compatibly with pullback, for any perfectoid space $S\in \Perf_{\F_p}$ which admits a morphism of finite $\dimtrg$ to a totally disconnected perfectoid space.

    \item If $f\colon S'\to S$ is a morphism of small v-stacks, then we denote by
    \begin{align*}
        f^\ast\colon \D_{[0,\infty)}(S) \rightleftarrows \D_{[0,\infty)}(S') \noloc f_\ast
    \end{align*}
    the restriction morphism $f^\ast$ for the v-sheaf $\D_{[0,\infty)}(-)$ and its right adjoint $f_\ast$.
\end{defenum}
\end{definition}

\begin{remark} \label{def:D-0-infty-without-0}
Replacing in \Cref{def:D-0-infty} (and \Cref{rslt:v-descent-for-D-0-infty}) the space $\mathcal{Y}_{[0,\infty),S}$ by $\mathcal{Y}_{(0,\infty),S}$ leads to the variant
\begin{align*}
    \Perf_{\F_p}^\opp \to \Cat, \qquad S \mapsto \D_{(0,\infty)}(S).
\end{align*}
\end{remark}

Simlarly to \cref{def:D-0-infty-without-0} one could also introduce a version of the sheaf categories where we replace $S$ by $\FF_S$. We will however take a slightly different approach to this category:

\begin{definition} \label{def:ff-categories}
For $S$ a small v-stack, we define
\[
    \DFF(S,\Z_p) := \D_{[0,\infty)}(S/\varphi^\Z), \qquad \DFF(S,\Q_p) := \D_{(0,\infty)}(S/\varphi^{\Z}).
\]
Here $\varphi$ denotes the canonical Frobenius on $S$.
\end{definition}

\begin{remarks}
\begin{remarksenum}
    \item If $S$ is a perfectoid space which admits a morphism of finite $\dimtrg$ to a totally disconnected perfectoid space, then $\DFF(S,\Q_p) = \D_{\hat\solid}(\mathcal{Y}_{(0,\infty),S})^\varphi = \D_{\hat\solid}(\FF_S)$ by analytic descent of $\D_{\hat\solid}(-)$.

    \item \label{rmk:choice-coefficients-Qp} In all the above, it would have been possible without any extra difficulty to replace the Fargues--Fontaine curves for the local field $\Q_p$ by Fargues--Fontaine curves for the local field $L$, if $L$ is a given finite extension of $\Q_p$ (or a local field of positive characteristic!). This entails to replacing Witt vectors by ramified Witt vectors and would lead to $\mathcal{O}_L$- or $L$-linear categories $\D_{[0,\infty),L}(-)$, $\D_{(0,\infty),L}(-)$, $\DFF(-,\mathcal{O}_L)$, $\DFF(-,L)$. In order to keep the notation simple we stick to the case $L=\Q_p$.
\end{remarksenum}
\end{remarks}

Using the results for the pushforward functors in the $\Dhsa{(-)}$ formalism we can formally deduce similar properties for $f_\ast$ in the setting of sheaves on the Fargues--Fontaine curve. It is convenient to allow coefficients in an arbitrary nuclear $\Z_p$-algebra $\Lambda$.

\begin{proposition} \label{rslt:D-0-infty-pushforward-colim-and-base-change}
Let $f\colon S^\prime\to S$ be a morphism of small v-stacks. Assume that $f$ is $p$-bounded (and in particular qcqs). Let $\Lambda$ be a nuclear $\Z_p$-algebra.
\begin{propenum}
    \item \label{rslt:D-0-infty-pushforward-preserve-colim} The functor $f_\ast\colon \Mod_\Lambda\D_{[0,\infty)}(S')\to \Mod_\Lambda\D_{[0,\infty)}(S)$ preserves colimits and is $\D_\nuc(\Lambda)$-linear.

    \item \label{rslt:D-0-infty-pushforward-base-change} If 
    \[\begin{tikzcd}
        {T'} & {S'} \\
        T & S
        \arrow["{g'}", from=1-1, to=1-2]
        \arrow["g", from=2-1, to=2-2]
        \arrow["f", from=1-2, to=2-2]
        \arrow["{f'}", from=1-1, to=2-1]
    \end{tikzcd}\]
    is a cartesian diagram of small v-stacks, then the natural transformation
    \[
        g^\ast f_\ast\to f'_\ast g^{\prime,\ast}\colon \Mod_\Lambda \D_{[0,\infty)}(S')\to \Mod_\Lambda \D_{[0,\infty)}(T)
    \]
    is an isomorphism.
\end{propenum}
\end{proposition}
\begin{proof}
Let $\Lambda \to \Lambda'$ be a morphism of nuclear $\Z_p$-algebras. If both assertions are true for $\Lambda$, then they are true for $\Lambda'$ as well. Indeed, if $f_\ast = f_\ast^\Lambda\colon \Mod_\Lambda \D_{[0,\infty)}(S')\to \Mod_\Lambda \D_{[0,\infty)}(S)$ commutes with colimits, then it is $\D_\nuc(\Lambda)$-linear (by rigidity of $\D_\nuc(\Lambda)$, \cite[Lemma 3.33.]{anschuetz_mann_descent_for_solid_on_perfectoids}) and hence the base change $f_\ast^\Lambda \tensor_\Lambda \Lambda'$ agrees with $f_\ast^{\Lambda'}$. Conversely, if $\Lambda \to \Lambda'$ is a descendable morphism of nuclear $\Z_p$-algebras then we claim that the assertion for $\Lambda'$ implies the one for $\Lambda$. Let $[n]\mapsto \Lambda'_n$ be the \v{C}ech nerve for $\Lambda \to \Lambda'$. Then the functor
\[
    \Mod_\Lambda \D_{[0,\infty)}(-) \isoto \varprojlim_{[n]\in \Delta} \Mod_{\Lambda'_n} \D_{[0,\infty)}(-)
\]
is an equivalence. Now the right adjoints $f_\ast^{\Lambda'_n}$ satisfy base-change by the above argument, and hence descend to the right adjoint $f_\ast^\Lambda$. In particular, $f_\ast^\Lambda$ commutes with colimits and base change.

Altogether the above arguments allow us to reduce the claim to the case $\Lambda = \Z_{p,\infty}$.
In this case we may by \Cref{sec:v-descent-d_ffs-1-modules-under-o-e-infty} reduce the assertion to \Cref{rslt:Dsha-pushforward-colim-and-base-change}.
We also note that the functor $S\mapsto \mathcal{Y}_{[0,\infty),S}\times_{\Spa(\Z_p)}\Spa(\Z_{p,\infty})$ from small $v$-stacks to untildet small $v$-stacks preserves the lpbc-condition for morphisms because on the tilts this functor is just the base change $S\mapsto S\times \Spd(\Z_{p,\infty})$.
\end{proof}

An important subcategory of $\D_{[0,\infty)}(S)$ is formed by the nuclear objects, which are roughly the ``ind-Banach sheaves''. In practice almost all of the relevant objects one deals with are of this form and they enjoy some particularly nice properties. Let us therefore introduce a corresponding notation.

\begin{definition}
  \label{sec:v-desc-texorpdfstr-definition-nuclear-subcats}
Let $S$ be a small v-stack. An object $M\in \D_{[0,\infty)}(S)$ (or $\D_{(0,\infty)}(S)$) is nuclear if its pullback to any perfectoid space which admits a morphisms of finite $\dimtrg$ to a totally disconnected perfectoid space, is nuclear in the sense of \cref{rmk:nuclear-objects-on-stably-uniform-adic-space}. We add a superscript $(-)^\nuc$ to denote the respective full subcategories of nuclear objects (the same for $\DFF(-,\Z_p)$ or $\DFF(-,\Q_p)$).
\end{definition}

\begin{lemma} \label{sec:texorpdfstr-desc-tex-nuclear-sheaves-for-hypercomplete-v-sheaf}
The functors $S\mapsto \D^\nuc_{[0,\infty)}(S)$ and $S\mapsto \D^\nuc_{(0,\infty)}(S)$ are hypercomplete v-sheaves.
\end{lemma}
\begin{proof}
We first note that \Cref{sec:defin-d_hats--4-comparison-for-modified-modules}, \cite[Theorem 2.38]{anschuetz_mann_descent_for_solid_on_perfectoids} and the argument of \cite[Lemma 3.30]{anschuetz_mann_descent_for_solid_on_perfectoids} imply analytic descent of nuclear modules on stably uniform analytic adic spaces. Then the argument of \cite[Lemma 3.30]{anschuetz_mann_descent_for_solid_on_perfectoids} shows hypercomplete v-descent of $\D^\nuc_{[0,\infty)}(-)$ on perfectoid spaces (similarly for $\D^\nuc_{(0,\infty)}(-)$). We additionally note that nuclearity can be checked by pullback to a single cover by perfectoid spaces which admit morphisms of finite $\dimtrg$ to totally disconnected perfectoid spaces.
\end{proof}

\subsection{Six functors for \texorpdfstring{$\D_{[0,\infty)}(-)$}{D\_[0,∞)(-))}} \label{sec:6-functor-formalism}

In this section we will establish a full $6$-functor formalism (in the sense of \cite[Definition A.5.7]{mann-mod-p-6-functors}) for the functor $S\mapsto \D_{[0,\infty)}(S)$ on small v-stacks $S$. It will be convenient to allow coefficients in an arbitrary nuclear $\Z_p$-algebra $\Lambda$.

As in \cref{sec:6ff-for-Dhsa} we first need to establish base-change and projection formula for étale and proper maps and then use the formal construction results in \cite[\S3]{heyer-mann-6ff} to get the desired 6-functor formalism. In \cref{rslt:D-0-infty-pushforward-colim-and-base-change} we already proved qcqs base-change. The following results provide the missing properties:

\begin{lemma} \label{sec:6-functors-dywzs-etale-morphisms-for-dywz}
Let $j\colon U\to S$ be an étale morphism of small v-stacks and $\Lambda$ a nuclear $\Z_p$-algebra. Then the functor $j^\ast =j_\Lambda^\ast = \Lambda \tensor_{\Z_p} j^\ast\colon \Mod_\Lambda\D_{[0,\infty)}(S)\to \Mod_\Lambda\D_{[0,\infty)}(U)$ admits a left adjoint
\[
    j_!=j_!^\Lambda\colon \Mod_\Lambda \D_{[0,\infty)}(U)\to \Mod_\Lambda \D_{[0,\infty)}(S),
\]
which satisfies the projection formula. Moreover, for every morphism $f\colon S'\to S$ of small v-stacks with base change $f'\colon U':=S'\times_S U\to U$, $j'\colon U'\to S'$ the natural morphism
\[
    j_!'f^{\prime,\ast}\to f^\ast j_!
\]
of functors $\Mod_\Lambda \D_{[0,\infty)}(U) \to \Mod_\Lambda \D_{[0,\infty)}(S')$ is an isomorphism.
\end{lemma}
\begin{proof}
Base change reduces to the case that $\Lambda = \Z_p$ (note that the projection formula for $j_!$ implies that $j_!$ is $\D_\nuc(\Z_p)$-linear). Then the assertion follows from \Cref{sec:defin-d_hats--1-if-sousperfectoid-can-extend-to-etale}.
\end{proof}

\begin{lemma} \label{-projection-formula-for-dywz}
Let $f\colon S'\to S$ be a proper $p$-bounded morphism of small v-stacks and let $\Lambda$ be a nuclear $\Z_p$-algebra. Then $f_\ast\colon \Mod_\Lambda\D_{[0,\infty)}(S')\to \Mod_\Lambda \D_{[0,\infty)}(S)$ satisfies the projection formula, i.e. for $\mathcal{M}\in \Mod_\Lambda \D_{[0,\infty)}(S)$ and $\mathcal{N}\in \Mod_\Lambda \D_{[0,\infty)}(S')$ the natural morphism
\[
    \mathcal{M}\otimes f_\ast(\mathcal{N})\to f_\ast(f^\ast\mathcal{M}\otimes \mathcal{N})
\]
is an isomorphism.
\end{lemma}
\begin{proof}
By \cref{rslt:D-0-infty-pushforward-preserve-colim} we know that $f_\ast$ is $\Lambda$-linear. Using base change and descent in $\Lambda$ as in the proof of \Cref{rslt:D-0-infty-pushforward-colim-and-base-change} we can reduce to the case that $\Lambda = \Z_{p,\infty}$. Then the assertion follows from \Cref{rslt:Dhsa-projection-formula-for-proper-map}.
\end{proof}

We can now construct the $6$-functor formalism for $\Mod_\Lambda \D_{[0,\infty)}(S)$. The formulation including ``change-of-coefficients'' is taken from \cite[Proposition 5.14]{mann-nuclear-sheaves}. For the following result, recall the definition of lpbc maps of v-stacks from \cref{def:lpbc-maps}.

\begin{theorem} \label{rslt:6ff-for-D-0-infty}
There exists a class $E$ of morphisms in $\vStacks \times \CAlg(\D_\nuc(\Z_p))^\opp$ and a $\D_\solid(\Z_p)$-linear 6-functor formalism
\[
    \D_{[0,\infty)}(-,-)\colon \Corr(\vStacks\times \CAlg(\D_\nuc(\Z_p))^\opp, E)\to \Pr^L_{\D_\solid(\Z_p)}
\]
satisfying the following properties:
\begin{thmenum}
    \item The underlying functor 
    \begin{align*}
        \vStacks^\opp\times \CAlg(\D_\nuc(\Z_p))\to \Pr^L_{\D_\solid(\Z_p)}
    \end{align*}
    sends $(S,\Lambda)\in \vStacks\times \CAlg(\D_\nuc(\Z_p))$ to $\Mod_\Lambda\D_{[0,\infty)}(S)$ and morphisms to pullback/base change functors.
    
    \item $E$ is stable under pullback and composition, and if $f,g$ are composable maps such that $f,f\circ g\in E$, then $g\in E$.\footnote{These properties are part of the definition of ``geometric setup'', see \cite[Remark~2.1.2]{heyer-mann-6ff}, and are thus implicit in the claim that there is a 6-functor formalism.}
  
    \item Let $f\colon (S',\Lambda')\to (S,\Lambda)$ be given. If $S'\to S$ is lpbc and $\Lambda \to \Lambda'$ an isomorphism, then $f\in E$. If $f\in E$, then the morphism $\Lambda \to \Lambda'$ is an isomorphism.
    
    \item Every étale map $j\colon S' \to S$ of small v-stacks is cohomologically étale, and thus $j_!$ agrees with the functor $j_!$ from \cref{sec:6-functors-dywzs-etale-morphisms-for-dywz}.
    
    \item Every proper lpbc map of small v-stacks is cohomologically proper, and thus $f_! = f_\ast$.
    
    \item $E$ is $\ast$-local on the target in the sense that if $f\colon (S',\Lambda')\to (S,\Lambda)$ is a morphism with $\Lambda \to \Lambda'$ an isomorphism, then $f$ lies in $E$ if it does so after pullback to any spatial diamond over $S$ which admits a morphism to a strictly totally disconnected perfectoid space.
    
    \item $E$ is $!$-local on the source and target.
\end{thmenum}
\end{theorem}
\begin{proof}
Using \cref{rslt:D-0-infty-pushforward-base-change,sec:6-functors-dywzs-etale-morphisms-for-dywz,-projection-formula-for-dywz} the proof of \cref{rslt:6ff-for-Dhsa} applies here as well and shows the existence of the 6-functor formalism. The extra functoriality on nuclear $\Z_p$-algebras can be obtained analogously to the construction in \cite[Proposition~3.5.22]{heyer-mann-6ff}. Moreover, note that the $\D_\solid(\Z_p)$-linearity of the 6-functor formalism is automatic by \cite[Lemma~3.2.5]{heyer-mann-6ff}. 
\end{proof}

\begin{definition} \label{def:6ff-for-D-0-shriekable-maps}
  Choose $E$ as in \ref{rslt:6ff-for-D-0-infty}.
  A morphism $f\colon S'\to S$ of small v-stacks is called $\D_{[0,\infty)}$-$!$-able if $(f,\mathrm{Id}_{\Z_p})$ lies in $E$ (cf., \cref{rmk:choice-of-E}).
\end{definition}

We note that \Cref{def:6ff-for-D-0-shriekable-maps} and \Cref{sec:6-funct-da_h-1-definition-p-shriekable-morphism} are probably not comparable as the involved notions of universal $!$-descent differ.
However, we'd like to have a stronger notion of a $!$-able map than $\D_{[0,\infty)}$-$!$-able maps, in order to compare the $\D_{[0,\infty)}$- and the $\D^a_{\hat\solid}$-formalism after base change to $\Z_{p,\infty}$. We develop such a notion in the following.

We note that for any small v-stack $S$ there exists by v-descent an untilted small v-stack, written $\mathcal{Y}_{[0,\infty),S}\times_{\Spa \Z_p}\Spa(\Z_{p,\infty})$, extending the case that $S$ is a perfectoid space (but only in the perfectoid case it is associated with an analytic space over $\Spa(\Z_{p,\infty})$). 

\begin{lemma} \label{sec:definition--able-1-existence-of-class-of-shriekable-maps}
There exists a class $E$ of morphisms of small v-stacks with the following properties:
\begin{lemenum}
    \item $E$ is stable under pullback and composition, and if $f,g$ are composable maps such that $f, f\circ g\in E$, then $g\in E$.

    \item Each lpbc-morphism in the sense of \Cref{def:lpbc-maps} lies in $E$.

    \item $E$ is $\ast$-local on the target in the sense that if $f\colon S'\to S$ is a morphism of small v-stacks whose base change to every spatial diamond over $S$ which admits a morphism to a strictly totally disconnected perfectoid space lies in $E$, then $f$ lies in $E$.

    \item $E$ is $!$-local on the source in the following sense: assume $f\colon S'\to S, g\colon S''\to S'$ are morphisms of small v-stacks, such that $f\circ g, g$ lie in $E$ and the morphism $\mathcal{Y}_{[0,\infty),S''}\times_{\Spa \Z_p}\Spa(\Z_{p,\infty})\to \mathcal{Y}_{[0,\infty),S'}\times_{\Spa \Z_p}\Spa(\Z_{p,\infty})$ of untilted small v-stacks is $\D^a$-$!$-able and satisfies universal $!$-descent for $\D^a_{\hat\solid}$. Then $f\in E$.
    
    \item $E$ is $!$-local on the target in the following sense: assume $f\colon S'\to S$, $g\colon S''\to S$ are morphisms of small v-stacks such that $\mathcal{Y}_{[0,\infty),S''}\times_{\Spa \Z_p}\Spa(\Z_{p,\infty})\to \mathcal{Y}_{[0,\infty),S}\times_{\Spa \Z_p}\Spa(\Z_{p,\infty})$ is $\D^a$-$!$-able and satisfies universal $!$-descent for $\D^a_{\hat\solid}$. If $g$ and the base change of $f$ along $g$ lie in $E$, then $f$ lies in $E$.
    
    \item If $f\colon S'\to S$ lies in $E$, then $f$ is $\D_{[0,\infty)}$-$!$-able, and the morphism $\mathcal{Y}_{[0,\infty),S'}\times_{\Spa(\Z_p)}\Spa(\Z_{p,\infty})\to \mathcal{Y}_{[0,\infty),S}\times_{\Spa(\Z_p)}\Spa(\Z_{p,\infty})$ is $\D^a_{\hat\solid}$-$!$-able.
    
    \item The two 6-functor formalisms
    \[
        \Corr(\vStacks,E)\to \Corr(\vStacks, \text{$\D_{[0,\infty)}$-!-able})\xto{\Mod_{\Z_{p,\infty}}\D_{[0,\infty)}(-)} \Cat
    \]
    and
    \[
        \Corr(\vStacks,E)\xto{\mathcal{Y}_{[0,\infty),(-)}\times \Spa(\Z_{p,\infty})} \Corr(\vStacks^\sharp,\text{$\D^a_{\hat\solid}$-!-able})\xto{\Mod_{\ri}\D^a_{\hat\solid}(-)} \Cat 
    \]
    are uniquely isomorphic if the restrictions to $\vStacks^\opp$ are identified through the natural isomorphism \Cref{sec:v-descent-d_ffs-1-modules-under-o-e-infty}.
\end{lemenum}
\end{lemma}
\begin{proof}
  As in \Cref{rslt:6ff-for-D-0-infty} and \Cref{rslt:6ff-for-Dhsa} we use \cite[Theorem 3.4.11]{heyer-mann-6ff}, but with the following modifications:
  the $!$-locality condition in (iii) and (iv) in loc.\ cit.\ is replaced by the stronger notion as used in (iv) and (v) stated here, i.e. universal $!$-descent is asked for maps $f\colon S'\to S$ such that $\mathcal{Y}_{[0,\infty),S'}\times_{\Spa \Z_p}\Spa(\Z_{p,\infty})\to \mathcal{Y}_{[0,\infty),S}\times_{\Spa \Z_p}\Spa(\Z_{p,\infty})$ satisfies universal $!$-descent for $\D^a_{\hat\solid}$ (note that this involves requiring $!$-descent on the $\D^{a,+}_{\hat\solid}$-categories after base change to \textit{any} untilted small v-stack over $\mathcal{Y}_{[0,\infty),S}\times_{\Spa \Z_p}\Spa(\Z_{p,\infty})$, not just whose coming from a morphism $T\to S$ of small v-stacks).
  With similar modifications, the proof of \cite[Theorem 3.4.11]{heyer-mann-6ff} goes through and provides us with a class satisfying (i)-(v). We may also assume that condition (vi) is satisfied because a morphism $f\colon S'\to S$ satisfying the stronger universal $!$-descent as above in particular satisfies universal $!$-descent for $\D_{[0,\infty)}$ (using that $\Z_p\to \Z_{p,\infty}$ is descendable in $\D_{\solid}(\Z_p)$).
  The final assertion (vii) follows from \cite[Proposition 3.4.8]{heyer-mann-6ff}, and the construction of $\D_{[0,\infty)}$- resp.\ $\D^a_{\hat\solid}$-$!$-able maps through \cite[Theorem 3.4.11]{heyer-mann-6ff}: namely, both 6-functor formalisms agree by construction on compactifiable $p$-bounded morphisms (using $\Lambda$-linearity of pushforward, see \cref{rslt:D-0-infty-pushforward-preserve-colim}, to handle the proper case), and this propagates through the inductive construction in \cite[Theorem 3.4.11]{heyer-mann-6ff} using the uniqueness of the extension in \cite[Proposition 3.4.8]{heyer-mann-6ff}.
\end{proof}

\begin{definition} \label{def:shriekable-maps}
Choose a class $E$ as in \cref{sec:definition--able-1-existence-of-class-of-shriekable-maps}. We call a map of small v-stacks $!$-able if it lies in this class $E$ (cf. \cref{rmk:choice-of-E}).
\end{definition}

\begin{remark} \label{rslt:relate-sheaves-on-FF-curve-to-sheaves-at-divisor}
Similar to \cref{sec:definition--able-1-existence-of-class-of-shriekable-maps} one can construct the following comparison of 6-functor formalisms. Fix an untilted small v-stack $X^\sharp$ with underlying v-stack $X$. Then there is an algebra $\ri_{X^\sharp} \subseteq \D_{[0,\infty)}(X)$ which intuitively is the structure sheaf of the Cartier divisor on $\mathcal Y_{[0,\infty),X}$ corresponding to the untitlt $X^\sharp$ and there is an equivalence
\begin{align*}
  \Mod_{\ri_{X^\sharp}} \D_{[0,\infty)}(X) = \D_{\hat\solid}(\widehat \ri_{X^\sharp})
\end{align*}
which extends to an isomorphism of 6-functor formalisms on untilted small v-stacks, for a certain class $E$ of $!$-able maps. Namely, as in \cref{sec:definition--able-1-existence-of-class-of-shriekable-maps} this claim reduces to constructing the above equivalence on a basis of affinoid perfectoid spaces $X = \Spa(A, A^+)$ and show that it is compatible with pullbacks and tensor products there. But then $\D_{[0,\infty)}(X)$ is the category of solid sheaves on $\mathcal Y_{[0,\infty),X}$ and $X^\sharp$ corresponds to a closed subspace of $\mathcal Y_{[0,\infty),X}$, functorially in $X^\sharp$; this easily provides the desired equivalence and functoriality.

As explained in the beginning of \cref{sec:cohom-smooth-cohom}, $\D_{\hat\solid}(\ri_{X^\sharp})$ is closely related to the 6-functor formalism from \cite{mann-mod-p-6-functors}, so by the above construction we get a further connection of this 6-functor formalism with $\D_{[0,\infty)}$.
\end{remark}

Whenever there can arise confusion we clarify if we consider a $!$-able, a $\D_{[0,\infty)}$-$!$-able, or a $\D^a_{\hat\solid}$-$!$-able map in the sense of \Cref{def:shriekable-maps}, \Cref{def:6ff-for-D-0-shriekable-maps} or \Cref{sec:6-funct-da_h-1-definition-p-shriekable-morphism}. 

In practice, the strongest form (as in \Cref{def:shriekable-maps}) is usually satisfied. In fact, in the following we list several useful examples of $!$-able maps.

\begin{proposition} \label{sec:exampl-da_h-morphisms-of-weakly-finite-type}
Let $R$ be a perfect $\F_p$-algebra which is perfectly of finite type over $\F_p$. Then the morphism $f\colon \Spd(R) \to \Spd(\mathbb{F}_p)$ is lpbc, and hence $!$-able. Similarly, for any finite set $R_0\subseteq R$ the morphism $\Spd(R,R_0)\to \Spd(\F_p)$ is lpbc, and hence $!$-able.
\end{proposition}
\begin{proof}
We use \Cref{sec:exampl-da_h-criterion-to-check-lpbc}. Namely, $f$ is representable in locally spatial diamonds, compactifable and of locally bounded dimension (the last two properties use that $R$ is of perfectly finite type). After base change to a strictly totally disconnected perfectoid space $S$ we can embed $S\times \Spd(R)$ as a Zariski-closed subspace into some affine space $\mathbb{A}^{n}_S$. Zariski-closed immersions are qcqs quasi-pro-\'etale, and hence $p$-bounded (see \cref{rslt:qcqs-qproet-map-is-p-bounded}), so it suffices to see that $\mathbb{A}^n_S\to S$ is $p$-bounded, but this is clear.
As $R_0$ is a finite set, the morphism $\Spd(R,R_0)\to \Spd(R)$ is an open immersion, and hence lpbc by \Cref{sec:6-funct-da_h-1-properties-of-lpbc-maps}. Thus, $\Spd(R,R_0)\to \Spd(\F_p)$ is lpbc as desired.
\end{proof}

\begin{proposition} \label{sec:exampl-da_h-morphism-from-classifying-stack-shriek-able}
Let $H$ be a locally profinite group with virtually finite $p$-cohomological dimension.\footnote{By this we mean that there exists an open subgroup with finite $p$-cohomological dimension.} Then the morphism $f\colon \Spd(\mathbb{F}_p)/H \to \Spd(\mathbb{F}_p)$ is $!$-able.
\end{proposition}
\begin{proof}
By the $\ast$-locality of $!$-able maps as in \Cref{rslt:6ff-for-Dhsa}, we may base change to a strictly totally disconnected perfectoid space $S$ over $\Spd(\F_p)$. We claim that the morphism
\begin{align*}
    g\colon \mathcal{Y}_{[0,\infty),S}\times_{\Spa(\Z_p)}\Spa(\Z_{p,\infty}) &\to \mathcal{Y}_{[0,\infty), S/H}\times_{\Spa(\Z_p)}\Spa(\Z_{p,\infty}) \\&\qquad= (\mathcal{Y}_{[0,\infty),S}\times_{\Spa(\Z_p)}\Spa(\Z_{p,\infty}))/H
\end{align*}
is of universal $!$-descent for $\D^a_{\hat\solid}$. In fact, we claim this for $\mathcal{Y}_{[0,\infty),S}\times_{\Spa(\Z_p)}\Spa(\Z_{p,\infty})$ replaced by any untilted small v-stack $T$, which admits a pseudo-uniformizer $\pi$. We may assume that $H$ is compact and of finite $p$-cohomological dimension (because \'etale covers satisfy $!$-descent, cf.\ \cite[Lemma 4.7.1]{heyer-mann-6ff}). We note that $g\colon T \to T/H$ is proper and lpbc. Now, the chosen pseudo-uniformizer $\pi$ on $T$ yields a symmetric monoidal functor $\D_\solid(\Z_p[[\pi]])\to \D^a_{\hat\solid}(\mathcal{O}_T^+)$ and similarly to the proof of \cite[Lemma 3.11]{mod-p-stacky-6-functors}, it suffices to see that the morphism $\Z_p[[\pi]]\to C(H,\Z_p[[\pi]])$ has descent in the category of $C(H,\Z_p[[\pi]])$-comodules in $\D_\solid(\Z_p[[\pi]])$. This can be checked modulo $(p^n,\pi^n)$ for $n\geq 0$ (provided that the bound is uniform in $n$). As in \cite[Lemma 3.11]{mod-p-stacky-6-functors} this follows from the assumption that $H$ has finite $p$-cohomological dimension.
\end{proof}

Another source of examples of $!$-able maps comes from the theory of Banach-Colmez spaces. Fix a finite extension $E$ of $\Q_p$. Following Fargues-Scholze's notation, if $S$ is a perfectoid space over the residue field $k_E$ of $E$, and $\mathcal{E}_0$, $\mathcal{E}_1$ are vector bundles on $\FF_{S,E}$, with a map $\mathcal{E}_1 \to \mathcal{E}_0$, we denote by 
$$
    \mathcal{BC}([\mathcal{E}_1 \to \mathcal{E}_0]): T \in \Perf_S \mapsto \mathbb{H}^0(\FF_{T,E},[\mathcal{E}_1 \to \mathcal{E}_0]_{|_{\FF_{T,E}}}).
$$
We will assume that $\mathcal{E}_1$ has only negative Harder-Narasimhan slopes.

\begin{proposition} \label{sec:exampl-da_h-bc-space-shriekable}
With the above notations, the morphism $f\colon \mathcal{BC}([\mathcal{E}_1 \to \mathcal{E}_0]) \to S$ is $!$-able.
\end{proposition}
\begin{proof}
By \cite[Proposition II.3.5.(i)]{fargues-scholze-geometrization} the morphism $f$ is representable in locally spatial diamonds and partially proper (hence compactifiable). Using \Cref{sec:exampl-da_h-criterion-to-check-lpbc} it suffices to see that $f$ is $p$-bounded, which is easy: one can reduce to the case that $S=\Spa(C)$ for an algebraically closed, non-archimedean field and show that the cohomological dimension of points on $\mathcal{BC}([\mathcal{E}_1 \to \mathcal{E}_0])$ is bounded in terms of the ranks and degrees of $\mathcal{E}_1, \mathcal{E}_0$.
\end{proof}

\begin{remark}
In \Cref{sec:geometric-examples-coh-smooth-morphisms} we will also discuss classifying stacks of Banach-Colmez spaces.
\end{remark}

Having discussed many examples of $!$-able maps, we finish this subsection with the following useful criterion for checking cohomological properness and smoothness in $\D_{[0,\infty)}(-)$ by reducing the question modulo a pseudo-uniformizer.

\begin{lemma} \label{rslt:smooth-and-proper-for-D-0-infty-can-be-checked-mod-pi}
  Let $f\colon S'\to S$ be a $!$-able morphism of small v-stacks.
  Assume that $\pi$ is a pseudo-uniformizer on $S$.
  If $f$ is cohomologically smooth (resp.\ cohomologically proper) for $\D^a_\solid(\ri^+_{(-)}/\pi)$, then $f$ is cohomologically smooth (resp.\ cohomologically proper) for $\D_{[0,\infty)}(-)$.
\end{lemma}
\begin{proof}
  This proof illustrates very well how the theory developed in this paper works on a technical level.
  We first note that \Cref{rslt:descent-of-6ff} implies that cohomological smoothness/properness in a $\D_{\solid}(\Z_p)$-linear $6$-functor formalism can be checked after the base change to $\D_{\solid}(\Z_{p,\infty})$.
  In our case we may therefore argue for the $6$-functor formalism $\Mod_{\Z_{p,\infty}}\D_{[0,\infty)}(-)$ on small $v$-stacks instead.
  By \Cref{sec:definition--able-1-existence-of-class-of-shriekable-maps} and the arguments in \ref{rslt:v-descent-for-D-0-infty} this $6$-functor formalism is (up to a restriction of the class of $!$-able maps) the composition of the functor $S\mapsto \mathcal{Y}_{[0,\infty),S}\times_{\Spa(\Z_p)}\Spa(\Z_{p,\infty})$ from small $v$-stacks to untilted small $v$-stacks with the $6$-functor formalism $Z\mapsto \D_{\hat\solid}(\mathcal{O}_Z)$ on untilted small $v$-stacks.
  Moreover, by \Cref{sec:definition--able-1-existence-of-class-of-shriekable-maps} a $!$-able map $S'\to S$ of small $v$-stacks yields a map $\mathcal{Y}_{[0,\infty),S'}\times_{\Spa(\Z_p)}\Spa(\Z_{p,\infty})\to \mathcal{Y}_{[0,\infty),S}\times_{\Spa(\Z_p)}\Spa(\Z_{p,\infty})$ on untilted small $v$-stacks, which is even $!$-able for the $6$-functor formalism $Z\mapsto \D_{\hat\solid}^a(\mathcal{O}_Z^+)$ on untilted small $v$-stacks.
  By \Cref{sec:suppl-6-funct-base-change-for-6-functor-formalisms} cohomological smoothness/properness for $\D_{\hat\solid}^a(\mathcal{O}_{(-)}^+)$ implies the same for $\D_{\hat\solid}(\mathcal{O}_{(-)})$.
  By the arguments in \Cref{sec:f-smooth-f-checking-cohom-smoothness-mod-pi} and \cref{sec:f-smooth-f-checking-cohomological-properness} we can check cohomological smoothness/properness for $\D_{\hat\solid}^a(\mathcal{O}_{(-)}^+)$ after modding out $p$, i.e., on $\D_{\hat\solid}^a(\mathcal{O}_{(-)}^+)\otimes_{\D_\solid(\Z_p)}\D_{\solid}(\F_p)$.
  Due to the existence of a pseudo-uniformizer $\pi$ on $S$, we see that the tilt $S\times \Spd(\Z_{p,\infty})$ of $\mathcal{Y}_{[0,\infty),S}\times_{\Spa(\Z_p)}\Spa(\Z_{p,\infty})$ admits a pseudo-uniformizer, namely $\pi$.
  Using \Cref{sec:f-smooth-f-checking-cohom-smoothness-mod-pi} and \cref{sec:f-smooth-f-checking-cohomological-properness} again we see that the statement of the lemma reduces to the assertion that if $S'\to S$ is cohomologically smooth/proper for $\D_{\solid}^a(\mathcal{O}_{(-)}^+)$, then $S'\times \Spd(\Z_{p,\infty})\to S\times \Spd(\Z_{p,\infty})$ is cohomologically smooth/proper.
  This in turn is clear.
\end{proof}

\begin{remark} \label{sec:6-funct-texorpdfstr-variant-without-zero}
The results of this subsection work verbatim for $\D_{(0,\infty)}(-)$, i.e., for $\mathcal{Y}_{[0,\infty),S}$ replaced by $\mathcal{Y}_{(0,\infty),S}$, and yield a $\D_\solid(\Q_p)$-linear $6$-functor formalism $\D_{(0,\infty)}(-)$ on small v-stacks.
Alternatively, set $R:=\D_{[0,\infty)}(\Spd(\F_p))$. As $\Spd(\F_p)$ is the terminal v-stack, the $6$-functor formalism $S\mapsto \D_{[0,\infty)}(S)$ is $R$-linear. In $R$ one can now by descent define an idempotent algebra $A$, which has the property that its pullback $A_S$ to $\D_{[0,\infty)}(S)=\D_{\hat\solid}(\mathcal{Y}_{[0,\infty),S})$ is the integral Robba ring for any perfectoid space $S$, which admits a morphism of finite $\dimtrg$ to a totally disconnected perfectoid space, i.e., $A_S$ is the ring of functions converging in a neighborhood of the locus $\{p=0\}$ in $\mathcal{Y}_{[0,\infty),S}$, cf.\ \ref{sec:main-results-remark-on-qp-version}.
We can then set $R':=R/\mathrm{Mod}_A(R)$, and consider the base change $R'\otimes_{R}\D_{[0,\infty)}(-)$ as in \Cref{sec:suppl-6-funct-base-change-for-6-functor-formalisms}. This base change is then $\D_{(0,\infty)}(-)$. Here, the implicit class of $!$-able maps is still the one from \Cref{def:shriekable-maps}.
\end{remark}

\subsection{Geometric examples of cohomologically smooth morphisms} \label{sec:geometric-examples-coh-smooth-morphisms}
 
In this subsection we provide examples of cohomological smooth morphisms arising from geometry. The next subsection will discuss the arithmetic example of $\Spd(\Q_p)\to \Spd(\F_p)$.

We will make use of the following abstract criterion for checking cohomological smoothness based on the results in \cite{heyer-mann-6ff}. We make use of the solid 6-functor formalism on schemes (see e.g. \cite[\S2.9]{mann-mod-p-6-functors}) which we implicitly extend to stacks via \cite[Theorem~3.4.11]{heyer-mann-6ff} (using the $\D_\solid$-topology, i.e. where covers are universal $!$- and $*$-covers; this includes smooth covers by \cite[Lemma~4.7.1]{heyer-mann-6ff}). Also, in the following result we implicitly identify a profinite group $\Gamma$ with the group scheme $\Spec C(\Gamma,\Z)$.

\begin{lemma}
  \label{sec:geom-exampl-cohom-crit-cohom-smooth}
Let $\Gamma$ be a profinite group. Let $A$ be a (classical) ring such that $\Gamma$ has finite cohomological dimension over $A$, and let $\widetilde{A}$ be an integral $A$-algebra with smooth action by $\Gamma$ over $A$. Denote $f\colon \Spec(\widetilde{A})/\Gamma \to \Spec(A)$ the natural map. Then:
\begin{lemenum}
    \item \label{rslt:checking-smoothness-of-quotient-by-profin-group-properness} $f$ is $\D_\solid(-)$-prim with trivial codualizing sheaf.

    \item \label{rslt:checking-smoothness-of-quotient-by-profin-group-criterion} $f$ is $\D_\solid(-)$-suave if and only if for every compact open subgroup $H\subseteq \Gamma$ the invariants $\Gamma(H,\widetilde{A}) \in \D(A)$ are dualizable.
\end{lemenum}
\end{lemma}
\begin{proof}
Set $X := \Spec(\widetilde{A})/\Gamma$ and $S := \Spec(A)$. We first prove (i). In the case $\Gamma = 1$, the map $f\colon \Spec(\widetilde A) \to \Spec(A)$ is proper because $\widetilde A$ is integral over $A$ (so that $\Spec(\widetilde A) = \Spa(\widetilde A, A)$). In particular $f$ is ($!$-able and) prim with trivial codualizing sheaf. For general $\Gamma$ we employ \cite[Lemma~4.7.4(i)]{heyer-mann-6ff} and \cite[Corollary~4.7.5(i)]{heyer-mann-6ff} using the prim cover $\Spec(\widetilde A) \to \Spec(\widetilde A)/\Gamma$. We then need to show that this prim cover is descendable, which by base-change immediately reduces to showing that $S \to S/\Gamma$ is descendable (for the trivial $\Gamma$-action on $S$). This follows from the finite cohomological dimension of $\Gamma$ over $A$, see \cite[Proposition~5.2.5]{heyer-mann-6ff}. We have thus shown that $f$ is $!$-able and prim in general. To identify the codualizing sheaf, we observe that the diagonal of $f$ has trivial codualizing sheaf, which as in the proof of \cite[Lemma~6.4.6(ii)]{heyer-mann-6ff} induces a map $\delta_f \to 1$ that needs to be an isomorphism because $f$ is prim.

For $H\subseteq \Gamma$ an open compact subgroup we set $Q_H:=\mathrm{Ind}^\Gamma_H \widetilde A$, or more precisely, $Q_H=g_{H,\ast} 1$, where $g_H\colon X_H:=\Spec(\widetilde{A})/H \to X$ is the base change of the proper, smooth morphism $\Spec(\Z)/H \to \Spec(\Z)/\Gamma$. It suffices to check that the collection $Q_H$ for $H\subseteq \Gamma$ running through the open compact subgroups, satisfies the assumptions of \cite[Lemma~4.4.14]{heyer-mann-6ff}. As $g_H$ is proper and smooth, the object $Q_H$ is $f$-prim if $1$ is $(X_H\to S)$-prim (see \cite[Lemma~4.5.16]{heyer-mann-6ff}), which is true because both $g_H$ and $f$ are prim. Next, we have to see that the collection of functors
\[
    \pi_{2,\ast}\IHom(\pi^\ast_1Q_H,-)\colon \D_\solid(X\times_S X)\to \D_\solid(X)
\]
is conservative. This may be checked after base change along morphisms $\Spec(B) \to X$ (as $\pi^\ast_1Q_H$ is $\pi_2$-prim), and then reduces by renaming $B$ to $A$ to the assertion that the collection of functors $f_\ast(\IHom(Q_H,-))\colon \D_\solid(X)\to \D_\solid(S)$ is conservative. Let $h\colon X\to S/\Gamma$ be the natural map. Then $h_\ast$ is conservative and $Q_H=h^\ast(\mathrm{Ind}^\Gamma_H 1)$. This reduces the assertion of conservativity to the case that that $\widetilde{A}=A$ with trivial $\Gamma$-action. But any $M\in \D_\solid(S/\Gamma)$ is isomorphic to $\varinjlim_H M^H$ with $M^H=\IHom(\mathrm{Ind}^\Gamma_H 1,M)$ because $C(\Gamma,\Z)\otimes A = \varinjlim_H C(\Gamma/H,\Z) \otimes A$. This shows the desired conservativity. Thus we may conclude by \cite[Lemma~4.4.14]{heyer-mann-6ff}.
\end{proof}

\begin{remark} 
\begin{remarksenum}
    \item \Cref{rslt:checking-smoothness-of-quotient-by-profin-group-criterion} extends easily to the setting of analytic rings in the sense of Clausen--Scholze, as the argument is mostly formal. The integrality assumption on $\widetilde{A}$ over $A$ can be avoided if one equips $\widetilde{A}$ with the relative analytic ring structure over $A$ (i.e. take $\Spa(\widetilde{A}, A)$ instead of $\Spec(\widetilde{A})$).

    \item \label{extension-analytic-almost} Given an almost setup (in the sense of \cite[Definition~2.2.1]{mann-mod-p-6-functors}, then the assertion of \cref{rslt:checking-smoothness-of-quotient-by-profin-group-criterion} also holds true for $\D^a_\solid(-)$.
      For example, given an algebraically closed non-archimedean field $C$ with ring of integers $\mathcal{O}_C$ and assuming in \ref{sec:geom-exampl-cohom-crit-cohom-smooth} that everything lives over $\mathcal{O}_C$, we can consider the base change $\D^a_{\solid}(-):=\D^a(\mathcal{O}_C)\otimes_{\D(\mathcal{O}_C)}\D_\solid(-)$ of the $6$-functor formalism $\mathcal{D}_\solid(-)$ on analytic stacks over $\mathcal{O}_C$ in the sense of \ref{sec:suppl-6-funct-base-change-for-6-functor-formalisms}.
      This will be useful in \ref{rslt:smooth-maps-of-adic-spaces-are-cohom-smooth}.
\end{remarksenum}
\end{remark}

The following result might be surprising at first -- it shows that the almost world can be rather subtle.
\cite[Proposition 3.7.5]{mann-mod-p-6-functors} contains a more general statement, but the easier statement \ref{rslt:direct-sum-dualizability-in-almost-world} suffices for now.

\begin{lemma} \label{rslt:direct-sum-dualizability-in-almost-world}
Let $A$ be a (classical) ring and let $(a_1)\subseteq (a_2)\subseteq \ldots \subseteq I=\bigcup_{n\geq 0}(a_n)\subseteq A$ be an idempotent ideal, i.e., the derived tensor product of $I$ with itself is again $I$. Assume that $K_n$ are perfect complexes over $A$ such that $a_n\cdot H^i(K_n)=0$ for $i\in \Z$. Then $K:=\bigoplus_{n\geq 0} K_n$ is dualizable in $\D^a(A)$.
\end{lemma}
Here, $\D^a(A)$ denotes the classical derived category of almost $A$-modules, with almost refering to $I$.
\begin{proof}
We note that the morphism $\bigoplus_{n\geq 0} K_n \to \prod_{n\geq 0} K_n$ is an almost isomorphism. Indeed, for fixed $m\geq 0$, the objects
\[
    \bigoplus_{k>m}K_k,\quad \prod_{k>m}K_k
\]
are killed by $a_m$. From here, we see that the natural morphism
\[
    K^\vee\otimes K\to \IHom(K,K)
\]
in $\D^a(A)$ is an isomorphism as each $K_n$ is dualizable. This implies that $K$ is dualizable by \cite[Lemma 6.2]{mann-nuclear-sheaves}.
\end{proof}

With the above preparations at hand, we can now come to the promised discussion of cohomological smoothness in our 6-functor formalisms. We start with the cohomological smoothness of smooth morphisms of analytic adic spaces over $\Q_p$.

\begin{theorem} \label{rslt:smooth-maps-of-adic-spaces-are-cohom-smooth}
Let $g\colon Y\to X$ be a smooth morphism of analytic adic spaces over $\Q_p$. Then $f:=g^\diamond\colon Y^\diamond\to X^\diamond$ is $\D_{[0,\infty)}$-smooth.
\end{theorem}
\begin{proof}
By \cref{rslt:smooth-and-proper-for-D-0-infty-can-be-checked-mod-pi} the claim follows immediately from \cite[Theorem~3.10.17]{mann-mod-p-6-functors}, but we decided to provide a new and more conceptual proof here, using \cref{rslt:checking-smoothness-of-quotient-by-profin-group-criterion} and \Cref{extension-analytic-almost} to streamline the argument. Although it uses a different language, this argument shares some similarity with the argument of \cite{zavyalov2025mod}, as was pointed out to us by the referee.

Using that étale morphisms are cohomologically smooth (by \Cref{rslt:6ff-for-D-0-infty}), and descent of cohomologically smoothness (see \cite[Lemma 4.4.9]{heyer-mann-6ff}), we may reduce to the case of a relative torus $Y=\mathbb{T}_{\Q_p}\times_{\Spa(\Q_p)}X\to X$.
By base change, we reduce then to $X=\Spa(\Q_p)$, and then by descent to the case that $Y=\mathbb{T}_C\to X=\Spa(C)$ for a non-archimedean, algebraically closed extension $C$ of $\Q_p$.
We fix a pseudo-uniformizer $\pi\in C$ and set $R:=\mathcal{O}_C/\pi$. As mentioned above, by \Cref{rslt:smooth-and-proper-for-D-0-infty-can-be-checked-mod-pi} it is sufficient to check that $f$ is $\D^a_{\solid}(\mathcal{O}^+_{(-)}/\pi)$-smooth (note that $f$ is lpbc, and hence $!$-able). We set $A:=R[T^{\pm 1}]=\mathcal{O}_C\langle T^{\pm 1}\rangle/\pi$ and $\widetilde{A}:=R[T^{\pm 1/p^\infty}]$. Fix $\varepsilon:=(1,\zeta_p,\ldots)\in C^\flat$ with $\zeta_p\neq 1$, and let $\Gamma:=\Z_p\cdot \sigma$ act on $\widetilde{A}$ as usual, i.e., $\sigma(T^{a/p^j})={\varepsilon^{a/p^j}}^\sharp T^{a/p^j}$. Here, $\sharp\colon C^\flat\cong \varprojlim_{x\mapsto x^p}C\to C$ is the projection to the first component. Unraveling the definitions, it is sufficient to check that the morphism
\[
    \Spec(\widetilde A)/\Gamma \to \Spec(R)
\]
of stacks over $\Spec(R)$ is $\D^a_\solid$-smooth. By smoothness of $\Spec(A) \to \Spec(R)$ it suffices to check that $h\colon \Spec(\widetilde{A})/\Gamma \to \Spec(A)$ is $\D^a_\solid$-smooth. Let us first check that this map is $\D^a_\solid$-suave, for which by \Cref{rslt:checking-smoothness-of-quotient-by-profin-group-criterion} it is sufficient to see that for any open subgroup $\Gamma'$ the invariants $\Gamma(\Gamma', \widetilde{A})$ (a priori a solid $A$-module) are a dualizable object in $\D^a(A)$. We may reduce to the case that $\Gamma' = \Gamma$. Then
\[
    \Gamma(\Gamma,\widetilde{A})=[\widetilde{A}\xto{\gamma-1} \widetilde{A}]=\bigoplus\limits_{a\in \Z[1/p]} [R\cdot T^a \xto{{\varepsilon^{a}}^\sharp-1} R\cdot T^a],
\]
which as an $A$-module is the direct sum of the complexes
\[
    K_b:=\bigoplus\limits_{a\in b+\Z} [R\cdot T^a\xto{{\varepsilon^{a}}^\sharp-1} R\cdot T^a]
\]
over $b\in \Z[1/p]/\Z$. Writing $b=c/p^j$ with $c\in \Z$, the complex $K_b$ is isomorphic to $[A\cdot T^{c/p^j}\xto{\zeta_{p^j}^c-1}A\cdot T^{c/p^j}]$ as a complex of $A$-modules, and hence perfect over $A$. Moreover, $K_b$ is killed by $\zeta_{p^j}^c-1$. This implies by \Cref{rslt:direct-sum-dualizability-in-almost-world} that $\bigoplus_{b\in \Z[1/p]/\Z} K_b$ is again dualizable in $\D^a$. This finishes the proof that $h$ is $\D^a_\solid$-suave.
\[
    \Gamma(\Gamma,\widetilde{A})=[\widetilde{A}\xto{\gamma-1} \widetilde{A}]=\bigoplus\limits_{a\in \Z[1/p]} [R\cdot T^a \xto{{\varepsilon^{a}}^\sharp-1} R\cdot T^a],
\]
which as an $A$-module is the direct sum of the complexes
\[
    K_b:=\bigoplus\limits_{a\in b+\Z} [R\cdot T^a\xto{{\varepsilon^{a}}^\sharp-1} R\cdot T^a]
\]
over $b\in \Z[1/p]/\Z$. Writing $b=c/p^j$ with $c\in \Z$, the complex $K_b$ is isomorphic to $[A\cdot T^{c/p^j}\xto{\zeta_{p^j}^c-1}A\cdot T^{c/p^j}]$ as a complex of $A$-modules, and hence perfect over $A$. Moreover, $K_b$ is killed by $\zeta_{p^j}^c-1$. This implies by \Cref{rslt:direct-sum-dualizability-in-almost-world} that $\bigoplus_{b\in \Z[1/p]/\Z} K_b$ is again dualizable in $\D^a(A)$. This finishes the proof that $h$ is $\D^a$-suave.

It remains to show that the dualizing complex $\omega_h = h^! 1$ is invertible. By construction $\omega_h$ is an object in $\D^a_\solid(\Spec(\widetilde A)/\Gamma)$, i.e. it is a (solid) almost $\widetilde A$-module equipped with a smooth $\Gamma$-action. To prove its invertibility, it is enough to check the invertibility of the underlying $\widetilde A^a$-module. By \cite[Lemma~3.4.19]{mann-mod-p-6-functors} this underlying module is computed as $\omega_h = \varinjlim_{H \subseteq \Gamma} \omega_h^H$. For every compact open subgroup $H \subseteq \Gamma$ we denote $h_H\colon \Spec(\widetilde A)/H_m \to \Spec(A)$ and observe
\begin{align*}
    \omega_h^H = h_{H*} h_H^! 1 = \IHom(h_{H!} 1, 1) = \IHom_A(\widetilde A^H, A) =: (\widetilde A^H)^\vee,
\end{align*}
where we used that $h_H$ is $\D^a_\solid$-prim with trivial codualizing complex (see \cref{rslt:checking-smoothness-of-quotient-by-profin-group-properness}). The right-hand side can be explicitly computed as above, from which we deduce that $\omega_h \isom \widetilde A[1]$, as desired.
\end{proof}

We now turn to the identification of the dualizing complex that was omitted in the proof of \Cref{rslt:smooth-maps-of-adic-spaces-are-cohom-smooth}. We will mostly follow the proof of \cite[Theorem 3.10.20]{mann-mod-p-6-functors} (see \cite[Theorem 1.2.13]{zavyalov-poincare-duality} for a similar approach) applied to the $6$-functor formalism $\DFF(-,\Z_p)$
on small v-stacks over $\Q_p$. We will use results on a Riemann-Hilbert functor, that will be proven in \Cref{sec:riem-hilb-funct}.

\begin{lemma} \label{sec:poincare-duality-pro-1-d-x-z-p-is-geometric}
The 6-functor formalism $X\mapsto \DFF(X,\Z_p)$ is geometric in the sense of \cite[Definition 4.2.14]{zavyalov-poincare-duality}, i.e., smooth morphisms of analytic adic spaces over $\Q_p$ are cohomologically smooth and for any small v-stack $X$ each invertible object in $\DFF(\mathbb{P}^1_{X},\Z_p)$ is pulled back from $\DFF(X,\Z_p)$. 
\end{lemma}
\begin{proof}
Cohomologically smoothness of smooth morphisms was established in \Cref{rslt:smooth-maps-of-adic-spaces-are-cohom-smooth}. Let $X$ be a small v-stack and $\mathbb{L}\in \DFF(\mathbb{P}^1_X,\Z_p)$ be invertible. By the $\Z_p$-Riemann-Hilbert functor constructed in \Cref{rslt:construction-of-RH-functors}, $\mathbb{L}=\RH_{\Z_p,X^\diamond}(L)$ for a dualizable object $L\in \D_\nuc(X,\Z_p)$. We have a natural morphism $f^{\nuc,\ast}f^\nuc_{\ast}(L)\to L$ and using (after pullback to a strictly totally disconnected perfectoid space) the $t$-structure on $\omega_1$-solid sheaves \cite[Proposition 3.16]{anschuetz_mann_descent_for_solid_on_perfectoids} (that preserves dualizable objects as can be checked in this case) we can check that it induces an isomorphism $f^{\nuc,\ast}\mathcal{H}^0(f^\nuc_{\ast}(L))\to L$. This finishes the proof.
\end{proof}

\begin{theorem} \label{rslt:dualizing-complex-on-smooth-adic-space}
Let $g\colon Y\to X$ be a smooth morphism of analytic adic spaces over $\Q_p$, which is equidimensional of dimension $d$. Then, in $\DFF(Y,\Z_p)$, $(g^\diamond)^!(1)\cong 1(d)[2d]$ with $-(d)$ refering to the Tate twist.  
\end{theorem}
As the morphism $Z\to Z/\varphi^\Z$ is \'etale for any small v-stack $Z$, \Cref{rslt:dualizing-complex-on-smooth-adic-space} implies that the same formula holds for the $\D_{[0,\infty)}(-)$-formalism (and therefore also for $\D_{(0,\infty)}(-)$ by \ref{rslt:base-change-of-6ff-preserves-suave-and-prim} and \ref{sec:6-funct-texorpdfstr-variant-without-zero}).
\begin{proof}
Given \Cref{sec:poincare-duality-pro-1-d-x-z-p-is-geometric} the proof of \cite[Theorem 3.10.20]{mann-mod-p-6-functors} works in this case as well, and reduces the assertion to the calculation of $s^!1[2d]$ for the unit section $\Spd(\Q_p)\to \mathbb{A}^{d,\diamond}_{\Q_p}$, which has the expected form (by \Cref{rslt:construction-of-RH-functors} and the known $\Z_p$-cohomology of $\mathbb{A}^1_{\mathbb{C}_p}$).
\end{proof}

Another important example for smooth morphisms comes from classifying stacks of $p$-adic Lie groups. 

\begin{proposition} \label{prop:classifying-stack-loc-profinite-coh-smooth}
Let $H$ be a $p$-adic Lie group (or more generally a locally profinite group which has locally finite $p$-cohomological dimension and is virtually $p$-Poincaré, in the sense of \cite[Definition 3.15]{mod-p-stacky-6-functors}), then $\Spd(\F_p)/H \to \Spd(\F_p)$ is $\D_{[0,\infty)}$-smooth.
\end{proposition}
\begin{proof}
We already proved $!$-ability in \Cref{sec:exampl-da_h-morphism-from-classifying-stack-shriek-able}. Cohomological smoothness can be checked after base change to a perfectoid field $K$ (see \cite[Lemma~4.5.7]{heyer-mann-6ff}). It then follows from \Cref{rslt:smooth-and-proper-for-D-0-infty-can-be-checked-mod-pi} and \cite[Theorem 3.16, Theorem 3.18]{mod-p-stacky-6-functors}.
\end{proof}

Finally, we discuss examples coming from the theory of Banach-Colmez spaces. Fix a finite extension $E$ of $\Q_p$. If $S$ is a perfectoid space over the residue field $k_E$ of $E$, and $\mathcal{E}_0, \mathcal{E}_1$ are vector bundles on $\FF_{S,E}$, with a map $\mathcal{E}_1 \to \mathcal{E}_0$, we denote by 
$$
    \mathcal{BC}([\mathcal{E}_1 \to \mathcal{E}_0]): T \in \mathrm{Perf}_S \mapsto \mathbb{H}^0(\FF_{T,E},[\mathcal{E}_1 \to \mathcal{E}_0]_{|_{\FF_{T,E}}}).
$$
Recall that for $\ell$-adic coefficients, $\ell \neq p$, if $\mathcal{E}_1$ only has negative Harder-Narasimhan slopes and $\mathcal{E}_0$ only has positive slopes, $\mathcal{BC}([\mathcal{E}_1 \to \mathcal{E}_0])$ is cohomologically smooth (cf. \cite[Proposition II.3.5]{fargues-scholze-geometrization}). This is not true in $\D_{[0,\infty)}(-)$: for $E=\Q_p$, we have $\mathcal{BC}(\mathcal{O}(1))= \Spd \mathbb{F}_p[[t]]$ and hence 
$$
    \D_{(0,\infty)}(\mathcal{BC}(\mathcal{O}(1)))  = \D_{\solid}(\widetilde{\mathbb{D}}_{\Q_p})
$$
with $\widetilde{\mathbb{D}}_{\Q_p}= \lim_{T\mapsto (1+T)^p-1} \mathbb{D}_{\Q_p}$ the (pre)perfectoid open unit disk over $\Q_p$ (this will be justified later, see \Cref{rmk:result-of-zillinger-bc(O(1))} and the discussion in \Cref{ex:classifying-stack-bc-O(1)}), which is not cohomologically smooth: indeed this would imply that the perfectoid torus $\widetilde{X}:= \Spa(C\langle T^{\pm 1/p^{\infty}} \rangle,\mathcal{O}_C\langle T^{\pm 1/p^{\infty}} \rangle)$ itself is cohomologically smooth. But if $f: \widetilde{X} \to X:=\Spa(C\langle T^{\pm 1} \rangle,\mathcal{O}_C\langle T^{\pm 1} \rangle)$, since $X$ is cohohomologically smooth, this would in turn imply that $f^! 1$ is invertible. Since $f$ is proper, we have
$$
\Gamma(\widetilde{X}, f^!1)= \mathrm{Hom}_{\D_{\solid}(\widetilde{X})}(1,f^!1)=\mathrm{Hom}_{\D_{\solid}(X)}(f_\ast 1, 1) = \mathrm{Hom}_{C\langle T^{\pm 1} \rangle}(C\langle T^{\pm 1/p^{\infty}} \rangle, C\langle T^{\pm 1} \rangle).
$$
Since $C\langle T^{\pm 1/p^{\infty}} \rangle$ is, as a $C\langle T^{\pm 1} \rangle$-module, a completed infinite countable direct sum of copies of $C\langle T^{\pm 1} \rangle$, the right hand side isn't an invertible $C\langle T^{\pm 1/p^{\infty}} \rangle$-module, contradiction.

\begin{proposition} \label{prop:classifying-stacks-bc-spaces-coh-smooth}
In the above situation, assume that $\mathcal{E}_1$ only has negative Harder-Narasimhan slopes and that $\mathcal{E}_0$ only has non-negative slopes. Then the classifying stack 
$$
    S/\mathcal{BC}([\mathcal{E}_1 \to \mathcal{E}_0])
$$
is $\D_{[0,\infty)}$-smooth over $S$. 
\end{proposition}
\begin{proof}
We first check $!$-ability for the morphism $S/\mathcal{BC}([\mathcal{E}_1\to \mathcal{E}_0])\to S$. We may check this statement locally in the analytic topology on $S$. By the assumption on the slopes, we have a short exact sequence
\[
    0\to \mathcal{BC}(\mathcal{E}_0)\to \mathcal{BC}([\mathcal{E}_1\to \mathcal{E}_0])\to \mathcal{BC}(\mathcal{E}_1[1])\to 0.
\]
This reduces the assertion to the following statements: 1) $S/\mathcal{BC}(\mathcal{E}_0) \to S$ is $!$-able, and 2) the morphism $h\colon S\to S/\mathcal{BC}(\mathcal{E}_1[1])$ is of universal $!$-descent (in the strong form as required in \Cref{sec:definition--able-1-existence-of-class-of-shriekable-maps}). Namely, given 2) the morphism $S/\mathcal{BC}(\mathcal{E}_1[1])\to S$ is $!$-able, and by pullback along $h$ the desired $!$-ability of $S/\mathcal{BC}([\mathcal{E}_1\to \mathcal{E}_0]) \to S$ reduces to 1).

We first show that 1) holds. Let $d$ be the degree of $\mathcal{E}_0$. From the proof of \cite[Proposition II.3.1]{fargues-scholze-geometrization} we see that there exists (locally in the analytic topology on $S$) a short exact sequence
\[
    0\to \mathcal{E}_2\to \mathcal{E}_0\to \bigoplus\limits_{i=1}^d \mathcal{O}_{S_i^\sharp}\to 0, 
\]
where $S_i^\sharp, i=1,\ldots, d,$ are untilts of $S$ over $E$, and $\mathcal{E}_2$ is semi-stable of slope $0$. \Cref{rslt:smooth-maps-of-adic-spaces-are-cohom-smooth} shows that the morphism $\mathcal{A}^{1,\diamond}_{S_i^\sharp}\to S$ is cohomologically smooth. This implies that the morphism $h_i\colon S\to S/\mathbb{A}^{1,\diamond}_{S_i^\sharp}$ satisfies universal $!$-descent for any $i=1,\ldots, d$. Pulling back along the product of the $h_i's$, we can reduce to the case $\mathcal{E}_0=\mathcal{E}_2$, i.e., that $\mathcal{E}_2$ is semi-stable of slope $0$. Let $r$ be the rank of $\mathcal{E}_2$. By the proof of \Cref{sec:exampl-da_h-morphism-from-classifying-stack-shriek-able} the morphism $S\to S/\mathrm{GL}_r(E)$ is of universal $!$-descent, and thus there exists a morphism $S'\to S$ of universal $!$-descent, such that the $E$-local system $\mathcal{BC}(\mathcal{E}_2)$ on $S$ is trivial on $S'$. As we may check $!$-ability of $S/\mathcal{BC}(\mathcal{E}_2)\to S$ after pullback to $S'$, this reduces to the case that $\mathcal{E}_2$ is trivial, where we can apply \Cref{sec:exampl-da_h-morphism-from-classifying-stack-shriek-able}.

Let us now check statement 2). It suffices to see that $\mathcal{BC}(\mathcal{E}_1[1])\to S$ is $\D_{[0,\infty)}$-smooth. Arguing for $\mathcal{E}_1^\vee$ as in the previous assertion there exists a short exact sequence
\[
    0 \to \mathcal{E}_1\to \mathcal{E}_3\to \bigoplus\limits_{j=1,\ldots, d'} \mathcal{O}_{S_j^\sharp} \to 0
\]
for untilts $S_j^\sharp, j=1,\ldots, d':=-\mathrm{deg}(\mathcal{E}_1)$ and $\mathcal{E}_3$ semistable of slope $0$. This yields the short exact sequence
\[
    0 \to \mathbb{L}:=\mathcal{BC}(\mathcal{E}_3)\to Z:=\bigoplus\limits_{j=1,\ldots} \mathbb{A}^{1,\diamond}_{S_j^\sharp}\to \mathcal{BC}(\mathcal{E}_1[1])\to 0.
\]
The morphism $Z\to S$ is cohomologically smooth by \Cref{rslt:smooth-maps-of-adic-spaces-are-cohom-smooth}, and the morphism $g\colon S/\mathbb{L} \to S$ is cohomologically smooth by \Cref{sec:exampl-da_h-morphism-from-classifying-stack-shriek-able} (in fact we already know $!$-ability of $g$ from the previous case, and hence the claim is v-local on $S$, which reduces to the case that $\mathbb{L}$ is trivial). The factorization $\mathcal{BC}(\mathcal{E}_1[1])\cong Z/E \to S/E \to S$ yields the desired smoothness. This finishes the proof.

Finally, we prove cohomological smoothness. The short exact sequence
$$
0 \to \mathcal{BC}(\mathcal{E}_0) \to \mathcal{BC}([\mathcal{E}_1 \to \mathcal{E}_0]) \to \mathcal{BC}(\mathcal{E}_1][1]) \to 0
$$
reduces us to prove cohomological smoothness for the classifying stacks of $\mathcal{BC}(\mathcal{E}_0)$ and $\mathcal{BC}(\mathcal{E}_1[1])$. We already established that $\mathcal{BC}(\mathcal{E}_1[1])$ is $\D_{[0,\infty)}$-smooth, so a fortiori its classifying stack is. So it only remains to prove that the classifying stack of $\mathcal{BC}(\mathcal{E}_0)$ is $\D_{[0,\infty)}$-smooth. The statement to be proven being local on $S$, we can apply \cite[Proposition II.3.1]{fargues-scholze-geometrization} and assume that $\mathcal{E}_0$ sits in a short exact sequence
$$
0 \to \mathcal{O}_{\FF_{S,E}}(-1)^r \to \mathcal{E}_0 \to \mathcal{F} \to 0
$$
for some $r\geq 0$ and some vector bundle $\mathcal{F}$ semi-stable of degree $0$ at every geometric point. Localizing further for the pro-\'etale topology, we can assume that $\mathcal{F}=\mathcal{O}_{\FF_{S,E}}^s$ is trivial. We get a short exact sequence
$$
0 \to \underline{E}^s \to \mathcal{BC}(\mathcal{E}_0) \to \mathcal{BC}(\mathcal{O}_{\FF_{S,E}}(-1)[1])^r \to 0.
$$
The classifying stack of $\underline{E}^s$ is $\D_{[0,\infty)}$-smooth by \Cref{prop:classifying-stack-loc-profinite-coh-smooth} and we already proved that $\mathcal{BC}(\mathcal{O}_{\FF_{S,E}}(-1)[1])$ is $\D_{[0,\infty)}$-smooth, and thus so also is its classifying stack. This finishes the proof.
\end{proof}

\subsection{The case of \texorpdfstring{$\Spd(\Q_p)$}{Spd(Qp)}} \label{sec:example:-div1}

We know discuss the (surprisingly subtle) arithmetic example $\Spd(\Q_p)\to \Spd(\F_p)$. Let us first establish smoothness of this map. Afterwards we will determine the dualizing complex. 

\begin{theorem} \label{sec:case-texorpdfstr-smoothness-of-spd-q-p}
The map $\Spd \Q_p \to \Spd(\F_p)$ is $\D_{[0,\infty)}$-smooth (and thus $\D_{(0,\infty)}$-smooth, \ref{sec:6-funct-texorpdfstr-variant-without-zero}, \ref{rslt:base-change-of-6ff-preserves-suave-and-prim}) and the dualizing complex is concentrated in cohomological degree $-2$.
\end{theorem}
\begin{proof}
  Fix some algebraically closed non-archimedean field $C$ of characteristic $p$. Then the $\D_{[0,\infty)}$-cohomological smoothness of $\Spd \Q_p \to \Spd(\F_p)$ can be checked after pullback along the v-cover $\Spa C \to  \Spd(\F_p)$ (see \cite[Lemma~4.5.7]{heyer-mann-6ff}).
  As in the proof of \cite[Proposition~24.5]{etale-cohomology-of-diamonds} we have that $\Spd \Q_p\times \Spa C$ admits a finite \'etale surjection from
\begin{align*}
    \perfPunctDisc_C/\Gamma,
\end{align*}
where $\Gamma = \Z_p$ and $\perfPunctDisc_C$ is the punctured perfectoid open unit disc over $C$.
So we reduce to showing that the map $\perfPunctDisc_C/\Gamma \to \Spa C$ is $\D_{[0,\infty)}$-cohomologically smooth.
We proceed in a similar way as in the torus case.

If $t$ is the coordinate on $\perfDisc_C$ then $\gamma \in \Gamma$ acts on $t$ by $\gamma \cdot t = (1 + t)^{1 + p\gamma} - 1$. In particular, for fixed $\gamma$ the action is given by a power series in $t$ which starts with $t + \dots$, which shows that for all rational radii $0 < r < s < 1$ the $\Gamma$-action on $\perfPunctDisc_C$ restricts to a $\Gamma$-action on the perfectoid annulus $\perfBall[r,s]_C \subset \perfDisc_C$ where $r \le \abs t \le s$. Since these annuli form an open cover of $\perfPunctDisc_C$, it is enough to show that each map $\perfBall[r,s]_C/\Gamma \to \Spa C$ is $\D_{[0,\infty)}$-cohomologically smooth. We now fix $r$ and $s$. After rescaling the coordinate on $\perfBall[r,s]_C$ we have $\perfBall[r,s]_C = \Spa(C\langle T^{1/p^\infty}, \frac{\varpi}{T}^{1/p^\infty}\rangle)$ for some pseudo-uniformizer $\varpi \in \mm_C$ depending on $\frac sr$, and there is a pseudo-uniformizer $\pi \in \mm_C$ (depending on $r$) such that $\gamma \in \Gamma$ acts on $T$ by $T \mapsto \frac1{\pi} ((1 + \pi T)^{1+p\gamma} - 1)$. We set
\begin{align*}
    R &:= \ri_C/\pi,\\
    A &:= R[T, S],\\
    A_n &:= R[T^{1/p^n}, S^{1/p^n}]/(S^{1/p^n} T^{1/p^n} - \varpi^{1/p^n}),\\
    \widetilde A &:= \varinjlim_n A_n.
\end{align*}
One checks that $\widetilde A$ is almost isomorphic to $\ri^+(\perfBall[r,s]_C)/\pi$ (e.g. observe that the version of $\widetilde A$ without modding out by $\pi$ is a perfectoid ring and use \cite[Lemma 3.21]{bms1}).
In order to shorten notation we will from now on abbreviate $T_n = T^{1/p^n}$, $\pi_n = \pi^{1/p^n}$ and so on. The element $\gamma = 1 \in \Gamma$ acts on $T_n$ and $S_n$ as follows:
\begin{align*}
    \gamma T_n &= T_n + \pi_n^{p-1} T_n^p + \pi_n^p T_n^{p+1},\\
    \gamma S_n &= S_n \cdot \sum_{k=0}^\infty (\pi_n^{p-1} T_n^{p-1} + \pi_n^p T_n^p)^k.
\end{align*}
As in the proof of \cref{rslt:smooth-maps-of-adic-spaces-are-cohom-smooth}, the claim reduces to showing that the map
\begin{align*}
    \Spec(\widetilde A)/\Gamma \to \Spec(R)
\end{align*}
is $\D^a_\solid$-smooth. Since $\Spec(A_\solid)$ is cohomologically smooth over $\Spec(R)$ and $\Gamma$ acts trivially on $A$, it suffices to check that the map $h\colon \Spec(\widetilde A)/\Gamma \to \Spec(A)$ is $\D^a_\solid$-smooth. We make the following preliminary observations, for each $n \ge 0$:
\begin{enumerate}[(a)]
    \item $A_n$ is perfect as an $A$-module\footnote{We warn the reader that $A_n$ is not perfect as an $A_0$-module, although it is of course finitely presented.} and $A_n^\vee := \IHom_A(A_n, A) \isom A_n[-1]$ as $A_n$-modules.

    \item The map $A_n \injto A_{n+1}$ is split injective, where the splitting $A_{n+1} \to A_n$ is given by forgetting the coefficients of $T_{n+1}^k$ and $S_{n+1}^k$ for all $k$ coprime to $p$.
\end{enumerate}
With these preparations at hand, we now prove the following crucial claim:
\begin{itemize}
    \item[($*$)] For every $\varepsilon \in \mm_C$ there is an integer $k \ge 0$ such that for every integer $m \ge 0$ the $A_m$-module
    \begin{align*}
        \cofib(A_{m+k} \injto \widetilde A)^{H_m}
    \end{align*}
    is killed by $\varepsilon$, where $H_m := p^m \Z_p \subset \Gamma$.
\end{itemize}
To prove claim ($*$) we will first prove an intermediate claim: For all integers $n, m \ge 0$ all the cohomologies of the complex $(A_n/A_{n-1})^{H_m}$ are killed by $\varepsilon_{n,m} := (\varpi_n \pi_n)^{p^{m+1} + p}$. Let us first argue why the intermediate claim implies ($*$). By using the long exact cohomology sequence of the fiber sequence
\begin{align*}
    (A_{n+\ell}/A_{n-1})^{H_m} \to (A_{n+\ell+1}/A_{n-1})^{H_m} \to (A_{n+\ell+1}/A_{n+\ell})^{H_m}
\end{align*}
for $\ell \ge 0$, we inductively deduce that all cohomologies of the complex $(A_{n+\ell}/A_{n-1})^{H_m}$ are killed by $\varepsilon_{n,m,\ell} := \varepsilon_{n,m} \varepsilon_{n+1,m} \dots \varepsilon_{n+\ell,m}$. In particular all cohomologies are killed by $\varepsilon_{n,m}^2$ and by passing to the colimit over $\ell$ we deduce that all cohomologies of $(\widetilde A/A_{n-1})^{H_m}$ are killed by $\varepsilon_{n,m}^2$. Since this complex lives in only two degrees, we deduce that $(\widetilde A/A_{n-1})^{H_m}$ is killed by $\varepsilon_{n,m}^4$ (see \cite[Lemma~05QP]{stacks-project}). This easily implies ($*$).

We now prove the intermediate claim, so let $n, m \ge 0$ be given. We denote from now on $q := p^{m+1}$ and $\gamma_m := p^m \in \Gamma$. Then
\begin{align*}
    (A_n/A_{n-1})^{H_m} = \fib(A_n/A_{n-1} \xto{\gamma_m - \id} A_n/A_{n-1}),
\end{align*}
i.e. the two cohomologies of $(A_{n}/A_{n-1})^{H_m}$ are given by the kernel and cokernel of the map $\gamma_m - \id$. We will separately show that both are killed by $\varepsilon_n := \varepsilon_{n,m} = (\varpi_n\pi_n)^{q+p}$. To this end, we first make the following preliminary computations: For an integer $k > 0$ we have
\begin{align*}
    (\gamma_m - \id) T_n^k &= k \pi_n^{q-1} T_n^{q+k-1} + k \pi_n^q T_n^{q+k} + \pi_n^{2q-2} T_n^{2q+k-2} r(T_n),\\
    (\gamma_m - \id) S_n^k &= S_n^k (k \pi_n^{q-1} T_n^{q-1} + k \pi_n^q T_n^q + \pi_n^{2q-2} T_n^{2q-2} r'(T_n)),
\end{align*}
for certain polynomials $r(T_n)$ and $r'(T_n)$ in $T_n$.

We now handle the kernel of $\gamma_m - \id$ on $A_{n}/A_{n-1}$, so fix some element $f$ in there. The element $f$ has a unique representation of the form $f = \sum_{k>0} b_k S_n^k + \sum_{k>0} a_k T_n^k$, where we only need to consider $k$ that are coprime to $p$ (otherwise $S_n^k$ and $T_n^k$ lie in $A_{n-1}$ and are thus killed in $A_n/A_{n-1}$) and of course only finitely many $b_k$ and $a_k$ are non-zero. We need to show that $\varepsilon_n f = 0$. Suppose this is not the case and suppose that there is some $k > 0$ such that $\varepsilon_n b_k \ne 0$. We consider the maximal such $k$. First assume $k > q$ and let $k_0 < k$ be the largest integer such that $p \divides k_0$. Then by the above formulas and the assumption on $k$, for every $\ell = k_0, \dots, k-1$ the coefficient of $S_n^{l-q+1}$ in $(\gamma_m - \id) f$ is
\begin{align}
    b_\ell \ell \pi_n^{q-1} \varpi_n^{q-1} + b_{\ell+1} (\ell+1) \pi_n^q \varpi_n^q. \label{eq:Div1-smoothness-computation-of-coefficient}
\end{align}
All of these coefficients must be $0$ because $l-q+1$ is not divisible by $p$ and hence $S_n^{l-q+1}$ does not lie in $A_{n-1}$. For $\ell = k_0$ this shows $\pi_n^q \varpi_n^q b_{k_0+1} = 0$. By induction on $\ell$ we deduce that $\pi_n^{q+\ell-k_0} \varpi_n^{q+\ell-k_0} b_{\ell+1} = 0$ for $\ell = k_0, \dots, k-1$. In particular $\varepsilon_n b_k = 0$, contradiction! This shows that $k < q$. Take $k_0$ as before. Then for $\ell = k_0, \dots, k-1$ the coefficient of $T_n^{q-1-\ell}$ in $\varpi_n^{q-\ell-1} (\gamma_m - \id) f$ is as in \cref{eq:Div1-smoothness-computation-of-coefficient}, so we can argue as before. This altogether shows that $\varepsilon_n b_k = 0$ for all $k$. Now suppose that $\varepsilon_n a_k \ne 0$ for some $k$ and pick the smallest such $k$. Let $k_0 > k$ be the smallest integer such that $p \divides k$. For $\ell = k+1, \dots, k_0$ we consider the coefficient of $T_n^{q+\ell-1}$ in $\varpi_n^{q+p} (\gamma_m - \id) f$ and argue again as above in order to deduce $\varepsilon_n a_k = 0$. We conclude that $\varepsilon_n f = 0$, finishing the proof that $H^0((A_n/A_{n-1})^H)$ is killed by $\varepsilon_n$.

It remains to show that the cokernel of $\gamma_m - \id$ on $A_n/A_{n-1}$ is killed by $\varepsilon_n = \varpi_n^{q+p} \pi_n^{q+p}$. For this we need to see that $\varepsilon_n S_n^k = 0$ and $\varepsilon_n T_n^k = 0$ in the cokernel, for each integer $k > 0$ that is coprime to $p$. In fact it is enough to show that $\varepsilon_n S_n^k$ and $\varepsilon_n T_n^k$ are multiples of $\pi_n \varepsilon_n$ in the cokernel---then inductively they are multiples of $\pi_n^\ell \varepsilon_n$ for all $\ell \ge 1$ and hence they must be $0$. Let us first consider $\varepsilon_n T_n^k$ for $k > q$ coprime to $p$. By the above computation of $(\gamma_m - \id) T_n^{k-q}$ we deduce that in the cokernel we have
\begin{align*}
    \varepsilon_n T_n^k = \frac{\varepsilon_n}{\pi_n} T_n^{k-1} + \pi_n^{q-2} \varepsilon_n k^{-1} r(T_n).
\end{align*}
If $k - 1$ is divisible by $p$ then $T_n^{k-1}$ lies in $A_{n-1}$ and hence vanishes in the cokernel, and we are done. If not, we apply the same procedure to $k-1$ and so on. After at most $p-1$ steps we arrive at the desired conclusion. One can similarly handle $\varepsilon_n T_n^k$ for $k < q$ and $\varepsilon_n S_n^k$ for $k > 0$ by employing the above computation of $(\gamma_m - \id) S_n^k$. This finally finishes the proof of the intermediate claim and hence the proof of claim ($*$) above.

We now come back to the proof of the $\D^a_\solid$-smoothness of the map $h\colon \Spec(\widetilde A)/\Gamma \to \Spec(A)$. Let us first check that $h$ is $\D^a_\solid$-suave. By \cref{rslt:checking-smoothness-of-quotient-by-profin-group-criterion} this reduces to showing that for all $m \ge 0$ the $A$-module $\widetilde A^{H_m}$ is dualizable in $\D^a_\solid(A)$. By \cite[Proposition~3.7.5(iii)]{mann-mod-p-6-functors} this reduces to showing that $\widetilde A^{H_m}$ is weakly almost perfect i.e. for every $\varepsilon \in \mm_C$ it is an $\varepsilon$-retract of a perfect $A$-module. But from ($*$) it follows easily that $\widetilde A^{H_m}$ is an $\varepsilon$-retract of $A_{m+k}^{H_m}$ and the latter is indeed a perfect $A$-module (by observation (a) above and the explicit formula for $H_m$-cohomology).

To finish the proof, it remains to show that the dualizing complex $\omega_h = h^! 1$ of $h$ is invertible and concentrated in degree $0$ (as the dualizing complex of $A$ over $R$ is concentrated in cohomological degree $-2$). As in the proof of \cref{rslt:smooth-maps-of-adic-spaces-are-cohom-smooth} we compute
\begin{align*}
    \omega_h = \varinjlim_{m \ge 0} (\widetilde A^{H_m})^\vee = \varinjlim_{m\ge0} \big[ \widetilde A^\vee \xto{\gamma_m^\vee - \id} \widetilde A^\vee \big][1] = \varinjlim_{m\ge0} (\widetilde A^\vee)^{H_m}[1].
\end{align*}
Here the $\widetilde A$-module structure on the colimit comes as the colimit of the $A_m$-module structures on $\widetilde A^{H_m}$ (where we observe that $H_m$ acts trivially on $A_m$). Note that using observation (a) we have for all $n \ge 0$
\begin{align*}
    \widetilde A^\vee = \IHom_A(\widetilde A, A) = \IHom_{A_n}(\widetilde A, \IHom_A(A_n, A)) = \IHom_{A_n}(\widetilde A, A_n)[-1]
\end{align*}
Using observations (a) and (b) we further note that $\widetilde A^\vee = \prod_n \IHom_A((A_n/A_{n-1}), A)$ and this complex is concentrated in cohomological degree $1$. Altogether we see that $\widetilde A^\vee$ is up to a shift by $-1$ simply the (classical) $A_n$-module of $A_n$-linear maps $\widetilde A \to A_n$. In particular $\omega_h$ is a priori concentrated in degrees $0$ and $1$. We claim that $\omega_h$ is concentrated in degree $0$ and invertible in $\D^a(\widetilde A)$.

We first show that $H^1(\omega_h) = 0$. By the above description of $\omega_h$ we see that $H^1(\omega_h) = \varinjlim_m H^1(\omega_{h,m})$, where $H^1(\omega_{h,m}) = \coker(\gamma_m^\vee - \id\colon \widetilde A^\vee[1] \to \widetilde A^\vee[1])$. Here the transition maps are given by
\begin{align*}
    \alpha_m := (\gamma_m^\vee - \id)^{p-1}\colon H^1(\omega_{h,m}) \to H^1(\omega_{h,m+1})
\end{align*}
Fix some $m \ge 0$, some $f \in H^1(\omega_{h,m})$ (which we identify with an $A_m$-linear map $\widetilde A \to A_m$) and some $\varepsilon \in \mm_C$. We need to see that the image of $\varepsilon f$ in $H^1(\omega_{h,m'})$ is $0$ for $m' \gg m$. We fix some $k \ge 0$ as in claim ($*$) and we claim that $\alpha_m(\varepsilon f)$ is divisible by $\varepsilon \pi_{k+1}$. This is enough to conclude, because then we can run the same argument for $m+1$ in place of $m$ to deduce that $\alpha_{m+1}(\alpha_m(\varepsilon f))$ is divisible by $\varepsilon \pi_{k+1}^2$ and so on.

Let $f' := \alpha_m f$ and write $f' = f'_1 + f'_2$, where $f'_1 := f' \comp i_{m+k}$, with $i_{m+k}$ being the composition $\widetilde A \surjto A_{m+k} \injto \widetilde A$; in other words, $f'_1$ is obtained from $f'$ by forgetting the values of $f'$ on $S_n$ and $T_n$ for $n > m+k$. By ($*$) the fiber of $\pr\colon (\widetilde A^\vee)^{H_{m+1}} \to (A_{m+k}^\vee)^{H_{m+1}}$ is killed by $\varepsilon$. Since $\pr(f'_2) = \pr(f' - f'_1) = 0$, we deduce that $\varepsilon f'_2 = 0$. Thus it only remains to prove that $f'_1$ is divisible by $\pi_{k+1}$. Since $\alpha_m$ is an iterated composition of $(\gamma_m^\vee - \id)$, it is now enough to show that the restriction of $(\gamma_m^\vee - \id) f$ to $A_{m+k}$ is divisible by $\pi_{k+1}$. But for all $\ell \ge 0$ we have
\begin{align*}
    ((\gamma_m^\vee - \id) f)(T_{m+k}^\ell) = f((\gamma_m - \id) T_{m+k}^\ell), \qquad ((\gamma_m^\vee - \id) f)(S_{m+k}^\ell) = f((\gamma_m - \id) S_{m+k}^\ell),
\end{align*}
so the claim reduces to the observation that the formulas for $(\gamma_m - \id) T_{m+k}^\ell$ and $(\gamma_m - \id) S_{m+k}^\ell$ are divisible by $\pi_{k+1}$.

We have shown that $H^1(\omega_h) = 0$, so in particular $\omega_h = H^0(\omega_h)$ is concentrated in degree $0$. We now show that $\omega_h$ is invertible in $\D^a(\widetilde A)$. For $n \ge 0$ let $B_n := C\langle T_n\rangle[T_n^{-1}]$. Then for $k \ge 0$ the ring $B_{n+k}$ is free as a $B_n$-module with basis given by $T_{n+k}^\ell$ for $\ell = 0, \dots, p^k-1$. There is a different basis of $B_{n+k}$ over $B_n$ given by the elements $(1 + \pi_{n+k} T_{n+k})^\ell$ for $\ell = 0, \dots, p^k - 1$. In particular, projecting to the coefficient of $1$ yields a splitting $B_{n+k} \to B_n$ of the canonical inclusion. Now let $B^+_n \subset B_n$ be the $\ri_C$-subalgebra generated by $T_n$ and $\frac{\varpi_n}{T_n}$, i.e. $B^+_n \isom \ri_C\langle T_n\rangle[\frac{\varpi_n}{T_n}]$ and $B^+_n/\pi = A_n$. We observe that the splitting $B_{n+k} \to B_n$ restricts to a $B^+_n$-linear map $\pi_n B^+_{n+k} \to B^+_n$ (because the relevant base-change matrices have a maximum of $\pi_n$ in the denominator) and hence to an $A_n$-linear map $\pi_n A_{n+k} \to A_n$. These maps are compatible for varying $k$ and thus induce an $A_n$-linear map $\pi_n \widetilde A \to A_n$. Now note that for an integer $m \ge 0$ we have
\begin{align*}
    \gamma_m (1 + \pi_{n+k} T_{n+k})^\ell = (1 + \pi_{n+k} T_{n+k})^{\ell(p^m + 1)}.
\end{align*}
We deduce that for $m \ge n$ (so that $H_m$ acts trivially on $A_n$) the map $\pi_n\widetilde A \to A_n$ from above is $H_m$-equivariant. We thus obtain a map $(A_n^\vee)^{H_m} \to (\widetilde A^\vee)^{H_m}$ and by evaluating this on $H^1$ and passing to the colimit over $m$ we obtain an $A_n$-linear map $A_n \to \omega_h$ and hence an induced $\widetilde A$-linear map $s_n\colon \widetilde A \to \omega_h$. Pick some $\varepsilon \in \mm_C$ and let $n \ge 0$ be big enough so that $\pi_n \divides \varepsilon$. We claim that $\fib(s_n)$ is killed by $\varepsilon^2$, or equivalently that both $\ker(s_n)$ and $\coker(s_n)$ are killed by $\varepsilon$.

To handle the kernel of $s_n$, suppose that $s_n(a) = 0$ for some $a \in \widetilde A$. Pick $n' \ge n$ big enough so that $a \in A_{n'}$. Then $s_n(a) = 0$ implies that the $A_n$-linear map
\begin{align}
    A_{n'} \xto{\pi_n a} \pi_n A_{n'} \to A_n \label{eq:composition-of-section-in-invertibility-on-Spa-Qp-proof}
\end{align}
is zero, where the second map is the section from above. By writing $\pi_n a$ in coordinates for the basis $(1 + \pi_{n'} T_{n'})^\ell$ from above (more precisely, we lift $\pi_n a$ to $B^+_{n'}$ and then use the above basis for $B_{n'}$) and plugging $(1 + \pi_{n'} T_{n'})^{-\ell}$ into the above zero map, we deduce that $\pi_n a = 0$; hence also $\varepsilon a = 0$. This shows that $\ker(s_n)$ is killed by $\varepsilon$.

To handle the cokernel of $s_n$, fix some $f \in \omega_h$. We have to show that $\varepsilon f$ lies in the image of $s_n$. Pick $m \ge n$ big enough so that $f \in \omega_{h,m}$ and pick $k \ge 0$ as in ($*$) for the given $\varepsilon$. We may view $f$ as an $A_n$-linear $H_m$-equivariant map $\widetilde A \to A_n$ and by our choice of $k$ we know that $\varepsilon f$ factors over the canonical projection $\widetilde A \to A_{m+k}$. Thus in order to show that $\varepsilon f$ lies in the image of $s_n$, it is enough to show the same claim after restricting to $A_{m+k}$ in place of $\widetilde A$. Thus for $n' := m + k$ we have reduced the claim to showing that for an $A_n$-linear map $f\colon A_{n'} \to A_n$ there is some $a \in A_{n'}$ such that $\varepsilon f$ is equal to the composition in \cref{eq:composition-of-section-in-invertibility-on-Spa-Qp-proof}. But the appropriate coordinates for $a$ in terms of the basis $(1 + \pi_{n'} T_{n'})^\ell$ are given by evaluating $\frac{\varepsilon}{\pi_n} f$ on $(1 + \pi_{n'} T_{n'})^{-\ell}$. This finishes the proof that $\coker(s_n)$ is killed by $\varepsilon$.

Altogether we have shown that for every $\varepsilon \in \mm_C$ there is a $\widetilde A$-linear map $\widetilde A \to \omega_h$ whose fiber is killed by $\varepsilon$. This shows that $\omega_h$ is an $\varepsilon$-retract of $\widetilde A$ for every $\varepsilon$, hence $\omega_h$ is weakly almost perfect. By \cite[Proposition~3.7.11]{mann-mod-p-6-functors} we conclude that $\omega_h$ is dualizable in $\D^a(\widetilde A)$. Thus by \cite[Remark~4.5.12]{heyer-mann-6ff} we deduce that $\omega_h$ is invertible, as desired.
\end{proof}

\begin{corollary} \label{sec:case-texorpdfstr-smoothness-of-spd-q-p-mod-hi}
The map $\Spd \Q_p/\varphi^{\Z} \to \Spd(\F_p)$ is $\D_{[0,\infty)}$-smooth and $\D_{[0,\infty)}$-proper, with dualizing complex $\omega \in \DFF(\Spd \Q_p,\Z_p)$ given by
$$
    \omega = \mathrm{RH}_{\Z_p}(\chi_{\mathrm{cycl}})[2],
$$
where $\chi_{\mathrm{cycl}}$ denotes the cyclotomic character.
\end{corollary}
Again, the same statements hold true for the $\D_{(0,\infty)}(-)$-formalism.
\begin{proof}
Cohomological smoothness immediately follows from \Cref{sec:case-texorpdfstr-smoothness-of-spd-q-p}. From \Cref{rslt:construction-of-RH-functors} proved below we conclude that $\omega=\mathrm{RH}_{\Z_p}(\mathbb{L}_\chi)[2]$ for a $\Z_p$-local system on $\Spd(\Q_p)$ associated with a continuous character $\chi\colon \mathrm{Gal}(\overline{\Q_p}/\Q_p)\to \Z_p^\times$. It suffices to identify the $\Q_p$-local system $\mathbb{L}_{\chi}[1/p]$, and hence to identify the image of $\omega$ in $\D_{(0,\infty)}(\Spd(\Q_p)/\varphi^\Z)$. Here, we can invoke the classical calculation of cohomology of rank $1$ $(\varphi,\Gamma)$-modules (e.g., \cite[Proposition 6.2.8]{kedlaya2014cohomology}) to see that $\chi$ must be given by the cyclotomic character $\chi_{\cycl}$. More precisely, it is sufficient to know which characters have non-trivial $H^0, H^2$, and then one sees that only the cyclotomic character can yield the duality. The $\D_{[0,\infty)}(-)$-properness follows from \Cref{rslt:6ff-for-D-0-infty} as $\Spd(\Q_p)/\varphi^\Z\to \Spd(\F_p)$ is proper and lpbc.
\end{proof}

\begin{remark}
Let $A\in \CAlg(\D_{[0,\infty)}(\Spd(\F_p)))$, e.g., $A$ could be pulled back from $\D_\solid(\Z_p)$. Using the results of \Cref{sec:abstract-results-6} we obtain the $6$-functor formalism $S\mapsto \mathrm{Mod}_A \D_{[0,\infty)}(S)$ on small v-stacks. For suitable $A$, dualizable objects in $\mathrm{Mod}_A\D_{[0,\infty)}(\Spd(\Q_p)/\varphi^\Z)$ should identify with (perfect complexes) of $(\varphi,\Gamma)$-modules with coefficients in $A$ (see \Cref{relation-phi-gamma-modules} for the case $A=\Q_p$). Hence, \Cref{sec:case-texorpdfstr-smoothness-of-spd-q-p} should imply duality and finiteness statements for families of $(\varphi,\Gamma)$-modules. See \cite{liu2007cohomology} and the recent paper \cite{mikami2024varphi}.
\end{remark}

\section{Quasi-coherent sheaves on Fargues--Fontaine curves and categories of \texorpdfstring{$\Z_p$}{Z\_p}- and \texorpdfstring{$\Q_p$}{Q\_p}-sheaves} \label{sec:an-overc-riem}

One of the main objectives of this text is to find a nice conceptual framework for understanding properties, such as finiteness or duality, of the pro-\'etale $\Q_p$-cohomology of rigid analytic spaces (or with $\Z_p$-coefficients, but for this short introduction we stick to the case of $\Q_p$). We just achieved the construction of a $6$-functor formalism for $\D_{(0,\infty)}(-)$. The goal of this section is therefore twofold. First, we need to see how to interpret the pro-\'etale cohomology of $\Q_p$, or of more general pro-\'etale $\Q_p$-local systems, inside this $6$-functor formalism. This is realized by the construction, valid for any small v-stack $S$, of a fully faithful Riemann-Hilbert functor from a certain category of nuclear $\Q_p$-sheaves on $S$ to $\DFF(S,\Q_p)=\D_{(0,\infty)}(S/\varphi^{\Z})$, similar to the results of \cite[\S 3.9]{mann-mod-p-6-functors}. This in particular tells us that we compute pro-\'etale $\Q_p$-cohomology inside our $6$-functor formalism $\DFF(-,\Q_p)$. Coupled with appropriate cohomological smoothness and properness results, this leads to finiteness and duality statements, but formulated at the level of $\DFF(-,\Q_p)$. To deduce more familiar looking statements, we therefore need to ``go back'' to the world of $\Q_p$-sheaves and interpret the results obtained there. The second objective of this section is to make this precise. Here, we do not obtain the most general, desirable results, but rather statements tailored to the applications we need.  

\subsection{The Riemann--Hilbert functors}
\label{sec:riem-hilb-funct}

In this subsection we want to construct $\Z_p$- and $\Q_p$-versions of a Riemann--Hilbert functor for overconvergent sheaves. The discussion is similar to the one leading to \cite[Theorem 3.9.23]{mann-mod-p-6-functors}. We note that for $\F_p$-coefficients our Riemann--Hilbert functor takes values in $\varphi$-modules for the tilted structure sheaf $\ri^\flat$ rather than for the sheaf $\ri^{+,\flat,a}/\pi$, see the end of this paragraph for a short related discussion.

\begin{definition} \label{sec:an-overc-riem-2-d-nuc-for-a-diamond}
Let $S$ be a diamond. We denote by $\D_\nuc(S,\Z_p)$ the value of the hypercomplete quasi-pro-\'etale sheaf from \cite[Lemma 3.30]{anschuetz_mann_descent_for_solid_on_perfectoids}, i.e., if $S$ is a $p$-bounded spatial diamond, then $\D_\nuc(S,\Z_p)$ is equivalent to the category of nuclear objects in the category of $\omega_1$-solid sheaves on $S_\qproet$ (see \cite[Definition 3.26]{anschuetz_mann_descent_for_solid_on_perfectoids}). For a nuclear $\Z_p$-algebra $\Lambda$ we set $\D_\nuc(S,\Lambda) := \Mod_\Lambda \D_\nuc(S,\Z_p)$.
\end{definition}

If $S$ is a strictly totally disconnected perfectoid space, then \cite[Lemma 3.29]{anschuetz_mann_descent_for_solid_on_perfectoids} implies that $\D_\nuc(S,\Z_p)\cong \D_\nuc(C(S,\Z_p))$, where $C(S,\Z_p)$ denotes the $p$-complete and nuclear $\Z_p$-algebra of continuous functions $|S|\to \Z_p$ (equivalently, of continuous functions $\pi_0(S)\to \Z_p$).

\begin{remark}
We see that $\mathrm{Mod}_{\F_p} \D_{\nuc}(S,\Z_p)\cong \D_{\et}(S,\F_p)^{\mathrm{oc}}$ is equivalent to the category of overconvergent \'etale $\F_p$-sheaves constructed in \cite[Definition 3.9.17]{mann-mod-p-6-functors}. Hence we think of $\D_\nuc(S,\Lambda)$ as a category of ``overconvergent étale $\Lambda$-sheaves''. It embeds into the category of pro-étale sheaves of $\Lambda$-modules and contains all pro-étale $\Lambda$-local systems. We stress that $\D_\nuc(S,\Z_p)$ is very different to the category $\D_\FF(S,\Z_p)$ (or to the category of nuclear objects in the latter).
\end{remark}

\begin{lemma} \label{sec:an-overc-riem-3-v-descent-for-d-nuc}
The functor $S\mapsto \D_\nuc(S,\Z_p)$ is a hypercomplete v-sheaf on the category of diamonds. In particular, it extends uniquely to a hypercomplete v-sheaf $S\mapsto \D_\nuc(S,\Z_p)$ on the category of small v-stacks.
\end{lemma}
\begin{proof}
By quasi-pro-étale descent, we are reduced to showing v-hyperdescent on strictly totally disconnected perfectoid spaces, so let $f_\bullet\colon S_\bullet\to S$ be a v-hypercover consisting of strictly totally disconnected perfectoid spaces. Then $g_\bullet:=\pi_0(f_\bullet)\colon T_\bullet:=\pi_0(S_\bullet)\to T:=\pi_0(S)$ is a v-hypercover of profinite sets. As $\D_\nuc(S_n,\Z_p)\cong \D_\nuc(C(T_n,\Z_p))\cong \mathrm{Mod}_{C(T_n,\Z_p)}(\D_\nuc(\Z_p))$ it suffices to show that $T\mapsto \D_\nuc(C(T,\Z_p))$ is a hypercomplete v-sheaf on the pro-\'etale site of profinite sets. This is implied by \cite[Theorem 3.9.(ii)]{mann-nuclear-sheaves} (and \cite[Remark 3.10]{mann-nuclear-sheaves} to identify the two different notions of nuclear sheaves).
\end{proof}

Before we can come to the promised construction of the Riemann--Hilbert functors, we need the following preparations.

\begin{lemma} \label{rslt:dualizable-object-on-FF-curve-is-compact}
Let $S$ be a qcqs perfectoid space, which is of characteristic $p$ and which admits a morphism of finite $\dimtrg$ to a totally disconnected perfectoid space. Then each dualizable object in $\DFF(S,\Z_p)$ or $\DFF(S,\Q_p)$ is compact, even compact in the enriched sense over $\D_\solid(\Z_p)$.
\end{lemma}
\begin{proof}
We may reduce to the case that $S$ is affinoid. By assumption, $\DFF(S,\Z_p)$ resp.\ $\DFF(S,\Q_p)$ identify with the category of $\varphi$-modules on $\mathcal{Y}_{[0,\infty),S}$ resp.\ $\mathcal{Y}_{(0,\infty),S}$. Fix a pseudo-uniformizer $\pi$ on $S$ with associated coordinate function $\kappa\colon \mathcal{Y}_{[0,\infty),S}\to [0,\infty)$. We know that $\kappa(\varphi(-))=p\kappa(-)$. For an interval $I\subseteq [0,\infty)$ let $U_{I}$ be (the interior of) $\kappa^{-1}(I)$. By analytic descent for $\D_{\hat\solid}$ (see \Cref{sec:defin-d_hats--1}) we can conclude that
\[
    \DFF(S,\Z_p)\cong \mathrm{eq}(\D_{\hat\solid}(U_{[0,1]})\rightrightarrows \D_{\hat\solid}(U_{[0,1/p]})), 
\]
where the two morphisms are induced by the inclusion $U_{[0,1/p]}\to U_{[0,1]}$ and the inverse of the Frobenius $U_{[0,1/p]}\overset{\varphi}{\to} U_{[0,1/p]}\to U_{[0,1]}$ (followed by restriction). As $U_{[0,1]}, U_{[0,1/p]}$ are affinoid, their structure sheaf is compact in $\D_{\hat\solid}$. As the above equalizer is a finite limit we can conclude that $1\in \DYwz{S/\varphi^\Z}$ is compact. This implies that each dualizable object in $\DFF(S,\Z_p)$ is compact. The argument for $\DFF(S,\Q_p)$ is similar by glueing $\mathcal{Y}_{(0,\infty),S}/\varphi^\Z$ from $U_{[1,p]}$ along the Frobenius on the ``boundary'' $U_{[1]}\cup U_{[p]}$. The compactness in the enriched sense, i.e., that the enriched homomorphisms out of a dualizable object preserve colimits, follows similarly.
\end{proof}

\begin{remark}
We note that \cref{rslt:dualizable-object-on-FF-curve-is-compact} is wrong for $\D_{[0,\infty)}(S)$ or $\D_{(0,\infty)}(S)$.
\end{remark}

\begin{lemma} \label{rslt:automatic-fully-faithfulness-for-rigid-cat}
Let $F\colon \calC\to \mathcal{R}$ be a morphism in $\CAlg(\Pr^L_{\mathrm{Sp}})$. Assume that $\calC$ is rigid and that $1\in \mathcal{R}$ is compact. Then $F$ is fully faithful if and only if the natural morphism
$$1_{\mathcal{C}}\to \Hom^{\calC}_{\mathcal{R}}(1_{\mathcal{R}},1_{\mathcal{R}})
$$ of the unit in $\calC$ to the $\calC$-linear endomorphism object of $1_{\mathcal{R}} \in \mathcal{R}$ is an isomorphism.
\end{lemma}
\begin{proof}
Let $G$ be the right adjoint of $F$. By \cite[Lemma 3.33(iii)]{anschuetz_mann_descent_for_solid_on_perfectoids} the functor $G$ commutes with colimits. By \cite[Lemma 3.33(ii)]{anschuetz_mann_descent_for_solid_on_perfectoids} the functor $G$ is therefore $\calC$-linear. Now, $F$ is fully faithful if and only if for all $X\in \calC$ the natural morphism $X\to GF(X)$ is an isomorphism. Now, $GF(X)\cong G(F(X)\otimes_{\mathcal{R}} 1_{\mathcal{R}})\cong X\otimes_{\mathcal{C}} GF(1_{\mathcal{C}})$, and hence $F$ is fully faithful if and only if the natural morphism $1_{\mathcal{C}}\to GF(1_{\mathcal{C}})=G(1_{\mathcal{R}})$ is an isomorphism. If $X\in \mathcal{C}$, then
\[
    \Hom_{\calC}(X,GF(1_{\mathcal{C}}))\cong \Hom_{\mathcal{R}}(F(X)\otimes_{\mathcal{R}} 1_{\mathcal{R}}, 1_{\mathcal{R}}),
\]
i.e. $GF(1_{\mathcal{C}})$ is the $\calC$-linear endomorphism object of $1_{\mathcal{R}} \in \mathcal{R}$. This finishes the assertion.
\end{proof}

We can now come to the construction of our Riemann--Hilbert functors, relating nuclear $\Z_p$- and $\Q_p$-sheaves to the 6-functor formalism of quasi-coherent sheaves on the Fargues--Fontaine curve.

\begin{theorem} \label{rslt:construction-of-RH-functors}
Let $S$ be a small v-stack. There exist symmetric monoidal, colimit preserving, fully faithful functors
\begin{align*}
    \RH_{S,\Z_p}&\colon \D_\nuc(S,\Z_p)\to \DFF(S,\Z_p),\\
    \RH_{S,\Q_p}&\colon \D_\nuc(S,\Q_p)\to \DFF(S,\Q_p),
\end{align*}
such that the following hold true:
\begin{thmenum}
    \item Both Riemann--Hilbert functors are compatible with pullback, i.e. they upgrade to natural transformations of functors from $\vStacks^\opp$ to the category of symmetric monoidal categories.
    \item \label{rslt:RH-identifies-dualizable-objects} The functor $\RH_{S,\Z_p}$ identifies dualizable objects on both sides.
    \item The functors $\RH_{S,\Z_p}$ and $\RH_{S,\Q_p}$ are compatible via the natural functors $\D_\nuc(S,\Z_p)\to \D_\nuc(S,\Q_p)$, $\DFF(S,\Z_p)\to \DFF(S,\Q_p)$.
\end{thmenum}
\end{theorem}
\begin{proof}
By v-descent on both sides (see \Cref{sec:an-overc-riem-3-v-descent-for-d-nuc} and \Cref{rslt:v-descent-for-D-0-infty}) we may reduce to the case that $S$ is strictly totally disconnected. To simplify notation we set $\D:=\DFF(S,\Z_p)$ and $\D':=\D_{[0,\infty)}(S)$. Note that $\D$ is the category of $\varphi$-modules in $\D'$. Via pullback $\D$ is naturally a $\D_\solid(\Z_p)$-linear category, and thus enriched over $\D_\solid(\Z_p)$. Let $\mathrm{1}\in \D$ be the unit object, and $A:=\mathrm{Hom}^{\D_\solid(\Z_p)}_{\D}(\mathrm{1},\mathrm{1})\in \D_\solid(\Z_p)$ its enriched endomorphism object. Thus, for $M\in \D_\solid(\Z_p)$ there exists a natural isomorphism
\[
    \Hom_{\D_\solid(\Z_p)}(M,A) = \Hom_{\D}(M\otimes \mathrm{1},\mathrm{1}),
\]
where $M\otimes(-)$ denotes the $\D_\solid(\Z_p)$-action on $\D$. We claim that there exists a natural isomorphism $\Phi\colon C(S,\Z_p)\to A$ of solid $\Z_p$-algebras. Let $T$ be a profinite set. Choose a pseudo-uniformizer on $S$ with corresponding radius function $\kappa\colon \mathcal{Y}_{[0,\infty),S}\to [0,\infty)$. Choose a sequence $r_n\in [0,\infty)$ with $r_n\to \infty$ for $n\to \infty$, and set $U_n:=\mathcal{Y}_{[0,r_n],S}=\Spa(B_n,B_n^+)$. We calculate
\begin{align*}
    &\Hom_{\D_\solid(\Z_p)}(\Z_{p,\solid}[T],A) \\
    &\qquad\qquad= \Hom_{\D}(\Z_{p,\solid}[T]\otimes \mathrm{1},\mathrm{1}) \\
    &\qquad\qquad= \Hom_{\D'}(\Z_{p,\solid}[T]\otimes\mathrm{1}_{\D'},\mathrm{1}_{\D'})^{\varphi=1}, \\
\intertext{and by writing $\mathcal{Y}_{[0,\infty),S}=\bigcup\limits_{n} U_n$ and unraveling the definitions,}
    &\qquad\qquad= (\varprojlim_n \Hom_{\D_{\hat\solid}(U_n)}(\Z_{p,\solid}[T]\otimes_{\Z_{p,\solid}} \widehat{(B_n^+)_\solid}\otimes_{B_n^{+}}B_n,B_n))^{\varphi=1}, \\
\intertext{and using adjunction, nuclearity of $B_n$ and \cite[Proposition 2.17]{anschuetz_mann_descent_for_solid_on_perfectoids},}
    &\qquad\qquad= (\varprojlim_n \Hom_{\D_\solid(\Z_p)}(\Z_{p,\solid}[T],B_n))^{\varphi=1} \\
    &\qquad\qquad= (\varprojlim_n \Hom_{\D_\solid(\Z_p)}(\Z_p,\IHom_{\D_\solid(\Z_p)}(\Z_{p,\solid}[T],B_n)))^{\varphi=1}, \\
\intertext{and by nuclearity of $B_n$ in $\D_\solid(\Z_p)$,}
    &\qquad\qquad= (\varprojlim\limits_n \Hom_{\D_\solid(\Z_p)}(\Z_p,C(T,\Z_p)\otimes B_n))^{\varphi=1} \\
    &\qquad\qquad= (\varprojlim\limits_{n} \Hom_{\D_\solid(\Z_p)}(\Z_p, B_n'))^{\varphi=1}, \\
\intertext{where $U_n':=\underline{T}\times U_n=\Spa(B_n',B_n^{\prime,+})$,}
    &\qquad\qquad= \Gamma(\mathcal{Y}_{[0,\infty),\underline{T}\times S},\mathcal{O})^{\varphi=1}, \\
\intertext{and by Artin--Schreier--Witt theory over the integral Robba ring,}
    &\qquad\qquad\isom C(T\times \pi_0(S),\Z_p),
\end{align*}
as desired. This shows that $C(S,\Z_p) \isom A$ as solid $\Z_p$-algebras, which are concentrated in degree $0$. This morphism is natural in $S$ and independent of the choice of the radius function (by a cofinality argument); here we use the fact that both sides of the isomorphism are concentrated in degree $0$, hence the desired functoriality is really a \emph{1-categorical} statement and can thus easily be checked by hand. Furthermore, we note that replacing $\mathcal{Y}_{[0,\infty),S}$ by $\mathcal{Y}_{(0,\infty),S}$, the use of Artin--Schreier--Witt theory by \cite[Proposition II.2.5]{fargues-scholze-geometrization} and the $U_n$ by increasing annuli, the same calculation shows that $C(S,\Q_p)$ identifies with the $\D_\solid(\Q_p)$-linear endomorphism object of $\mathrm{1}\in \DY{S/\varphi^\Z}$.
The existence of $\RH_{S,\Z_p}, \RH_{S,\Q_p}$ is now formal: given a $C(S,\Z_p)$-module in $\D_\nuc(\Z_p)$, we can form the tensor product $M\otimes_{C(S,\Z_p)} 1_{\D}$ using the morphism of $\mathbb E_\infty$-rings $C(S,\Z_p)\to A$ to make $1_{\D}$ into a $C(S,\Z_p)$-module.\footnote{More precisely, for any nuclear $\Z_p$-$\mathbb E_\infty$-algebra $R$, $\D_\nuc(\Z_p)$-linear symmetric monoidal functors $\Mod_R(\D_\nuc(\Z_p))\to \mathcal{D}$ correspond to $\mathbb E_\infty$-algebra morphisms from $R$ into the $\D_\nuc(\Z_p)$-enriched endomorphisms of $1_{\D}$. In fact, \cite[Section 4.8.5]{lurie-higher-algebra}, more specifically \cite[Remark 4.8.5.12, Corollary 4.8.5.21]{lurie-higher-algebra} provide the symmetric monoidal functor $\CAlg(\D_\nuc(\Z_p))\to \CAlg(\Pr^L_{\D_\nuc(\Z_p)}),\ R\mapsto \Mod_R(\D_\nuc(\Z_p))$ with right adjoint sending $\D$ to $\End^{\D_\nuc(\Z_p)}(1_\D)$.} 

Now, fully faithfulness of $\RH_{S,\Z_p}$ and $\RH_{S,\Q_p}$ follow from \Cref{rslt:automatic-fully-faithfulness-for-rigid-cat} and \Cref{rslt:dualizable-object-on-FF-curve-is-compact}. We are left with showing that $\RH_{S,\Z_p}$ identifies dualizable objects. By \Cref{rmk:nuclear-objects-on-stably-uniform-adic-space} dualizable objects in $\D_{[0,\infty)}(S)$ are exactly the perfect complexes on $\mathcal{Y}_{[0,\infty),S}$. By \cite[Proposition 2.6]{anschutz2021fourier} (and the presentation in \Cref{rslt:dualizable-object-on-FF-curve-is-compact}) we can then conclude that each perfect complex with $\varphi$-module structure is actually strictly perfect, i.e., represented by a complex of vector bundles with $\varphi$-module structures (and differential respecting $\varphi$). This reduces the assertion to the case of vector bundles where it follows from Artin--Schreier--Witt theory over the integral Robba ring, \cite[Theorem 8.5.3]{relative-p-adic-hodge-1}.
\end{proof}

\begin{remark}
The classification of vector bundles on the Fargues--Fontaine curve implies that even for $S=\Spa(C,\mathcal{O}_C)$ with $C$ a non-archimedean algebraically closed extension of $\F_p$, the functor $\RH_{\Q_p}\colon \D_{\nuc}(S, \Q_p)\to \DFF(S,\Q_p)$ is not an equivalence on dualizable objects: each vector bundle, which is not semi-stable of slope $0$, does not lie in the essential image.
\end{remark}

\begin{remark}
Finally, we want to relate the Riemann--Hilbert functor $\RH_{\Z_p}$ to the one considered in \cite[Section 3.9]{mann-mod-p-6-functors}. To fix notation assume that $S$ admits a pseudo-uniformizer $\pi$ in the sense of \cite[Definition 3.2.2]{mann-mod-p-6-functors} dividing $p$. Let
\[
    \RH_{\F_p}\colon \D_\et(S,\F_p)^{\mathrm{oc}}\to \D^a_{\solid}(\mathcal{O}^+_S/\pi)^{\varphi}
\]
be the functor constructed in \cite[Definition 3.9.21]{mann-mod-p-6-functors}. We want to relate $\RH_{\F_p}$ to a mod-$p$-version of $\RH_{\Z_p}$. As noted before
\[
    \Mod_{\F_p}(\D_\nuc(S,\Z_p))\cong \D_\et(S,\F_p)^{\mathrm{oc}}.
\]
Moreover, we have
\[
    \DFF(S,\F_p) := \Mod_{\F_p}(\DFF(S,\Z_p))\cong \D_{\hat\solid}(\mathcal{O}_{S/\varphi^\Z}),
\]
which receives a functor from $\D_{\hat\solid}^a(\mathcal{O}^+_{S/\varphi^\Z})$ and this last category maps to $\D_{\solid}^a(\mathcal{O}^+_{S}/\pi)^\varphi$.
From \cite[Lemma 3.9.4]{mann-mod-p-6-functors} we can conclude that the functor
\[
    \RH_{\Z_p}\colon \D_{\et}(S,\F_p)^{\mathrm{oc}}\cong \mathrm{Mod}_{\F_p}\D_{\nuc}(S,\Z_p)\to  \DFF(S,\mathbb{F}_p) 
\]
factors naturally over a functor $\widetilde{\RH}_{\Z_p}$ with values in $\D_{\hat\solid}^a(\mathcal{O}^+_{S/\varphi^\Z})$.
We can derive that there exists the following commutative diagram:
\[\begin{tikzcd}
    & {\D_\et(S,\F_p)^{\mathrm{oc}}} \\
    { \DFF(S,\mathbb{F}_p) } & {\D^a_{\hat\solid}(\mathcal{O}^+_{S/\varphi^\Z})} & {\D^a_{\hat\solid}(\mathcal{O}^+_{S}/\pi)^\varphi}
    \arrow["{\RH_{\Z_p}}"', from=1-2, to=2-1]
    \arrow["{\widetilde{\RH}_{\Z_p}}"', from=1-2, to=2-2]
    \arrow["{\RH_{\F_p}}", from=1-2, to=2-3]
    \arrow[from=2-2, to=2-1]
    \arrow[from=2-2, to=2-3]
\end{tikzcd}\]
All of the above functors induce equivalences on dualizable objects and hence identify dualizable objects in each of these categories with $\Fld_{p}$-local systems of perfect complexes.
Indeed, for $\mathrm{RH}_{\Z_p}$ this follows from \ref{rslt:construction-of-RH-functors} (because dualizable objects over $\F_p$ are dualizable over $\Z_p$), for $\mathrm{RH}_{\F_p}$ this is \cite[Theorem 3.9.23]{mann-mod-p-6-functors}, and for $\widetilde{\RH}_{\Z_p}$ this follows from \cite[Lemma 3.9.2]{mann-mod-p-6-functors} (and the observation that each dualizable object in $\D^a_{\hat\solid}(\mathcal{O}^+_{S/\varphi^\Z})$ is $\pi$-complete as can be checked on affinoids).

As a warning we note that in general $\widetilde{\RH}_{\Z_p}$ does not take values in $\pi$-complete modules (e.g., on strictly totally disconnected perfectoid spaces), and hence does not agree with the composition of $\RH_{\F_p}$ with the equivalence constructed in \cite[Lemma 3.9.2]{mann-mod-p-6-functors}.
\end{remark}

\subsection{Supplements on the primitive comparison theorem}
\label{sec:suppl-prim-comp}

We now want to provide some supplements to the ``primitive comparison theorem'' (\cite[Theorem 3.13]{perfectoid-spaces-survey}, \cite[Theorem 5.1]{rigid-p-adic-hodge}, \cite[Corollary 3.9.24]{mann-mod-p-6-functors}). Let $f\colon Y'\to Y$ be a morphism of small v-stacks and $\mathcal{F}\in \D_\nuc(Y,\Z_p)$. In the language of this paper, the primitive comparison theorem can be formulated as the question whether the natural morphism
\[
    \Phi_{\mathcal{F},f}\colon \RH_{\Z_p,Y}(f^\nuc_\ast(\mathcal{F}))\to f_\ast \RH_{\Z_p,Y'}(\mathcal{F})
\]
is an isomorphism, where $f^\nuc_\ast\colon \D_\nuc(Y',\Z_p)\to \D_\nuc(Y,\Z_p)$ and $f_\ast\colon \DFF(Y',\Z_p) \to \DFF(Y,\Z_p)$ are the natural pushforwards. If this is the case then this tells us that the (derived) pushforward of $\mathcal F$ along $f$ can be computed by the pushforward in $\DFF(-,\Z_p)$. If one takes $\F_p$-coefficients then $\DFF(-,\F_p) = \D_{\hat\solid}(\ri^\flat_{(-)})^\varphi$ and hence the pushforward on the FF-side computes $\ri^{+a}/p$-cohomology, resulting in the more classical formulations of the primitive comparison theorem.

Our first result towards the primitive comparison theorem simplifies the problem slightly by observing that $\Phi_{\mathcal F, f}$ is automatically an isomorphism if the right-hand side lies in the image of the Riemann--Hilbert functor:

\begin{lemma} \label{rslt:primitive-comparison-reduces-to-image-of-RH}
In the above setup, the morphism $\Phi_{\mathcal{F},f}$ is an isomorphism if and only if $f_\ast(\RH_{\Z_p,Y'}(\mathcal{F}))$ lies in the essential image of $\RH_{\Z_p,Y}$.  
\end{lemma}
\begin{proof}
Necessity is clear, so assume that conversely $f_\ast(\RH_{\Z_p,Y'}(\mathcal{F}))$ lies in the essential image of $\RH_{\Z_p,Y}$.
Let $\mathrm{Sol}_{\Z_p,-}$ be the right adjoint to $\mathrm{RH}_{\Z_p,-}$ (which exists for formal reasons). By fully faithfulness of $\RH_{\Z_p}$ (see \Cref{rslt:construction-of-RH-functors}) we see that
\begin{align*}
    f_\ast(\RH_{\Z_p,Y'}(\mathcal{F})) & = \RH_{\Z_p,Y}\mathrm{Sol}_{\Z_p,Y} f_\ast(\RH_{\Z_p,Y'}(\mathcal{F})) \\
    & = \RH_{\Z_p,Y}(f^\nuc_{\ast}\circ \mathrm{Sol}_{\Z_p,Y'}\circ \RH_{\Z_p,Y'}(\mathcal{F})) \\
    & = \RH_{\Z_p,Y}f^\nuc(\mathcal{F}). 
\end{align*}
Here, we used that $\mathrm{Sol}$ commutes with $\ast$-pushforward as $\RH$ commutes with $\ast$-pullback.
\end{proof}
 
The cases in which $\Phi_{\mathcal{F},f}$ is an isomorphism require special assumptions. The next lemma is a $\Z_p$-version of \cite[Corollary 3.9.24]{mann-mod-p-6-functors} (which in turn implies \cite[Theorem 3.13]{rigid-p-adic-hodge}).

\begin{proposition} \label{rslt:primitive-comparison-for-smooth-proper-map}
Assume that $f$ is $\D_{[0,\infty)}$-proper and $\RH_{\Z_p,Y'}(\mathcal{F})$ is $f$-suave. Then $\Phi_{\mathcal{F},f}$ is an isomorphism. 
\end{proposition}
\begin{proof}
By \cite[Lemma~4.5.16(ii)]{heyer-mann-6ff} the object $\mathcal F' := f_* \RH_{\Z_p,Y'}(\mathcal{F})$ is suave over $Y$ and hence dualizable (see \cite[Example~4.4.3]{heyer-mann-6ff}). Since $\RH_{\Z_p,Y}$ induces an equivalence on dualizable objects (see \cref{rslt:RH-identifies-dualizable-objects}) we deduce that $\mathcal F'$ lies in the image of $\RH_{\Z_p,Y}$, hence the claim follows from \cref{rslt:primitive-comparison-reduces-to-image-of-RH}.  
\end{proof}

\begin{example}
The assumption of \cref{rslt:primitive-comparison-for-smooth-proper-map} is for example satisfied if $f$ is $\D_{[0,\infty)}$-smooth and $\D_{[0,\infty)]}$-proper and $\mathcal{F}$ is dualizable. By \cref{rslt:smooth-maps-of-adic-spaces-are-cohom-smooth} this is satisfied if $f$ comes from a smooth proper map of adic spaces over $\Q_p$, resulting in the usual formulation of the primitive comparison theorem. The same assertion holds modulo $p$, i.e. for $\RH_{\F_p}$.
\end{example}

We will now investigate in what generality (beyond the smooth proper case) one can expect a primitive comparison. The next result is useful to localize the question on the source via an excision argument:

\begin{lemma} \label{rslt:primitive-comparison-and-excision}
\begin{lemenum}
    \item Assume that $f=j\colon Y'\to Y$ is a partially proper \'etale morphism. Then the natural morphism
    \[
        j_!\circ \RH_{\Z_p,Y'}\to \RH_{\Z_p,Y}\circ j_!^\nuc
    \]
    is an isomorphism. Here, $j_!^\nuc\colon \D_\nuc(Y',\Z_p)\to \D_\nuc(Y,\Z_p)$ is the left adjoint to $j^\ast$ (which exists by partial properness of $j$\footnote{Indeed, the existence of $j_!^\nuc$, and that it is compatible with base change, can be checked in the case of a strictly totally disconnected perfectoid space $Y$ with a partially proper open subset $Y'$, where $Y'$ is a union of clopen immersions. Namely, by partial properness we get that $Y'\cong \pi_0(Y')\times_{\pi_0(Y)}Y$ and $\pi_0(Y')\subseteq \pi_0(Y)$ is open in a profinite set. }). If $Y$ admits a pseudo-uniformizer $\pi$, the analogous assertion holds for $\RH_{\F_p}$.

    \item Assume that $f=i\colon Y'\to Y$ is a closed immersion and that $\pi$ is a pseudo-uniformizer on $Y$. Then the natural morphism
    \[
        \RH_{\F_p,Y'}\circ i_\ast^\et\to i_\ast\circ \RH_{\F_p,Y}
    \]
    is an isomorphism.
\end{lemenum}
\end{lemma}
\begin{proof}
Part (ii) follows from (i) by excision (the proof of \cite[Lemma 6.2.1]{zavyalov-poincare-duality} applies as well as the complementary open immersion $j$ is partially proper).
We first note that (i) is clear if $j$ is a clopen immersion. The general case reduces to this by base changing to a strictly totally disconnected perfectoid space and writing $j$ as a filtered colimit of disjoint unions of clopen immersions as in \cite[Lemma 6.2.1, Step 1]{zavyalov-poincare-duality}.
\end{proof}

We now give some negative examples, showing that in general the primitive comparison theorem should only be asked for $\RH_{\F_p}$ and $f$ a proper morphism of rigid-analytic varieties.

\begin{examples}
\begin{examplesenum}
    \item In general, $\RH_{\Q_p}$ does not commute with proper, smooth pushforward. Indeed, \Cref{sec:pro-etale-cohomology-example-introduction-lubin-tate} yields a counterexample as there the pushforward does not lie in the essential image of $\RH_{\Q_p}$. Similarly, \Cref{sec:pro-etale-cohomology-example-introduction-affine-line} shows that if $f\colon (\mathbb{P}^1_{\mathbb{C}_p})^\diamond\to S:=\Spd(\mathbb{C}_p)$, then $f_\ast\RH_{\Z_p}(\mathcal{F})\ncong \RH_{\Z_p} f_\ast^\nuc(\mathcal{F})$ for $\mathcal{F}=j_!^\nuc\Z_p$ for the open immersion $(\A^{1,\an}_{\mathbb{C}_p})^\diamond\to (\mathbb{P}^1_{\mathbb{C}_p})^\diamond$. Namely, otherwise, then $f_\ast\circ j_!(\mathcal{O})\in \DFF(S,\Q_p)$ would be $\RH_{\Q_p}(f_\ast^\nuc\circ j_!^\nuc(\Q_p))$, which is not true. Thus, we see that \Cref{rslt:primitive-comparison-for-smooth-proper-map} only works for $\RH_{\Z_p}$ and some restriction on $\mathcal{F}$.

    \item The primitive comparison theorem should only be considered for proper morphisms of \textit{rigid-analytic varieties}, and not general proper morphisms of analytic adic spaces. Indeed, by a direct calculation the primitive comparison theorem fails for $\mathcal{F}=\F_p$ on the canonical compactification of the torus $\mathbb{T}_{\mathbb{C}_p}$ over $\mathbb{C}_p$ (as the $H^1$ with $\mathcal{O}^+/\pi$-coefficients is not free). From here one can deduce from excision (see \Cref{rslt:primitive-comparison-and-excision}) that the primitive comparison fails for the open unit disc by considering the complement of $\overline{\mathbb{T}_{\mathbb{C}_p}}$ in $\mathbb{P}^1_{\mathbb{C}_p}$. 
\end{examplesenum}
\end{examples}

The main case where the primitive comparison theorem holds is provided by the next assertion. To simplify notation, let us first introduce some terminology:

\begin{definition}
Given a $!$-able map $f\colon Y \to X$ of small v-stacks, we say that a sheaf $\mathcal F \in \D_\et(Y,\F_p)^\oc$ is \emph{$f$-ULA} if $\RH_{\F_p,Y}(\mathcal F)$ is $f$-suave. In the case where $X = \Spd K$ for some non-archimedean extension of $\Q_p$ and $f$ is the structure map, we abbreviate $f$-ULA as ULA.
\end{definition}

In the following result we use the notation of Zariski-constructible sheaves from \cite[Definition 3.1]{bhatt2022six}, and their bounded derived category $\D_{zc}^b(-,\F_p)\subseteq \D_\et((-)^\diamond,\F_p)^{\mathrm{oc}}$. The following result is based on the proof of \cite[Theorem 3.13]{perfectoid-spaces-survey}.
See \cite[Corollary 7.3.7]{bhatt_pht} and \cite[Corollary 3.9.24]{mann-mod-p-6-functors} for analogous results (for $\mathcal{O}^+/\pi$-modules).
We note that \cite[Corollary 3.9.24]{mann-mod-p-6-functors} does not assume Zariski constructibility (only overconvergence), but does not prove the ULA property. 

\begin{theorem} \label{sec:suppl-prim-comp-1-zariski-constructible-implies-ula}
Let $K$ be a non-archimedean field over $\Q_p$, let $f\colon Y\to X$ be a proper morphism of rigid-analytic varieties over $K$ and let $\mathcal F \in \D_{zc}^b(Y,\F_p)$ be given. Then $\mathcal F$ is ULA and the natural map
\[
    \RH_{\F_p,X^\diamond}f^\et_{\ast}(\mathcal{F})\to f_\ast\RH_{\F_p,Y^\diamond}(\mathcal{F})
\]
is an isomorphism.
\end{theorem}

\begin{proof}
By base change to an algebraic closure (using \cite[Lemma 4.4.9]{heyer-mann-6ff}) we may assume that $K$ is algebraically closed. We prove the assertion by induction on the dimension $d$ of the support of the Zariski-constructible complex (with support defined as the union of the supports of the cohomology objects). The assertion is local on $X$, so we may assume that $X$ (and therefore $Y$) is qcqs, even separated. We may moreover assume that $X,Y$ are reduced. Assume that $d=0$. Then $\mathcal{F}=i_\ast^\et\mathcal{G}$ for a dualizable object $\mathcal{G}\in \D_{zc}^b(Z,\F_p)$ for a finite set $Z$ of (reduced) $K$-rational points. In particular, $\mathcal{G}$ is ULA. From \Cref{rslt:primitive-comparison-and-excision} and \'etaleness of $Z^\diamond\to \Spd(K)$ we can conclude that $\mathcal{F}$ is ULA as the proper pushforward $i_\ast$ preserves ULA objects (see \cite[Lemma~4.5.16]{heyer-mann-6ff}). Furthermore, we see from \Cref{rslt:primitive-comparison-reduces-to-image-of-RH} and \Cref{rslt:primitive-comparison-and-excision} that
\[
    \RH_{\F_p,X^\diamond}f^\et_{\ast}(i^\et_\ast \mathcal{G})\cong f_\ast i_\ast \RH_{\F_p,Z^\diamond}(\mathcal{G}) \cong f_\ast \RH_{\F_p,Y^\diamond}(i_\ast^\et\mathcal{G}),
\]
which implies the second assertion. Now assume $d>0$ and that the claim is proven for all Zariski-constructible $\mathcal{F}$ with support of dimension $<d$. Let $\mathcal{F}\in \D^b_{zc}(Y,\F_p)$. From the reasoning above we see that we may assume that $\mathrm{supp}(\mathcal{F})=Y$. By \cite[Proposition 3.6]{bhatt2022six} we may assume that $\mathcal{F}=g_\ast^\et \F_p$ for a finite morphism $g\colon Y'\to Y$. Using the above reasoning for $f\circ g$ shows that it is sufficient to handle the case $\mathcal{F}=\F_p$ (by replacing $Y$ by $Y'$). We may again reduce to the case that $Y$ is reduced. By resolution of singularities for rigid-analytic varieties in characteristic $0$ (see the reasoning in the proof of Theorem 2.10 of \cite{hansen2020vanishing}, relying on \cite[Theorem 1.1.11]{temkin2018functorial}), we may find a proper morphism $g\colon Y'\to Y$ with $Y'$ smooth and an isomorphism over a Zariski-open subset $j\colon U\subseteq Y$ with (reduced) complement $i\colon Z\subseteq Y$.
We obtain a short distinguished triangle
\[
    \F_p\to g_\ast^\et(\F_p)\to \mathcal{G}
\]
with $\mathcal{G}$ Zariski-constructible (see \cite[2.4.iii]{hansen2020vanishing}, \cite[Theorem 3.10]{bhatt2022six}) and $\mathrm{supp}(\mathcal{G})<d$. To see that $\F_p$ is ULA it is therefore by induction enough to see that $g_\ast(\F_p)$ is ULA. This in turn is implied by \Cref{rslt:smooth-maps-of-adic-spaces-are-cohom-smooth} as $Y'$ is smooth and being ULA is stable under proper pushforward. We are left with checking that
\[
    \RH_{\F_p,X^\diamond}f^\et_\ast(\F_p) = f_\ast \RH_{\F_p,Y^\diamond}(\F_p),
\]
or equivalently (by induction and the above distinguished triangle) with $\F_p$ replaced by $g^\et_\ast(\F_p)$. We may therefore reduce to the case that $Y$ is smooth. By \cite[Theorem 2.29]{bhatt2022six} there exists a Zariski-open $j\colon V\to X$ with $f_V\colon W:=Y\times_X V\to V$ proper and smooth. Let $j'\colon W\to Y$ be the given open immersion. Using \Cref{rslt:primitive-comparison-and-excision} and \Cref{rslt:primitive-comparison-for-smooth-proper-map} we see that
\[
    \RH_{\F_p,X^\diamond}(f^\et_\ast(j_!^{\prime,\et}\F_p)) = \RH_{\F_p,X^\diamond}(j_!^\et(f^\et_{V,\ast}(\F_p))) = j_!f_{V,\ast}(\RH_{\F_p,Y^\diamond}(\F_p)), 
\]
which equals $f_\ast(\RH_{\F_p,Y^\diamond}(j^\et_{!}(\F_p))$ as desired. Thus, we may replace $\F_p$ on $Y$ by the cone of $j_!^\et\F_p\to \F_p$, which we also may assume to be supported in dimension $<d$ (e.g., by reducing to $Y$, $X$ Zariski-irreducible). Now, we apply induction and finish the proof.
\end{proof}

\begin{example}
The following example was communicated to us by David Hansen. Let $\ell\neq p$ be a prime. Then there exist a proper, smooth rigid-analytic variety $X$ over $\Q_p$ and an object $\D_\et(X^\diamond,\F_\ell)$, which is ULA, but not Zariski-constructible. In fact, let $(G,b,\{\mu\})$ be a local Shimura datum with $G=\mathrm{GL}_n$, $b$ basic and $\mu$ minuscule. Concretely, $b$ corresponds to a semistable vector bundle $\mathcal{E}_b=\mathcal{O}(\frac{d}{n})$ with $d=\kappa(b)$. Let $X=\mathcal{F}l(G,\mu)$ be the flag variety of type $\mu$, which we view as parametrizing modifications of the bundle $\mathcal{E}_b$ at infinity that are of type $\mu$. By \cite[Proposition A.12]{scholze2018p}, \cite[Lemma A.11]{scholze2018p} $X$ agrees with the full weakly admissible locus if and only if $G_b$ is the group of units in a division algebra, i.e., if $d$ is coprime to $n$. However, $X$ rarely agrees with its admissible locus (\cite[Corollary A.16]{scholze2018p}), and for example $d=2, n=5$ and $\mu=(1,1,0,0,0)$ yield a case with $X=\mathrm{Gr}(2,5)$ the Grassmannian of planes in $\Q_p^5$, where the admissible locus $U\subseteq X$ is strictly smaller than the weakly admissible locus. By ``weakly admissible implies admissible'' for $p$-adic fields, the complement $Z$ of $U$ in $X$ does not contain any $K$-valued points for $K$ a finite extension of $\breve{\Q}_p$. In particular, $U$ is not Zariski-constructible. If $d=2, n=5$, one checks that $Z$ is $\ell$-cohomologically smooth (being a union of quotients of a positive slope Banach--Colmez space by a $p$-adic Lie group). In particular, the pushfoward to $X$ of the constant sheaf on $Z$ is ULA, but not Zariski-constructible. It is likely that a similar counterexample also exists for $\F_p$-coefficients.
\end{example}

\subsection{Relation to v-sheaves: the functor \texorpdfstring{$\sigma$}{σ}} \label{sec:relation-v-sheaves-mod-p}

In this subsection and the next one, we investigate a relation between $\DFF(S,\Z_p)$ (resp. $\DFF(S,\Q_p)$) and v-sheaves of $\Z_p$-modules (resp. of $\Q_p$-modules) on a small v-stack $S$, which later will bring us closer to the formulation of $p$-adic Poincaré duality in the work of Colmez--Gilles--Nizio\l{}. The results presented here benefited a lot from conversations with Wies\l{}awa Nizio\l{}.

In the following, by $\Z_p$ we implicitly denote the v-sheaf sending a perfectoid space $S$ to continuous maps $\abs S \to \Z_p$ for the natural topology on $\Z_p$ (this is denoted $\hat{\Z}_p$ in \cite[Definition~8.1]{rigid-p-adic-hodge}), and similarly for $\Q_p$.

\begin{lemma} \label{rslt:construction-of-sigma-solid}
There exists a unique natural transformation 
\[
    \sigma=\sigma_{\Z_p}\colon \DFF(-,\Z_p) \to \D((-)_v,\Z_p)
\]
of functors $\vStacks^\opp \to \Cat$, such that for any $S\in \vStacks$ we have (naturally in $S, M$)
\[
    \Gamma(S_v,\sigma(M)) = \Hom_{\DFF(S,\Z_p)}(1,M)
\]
for any $M\in \DFF(S,\Z_p)$.
\end{lemma}
\begin{proof}
Uniqueness is implied by the formula. For existence, we first note that \cref{rslt:v-descent-for-D-0-infty} implies that for a given $M$ the formula
\[
    S'\mapsto \Hom_{\DFF(S',\Z_p)}(1,M_{S'})
\]
defines a v-sheaf on small v-stacks over $S$, where $M_{S'}$ denotes the pullback of $M$ to $S'$. It remains to construct the desired functoriality of $\sigma$ in order to establish it as a natural transformation as in the claim. By the above sheafiness, we do not need to worry about the sheaf property and can instead construct it as a map to v-presheaves. We can now perform a standard argument with straightening and unstraightening. To abbreviate notation, denote $\mathcal C := \vStacks$ and let $\mathcal E_\FF \to \mathcal C^\opp$ be the cocartesian unstraightening of the functor $\DFF(-,\Z_p)$. Let $\mathcal E \to \mathcal C^\opp$ be the cartesian unstraightening of the functor $\mathcal C \to \Cat$, $S \mapsto (\mathcal C_{/S})^\opp$. By \cite[Remark~3.78]{Heine.2023} this map is flat, hence the base-change $- \times_{\mathcal C^\opp} \mathcal E$ admits a right adjoint
\begin{align*}
    \Fun^{\mathcal C^\opp}(\mathcal E, -)\colon \Cat \to \Cat_{/\mathcal C^\opp}.
\end{align*}
We observe that $\Fun^{\mathcal C^\opp}(\mathcal E, \D(\Z_p)) \to \mathcal C^\opp$ is the cocartesian unstraightening of the functor $\mathcal C^\opp \to \Cat$,  $S \mapsto \Fun((\mathcal C_{/S})^\opp, \D(\Z_p))$. Therefore, the natural transformation $\sigma$ must be a map of cocartesian fibrations
\begin{align*}
    \mathcal E_\FF \to \Fun^{\mathcal C^\opp}(\mathcal E, \D(\Z_p))
\end{align*}
over $\mathcal C^\opp$. By adjointness, this map is equivalently a map $\mathcal E_\FF \times_{\mathcal C^\opp} \mathcal E \to \D(\Z_p)$. Now note that $\mathcal E = \Fun([1], \mathcal C)^\opp$. Moreover, it is enough to construct the composition of the above map with the Yoneda embedding of $\D(\Z_p)$. Altogether the construction of $\sigma$ reduces to the construction of the functor
\begin{align*}
    \mathcal E_\FF \times_{\mathcal C^\opp} \Fun([1], \mathcal C)^\opp \times \D(\Z_p)^\opp &\to \Ani,\\
    (M, [S' \to S], N) &\mapsto \Hom_{\DFF(S',\Z_p)}(N \tensor 1, M_{S'}).
\end{align*}
Let $\mathcal E_\FF^\vee \to \mathcal C^\opp$ be the cocartesian fibration classifying the functor $\DFF(-,\Z_p)^\opp$. Then by \cite[\S5]{Barwick-Glasman-Nardin.2018} there is a pairing $\mathcal E_\FF \times_{\mathcal C^\opp} \mathcal E_\FF^\vee \to \Ani$, given informally by $(M, N) \mapsto \Hom_{\DFF(S,\Z_p)}(N, M)$. This reduces the construction to the following two functors over $\mathcal C^\opp$:
\begin{align*}
    &\alpha\colon \mathcal E_\FF \times_{\mathcal C^\opp} \Fun([1], \mathcal C)^\opp \to \mathcal E_\FF, & (M, [S' \to S]) &\mapsto M_{S'},\\
    &\beta\colon \Fun([1], \mathcal C)^\opp \times \D(\Z_p)^\opp \to \mathcal E_\FF^\vee, & ([S' \to S], N) &\mapsto N \tensor 1_{S'}.
\end{align*}
Here in both cases the structure map to $\mathcal C^\opp$ is given by the source map on $\Fun([1], \mathcal C)^\opp$. Now $\alpha$ falls directly out of the definition of cartesian fibrations (as in \cite[Definition~A.2.4(b)]{heyer-mann-6ff}), which tells us that the opposite of the fiber product on the left is just the category of cartesian edges in $\mathcal E_\FF^\opp$---then simply take $\alpha$ to be the source map. To construct $\beta$, we first take the source map $\Fun([1], \mathcal C)^\opp \to \mathcal C^\opp$ in order to reduce the construction of $\beta$ to the functor $\mathcal C^\opp \times \D(\Z_p)^\opp \to \mathcal E_\FF^\vee$. This functor is then the map of cocartesian fibrations associated to the natural transformation of functors $\mathcal C^\opp \to \Cat$ from the constant functor $\D(\Z_p)^\opp$ to the functor $\DFF(-,\Z_p)^\opp$, sending $N \in \D(\Z_p)^\opp$ to $N \tensor 1$.
\end{proof}

\begin{remarks}
\begin{remarksenum}
    \item More precisely, we used here the notation $\D(S_v,\Z_p)$ for the category of v-sheaves with values in $\D(\Z_p)$. By v-hyperdescent (\Cref{rslt:v-descent-for-D-0-infty}) the image of $\sigma$ lands even in hypercomplete v-sheaves, and we will be implicitly use the notation $\D(S_v,\Z_p)$ also for the full subcategory of hypercomplete v-sheaves (as the difference between the two is not important for the results in this paper). By repleteness of $S_v$ and \cite[Theorem A]{mondal2022postnikov} hypercomplete v-sheaves are Postnikov complete and as a consequence the category of hypercomplete v-sheaves agrees with the derived category of static $\Z_p$-modules on $S_v$.
    
    \item Similarly, one can construct for a small v-stack $S$ a natural functor
    \[
        \sigma_{\Q_p}\colon \DFF(S,\Q_p)\to \D(S_v,\Q_p).
    \]
    By restricting \Cref{rslt:construction-of-sigma-solid} to modules over $\F_p$, we moreover obtain the functor
    \[
        \sigma_{\F_p}\colon \DFF(S,\mathbb{F}_p)\to \D(S_v,\F_p),
    \]
    where $\DFF(S,\mathbb{F}_p)=\mathrm{Mod}_{\mathcal{O}^\flat}\D_{\hat\solid}^a(\mathcal{O}^+_{S})^\varphi$ refers to the category defined in \Cref{def:Dhsa}. Here $S$ is viewed as an untilted small v-stack via $S\to \Spd(\F_p)\to \Spd(\Z_p)$, in particular $\DFF(S,\mathbb{F}_p)$ is $\F_p$-linear. If we want to stress that $\mathcal{O}$ refers to the structure sheaf in characteristic $p$, then we write $\mathcal{O}^\flat$. 
    
    \item In \cite[Corollary 3.11]{anschutz2021fourier} the functor $\sigma_{\Q_p}$ was denoted by $R\tau_\ast$ and it was shown that it is fully faithful on perfect complexes. The notation $R\tau_\ast$ was used in its relation to taking relative cohomology. We changed the notation because $\sigma$ is not a right adjoint as it does not commute with limits (because of the occuring pullbacks to small v-stacks over $S$). 

    \item \label{remark-rslt:sigma-commutes-countable-limits} For each small v-stack $S$ the functor $\sigma \colon \DFF(S,\Z_p) \to \D(S_v,\Z_p)$ commutes with colimits. Namely, this can be checked locally on $S_v$, and thus in the case that $S$ is a qcqs perfectoid space which admits a morphism of finite $\dimtrg$ to a totally disconnected perfectoid space. Then \Cref{rslt:dualizable-object-on-FF-curve-is-compact} and the formula for $\sigma$ in \Cref{rslt:construction-of-sigma-solid} show that $\sigma$ commutes with colimits. The same argument applies to $\sigma_{\Q_p}$ or $\sigma_{\F_p}$.
\end{remarksenum}
\end{remarks}

\begin{lemma} \label{rslt:sigma-is-lax-symmetric-monoidal}
Let $S$ be a small v-stack. Then
\[
   \sigma\colon \DFF(S,\Z_p)\to \D(S_v,\Z_p)
\]
is lax symmetric monoidal.
\end{lemma}
\begin{proof}
To show that $\sigma$ is lax symmetric monoidal, we upgrade the construction in \cref{rslt:construction-of-sigma-solid} in order to define $\sigma$ as a natural transformation of functors from $\vStacks^\opp$ to the category of operads. As before, we can reduce to constructing a functor to $\D(\Z_p)$-valued v-presheaves instead of v-sheaves, as even on the level of operads, sheaves embed fully faithfully into presheaves. Let again $\mathcal C = \vStacks$ and let $\mathcal E_\FF^\tensor \to \mathcal C^\opp \times \Comm^\tensor$ be the cocartesian fibration classifying the functor $S \mapsto \D_\FF(S,\Z_p)^\tensor$; here $\Comm^\tensor$ is the commutative operad (given by the category of finite pointed sets). We let $\mathcal E^\tensor := \mathcal E \times \Comm^\tensor$ with $\mathcal E$ as in the proof of \cref{rslt:construction-of-sigma-solid}. Then
\begin{align*}
   \Fun^{\mathcal C^\opp \times \Comm^\tensor}_{\Comm^\tensor}(\mathcal E^\tensor, \D(\Z_p)^\tensor) \to \mathcal C^\opp \times \Comm^\tensor
\end{align*}
is the cocartesian fibration corresponding to the functor $S \mapsto \Fun((\mathcal C_{/S})^\opp, \D(\Z_p))^\tensor$, where the decorations on $\Fun$ are understood in an operadic sense as introduced (in the non-symmetric case) right before \cite[Theorem~11.23]{Heine.2023}. More explicitly, by unraveling the definitions one sees that the fiber of the above functor category over some $S \in \mathcal C$ is given by the Day convolution operad on $\Fun((\mathcal C_{/S})^\opp, \D(\Z_p))$, where the first argument is the generalized operad $(\mathcal C_{/S})^\opp \times \Comm^\tensor$. Similarly to the proof of \cref{rslt:construction-of-sigma-solid} we can use the universal property of the $\Fun$-construction above in order to reduce the construction of $\sigma_\solid$ (with its lax symmetric monoidal structure) to constructing the map 
\begin{align*}
   \mathcal E_\FF^\tensor \times_{\mathcal C^\opp \times \Comm^\tensor} \mathcal E^\tensor \to \D(\Z_p)^\tensor
\end{align*}
over $\Comm^\tensor$. By expanding $\mathcal E^\tensor$ and using the Yoneda embedding for $\D(\Z_p)^\tensor$ (relative over $\Comm^\tensor$) we are reduced to constructing the functor
\begin{align*}
   \D(\Z_p)^{\opp,\tensor} \times_{\Comm^\tensor} \mathcal E_\FF^\tensor \times_{\mathcal C^\opp} \Fun([1], \mathcal C)^\opp &\to \Ani,\\
   (N_\bullet, M_\bullet, [S' \to S]) &\mapsto \prod_i \Hom(N_i \tensor 1, (M_i)_{S'}).
\end{align*}
As in the proof of \cref{rslt:construction-of-sigma-solid}, using the fibration $\mathcal E^{\tensor,\vee} \to \mathcal C^\opp \times \Comm^\tensor$ corresponding to the functor $S \mapsto \D_\FF(S,\Z_p)^{\opp,\tensor}$, we can reduce the construction to the following two functors over $\mathcal C^\opp \times \Comm^\tensor$:
\begin{align*}
   &\alpha\colon \mathcal E_\FF^\tensor \times_{\mathcal C^\opp} \Fun([1], \mathcal C)^\opp \to \mathcal E_\FF^\tensor, & ((M_i)_i, [S' \to S]) &\mapsto ((M_i)_{S'})_i,\\
   &\beta\colon \D(\Z_p)^{\tensor,\opp} \times \Fun([1], \mathcal C)^\opp \to \mathcal E_\FF^{\tensor,\vee}, & ((N_i)_i, [S' \to S]) &\mapsto (N_i \tensor 1_{S'})_i.
\end{align*}
These can be constructed as in the proof of \cref{rslt:construction-of-sigma-solid}.
\end{proof}

The next result concerns the compatibility of $\sigma$ with pushforwards.

\begin{lemma} \label{rslt:sigma-ast-compatible-with-pushforward}
Let $f\colon X\to S$ be a $p$-bounded morphism of small v-stacks, and let $f_{v,\ast}\colon \D(X_v,\Z_p)\to \D(S_v,\Z_p)$ be the pushforward. Then the natural transformation
\[
    \sigma f_\ast\to f_{v,\ast}\sigma
\]
of functors $\DFF(X, \Z_p) \to \D(S_v,\Z_p)$ is an isomorphism.
\end{lemma}
\begin{proof}
Let $M\in \DFF(X,\Z_p)$ and $g\colon S'\to S$ a morphism. Let $X':=X\times_S S'$ with projections $g'\colon X'\to X, f'\colon X'\to S'$. Using qcqs base-change (see \cref{rslt:D-0-infty-pushforward-base-change}) we calculate
\begin{align*}
    \Gamma(S',\sigma(f_\ast(M)))
    &= \Hom_{\DFF(S',\Z_p)}(1, g^\ast f_\ast M) \\
    &= \Hom_{\DFF(S',\Z_p)}(1, f'_\ast g^{\prime,\ast}M)\\
    &= \Hom_{\DFF(X',\Z_p)}(1, g^{\prime,\ast}M) \\
    &= \Gamma(X',\sigma(M))\\
    &= \Gamma(S',f_{v,\ast}(\sigma (M))).
\end{align*}
This shows the claim.
\end{proof}

\begin{remark}
\label{remark-about-solid-version}
One a priori drawback of $\sigma$ as defined above is the lack of a $\D_\solid(\Z_p)$-linear structure on its target making it $\D_\solid(\Z_p)$-linear. We will however see in the next subsection a way to circumvent this problem when restricting to nuclear objects on the source. Another way, which was our original approach and which would not entail such a restriction, would be to enrich $\sigma$ into a natural transformation of functors 
\[
    \sigma_\solid=\sigma_{\solid,\Z_p}\colon \DFF(-,\Z_p) \to \D((-)_v,\D_\solid(\Z_p))
\]
of functors $\vStacks^\opp \to \Cat$, uniquely characterized by the fact that for any $S\in \vStacks$ and any profinite set $T$ we have (naturally in $S,T,M$)
\[
    \Hom_{\D_\solid(\Z_p)}(\Z_{p,\solid}[T],\Gamma(S_v,\sigma_\solid(M))) = \Hom_{\DFF(S,\Z_p)}(\Z_{p,\solid}[T]\otimes 1,M)
\]
for any $M\in \DFF(S,\Z_p)$. (Then $\sigma$ would be recovered as the composition of $\sigma_\solid$ with $\Hom_{\D_\solid(\Z_p)}(\Z_p,-)$.)

All the properties of $\sigma$ discussed above hold for $\sigma_\solid$. Since we do not need $\sigma_\solid$ in the rest of the paper, we do not discuss it further, but believe it could be useful for other purposes.
\end{remark}

\subsection{Relation to v-sheaves: the \texorpdfstring{$\Q_p$}{Q\_p}-case}
\label{sec:q-p-case}

In order to relate duality on the Fargues--Fontaine curve with duality for pro-\'etale cohomology, one needs to compare suitable internal duals via $\sigma_{\Q_p}$. More precisely, we aim at proving the following result.

\begin{theorem} \label{rslt:sigma-Qp-comparison-of-internal-hom}
  Let $S$ be a small v-stack, and let $M,N\in \DFF(S,\Q_p)$.
  Assume that $M$ lies in the full subcategory generated under countable colimits by $V\otimes P$ for basic nuclear objects $V$ in $\D_\solid(\Q_p)$ and dualizable objects $P \in \DFF(S, \Q_p)$, and that $N$ is dualizable.
  Then the natural morphism
\[
    \sigma_{\Q_p}\IHom_{\DFF(S,\Q_p)}(M,N) \isoto \IHom_{\D(S_v,\Q_p)}(\sigma_{\Q_p}(M),\sigma_{\Q_p}(N))
\]
is an isomorphism.
\end{theorem}

This statement will only be proved at the end of this subsection. The case where $M$ is dualizable was, as we will recall below, already proved in \cite{anschutz2021fourier}, essentially as a consequence of Breen's computation of the (derived) endomorphisms of the structure sheaf on the (big) perfect site of $\F_p$. To obtain the theorem, one would like to first pass from the case where $M$ is dualizable to the case where $M$ is the tensor product of a dualizable object by a basic nuclear $\Q_p$-vector space, and then pass to colimits of such, but we need to overcome two issues:
\begin{itemize}
\item $\sigma_{\Q_p}$ is not symmetric monoidal if we endow $\D(S_v,\Q_p)$ with the ``naive'' tensor product of sheaves of $\Q_p$-modules on the v-site: on the source, the tensor product is a tensor product in solid quasi-coherent sheaves on the Fargues-Fontaine curve, but the one on the target is rather an ``algebraic''/``condensed'' tensor product. This is a problem when trying to pass from the case where $M$ is dualizable to the case where $M$ is the tensor product of a dualizable object by a basic nuclear $\Q_p$-vector space.  
\item $\sigma_{\Q_p}$ does not commute with arbitrary limits. Since by assumption, $M$ is presented as a colimit and $\sigma_{\Q_p}$ commutes with colimits, the right-hand side in \Cref{rslt:sigma-Qp-comparison-of-internal-hom} is naturally presented as a limit, but it is a priori not clear how to write the left-hand side in the same way due to the lack of commutation of $\sigma_{\Q_p}$ with limits.
\end{itemize}
Before getting to the proof of \Cref{rslt:sigma-Qp-comparison-of-internal-hom}, we therefore develop the necessary preliminary considerations to resolve these issues. Roughly:
\begin{itemize}
\item We introduce, based on the notion of solid sheaves introduced in \cite[\S VII]{fargues-scholze-geometrization}, a category of \textit{big solid sheaves} of $\Z_p$-modules or $\Q_p$-modules on the \textit{big} quasi-pro-\'etale site of any spatial diamond $S$ with a solid tensor product, into which the restriction of $\sigma_{\Q_p}$ to nuclear objects in $\DFF(S,\Q_p)$ naturally factors. This restriction of $\sigma_{\Q_p}$ commutes then with the tensor product with a nuclear $\Q_p$-vector space (understood as the solid tensor product on the target); in other words, this provides a $\D_\nuc(\Q_p)$-linear replacement of $\sigma_{\Q_p}$, rectifying the inconvenience mentioned in \cref{remark-about-solid-version}.
\item Even if $\sigma_{\Q_p}$ does not commute with arbitrary limits, we show that it commutes with countable limits of nuclear objects.
\end{itemize}

We start by the discussion of (big) solid sheaves. Let $S$ be a spatial diamond. We denote by $S_\qproet$ the small quasi-pro-\'etale site of $S$, and by $S_\Qproet$ the big quasi-pro-\'etale of all spatial diamonds over $S$ with the quasi-pro-\'etale topology. The category $\D(S_\qproet,\Z_p)$ admits the full subcategory $\D_\solid(S,\Z_p)$ of solid sheaves as defined in \cite[Definition VII.1.10]{fargues-scholze-geometrization} (and the paragraph following it), or \cite[Definition 3.14]{anschuetz_mann_descent_for_solid_on_perfectoids}.

We recall some properties of $\D_\solid(S,\Z_p)$.

\begin{lemma}
Let $f\colon S'\to S$ be a morphism of spatial diamonds.
\begin{lemenum}
    \item \label{rslt:solid-Zp-sheaves-stable-under-lim-colim-trunc} The subcategory $\D_\solid(S,\Z_p)\subseteq \D(S_\qproet,\Z_p)$ is stable under all colimits, all limits, all truncations and contains all \'etale sheaves, and $f^\ast_\qproet(\D_\solid(S,\Z_p))\subseteq \D_\solid(S',\Z_p)$.

    \item \label{rslt:solid-Zp-sheaves-stable-under-pushforward} The pushforward $f_{\qproet,\ast}\colon \D(S'_\qproet,\Z_p)\to \D(S_\qproet,\Z_p)$ restricts to a functor $$f_\ast\colon \D_\solid(S',\Z_p)\to \D_\solid(S,\Z_p).$$
\end{lemenum}
\end{lemma}
\begin{proof}
  For the first assertion see \cite[Proposition 3.15]{anschuetz_mann_descent_for_solid_on_perfectoids} and \cite[Proposition VII.1.8]{fargues-scholze-geometrization}. As in \cite[Proposition VII.2.1]{fargues-scholze-geometrization} one can reduce the second assertion to the assertion that $f_{\qproet,\ast}\mathcal{F}\in \D_\solid(S,\Z_p)$ if $\mathcal{F}\in \D^+(S'_\qproet,\Z_p)$ is pulled back from the small \'etale site of $S'$ and killed by some $p^n$, $n\geq 0$. Then \cite[Corollary 16.7]{etale-cohomology-of-diamonds} implies that $f_{\qproet,\ast}\mathcal{F}$ is solid (even pulled back from the \'etale site) as desired. 
\end{proof}

We want to define a version of solid sheaves in $\D(S_\Qproet,\Z_p)$, i.e., on the big quasi-pro-\'etale site over $S$. We note that this full subcategory will be different from the pullback of solid sheaves along $S_\Qproet\to S_\qproet$. Given a spatial diamond $S'\in S_\Qproet$, we let
\[
  \varepsilon_{S'}\colon S_\Qproet/S'\to S'_\qproet
\]
be the natural projection of sites.

\begin{definition}
We call $\mathcal{F}\in \D(S_\Qproet,\Z_p)$ a \textit{big solid sheaf} if for all spatial diamonds $S'\in S_\Qproet$, the sheaf $\varepsilon_{S',\ast}(\mathcal{F}_{|S'})\in \D(S'_\qproet,\Z_p)$ lies in $\D_\solid(S',\Z_p)$.
Given a solid $\Z_p$-algebra $\Lambda$, e.g., $\Lambda=\Q_p$, we let $\D_\solid(S_\Qproet, \Lambda)\subseteq \D(S_\Qproet, \Lambda)$ be the full subcategory of sheaves whose underlying $\Z_p$-sheaf is big solid.
\end{definition}

Here $\D(S_\Qproet,\Z_p)$ denotes the usual derived category of static $\Z_p$-modules on $S_\Qproet$. As $S_\Qproet$ is replete, countable products are exact on static $\Z_p$-modules on $S_\Qproet$, and as a consequence $\D(S_\Qproet,\Z_p)$ is left-complete.

\begin{lemma} \label{rslt:basic-properties-of-big-solid-sheaves}
Let $f\colon S'\to S$ be a morphism of spatial diamonds and let $\Lambda$ be a solid $\Z_p$-algebra.
\begin{lemenum}
    \item The subcategory $\D_\solid(S_\Qproet, \Lambda)\subseteq \D(S_\Qproet, \Lambda)$ is stable under all limits, all colimits and all truncations. In particular, the inclusion $\D_\solid(S_\Qproet, \Lambda)\to \D(S_\Qproet,\Lambda)$ admits a left adjoint $M\mapsto M_\solid$.

    \item \label{rslt:pullback-of-big-solid-sheaves} The pullback $f^\ast_{\Qproet}$ sends $\D_\solid(S_\Qproet,\Lambda)$ to $\D_\solid(S'_\Qproet,\Lambda)$. Moreover, if $f$ is surjective and $M \in \D(S_\Qproet,\Lambda)$ satisfies $f_{\Qproet}^\ast M \in \D_\solid(S'_\Qproet,\Lambda)$, then $M \in \D_\solid(S_\Qproet,\Lambda)$.
    
    \item The pushforward $f_{\Qproet,\ast}\colon \D(S'_\Qproet,\Lambda)\to \D(S_\Qproet,\Lambda)$ preserves big solid sheaves.
    
    \item If $M\in \D(S_\Qproet,\Lambda)$ and $N\in \D_\solid(S_\Qproet,\Lambda)$, then $$\IHom_{\D(S_\Qproet,\Lambda)}(M,N)\in \D_\solid(S_\Qproet,\Lambda).$$ In particular, $\D_\solid(S_\Qproet,\Lambda)$ admits a unique symmetric monoidal structure $-\otimes_{\Lambda}^\solid -$ making the left adjoint $(-)_\solid$ and the pullback $f_{\Qproet}^\ast$ symmetric monoidal. 
\end{lemenum}
\end{lemma}
\begin{proof}
The first assertion follows from \cref{rslt:solid-Zp-sheaves-stable-under-lim-colim-trunc} and the fact that for any spatial diamond $S'$ over $S$ the restriction functor along $S_{\Qproet/S'} \to S_\Qproet$ and the pushforward along $\varepsilon_{S'}\colon S_{\Qproet/S'}\to S'_\qproet$ commute with limits, colimits and truncations (we note that the functor $\varepsilon_{S',*}$ is t-exact). For the second assertion, preservation of big solid sheaves under pullback is clear and the second part follows from the corresponding statement for solid sheaves (\cite[Proposition VII.1.8]{fargues-scholze-geometrization}). Preservation under pushforward follows from \cref{rslt:solid-Zp-sheaves-stable-under-pushforward} because $\varepsilon_{S'',*} \comp f_{\Qproet,\ast} = g_{\qproet,\ast} \comp \varepsilon_{S''\times_S S',*}$ if $S''$ is a spatial diamond over $S$ with base change $g\colon S''\times_S S'\to S''$ of $f$.
Given $M, N$ as in the fourth assertion, we may write $M$ as a colimit of $\Lambda[U]$ with $h\colon U\to S$ a morphism from a spatial diamond. By stability of $\D_\solid(S_\Qproet,\Lambda)$ under limits, we may assume that $M=\Lambda[U]$. Then $\IHom_{\D(S_\Qproet,\Lambda)}(M,N)\cong h_{\Qproet,\ast}(N_{|U_\Qproet})$ is solid by stability of big solid sheaves under pullback and pushforward.
The existence of the symmetric monoidal structure making $(-)_\solid$ symmetric monoidal is formal, cf.\ \cite[Proposition VII.1.14]{fargues-scholze-geometrization}. It is clear that the pullback $f_{\Qproet}^\ast$ is naturally symmetric monoidal.
\end{proof}

\begin{lemma} \label{sec:prel-subs-pullback-and-solidification}
Let $S$ be a spatial diamond. Then $\varepsilon^\ast_S\colon \D(S_\qproet,\Z_p)\to \D(S_\Qproet,\Z_p)$ maps solid sheaves to big solid sheaves. Moreover, $\varepsilon^\ast_S$ commutes with solidification (i.e. the left adjoint $(-)_\solid$ on both sides) and is naturally symmetric monoidal for the solid tensor product.
\end{lemma}
\begin{proof}
Preservation of solid sheaves follows from \cite[Proposition VII.1.8]{fargues-scholze-geometrization}. The commutation with solidification can be checked on adjoints, where it is clear. Symmetric monoidality for the solid tensor product follows from symmetric monoidality for the usual tensor product and the commutation with solidification.
\end{proof}

The relevant examples for us of big solid, even ``big nuclear'', sheaves are the ones coming from quasi-coherent sheaves on the Fargues--Fontaine curve, as follows.

\begin{lemma} \label{sec:prel-subs-sigma-maps-nuclear-objects-to-big-nuclear-sheaves}
Let $S$ be a spatial diamond. Let $M\in \DFF(S,\Q_p)$ be nuclear. Then $\sigma_{\Q_p}(M)\in \D(S_\Qproet,\Q_p)$ is big solid.
\end{lemma}
\begin{proof}
As $\sigma_{\Q_p}$ is compatible with pullback we may assume by \cref{rslt:pullback-of-big-solid-sheaves} that $S$ is strictly totally disconnected. Fix a pseudo-uniformizer $\pi$ on $S$ with radius function $\kappa\colon \mathcal{Y}_{[0,\infty),S}\to [0,\infty)$. As in \cref{rslt:dualizable-object-on-FF-curve-is-compact} we use the notation $\mathcal{Y}_{I,S}$ for $I\subseteq [0,\infty)$ a finite union of compact intervals. For each quasi-pro-\'etale morphism $S'\to S$ with $S'$ affinoid we can write
\[
\Gamma(S',\sigma_{\Q_p}(M))=\mathrm{eq}(\Gamma(\mathcal{Y}_{[1,p],S'},M)\rightrightarrows \Gamma(\mathcal{Y}_{[1]\cup[1/p],S'},M)).
\]
We claim that $\varepsilon_{S,\ast}(\sigma_{\Q_p}(M))\in \D(S_\qproet,\Q_p)$ is even nuclear in the sense of \cite[Section 3.5]{anschuetz_mann_descent_for_solid_on_perfectoids}, and hence solid (even $\omega_1$-solid). By \cite[Lemma 3.29]{anschuetz_mann_descent_for_solid_on_perfectoids} the nuclear objects in ($\omega_1$-solid objects in) $\D(S_\qproet,\Q_p)$ are equivalent to $\D_\nuc(C(S,\Q_p))$ because $S$ is strictly totally disconnected. More precisely, to $N\in \D_\nuc(C(S,\Q_p))$ is associated the sheaf $S'\mapsto (C(S',\Q_p)\otimes_{C(S,\Q_p)} N)(\ast)$ (the tensor product refers to solid tensor product in $\D_\nuc(C(S,\Q_p))$) on strictly totally disconnected perfectoid spaces $S'\in S_\qproet$.
We note that there exists a natural continuous map $\pi_0(\mathcal{Y}_{I,S})\to \pi_0(S)$, and thus a natural morphism $C(S,\Q_p)=C(\pi_0(S),\Q_p)\to \mathcal{O}(\mathcal{Y}_{I,S})$ of (solid) Banach algebras over $\Q_p$.
We furthermore note that $\mathcal{O}(\mathcal{Y}_{I,S'})\cong C(S',\Q_p)\otimes_{C(S,\Q_p)} \mathcal{O}(\mathcal{Y}_{I,S})\in \D_\nuc(C(S,\Z_p))$ for $S'\in S_\qproet$ strictly totally disconnected. Given $I\subseteq (0,\infty)$ a finite union of compact intervals, we let $M_I$ be the nuclear $\mathcal{O}(\mathcal{Y}_{I,S})$-module given by the restriction of $M$ to $\mathcal{Y}_{I,S}$. It follows easily from the above equalizer diagram that $\varepsilon_{S,\ast}(\sigma_{\Q_p}(M))$ is associated with the nuclear $C(S,\Q_p)$-module $\mathrm{eq}(M_{[1,p]}\rightrightarrows M_{[1]\cup [1/p]})$, where the two arrows are induced by restriction respectively Frobenius (we note that we use here that $\mathcal{O}(\mathcal{Y}_{I,S})$ is an adic, hence nuclear, $C(S,\Q_p)$-algebra, and thus the $C(S,\Q_p)$-module $M_I$ is nuclear).
\end{proof}

We note that for any spatial diamond $S$ the site $S_\qproet$ (and hence $S_\Qproet$) lives naturally over the pro-\'etale site $\ast_\proet$ of a point. Indeed, given a profinite set $T$ and $S'\in S_\qproet$, we can form $T\times S'\in S_\qproet$. In particular, $\D(S_\Qproet,\Z_p)$ is naturally linear over $\D(\ast_\proet,\Z_p)$. From \cref{sec:prel-subs-pullback-and-solidification} we can conclude that the $\D(\ast_\proet,\Z_p)$-linear structure on $\D_\solid(S_\Qproet,\Z_p)$ factors canonically through the quotient $\D(\ast_\proet,\Z_p)\to \D_\solid(\Z_p)$. We note that by rigidity of $\D_\nuc(\Z_p)$ the functor $\varepsilon_{S,\ast}\colon \D_\solid(S_\Qproet,\Z_p)\to \D_\solid(S_\qproet,\Z_p)$ is $\D_\nuc(\Z_p)$-linear as it commutes with colimits (see \cite[Lemma 3.33]{anschuetz_mann_descent_for_solid_on_perfectoids}).

\begin{lemma} \label{sec:prel-subs-linearity-of-sigma-to-solid}
Let $S$ be a spatial diamond. Then the functor
\[
    \sigma_{\Q_p}\colon \DFF^\nuc(S,\Q_p)\to \D_\solid(S_\Qproet,\Q_p)
\]
is $\D_\nuc(\Q_p)$-linear. 
\end{lemma}
\begin{proof}
Let $V\in \D_\nuc(\Q_p)$ and $M\in \DFF^\nuc(S,\Q_p)$. By construction of $\sigma_{\Q_p}$, we have a natural map $V \tensor_{\Q_p} (\Q_p)_S \to \sigma_{\Q_p}(V \tensor \mathcal O_{\FF_S})$ of v-sheaves on $S$, because by adjunction such a map is the same as a map $V \to \Gamma(S, \sigma_{\Q_p}(V \tensor \mathcal O_{\FF_S}))$, which by definition of $\sigma_{\Q_p}$ is the same as a map $V \tensor \mathcal O_{\FF_S} \to V \tensor \mathcal O_{\FF_S}$ in $\DFF(S,\Q_p)$. Thus by the lax symmetric monoidal structure on $\sigma_{\Q_p}$ (\cref{rslt:sigma-is-lax-symmetric-monoidal}), we obtain a natural map
\begin{align*}
    &V \tensor^\solid_{\Q_p} \sigma_{\Q_p}(M) = (V \tensor_{\Q_p} (\Q_p)_S) \tensor^\solid_{\Q_p} \sigma_{\Q_p}(M) \to \sigma_{\Q_p}(V \tensor \mathcal O_{\FF_S}) \tensor^\solid_{\Q_p} \sigma_{\Q_p}(M) \to\\&\qquad\qquad\to \sigma_{\Q_p}((V \tensor \mathcal O_{\FF_S}) \tensor M) = \sigma_{\Q_p}(V \tensor M).
\end{align*}
Since the target is big solid, this supplies us with a natural morphism
$$
    \Phi\colon V\otimes_{\Q_p}^\solid \sigma_{\Q_p}(M)\to \sigma_{\Q_p}(V \otimes M)
$$
(which is compatible with pullback). It is sufficient to check that $\varepsilon_{S',\ast}(\Phi)$ is an isomorphism for any strictly totally disconnected perfectoid spaces $S'$ over $S$. As noted before this lemma, $\varepsilon_{S',\ast}$ is $\D_\nuc(\Z_p)$-linear. But from the proof of \cref{sec:prel-subs-sigma-maps-nuclear-objects-to-big-nuclear-sheaves} (more precisely, the equalizer diagram), it follows easily that $\varepsilon_{S',\ast}\circ \sigma_{\Q_p}$ is $\D_\nuc(\Q_p)$-linear as desired.
\end{proof}

\begin{remark}
\Cref{sec:prel-subs-linearity-of-sigma-to-solid} solves the problem of the missing $\D_\nuc(\Q_p)$-linear structure on the target $\D(S_v,\Q_p)$ of the functor $\sigma_{\Q_p}$. As mentioned in \Cref{remark-about-solid-version} another solution would have been to add an ``external condensed direction'' and to consider the functor $\sigma_{\solid,\Q_p}\colon \DFF(S,\Q_p)\to \D(S_v,\D_\solid(\Q_p))$ and the pointwise $\D_\solid(\Q_p)$-linear structure on $\D(S_v,\D_\solid(\Q_p))$. One can check that $\sigma_{\solid,\Q_p}$ is $\D_\nuc(\Q_p)$-linear (though probably not $\D_\solid(\Q_p)$-linear). As brought to our attention by Wies\l{}awa Nizio\l{} it is however difficult to apply the Breen--Deligne resolution in $\D(S_v,\D_\solid(\Q_p))$. For this reason, we work with the ``internal condensed direction'' in $\D(S_v,\Q_p)$. If $M\in \DFF(S,\Q_p)$ is nuclear one can check that the internal condensed structure $T\mapsto \Gamma(S\times T,\sigma_{\Q_p}(M))$ on $\Gamma(S,\sigma_{\Q_p}(M))=\Hom_{\D_\solid(\Q_p)}(\Q_p,\Gamma(S,\sigma_{\solid,\Q_p}(M)))$ is naturally equivalent to the external condensed structure $T\mapsto \Hom_{\D_\solid(\Q_p)}(\Q_{p,\solid}[T],\Gamma(S,\sigma_{\solid,\Q_p}(M)))$ provided by $\sigma_{\solid,\Q_p}$. On a heuristic level, this may explain why the functor $\sigma_{\Q_p}$ seems to have nice properties only when restricted to nuclear objects in $\DFF(S,\Q_p)$.
\end{remark}

\begin{proposition} \label{sec:prel-subs-internal-hom-formula-for-big-nuclear-sheaves}
Let $S$ be a small v-stack, and let $P, N\in \DFF(S,\Q_p)$ be nuclear objects. Let $V\in \D_\solid(\Q_p)$ be nuclear. Then there is a natural isomorphism
\begin{align*}
    &\IHom_{\D(S_v,\Q_p)}(\sigma_{\Q_p}(V\otimes P),\sigma_{\Q_p}(N))\\
    &\qquad\qquad= \IHom_{\D(S_v,\Q_p)}(\sigma_{\Q_p}(P),\IHom_{\D(S_v,\Q_p)}(V\otimes_{\Q_p} (\Q_p)_S,\sigma_{\Q_p}(N))) 
\end{align*}
in $\D(S_v, \Q_p)$.
\end{proposition}
\begin{proof}
Arguing as in the proof of \Cref{sec:prel-subs-linearity-of-sigma-to-solid}, we obtain a natural map
\[
    V \tensor^\solid_{\Q_p} \sigma_{\Q_p}(P) \to \sigma_{\Q_p}(V \tensor P).
\]
in $\D(S_v,\Q_p)$. Together with the natural isomorphism
\begin{align*}
    &\IHom_{\D(S_v,\Q_p)}(V\otimes^\solid_{\Q_p}\sigma_{\Q_p}(P),\sigma_{\Q_p}(N))\\
    &\qquad\qquad= \IHom_{\D(S_v,\Q_p)}(\sigma_{\Q_p}(P),\IHom_{\D(S_v,\Q_p)}(V\otimes_{\Q_p} (\Q_p)_S,\sigma_{\Q_p}(N)))
\end{align*}
this yields a natural morphism from the left to the right in the claim. To check that this morphism is an isomorphism, we may assume that $S$ is a spatial diamond. Moreover, the internal Hom's do not change if we pass to the big quasi-pro-\'etale site $S_{\Qproet}$ (as their second argument is a $v$-sheaf). By \cref{rslt:basic-properties-of-big-solid-sheaves} and \cref{sec:prel-subs-sigma-maps-nuclear-objects-to-big-nuclear-sheaves} we may pass to $\D_\solid(S_\Qproet,\Z_p)$ (implicitly solidifiying all tensor products in $\D(S_\Qproet,\Z_p)$). Then the assertion follows from \Cref{sec:prel-subs-linearity-of-sigma-to-solid}.
\end{proof}

We have seen how big solid sheaves relate to nuclear sheaves on the Fargues--Fontaine curve. We now investigate the behavior of $\sigma_{\Q_p}$ with respect to countable limits of nuclear sheaves. We start with some fundamental results on solid functional analysis over $\Q_p$.

\begin{proposition} \label{solid-functional-analysis-over-q-p}
\begin{propenum}
    \item If $V\in \D_\solid(\Q_p)$ is quasi-separated, then $V$ is flat for the non-derived solid tensor product $H_0(-\otimes_{\Q_{p\solid}} -)$. In particular, Fr\'echet spaces and Smith spaces are flat.\footnote{We use the terminology from \cite[Definition 3.2, Definition 3.22]{rodrigues2022solid}.}

    \item \label{rslt:D-nuc-Q-p-stable-under-countable-products} The category $\D_\nuc(\Q_p)\subseteq \D_\solid(\Q_p)$ is stable under countable products, and an object in $\D_\solid(\Q_p)$ is nuclear if and only if its cohomology objects are nuclear.
    
    \item If $W\in \D_\solid(\Q_p)$ is nuclear, then for any compact object $V\in \D_\solid(\Q_p)$, the morphism $V^\ast\otimes_{\Q_{p\solid}}W\to \IHom_{\D_\solid(\Q_p)}(V,W)$ is an isomorphism.
    
    \item \label{rslt:nuclear-tensor-over-Q-p-solid-preserves-countable-limits} Let $W\in \D_\solid(\Q_p)$ be a finite complex of Fr\'echet spaces. Then $W$ is nuclear and the functor
    \[
        -\tensor_{\Q_{p\solid}} W \colon \D_\nuc(\Q_p)\to \D_\nuc(\Q_p)
    \]
    commutes with countable limits.
\end{propenum}
\end{proposition}
\begin{proof}
The first statement is \cite[Corollary A.28]{bosco2021p}. The second assertion follows from \cite[Theorem A.43]{bosco2021p} and the first statement. Indeed, it suffices to observe that if $T$ is a profinite set with associated Smith space $\Q_{p,\solid}[T]$, then the functor $((\Q_{p,\solid}[T])^\vee\otimes (-))(\ast)$ is exact on the heart of the natural $t$-structure on $\D_\solid(\Q_p)$. The third assertion is proven in \cite[Proposition A.55]{bosco2021p}, or follows from the assertion that compact objects in $\D_\solid(\Q_p)$ are stable under tensor products. The nuclearity of $W$ in the fourth assertion is \cite[Proposition A.64]{bosco2021p}. The claim on commutation with countable limits follows from \cite[Proposition A.66]{bosco2021p}, the flatness of Fr\'echet spaces, and the stability of $\D_\nuc(\Q_p)$ under countable limits.
\end{proof}

The commutation of the solid tensor product of a Fr\'echet space with countable inverse limits is specific to $\Q_p$, for example it fails for $-\otimes_{\Z_{p\solid}} \Z_p\langle U\rangle $ for the Tate algebra $\Z_p\langle U\rangle$ as it fails for $-\otimes_{\F_{p\solid}} \F_p[U]$. Similarly, the stability of nuclear objects under countable limits is specific to $\Q_p$ and fails for $\Z_p$ or $\F_p$. We now extend \Cref{solid-functional-analysis-over-q-p} to Huber pairs $(A,A^+)$ over $(\Q_p,\Z_p)$.

\begin{proposition} 
Let $(A,A^+)\to (A',A^{\prime,+})$ be a morphism of Huber pairs over $(\Q_p,\Z_p)$. 
\begin{propenum}
    \item \label{rslt:internal-hom-from-nuclear-stable-under-base-change} If $K\in \D_\solid(A,A^+)$ is $\omega_1$-compact and $N\in \D_\solid(A,A^+)$ nuclear, then $\IHom_{\D_\solid(A,A^+)}(K,N)$ is nuclear and its formation commutes with the base change $(A',A^{\prime,+})_\solid \otimes_{(A,A^+)_\solid}(-)$ when the latter has finite Tor dimension.

    \item \label{rslt:solid-base-change-preserves-countable-limits-of-nuclears} $\D_\nuc(A)\subseteq \D_\solid(A,A^+)$ is stable under countable limits, and
    \begin{align*}
        (A',A'^+)_\solid \tensor_{(A,A^+)_\solid} - \colon \D_\nuc(A)\to \D_\nuc(A')
    \end{align*}
    commutes with right-bounded countable products. If $(A',A'^+)_\solid \tensor_{(A,A^+)_\solid} -$ has finite Tor dimension then it commutes with all countable limits of nuclear modules.
\end{propenum}
If $A,A'$ are stably uniform, the same assertions holds with $\D_\solid$ replaced by $\D_{\hat\solid}$.
\end{proposition}
\begin{proof}
For part (i), if $K$ is compact, we can moreover assume that $K=(A,A^+)_\solid[T]$ for some profinite set $T$. In this case,
$$\IHom_{\D_\solid(A,A^+)}(K,N) \cong (C(T,\Z_p) \otimes_{\Z_{p,\solid}} (A,A^+)_\solid) \otimes_{(A,A^+)_\solid} N,$$
which is nuclear and whose formation commutes with base change in $(A,A^+)$.

The case that $K$ is $\omega_1$-compact, i.e., a countable colimit of compact objects, follows from this case and part (ii). For (ii), we may assume that $A^+$ is minimal (because the inclusion $\D_\solid(A,A^+)\subseteq \D_\solid(A,\Z)\isom \mathrm{Mod}_A\D_\solid(\Q_p)$ preserves products). Then $\D_\nuc(A)\isom \mathrm{Mod}_A \D_\nuc(\Q_p)$ by nuclearity of $A$ in $\D_{\solid}(\Q_p)$ and the first part of (ii) follows from \cref{rslt:D-nuc-Q-p-stable-under-countable-products}. For the second part of (ii) we first note that $(A',A'^+)_\solid \tensor_{(A,A^+)_\solid} - = A' \tensor_A -$ on nuclear modules, because the right-hand side lands in nuclear $A'$-modules, which are automatically solid over $A'^+$ (e.g. by the explicit description of nuclear modules in \cite[Lemma 2.18]{anschuetz_mann_descent_for_solid_on_perfectoids}) Now note that for every nuclear $(A, A^+)_\solid$-module $M$ we have
\begin{align*}
    A' \tensor_A M = \varinjlim_{n\in\Delta^\opp} A' \tensor_{\Q_{p\solid}} A^{\tensor n} \tensor_{\Q_{p\solid}} M.
\end{align*}
By \cref{rslt:nuclear-tensor-over-Q-p-solid-preserves-countable-limits} it follows that $A' \tensor_{\Q_{p\solid}} A^{\tensor n} \tensor_{\Q_{p\solid}} -$ preserves countable products for each $n$, and clearly this functor is right-bounded. This implies that $A' \tensor_A -$ preserves right-bounded countable products, because a uniformly right-bounded colimit over $\Delta^\opp$ commutes with products (as it is computed by a spectral sequence, cf. \cite[Proposition~1.2.4.5]{lurie-higher-algebra}). To prove the last part of the claim, we note that in the case that $A' \tensor_A -$ has finite Tor dimension we can formally deduce that this functor preserves all countable products (and hence all countable limits) by passing through the canonical isomorphism $M = \varinjlim_n \tau^{\le n} M$.

It remains to settle the statements for $\D_{\hat\solid}$ if $A,A'$ are stably uniform. The argument for the first assertion is the same. For the second assertion it suffices to show that a countable product in $\D_{\hat\solid}(A,A^+)$ of nuclear objects $N_n\in \D_\nuc(A)\subseteq \D_{\hat\solid}(A,A^+)$, $n\in \N$ is again nuclear. This assertion follows from the second point in \Cref{sec:defin-d_hats--4-comparison-for-modified-modules} because the functor $\alpha_\ast\colon \D_{\hat\solid}(A,A^+)\to \D_\solid(A,A^+)$ preserves and reflects products and nuclearity.
\end{proof}

From the above results on $p$-adic functional analysis we can deduce the following consequence for nuclear sheaves on the Fargues--Fontaine curve.

\begin{proposition} \label{rslt:sigma-commutes-countable-limits}
  Let $S$ be a small v-stack.
  The category $\D_{(0,\infty)}^{\mathrm{nuc}}(S)$ is stable under countable limits in $\D_{(0,\infty)}(S)$, and the functor $\sigma_{\Q_p}\colon  \DFF(S,\Q_p)\to \D(S_v,\Q_p)$ commutes with countable limits in $\DFF(S,\Q_p)$ of nuclear objects in $\DFF(S,\Q_p)$ (as defined in \ref{sec:v-desc-texorpdfstr-definition-nuclear-subcats}).
\end{proposition}
\begin{proof}
By v-descent of both sides, we can assume that $S$ is strictly totally disconnected. By definition of $\sigma_{\Q_p}$ and \cref{rslt:solid-base-change-preserves-countable-limits-of-nuclears}, the claim is implied if we can show that for any $S'\in S_v$ and any compact interval $I\in [0,\infty)$ with rational ends, the morphism $\mathcal{Y}_{I,S'} \to \mathcal{Y}_{I,S}$ has finite Tor dimension for $\D_{\hat\solid}$. Using \cite[Corollary 2.26]{anschuetz_mann_descent_for_solid_on_perfectoids} one can check that the property of having finite Tor dimension is stable under base change, e.g.,  base change to rational opens by shrinking $I$. Moreover, having finite Tor dimension can be checked after base change to $\Z_{p,\infty}$ because $\Z_p\to \Z_{p,\infty}$ is $p$-completely faithfully flat and descendable in $\D_\solid(\Z_p)$. Write $I=[r,s]$, $S=\Spa(R,R^+)$, $S'=\Spa(R',R^{'+})$. As noted above, we may assume $r=0$. Let $\varpi$ be a pseudo-uniformizer of $S$. Replacing $\varpi$, we may assume $s=1$. Then one has an almost isomorphism (\cite[Proposition II.1.1]{fargues-scholze-geometrization})
$$
\mathcal{O}^+(\mathcal{Y}_{I,S,\infty}) \overset{a}{\cong} A_{\mathrm{inf}}(R^+)\left[(\frac{p}{[\varpi]})^{1/p^\infty}\right]^{\wedge_{[\varpi]}}.
$$
and similarly for $\mathcal{O}^+(\mathcal{Y}_{I,S',\infty})$. It suffices to show that the morphism of adic analytic rings (in the sense of \cite[Definition 2.1]{anschuetz_mann_descent_for_solid_on_perfectoids} for the ideal generated by $[\varpi]$)
$$
\mathcal{O}^+(\mathcal{Y}_{I,S,\infty})_\solid \to \mathcal{O}^+(\mathcal{Y}_{I,S',\infty})_\solid
$$
has finite Tor dimension in the almost version of $\D_{\hat\solid}$. By \cite[Corollary 2.14]{anschuetz_mann_descent_for_solid_on_perfectoids}, this can be checked modulo a power of $[\varpi]$.
Thanks to the explicit presentation above, one sees that one has an almost isomorphism
$$
\mathcal{O}^+(\mathcal{Y}_{I,S,\infty})_\solid/[\varpi] \overset{a}{\cong} (R^+/[\varpi]) [T^{1/p^\infty}]_\solid,
$$
and similarly for $S'$. Hence the desired finite Tor dimension follows by base change from \cite[Proposition 3.1.20]{mann-mod-p-6-functors} (which uses that $S$ is totally disconnected). 
\end{proof}

\begin{proposition}
\label{sec:prel-another-internal-hom-formula-for-big-nuclear-sheaves}
Let $S$ be a small v-stack, and let $N \in \DFF(S,\Q_p)$ be a nuclear object. Let $V\in \D_\solid(\Q_p)$ be basic nuclear. Then the natural morphism
$$
 \sigma_{\Q_p} (\IHom_{\DFF(S,\Q_p)}(V\otimes \mathcal O_{\FF_S},N)) \to \IHom_{\D(S_v,\Q_p)}(V\otimes_{\Q_p} (\Q_p)_S,\sigma_{\Q_p}(N))
 $$
 is an isomorphism.
\end{proposition}
\begin{proof}
We may assume that $S$ is a spatial diamond, and replace $S_v$ by $S_{\Qproet}$. By \Cref{rslt:sigma-commutes-countable-limits}, the functor $\sigma_{\Q_p}$ commutes with countable limits. We write $V=\varinjlim_{n\in \N} W_n$ with $W_n\in \D_\solid(\Q_p)$ compact and each $W_n\to W_{n+1}$ of trace class (in $\D_\solid(\Q_p)$). We can conclude that
\begin{align*}
    &\sigma_{\Q_p} (\IHom_{\DFF(S,\Q_p)}(V\otimes \mathcal O_{\FF_S},N)) \\
    &\qquad= \varprojlim\limits_{n\in \N} \sigma_{\Q_p}(\IHom_{\DFF(S,\Q_p)}(W_n\otimes \mathcal O_{\FF_S}, N))\\
    &\qquad= \varprojlim\limits_{n\in \N} \sigma_{\Q_p}(W_n^\vee\otimes N)\\
    &\qquad= \varprojlim\limits_{n\in \N} W_n^\vee\otimes_{\Q_p}^\solid\sigma_{\Q_p}(N),
\end{align*}
using that $N$ and $W_n^\vee$ are nuclear and \Cref{sec:prel-subs-linearity-of-sigma-to-solid}. Let $S'\in S_\Qproet$ be a spatial diamond. By \Cref{solid-functional-analysis-over-q-p} and the $\D_\nuc(\Q_p)$-linearity of $\varepsilon_{S',\ast}$ we can conclude that this last term has sections $\varprojlim_{n\in \N} W_n^\vee\otimes_{\Q_p}^\solid\Gamma(S',\sigma_{\Q_p}(N))$ over $S'$.

For the right-hand side of the claim of this lemma, we first note that $V\otimes_{\Q_p} (\Q_p)_S$ denotes the naive tensor product on $\D(S_\Qproet,\Q_p)$ (where, by abuse of notation, $V$ denotes the pullback of $V\in \D_\solid(\Q_p)\subseteq \D(\ast_\proet,\Q_p)$ along $S_\Qproet\to \ast_\proet$).
Let $S'\in S_\Qproet$ be a strictly totally disconnected perfectoid space. Then
\begin{align*}
    &\Gamma(S',\IHom_{\D(S_\Qproet,\Q_p)}(V\otimes_{\Q_p} (\Q_p)_S, \sigma_{\Q_p}(N))) \\
    &\qquad= \Hom_{\D(S'_\Qproet,\Q_p)}(\varepsilon^\ast_{S'}(V\otimes_{\Q_p} (\Q_p)_S),\sigma_{\Q_p}(N)) \\
    &\qquad= \Hom_{\D(S'_\qproet,\Q_p)}(V\otimes_{\Q_p} (\Q_p)_{S'},\varepsilon_{S',\ast} \sigma_{\Q_p}(N))
\end{align*}
with $V\otimes_{\Q_p} (\Q_p)_{S'}$ now denoting the naive tensor product on $S'_\qproet$. As $\varepsilon_{S',\ast}\sigma_{\Q_p}(N)$ is solid (even nuclear), we may replace $V\otimes_{\Q_p} (\Q_p)_{S'}$ by the solid tensor product $V\otimes_{\Q_p}^\solid (\Q_p)_{S'}$, which is nuclear. We can conclude that the last $\Hom$ identifies with
\begin{align*}
    \Hom_{\D_\nuc(C(S',\Q_p))}(V\otimes^\solid_{\Q_p} C(S',\Q_p),\Gamma(S',\sigma_{\Q_p}(N))) = \Hom_{\D_\nuc(\Q_p)}(V,\Gamma(S',\sigma_{\Q_p}(N))).
\end{align*}
Here we abuse notation and identify denote by $\Gamma(S',\sigma_{\Q_p}(N))$ the nuclear $C(S',\Q_p)$-module associated to the nuclear sheaf $\varepsilon_{S',\ast}\sigma_{\Q_p}(N)$ on $S'_\qproet$. Now the above expression identifies with
\[
    \varprojlim\limits_{n\in \N} W_n^\vee\otimes^\solid_{\Q_p} \Gamma(S',\sigma_{\Q_p}(N))
  \]
  as desired.
\end{proof}

We can now put everything together and finally prove the main result of this subsection, providing a relation between $\sigma_{\Q_p}$ and internal Hom's.

\begin{proof}[Proof of \Cref{rslt:sigma-Qp-comparison-of-internal-hom}.]
Assume first that $M$ is dualizable. Then the assertion follows from \cite[Corollary 3.10]{anschutz2021fourier}. 
Next, assume that $M$ is of the form $M=V \otimes P$ with $V \in D_\solid(\Q_p)$ basic nuclear and $P\in \DFF(S,\Q_p)$ dualizable. Using \Cref{sec:prel-subs-internal-hom-formula-for-big-nuclear-sheaves}, we get
\begin{align*}
    & \IHom_{\D(S_v,\Q_p)}(\sigma_{\Q_p}(M),\sigma_{\Q_p}(N)) \\
    &\qquad= \IHom_{\D(S_v,\Q_p)}(\sigma_{\Q_p}(P),\IHom_{\D(S_v,\Q_p)}(V\otimes_{\Q_p} (\Q_p)_S,\sigma_{\Q_p}(N))) \\
    &\qquad= \IHom_{\D(S_v,\Q_p)}(V\otimes_{\Q_p} (\Q_p)_S,\IHom_{\D(S_v,\Q_p)}(\sigma_{\Q_p}(P),\sigma_{\Q_p}(N))).
\end{align*}
By the dualizable case just treated, this can be rewritten as 
\[ 
    \IHom_{\D(S_v,\Q_p)}(V\otimes_{\Q_p} (\Q_p)_S,\sigma_{\Q_p}(\IHom_{\DFF(S,\Q_p)}(P,N))),
\]
which by \Cref{sec:prel-another-internal-hom-formula-for-big-nuclear-sheaves} is naturally isomorphic to
\[ 
    \sigma_{\Q_p} \IHom_{\DFF(S,\Q_p)}(V\otimes 1,\IHom_{\DFF(S,\Q_p)}(P,N)) = \sigma_{\Q_p} \IHom_{\DFF(S,\Q_p)}(V\otimes P,N), 
\]
as desired.
 
Finally we come to the general case.
Write $M$ as a countable colimit of $M_n= V_n \otimes P_n$ with $V_n\in \D_\solid(\Q_p)$ basic nuclear and $P_n\in \D_\FF(S,\Q_p)$ dualizable for all $n$.
Since $\sigma_{\Q_p}$ commutes with (arbitrary) colimits, the right-hand side in the statement of \Cref{rslt:sigma-Qp-comparison-of-internal-hom} identifies with
$$
    \varprojlim_n \IHom_{\D(S_v,\D_\solid(\Q_p))}(\sigma_{\Q_p}(M_n),\sigma_{\Q_p}(N)).
$$
The left-hand side can trivially be rewritten as 
$$
\sigma_{\Q_p}(\varprojlim_n \IHom_{\DFF(S,\Q_p)}(M_n,N)).
$$
Hence, because by the previous step we already know the statement for $M_n$ and $N$ for every $n$, it suffices to show that we can pull out the limit in the above expression. But we already noticed that $\sigma_{\Q_p}$ commutes with countable limits of nuclear objects (\Cref{rslt:sigma-commutes-countable-limits}), and each $\IHom_{\DFF(S,\Q_p)}(M_n,N)$ is nuclear (cf. \Cref{rslt:internal-hom-from-nuclear-stable-under-base-change}). This finishes the proof.  
\end{proof}

\subsection{Relation to v-sheaves: the mod \texorpdfstring{$p$}{p}-case} \label{sec:relation-to-v-sheaves-mod-p-case}

Although it will not be necessary for the rest of this paper, we briefly pause in this subsection and restrict our attention to the mod $p$-case and analyze the functor
\[
\sigma_{\F_p}\colon \DFF(S,\mathbb{F}_p) = \D_{\hat\solid}(\mathcal{O}_{S/\varphi^\Z})\to \D(S_v,\F_p)
\]
or variants thereof.

We start by calculating some examples.

\begin{example} \label{sec:relation-v-sheaves-3-examples-for-sigma-f-p}
Let $S=\Spa(R,R^+)$ be an affinoid perfectoid space over $\F_p$ which admits a morphism of finite $\dimtrg$ to a totally disconnected perfectoid space. Then $\DFF(S,\mathbb{F}_p)\cong \mathrm{Mod}_{R[F^{\pm 1}]} \D_{\hat\solid}(R,R^+)$ identifies with the category of (solid) $\varphi$-modules over $R$, i.e. with pairs $(M,\varphi_M)$ with $M\in \D_{\hat\solid}(R,R^+)$ and $\varphi_M\colon \varphi^\ast M\to M$ an isomorphism, or equivalently with $R[F^{\pm 1}]$-modules in $\D_{\hat\solid}(R,R^+)$. Here, $R[F^{\pm 1}]$ denotes the non-commutative $R$-algebra with $Fr=\varphi(r)F$ for $r\in R$. In the following we compute $\sigma_{\F_p}$ in several examples:
\begin{examplesenum}
    \item Assume $(M,\varphi_M) = (R,\varphi)$. Then we have
    \[
        \Hom_\varphi(1,M)=[R \xto{\varphi-1} R]=\Gamma(S_\et,\F_p).
    \]
    Varying $S$ this implies that $\sigma_{\mathbb{F}_p}(R,\varphi)=\F_p$ is the constant sheaf $\F_p$ on $S_v$. 
    \item Assume $(M,\varphi_M)=(R[F^{\pm 1}],F\cdot )$. Then
    \[
        \Hom_\varphi( 1,M)=[R[F^{\pm 1}]\xto{F-1}R[F^{\pm 1}]]\cong R[-1]
    \]
    by the map $\sum_{n\in \Z}r_nF^n\mapsto \sum_{n\in \Z}\varphi^{-n}(r_n)$ in cohomological degree $1$ (here $r_n\in R$ for $n\in \Z$). This implies that $\sigma_{\mathbb{F}_p}(R[F^{\pm 1}],F)$ is the structure sheaf $S'\mapsto \mathcal{O}(S')$.

    \item Assume $(M,\varphi_M)=(R\langle F^{\pm 1}\rangle,F\cdot)$, where $R\langle F^{\pm 1}\rangle=R^+[F^{\pm 1}]^\wedge_\pi[1/\pi]$ for a pseudo-uniformizer $\pi \in R$ (and the $\pi$-adic completion resp.\ localization is as a left module over $R^+$). Then
    \begin{align*}
        \Hom_\varphi(1,M) &= [R\langle F^{\pm 1}\rangle\xto{F-1}R \langle F^{\pm 1}\rangle]\\
        &\isom R/R^{\circ\circ})[-1]
    \end{align*}
    by \cite[Lemma 3.10]{anschutz2021fourier}. This implies that $\sigma_{\mathbb{F}_p}(R\langle F^{\pm 1}\rangle,F)$ is the quotient sheaf $\mathcal{O}/\mathcal{O}^{\circ\circ}$.
\end{examplesenum}
\end{example}

We now recall the following calculation on the v-site. It might look surprising at first, but we warn the reader that $R\langle F^{\pm 1}\rangle$ is not a ring.\footnote{For example, one cannot multiply $\sum_{n\geq 0}\pi^n F^n$ with $\pi^{-1}$: the product ought to be $\sum_{n\geq 0}\pi^n\varphi^n(\pi^{-1})F^n=\sum_{n\geq 0}\pi^{n-p^n}F^n$, which does not lie in $R\langle F^{\pm 1}\rangle$.}

\begin{proposition} \label{rslt:Hom-from-O+-to-O-on-v-site}
Let $S=\Spa(R,R^+)$ be an affinoid perfectoid space of characteristic $p$. Consider $\mathcal{O}^+, \mathcal{O} \in \D(S_v,\F_p)$. Then $\Hom_{\D(S_v,\F_p)}(\mathcal{O}^+,\mathcal{O})$ is given by $R\langle F^{\pm 1}\rangle$. Here, the element $F$ maps to the natural inclusion $\mathcal{O}^+\to \mathcal{O}$. 
\end{proposition}
\begin{proof}
By \cite[Proposition 3.7]{anschutz2021fourier} the natural map
\[
    R^+\langle F^{\pm 1}\rangle\to \Hom_{\D(S_v,\F_p)}(\mathcal{O}^+,\mathcal{O}^+)
\]
is an almost isomorphism, where $R^+$ acts on the right-hand side through the right $\mathcal{O}^+$. This implies that
\[
    R\langle F^{\pm 1}\rangle\cong \Hom_{\D(S_v,\F_p)}(\mathcal{O}^+,\mathcal{O})
\]
as in \cite[Corollary 3.8]{anschutz2021fourier}.
\end{proof}

We get the following consequence. Abusing notation, we denote by $\sigma_{\mathbb{F}_p}$ also the composition
\[
    \D^a_{\hat\solid}(\mathcal{O}^+_{S/\varphi^\Z})^\nuc_{\pi-\textrm{compl}}\to \D_{\hat\solid}^\nuc(\mathcal{O}_{S/\varphi^\Z})\to \D(S_v,\F_p),
\]
where the first morphism is given by inverting a pseudo-uniformizer $\pi\in R$ (where $S=\Spa(R,R^+)$ is an affinoid perfectoid space of characteristic $p$). We note that $\D^a_{\hat\solid}(\mathcal{O}^+_{S/\varphi^\Z})^\nuc_{\pi-\textrm{compl}}$ is just the classical almost derived category of $R^+[F^{\pm 1}]$-modules in $\D^a(R^+)$, whose underlying almost $R^+$-module is $\pi$-adically complete. 

\begin{lemma} \label{sec:relation-v-sheaves-morphisms-of-completed-sums}
Let $S=\Spa(R,R^+)$ be an affinoid perfectoid space of characteristic $p$. Let $I$ be a set. Then $\sigma_{\mathbb{F}_p}(\widehat{\bigoplus}_{I} R^+\langle F^{\pm 1} \rangle)\cong \bigoplus_I \mathcal{O}/\mathcal{O}^{\circ\circ}[-1]$ and the natural map
\[
    \Hom_{\D^a_{\hat\solid}(\mathcal{O}^+_{S/\varphi^\Z})_{\pi-\textrm{compl}}}(R^+\langle F^{\pm 1}\rangle,\widehat{\bigoplus}_I R^+\langle F^{\pm 1} \rangle) \isoto \Hom_{\D(S_v,\F_p)}(\mathcal{O}/\mathcal{O}^{\circ\circ},\bigoplus_I \mathcal{O}/\mathcal{O}^{\circ\circ})
\]
is an isomorphism in $\D(\F_p)$.
\end{lemma}
\begin{proof}
The claim $\sigma_{\mathbb{F}_p}(\widehat{\bigoplus}_{I} R^+\langle F^{\pm 1} \rangle)\cong \bigoplus_I \mathcal{O}/\mathcal{O}^{\circ\circ}[-1]$ follows from \Cref{sec:relation-v-sheaves-3-examples-for-sigma-f-p} and \cite[Lemma 9.2]{bhatt2022prisms}.
For the other statement we calculate both sides. On the one hand we have
\begin{align*}
    & \Hom_{\D^a_{\hat\solid}(\mathcal{O}^+_{S/\varphi^\Z})^\nuc_{\pi-\textrm{compl}}}(R^+\langle F^{\pm 1}\rangle,\widehat{\bigoplus}_{I} R^+\langle F^{\pm 1} \rangle) \\
    &\qquad\qquad= \Hom_{\D^a(R^+)}(R^+,\widehat{\bigoplus_I} R^+\langle F^{\pm 1}\rangle) \\
    &\qquad\qquad= \Hom_{\D(R^+)}(R^{\circ\circ},\widehat{\bigoplus_I} R^+\langle F^{\pm 1}\rangle) \\
    &\qquad\qquad= \varprojlim\limits_{\varepsilon \to 0, i\to \infty} \Hom_{\D(R^+)}(\pi^{\varepsilon}R^+, \bigoplus_I R^+/\pi^i[F^{\pm 1}]) \\
    &\qquad\qquad= \varprojlim\limits_{\varepsilon\to 0, i\to \infty}\bigoplus_I \Hom_{\D(R^+)}(\pi^\varepsilon R^+, R^+/\pi^i[F^{\pm 1}])\\
    &\qquad\qquad= \varprojlim\limits_{\varepsilon\to 0, i\to \infty}\bigoplus_I \pi^{-\varepsilon}R^+/\pi^{i-\varepsilon}R^+[F^{\pm 1}]
\end{align*}
On the other hand we have
\begin{align*}
    & \Hom_{\D(S_v,\F_p)}(\mathcal{O}/\mathcal{O}^{\circ\circ},\bigoplus \mathcal{O}/\mathcal{O}^{\circ\circ}) \\
    &\qquad\qquad= \varprojlim\limits_{i\to \infty,\varepsilon \to 0} \Hom_{\D(S_v,\F_p)}(\pi^{-i}\mathcal{O}^+/\pi^{\varepsilon}\mathcal{O}^+, \bigoplus_I \mathcal{O}/\mathcal{O}^{\circ\circ}) \\
    &\qquad\qquad= \varprojlim\limits_{i\to \infty,\varepsilon \to 0}  \bigoplus_I \Hom_{\D(S_v,\F_p)}(\pi^{-i}\mathcal{O}^+/\pi^{\varepsilon}\mathcal{O}^+, \mathcal{O}/\mathcal{O}^{\circ\circ})
\end{align*}
using that the v-sheaf $\pi^{-i}\mathcal{O}^+/\pi^{\varepsilon}\mathcal{O}^+$ is pseudo-coherent (\cite[Corollary 3.8]{anschutz2021fourier}) in the last step. Now, similarly to the proof of \Cref{rslt:Hom-from-O+-to-O-on-v-site} we have
\[
\Hom_{\D(S_v,\F_p)}(\mathcal{O}^+,\mathcal{O}/\mathcal{O}^{\circ\circ})\cong \ R/R^{\circ\circ}[F^{\pm 1}],
\]
and thus
\[
\Hom_{\D(S_v,\F_p)}(\pi^{-i}\mathcal{O}^+/\pi^{\varepsilon}\mathcal{O}^+, \mathcal{O}/\mathcal{O}^{\circ\circ})\cong \pi^{-i+\varepsilon}R^{\circ\circ}/R^{\circ\circ}[F^{\pm 1}].
\]
Now, the pro-systems $\{\bigoplus_I \pi^{-\varepsilon} R^+/\pi^{i-\varepsilon}R^+[F^{\pm 1}]\}_{\varepsilon,i}$ and $\{\bigoplus_I \pi^{-i+\varepsilon}R^{\circ\circ}/R^{\circ\circ}[F^{\pm 1}]\}_{\varepsilon,i}$ are isomorphic, which implies the lemma.
\end{proof}

With the above computations we can now easily deduce the main result of this subsection, showing that the $+$-version of $\sigma_{\F_p}$ is fully faithful:

\begin{theorem} \label{sec:relation-v-sheaves-sigma-fully-faithful}
  Let $S$ be a small v-stack.
  Then the functor
\[
    \sigma_{\F_p}\colon \D^a_{\hat\solid}(\mathcal{O}^+_{S/\varphi^\Z})^\nuc_{\pi-\mathrm{compl}}\to \D(S_v,\F_p)
\]
is fully faithful on the subcategory of objects, which v-locally on a strictly totally disconnected perfectoid space $\Spa(R,R^+)$ over $S$ lie in the subcategory generated under finite colimits by the objects $\widehat{\bigoplus}_I R^+\langle F^{\pm 1}\rangle$ for some set $I$.
\end{theorem}
We thank the referee for making us aware of a wrong statement of \ref{sec:relation-v-sheaves-sigma-fully-faithful} in a previous version of this paper.
\begin{proof}
  We note that both sides satisfy v-descent in $S$, and hence we may reduce to the case that $S=\Spa(R,R^+)$ is a strictly totally disconnected perfectoid space of characteristic $p$.
  Then we consider the full subcategory of $\D^a_{\hat\solid}(\mathcal{O}^+_{S/\varphi^\Z})^\nuc_{\pi-\textrm{compl}}\cong \mathrm{Mod}_{R^+[F^{\pm 1}]}(\D^a(R^+))_{\pi-\textrm{compl}}^{\mathrm{nuc}}$ generated under finite colimits by the objects $\widehat{\bigoplus_I} R^+\langle F^{\pm 1}\rangle$ for varying sets $I$.
  Now, \Cref{sec:relation-v-sheaves-morphisms-of-completed-sums} implies that $\sigma_{\mathbb{F}_p}$ is fully faithful on the category generated by these objects under finite colimits.
\end{proof}

We now discuss difficulties that occur if one tries to get a version of \Cref{sec:relation-v-sheaves-sigma-fully-faithful} for $\mathcal{O}_{S/\varphi^\Z}$ instead of $\mathcal O^{+a}_{S/\varphi^\Z}$. The same kind of difficulties arise if one tries to prove fully faithfulness for $\sigma_{\Z_p}$ or $\sigma_{\Q_p}$. If $S=\Spa(R,R^+)$ is an affinoid perfectoid space in characteristic $p$, we set
\[
    \mathcal{A}:=\Hom_{\D(S_v,\F_p)}(\mathcal{O},\mathcal{O}).
\]
By \Cref{rslt:Hom-from-O+-to-O-on-v-site} we have
\[
    \mathcal{A}\cong \varprojlim (\ldots \xto{\cdot \pi }R\langle F^{\pm 1}\rangle \xto{\cdot \pi} R\langle F^{\pm 1}\rangle)
\]
for a pseudo-uniformizer $\pi\in R$, where the action of $\pi$ is from the right. Equivalently, the transition maps are given by $\sum_{n\in \Z} r_nF^n\mapsto \sum_{n\in \Z} r_n\pi^{p^n}F^n$. Using topological Mittag--Leffler (see \cite[Lemma 4.8]{condensed-complex-geometry}) we see that $\mathcal{A}$ is concentrated in degree $0$, and given by the ring of power series $\sum_{n\in \Z}r_nF^n\in R\langle F^{\pm 1}\rangle$ such that $|r_n|s^{p^n}\to 0$ for $n\to \infty$ and any $s>0$ while $|r_n|\to 0$ for $n\to -\infty$. We note that the functor
\[
    \sigma_{\mathbb{F}_p}\colon \D_{\hat\solid}^\nuc(S/\varphi^\Z)\to \D(S_v,\F_p)
\]
is not fully faithful (assuming $S$ admits a morphism of finite $\dimtrg$ to a totally disconnected perfectoid space, to make the objects well-defined):
\[
    \Hom_{\D_{\hat\solid}^\nuc(S/\varphi^\Z)}(R[F^{\pm 1}],R[F^{\pm 1}])=R[F^{\pm 1}]\ncong \mathcal{A}\cong \Hom_{\D(S_v,\F_p)}(\mathcal{O},\mathcal{O}),
\]
but $\sigma_{\mathbb{F}_p}(R[F^{\pm 1}])=\mathcal{O}$ by \Cref{sec:relation-v-sheaves-3-examples-for-sigma-f-p}.

We expect that the non-trivial commutation of solid tensor products with countable inverse limits of Fr\'echet spaces (cf.\ \cite[Corollary A.24]{bosco2021p}) hold true in our situation and that they show that $\mathcal{A}$ is an idempotent $R[F^{\pm 1}]$-algebra, which is compatible with base change. Moreover, we expect that $\sigma_{\F_p}(\mathcal{A})=\mathcal{O}$, which then implies that $\sigma_{\F_p}\colon \DFF(S,\F_p)\to \D(S_v,\F_p)$ is fully faithful when restricted to the full subcategory generated by $\mathcal{A}$ under \textit{finite} colimits. The fully faithfulness does not extend to infinite colimits. Namely, the problem is caused by the fact that $\mathcal{O}$ is not a pseudo-coherent sheaf on the v-site.

\begin{remark} \label{analogy-with-complex-rh}
The situation seems reminiscent of what happens for analytic Riemann-Hilbert over the complex numbers. If $X$ is a complex manifold, Clausen and Scholze prove that the base change to $\mathbb{C}$ of the Betti stack of $X$ is isomorphic to the analytic de Rham stack of $X$. In this analogy, quasicoherent sheaves on (the base change of) the Betti stack function as an analog of nuclear objects in $\mathcal{D}(S_v,\mathbb{F}_p)$ and the sheaf of analytic differential operators $\mathcal{D}_\infty$ on $X$ as an analogue of $\mathcal{A}$: note that, similar to the definition of $\mathcal{A}$, $\mathcal{D}_\infty$ can be identified with the derived internal endomorphisms of the structure sheaf of $X$ over the ``constant" sheaf $\mathbb{C}$ (\cite{ishimura1978homomorphismes}). But quasicoherent sheaves on the analytic de Rham aren't the whole category of $\mathcal{D}_\infty$-modules in quasicoherent sheaves on $X$; rather, $!$-pullback identifies them with a full subcategory thereof, killed by a certain idempotent algebra. Hence, following this analogy, an expectation would be that a (to be found) category of nuclear objects in $\mathcal{D}(S_v,\mathbb{F}_p)$ identifies with a certain full subcategory of $\D_{\hat\solid}^\nuc(S/\varphi^\Z)$, formed by objects killed by a certain idempotent algebra $B$ in $\D_{\hat\solid}^\nuc(S/\varphi^\Z)$. This would explain why $\sigma_{\mathbb{F}_p}$ is not fully faithful: the correct statement would then be that it factors through the category on the ``complementary, non-commutative open subspace'' of the idempotent algebra $B$ and gives a fully faithful functor on this quotient. But it seems to us that this expectation is too naive and that the situation is more complicated. At least we do not know how to formulate a precise and reasonable guess.
\end{remark}

\section{Applications and examples}
\label{sec:applications}
In this section we want to present some examples and applications of our theory.

\subsection{Around excision}
\label{sec:excision}

Excision is a useful tool for cohomology computations. It does not work in the naive sense for pro-\'etale $\Q_p$- (or $\Z_p$-)cohomology. We explain here what one correct version of excision is in our setting, and more precisely how one extracts it from the general excision result at the level of categorified locales (\cite[Lecture V]{condensed-complex-geometry}). This will be used in the next subsection.

\begin{definition} \label{def:smashing-spectrum}
Let $\mathcal{C}$ be a presentably symmetric monoidal stable category. The \textit{smashing spectrum} of $\mathcal{C}$ is the poset $\mathcal{S}(\mathcal{C})$ of idempotent algebras in $\mathcal{C}$. In other words, $\mathcal{S}(\mathcal{C})$ is the category
of objects $A$ in $\mathcal{C}$ endowed with an arrow
$$
    \mu\colon 1_{\mathcal{C}} \to A
$$
such that the natural map
$$
    \mu \otimes \mathrm{id}_A\colon  A \to A\otimes A
$$
is an equivalence (any such $A$ is automatically equipped with a unique and functorial $\mathbb E_\infty$-algebra structure). This poset is a locale, whose closed subsets correspond to idempotent algebras. 
\end{definition}

\begin{definition} \label{def:excision-locales}
Let $Z \subseteq \mathcal{S}(\mathcal{C})$ be a closed subspace corresponding to the idempotent algebra $A = A_Z$. Write $U$ for its formal open complement. We define
$$
    \mathcal{C}(Z) := \mathrm{Mod}_A(\mathcal{C}),
$$
which, by idempotency of $A$, is a full subcategory of $\mathcal{C}$, and 
$$
    \mathcal{C}(U) =\mathcal{C}/\mathcal{C}(Z).
$$
Moreover, we define the following functors:
\begin{defenum}
    \item $i_Z^\ast\colon \mathcal{C} \to \mathcal{C}(Z)$ is the natural base change and $j_U^\ast \colon \mathcal{C} \to \mathcal{C}(U)$ is the natural localization functor.
    
    \item $i_{Z,\ast}$ and $j_{U,\ast}$ are the (fully faithful) right adjoints of $i_Z^\ast$ and $j_U^\ast$ respectively.
    
    \item $i_{Z,!} = i_{Z,\ast}$ and $j_{U,!}$ is the (fully faithful) left adjoint of $j_U^\ast$.
    
    \item $j_U^!= j_U^\ast$ and $i_Z^!$ is the right adjoint of $i_{Z,\ast}$.
\end{defenum}
\end{definition}

\noindent One has, for any $X\in \mathcal{C}$,
$$ 
    j_{U,\ast} j_U^\ast X = \IHom_{\mathcal{C}}([1_{\mathcal{C}} \to A],X)
$$
and
$$
    j_{U,!}j_U^\ast X = [X \to X\otimes A],
$$
so that in particular one has excision triangles
$$
    j_{U,!}j_U^\ast \to \id \to i_{U,\ast}i_U^\ast, \qquad i_{U,\ast}i_{U}^! \to \id \to j_{U,\ast} j_U^\ast.
$$
The above formalism can be applied to the category of modified solid quasicoherent sheaves on relative Fargues--Fontaine curves. What makes it useful is the following statement.

\begin{proposition}
\label{prop:morphism-of-locales}
Let $X$ be a small v-stack. There is a unique morphism of locales
$$
    \mathcal{S}(\mathcal{D}_{[0,\infty)}(X)) \to |X|,
$$
sending an open sub-v-stack $j\colon U \subset X$ to the open sublocale in $\mathcal{D}_{[0,\infty)}(X)$ determined by the coidempotent coalgebra $j_!(\mathcal{D}_{[0,\infty)}(U))\subseteq \D_{[0,\infty)}(X)$.
\end{proposition}
\begin{proof}
  This follows from \Cref{sec:6-functors-dywzs-etale-morphisms-for-dywz}, which shows that open immersions give rise to categorical open immersions for $\mathcal{D}_{[0,\infty)}$.
  More concretely, for $j\colon U\subseteq X$ open, set $I_U:=j_!(1)$, which is a coidempotent coalgebra in $\D_{[0,\infty)}$ (\cite[Definition 2.9]{aoki2023sheaves}, we note that by \cite[Proposition 2.14]{aoki2023sheaves} the smashing spectrum is here equivalently given by coidempotent coalgebras). Then $j_!(\D_{[0,\infty)}(U))$ identifies with the full subcategory of $I_U$-comodules in $\D_{[0,\infty)}(X)$. Now we have to check that $U \mapsto I_U$ is a morphism of locales. Let $U_1,\ldots, U_n\subseteq X$ be a (possibly empty) collection of open substacks, with intersection $U$. Then $I_{U}=I_{U_1}\otimes\ldots \otimes I_{U_n}$ by the projection formula, and the tensor product of coidempotent coalgebras defines the intersection in $\mathcal{S}(\D_{[0,\infty)}(X))$. Now, let $(U_j)_{j\in J}$ be an arbitrary diagram of open substacks with union $U$. If $J$ is filtered, then $I_U=\varinjlim_{j\in J} I_{U_j}$ (as follows formally from the $\ast$-descent $\varinjlim_{j\in J}\D_{[0,\infty)}(U_j) \cong \D_{[0,\infty)}(U)$). Similarly, we can check that $I_{U_1\cup U_2}=\mathrm{cone}(I_{U_1\cap U_2}\to I_{U_1}\oplus I_{U_2})$ for open substacks $U_1,U_2\subseteq X$. Using $I_{U_1\cap U_2}\cong I_{U_1}\otimes I_{U_2}$, this shows that $U\mapsto I_U\in \mathcal{S}(\D_{[0,\infty)}(U))$ is compatible with finite unions. This finishes the proof.    
\end{proof}

We note that for $U=U_1\cup U_2$ the category of $I_U$-comodules is not simply the union of the categories of $I_{U_1}$- or $I_{U_2}$-comodules.

\begin{corollary} \label{cor-exicsion-Qp-cohomology}
Let $X$ be a small v-stack, and let $Z\subseteq |X|$ be a closed subset. Assume that $Z=\bigcap_{i\in J} |U_i|$ for a cofiltered system of open substacks $j_i\colon U_i\to X$, such that for each $i\in J$ there exists $i'\geq i$ and a closed subset $Z_{i}\subseteq |X|$ with a factorization $|U_{i'}|\subseteq Z_i\subseteq |U_i|$.
Then the idempotent algebra $A_Z\in \D_{[0,\infty)}(X)$ associated with the closed subset $|Z|$ is isomorphic to the colimit $\varinjlim_{i\in J} j_{i,\ast}1$.
\end{corollary}
\begin{proof}
For $Z'\subseteq |X|$ closed, let $A_{Z'}\in \D_{[0,\infty)}$ be the corresponding idempotent algebra. Then $A_Z=\varinjlim_{i\in J} A_{Z_i}$ because a morphism of locales preserves infinite intersections of closed subspaces (as it preserves arbitrary unions of opens). Given $i\in J$, the given factorization $|U_{i'}|\subseteq Z_i\subseteq |U_i|$ implies that the morphism $j_{i',\ast}(1)\to j_{i,\ast}(1)$ factors over $A_{Z_i}$. Thus by cofinality, we can conclude the statement. 
\end{proof}

In particular, in the situation and notations of \Cref{cor-exicsion-Qp-cohomology}, if $j: U=X\backslash Z \to X$ is the open complement of $Z$, we get an exact triangle in $\D_{[0,\infty)}(-)$ of the form  
$$
    j_{!}j^\ast \to \id \to \colim_{i\in J} j_{i,\ast}j_i^\ast.
$$
This is what we will make use of in the next subsections.

\begin{remark}
By \cite[Section 2]{anschutz2022p} closed sub-v-stacks are in bijection with closed weakly generalizing subsets on $|X|$. In the case that $X$ is a spatial diamond, the closed weakly generalizing subsets are exactly the closed subsets that are stable under generalization. Only these closed subsets have the chance to satisfy the assumptions in \Cref{cor-exicsion-Qp-cohomology}.
\end{remark}

\begin{remark}
A different version for excision on the curve gives a new viewpoint on why the passage from pro-\'etale $\Z_p$- to $\Q_p$-cohomology is more complicated than just inverting $p$: if $S$ is a small v-stack, then the pullback of the idempotent algebra $A$ from \Cref{sec:6-funct-texorpdfstr-variant-without-zero} to $S/\varphi^\Z$ defines an idempotent algebra $A_S\in \D_{[0,\infty)}(S/\varphi^\Z)$ with complementary open $\D_{(0,\infty)}(S/\varphi^\Z)$. Now, excision yields for each $M\in \D_{[0,\infty)}(S/\varphi^\Z)$ a fiber sequence
\[
    \Hom_{\D_{[0,\infty)}(S/\varphi^\Z)}(A_S,M)\to \Hom_{\D_{[0,\infty)}(S/\varphi^\Z)}(1,M)\to \Hom_{\D_{(0,\infty)}(S/\varphi^\Z)}(1,M_{(0,\infty)}),
\]
where $M_{|(0,\infty)}$ is the restriction of $M$ to $\D_{(0,\infty)}(S/\varphi^\Z)$. If $M=f_\ast(1)$, for a morphism $f\colon S'/\varphi^\Z\to S/\varphi^\Z$, then the middle term of this fiber sequence is given by pro-\'etale cohomology with $\Z_p$-coefficients of $S'$, while the right term by pro-\'etale $\Q_p$-cohomology. Hence, their difference is governed by the object $\Hom_{\D_{[0,\infty)}(S/\varphi^\Z)}(A_S,M)$ (which might be difficult to access).
\end{remark}

\subsection{Pro-\'etale cohomology of smooth rigid spaces} \label{sec:misc}

An immediate corollary of \Cref{rslt:smooth-maps-of-adic-spaces-are-cohom-smooth} and \Cref{rslt:dualizing-complex-on-smooth-adic-space} is a version of Poincar\'e duality: 

\begin{theorem} \label{thm:poincare-duality-smooth-level-ff-curve}
Let $g\colon Y\to X$ be a smooth morphism of analytic adic spaces over $\Q_p$, which is equidimensional of dimension $d$. There is a canonical isomorphism in $\DFF(X,\Q_p)$
$$
    \underline{\mathrm{Hom}}_{\DFF(X,\Q_p)}(g_! M, 1) = g_\ast \underline{\mathrm{Hom}}_{\DFF(Y,\Q_p)}(M,1)(d)[2d].
$$
\end{theorem}
\begin{proof}
This is formally implied by $g^!(-) = g^\ast(-)\otimes g^!(1) = g^\ast(-)(d)[2d]$.
\end{proof}

As a first consequence, we get finiteness results for proper, smooth pushforward of $\Q_p$-local systems. A similar assertion has been announced by Kedlaya--Liu in relation with \cite{kedlaya2016finiteness}. Recall that in \cite[Remark 3.12]{anschutz2021fourier} a theory of relative Banach--Colmez spaces over a small v-stack $S$ is introduced, namely, as the essential image in $\D(S_v,\Q_p)$ under the functor $\sigma_{\Q_p}$ (\Cref{sec:relation-v-sheaves-mod-p}) of the category of dualizable objects in $\DFF(S,\Q_p)$.

\begin{theorem} \label{proper-pushforward-of-q-p-local-systems}
Let $f\colon Y\to X$ be a proper smooth morphism of analytic adic spaces over $\Q_p$, equidimensional of dimension $d$, and let $\mathbb{L}$ be a $\Q_p$-local system on $Y$ (or more generally a relative Banach--Colmez space). Then the derived pushforward $f_{v,\ast}\mathbb{L}\in \D(X_v,\Q_p)$ is a relative Banach--Colmez space. Moreover, if $f$ is of constant relative dimension $d$ then there is natural isomorphism
\[
    \IHom_{\D(X_v,\Q_p)}(f_{v,\ast}\mathbb{L},\Q_p) = f_{v,\ast}(\mathbb{L}^\vee(d)[2d]).
\]
\end{theorem}
\begin{proof}
First note that in any 6-functor formalism, pushforward along a suave and prim map $f\colon Y \to X$ preserves dualizable objects: By \cite[Corollary~4.5.18(i)]{heyer-mann-6ff} dualizable objects on $Y$ are $f$-suave, hence by \cite[Lemma~4.5.16(ii)]{heyer-mann-6ff} their pushforward is $\id$-suave on $X$, equivalently dualizable on $X$. Using this observation together with \cref{rslt:sigma-ast-compatible-with-pushforward}, the definition of relative Banach--Colmez spaces and \Cref{rslt:sigma-Qp-comparison-of-internal-hom} (applied here only for $M$ dualizable and $N$ being the unit), we may argue for the pushforward of dualizable objects along $f_\ast\colon \DFF(Y,\Q_p) \to \DFF(X,\Q_p)$ instead. Then the assertion follows from \Cref{thm:poincare-duality-smooth-level-ff-curve}.
\end{proof}

\begin{remark} \label{rmk:arithmetic-duality}
If in the previous discussion $X=\Spa(\Q_p)$ and the morphism $Y\to X$ is replaced by the morphism $Y^\diamond/\varphi^{\Z} \to \Spd(\Q_p)/\varphi^{\Z}$, one could also consider the composition $Y^\diamond/\varphi^{\Z} \to \Spd(\Q_p)/\varphi^{\Z} \to \Spd(\F_p)$ and use in addition \Cref{sec:case-texorpdfstr-smoothness-of-spd-q-p-mod-hi}. This would give Poincar\'e duality for ``arithmetic" pro-\'etale cohomology. Such results have been proved independently of ours by Zhenghui Li \cite{zhenghui_li_2024} (to compare his formulation with ours, use \Cref{prop:description-compactly-supp-cohomology-via-excision} proved below and note that by \Cref{rslt:compute-D-0-infty-on-Fp-algebra} down below, the category $\D_{(0,\infty)}(\Spd(\mathbb{F}_p))$ identifies with $\D_{\solid}(\Q_p)$ in such a way that pushforward in $\D_{(0,\infty)}(-)$ along $\Spd(\Q_p)/\varphi^{\Z} \to \Spd(\F_p)$ agrees with the functor $R\Gamma(\mathrm{Gal}_{\Q_p}, R\Gamma(\FF_C,-))$ from $\mathrm{Gal}_{\Q_p}$-equivariant objects in $\D_{\solid}(\FF_C)$ to $\D_\solid(\Q_p)$).
\end{remark}

For the rest of this subsection, we fix once and for all a complete algebraically closed non-archimedean extension $C$ of $\Q_p$. We want to make \Cref{thm:poincare-duality-smooth-level-ff-curve} more explicit for certain partially proper smooth rigid spaces over $\Spa(C)$. Unless the notation makes it clear, six functors always refer to the $6$-functor formalism for $\DFF(-,\Q_p)$. We will also abuse notation and not distinguish between a morphism of rigid spaces and the induced morphism of diamonds.

\begin{definition} \label{def:generalized-stein-space}
A partially proper rigid space $X$ over $\Spa(C)$ is said to be \emph{countable at infinity} if one can write $X=\cup_{n \in \N} X_n$ as a countable increasing union of qcqs open subspaces $X_n$, such that for each $n$ the inclusion $X_n \subset X_{n+1}$ factors through the adic compactification $Y_n$ of $X_n$ over $\Spa(C)$.
\end{definition}

Examples of \cref{def:generalized-stein-space} are proper rigid varieties and Stein rigid varieties (for which by definition the qcqs rigid spaces in the countable union can be taken to be affinoid). Local Shimura varieties conjecturally provide examples of arithmetic and representation-theoretic interest.

Let us first observe that one can rewrite compactly supported cohomology (in the sense of the 6-functor formalism $\DFF(-,\Q_p)$) of a partially proper rigid space $X$ over $\Spa(C)$ that is countable at infinity using excision. Namely, by definition, we can find an increasing open cover $X=\cup_{n \in \mathbb{N}} X_n$ by qcqs rigid spaces $X_n$, such that for each $n$ the inclusion $X_n \subset X_{n+1}$ factors through the adic compactification $Y_n$ of $X_n$ over $\Spa(C)$. We set $f_n\colon X_n \to \Spa(C)$, $h_n\colon X\backslash Y_n \to \Spa(C)$.

\begin{proposition} \label{prop:description-compactly-supp-cohomology-via-excision}
Let $f: X \to \Spa(C) $ be a partially proper countable at infinity rigid space. Using the notations introduced above, we have
$$
    f_!  1= \varinjlim\limits_{n\in \N} f_{n,!} 1 = \fib(f_\ast 1 \to  \colim_n h_{n,\ast} 1),
$$
\end{proposition}
\begin{proof}
Set $Z_n:=Y_n\setminus X_n$. Then $Z_{n+1}\cap X_n=\emptyset$ (as $X_n\subseteq X_{n+1}$) and $Z_{n+1}$ is quasi-compact (being closed in $Y_{n+1}$). This implies that there exists a quasi-compact open subset $V_{n+1}\subseteq Y_{n+1}$ such that $Z_{n+1}\subseteq V_{n+1}$ and $Y_n\cap V_{n+1}=\emptyset$. As $X_n$ is open in $Y_{n+1}$ we get that $X_n\cap \overline{V_{n+1}}=\emptyset$, and thus $X_n\subseteq U_n:=Y_{n+1}\setminus \overline{V_{n+1}}\subseteq X_{n+1}$. As $V_{n+1}$ is quasi-compact (and hence pro-constructible in the spectral space $Y_{n+1}$), $\overline{V_{n+1}}$ is the set of specializations of $V_{n+1}$ in $Y_{n+1}$, and as specializations of analytic adic spaces don't change the residue field we see that $\overline{V_{n+1}}$ is stable under generalizations. Thus, $U_n\subseteq X_{n+1}$ is a partially proper open subspace containing $X_n$. Set $g_n\colon U_n\to \Spa(C)$ as the natural morphism, and $j_n\colon U_n \injto X$. Then (using that $X_n, U_n$ are open subspaces of $X$)
  \[
    f_!  1=\varinjlim\limits_{n\in \N} f_{n,!} 1=\varinjlim\limits_{n\in \N} g_{n,!} 1.
  \]
By \Cref{prop:morphism-of-locales}, the closed subset $X \backslash U_n \subseteq X$ defines an idempotent algebra $A_n \in \DFF(X,\Q_p)$. We see that 
$$
j_{n,!} 1 = \fib(1 \to A_n).
$$
As $Y_n\subseteq U_n\subseteq Y_{n+1}$, we see that the systems $\{A_n\}_{n\geq 0}$ and $\{\iota_{n,\ast}(1)\}_{n\geq 0}$ are cofinal among each other, where $\iota_n\colon X\setminus Y_n\to X$. 
Applying $f_\ast$ and passing to the colimit over $n$, we therefore get by cofinality
$$
\colim_n f_\ast j_{n,!} 1 = \fib(f_\ast 1 \to  \colim_n h_{n,\ast}1).
$$
Since the morphism $j_n: U_n \to X$ factors through the adic compactification $Y_n$ we have $f_\ast j_{n,!} = g_{n,!}$. Hence, we finally deduce that
$$
    f_! 1 = \varinjlim_{n\in \N} g_{n,!} 1 = \fib(f_\ast 1 \to  \colim_n h_{n,\ast} 1),
$$
as desired.
\end{proof}

Based on the description in \cref{prop:description-compactly-supp-cohomology-via-excision}, our next goal, \Cref{prop:compactly-supported-coh-is-countable-colimit-compact-times-perfect-complex}, is to show that compactly supported cohomology of smooth partially proper countable at infinity rigid spaces is not an arbitrary solid quasi-coherent complex on the Fargues--Fontaine curve for $C$, but has a rather specific shape. We start with some basic results about nuclear sheaves and pushforward.

\begin{lemma} \label{rslt:pushforward-for-DFF-Qp-preserves-nuclearity}
Let $f\colon S'\to S$ be a quasi-separated morphism of small v-stacks. Assume that $S'=\bigcup_{n\in \N} S'_n$ is a countable union in the analytic topology with $f_n\colon S'_n\to S$ $p$-bounded. Let $M\in \D_{(0,\infty)}^\nuc(S')$ be a nuclear object. Then $f_\ast(M)$ is nuclear. Moreover, in this case the formation of $f_\ast(M)$ commutes with base change in $S$.
\end{lemma}
\begin{proof}
By \cref{rslt:D-0-infty-pushforward-base-change} and \cref{rslt:sigma-commutes-countable-limits} we may reduce to the case that $S$ is a totally disconnected perfectoid space and that $f$ is qcqs and $p$-bounded. Indeed, the compatiblity with base change follows by the fact that base change of nuclear modules commutes with the countable limit $f_\ast(M)=\varprojlim_{n\in \N}f_{n,\ast}(M_{|S'_n})$. By stability of $\D_{(0,\infty)}^\nuc(S)$ under countable limits (this is proven in \Cref{rslt:sigma-commutes-countable-limits}) and by picking a quasi-pro-\'etale hypercover of $S'$ by totally disconnected perfectoid spaces, we see that we may reduce to the case that $S'$ is a totally disconnected perfectoid space. In this case, the claim is clear.
\end{proof}

Using \cref{rslt:pushforward-for-DFF-Qp-preserves-nuclearity} we can obtain the following strengthening of \Cref{rslt:sigma-ast-compatible-with-pushforward} for $\Q_p$-coefficients. 

\begin{corollary} \label{rslt:pushforward-commutes-with-sigma-nuclear-Qp-case}
In the situation of \Cref{rslt:pushforward-for-DFF-Qp-preserves-nuclearity}, the natural morphism
\begin{align*}
    \sigma_{\Q_p} f_\ast(M) \isoto f_{v,\ast} \sigma_{\Q_p}(M)   
\end{align*}
is an isomorphism for every nuclear $M\in \DFF(S',\Q_p)$.
\end{corollary}
\begin{proof}
The proof is identical to the one of \Cref{rslt:sigma-ast-compatible-with-pushforward} except that \Cref{rslt:D-0-infty-pushforward-colim-and-base-change} is replaced by \Cref{rslt:pushforward-for-DFF-Qp-preserves-nuclearity} (the critical assertion is the compatibility with base change).
\end{proof}

\begin{remark} \label{rmk:Bosco-nuclear-sheave-examples}
Let $f\colon S'\to S$ be a $p$-bounded morphism between small v-stacks. Assume that $S$ is an affinoid perfectoid space. Then the choice of a pseudo-uniformizer on $S$ defines a ``radius'' function $\mathrm{rad}$ on $|Y_{(0,\infty),S}|$, and for a closed interval $I\subseteq (0,\infty)$ we denote by $\mathcal{Y}_{I,S}=\Spa(B_{S,I},B_{S,I}^+)$ the corresponding quasi-compact open subspace (cf.\ \cite[Proposition II.1.16]{fargues-scholze-geometrization}). If $N\in\D^\nuc_{\hat\solid}(B_{S,I},B_{S,I}^+)$, then we obtain by $\ast$-pushforward to $\mathcal{Y}_{(0,\infty),S}$ a nuclear object $\widetilde{N}\in \D_{(0,\infty)}(S)$. Moreover, for a perfectoid space $T$ with a morphism $g\colon T\to S'$, we set $N_{T}$ as the $B_{T,I}$-module corresponding to $g^\ast f^\ast(\widetilde{N})$, in other words, $N_T=B_{T,I}\otimes_{B_{S,I}}N$, with the tensor product happening in $\D_\solid(\Q_p)$ (we use here that $N$ is nuclear). Calculating the pushforward $f_\ast(N_{S'})$ via a \v{C}ech nerve for a quasi-pro-\'etale surjection $T\to S'$ by an affinoid perfectoid space, we can conclude that $f_\ast(N_{S'})\in \D_{(0,\infty)}(S)$ lies in the full subcategory $\D^\nuc_{\hat\solid}(B_{S,I},B_{S,I}^+)$ (embedded via $\ast$-pushforward) and is given by the $B_{S,I}$-module $M_{I}:=\Gamma(S'_{\qproet},\mathcal{F}_N)$, where $\mathcal{F}_N$ is the pro-\'etale sheaf over $S'$ sending an affinoid perfectoid space $T\to S'$ to $N_T$. As $f_\ast(N_{S'})$ is a sheaf on $\mathcal{Y}_{(0,\infty),S}$, the collection $\{M_I\}_{I\subseteq (0,\infty) \textrm{ compact}}$ is a ``coadmissible'' $B_S=\mathcal{O}({Y_{(0,\infty),S}})$-module in the terminology of \cite[Construction 6.5]{bosco2023rational}.
In the special cases that $S=\Spa(C)$ and $N=B_{S,I}, B_{\mathrm{dR}}, B_{\mathrm{dR}}^+$ (the latter two are nuclear by \cref{rslt:solid-base-change-preserves-countable-limits-of-nuclears}) we therefore recover objects considered by Bosco in \cite{bosco2023rational}, \cite{bosco2021p}, e.g. \cite[Corollary 6.17]{bosco2021p}.
\end{remark}

We now come back to the problem of computing the cohomology with compact support. Before addressing the general case, we illustrate the situation by computing the example of the analytic affine line.

\begin{lemma}
Let $X:=\mathbb{A}^{1,\an}_{C}\subseteq \overline{X}:=\mathbb{P}^{1,\an}_{C}$ with complement the point at infinity $\infty$. Let $f: X \to \mathrm{Spa}(C)$ be the structure morphism, and $\overline{f}\colon \overline{X}\to \Spa(C)$. For an open quasi-compact subset $U\subseteq \overline{X}$, we let $f_U\colon U \to \Spa(C)$ be the induced morphism. Then
\[
    f_! 1_X=\fib(\overline{f}_\ast(1_{\overline{X}})\to \varinjlim_{U \ni \infty} f_{U,\ast}(1_U)) \isom i_{x_C,\ast}(\mathcal{O}_{\overline{X},\infty})[-2]\oplus 1[-2].
\]
Here, $i_{x_C}\colon \Spa(C) \injto \FF_{C^\flat}$ is the closed Cartier divisor defined by the untilt $C$, and this is a closed Cartier divisor on the Fargues--Fontaine curve $\FF_{C^\flat}$ for $C^\flat$.
\end{lemma}
The proof is very similar to calculating the compactly supported pro-\'etale cohomology of $\mathbb{A}^{1,\an}_{C}$. However, we want to argue for the compactly supported cohomology in the $\D_{(0,\infty)}$-formalism, i.e., on the curve, and this makes our life a bit more difficult. This is why we first had to add \Cref{rslt:pushforward-for-DFF-Qp-preserves-nuclearity} and \Cref{rmk:Bosco-nuclear-sheave-examples} before we could invoke the results of \cite{bosco2021p}.

\begin{proof}
The first isomorphism has been discussed in \Cref{prop:description-compactly-supp-cohomology-via-excision}. For the second, we use \Cref{rmk:Bosco-nuclear-sheave-examples} with $N$ running through the (non-zero) terms of the fundamental exact sequence
\[
    0\to \Q_p\to B_e\to B_{\mathrm{dR}}/B_{\mathrm{dR}}^+\to 0.
\]
To avoid confusion, the corresponding objects in $\DFF(\Spa(C),\Q_p)$ are denoted in this proof by $1$ (the monoidal unit), $\mathcal{O}(\infty \cdot x_C):=\varinjlim_{n\in \N} \mathcal{O}(n \cdot x_C)$ (with transition maps given by multiplication by Fontaine's element $t$) and $\mathcal{F}$, and similarly for the corresponding objects in $\DFF(U,\Q_p)$, with a subscript $U$ added. Then from \cite[Proposition 7.17]{bosco2021p} we can conclude that the natural morphism
\[
    \mathcal{O}(\infty \cdot x_C) \to \varinjlim_{U \ni \infty} f_{U,\ast}(\mathcal{O}(\infty \cdot x_C)_U)
\]
is an isomorphism, because the colimit of the higher de Rham cohomologies of $U$ vanishes. Similarly, \cite[Corollary 6.17]{bosco2021p} implies that the cofiber $M$ of the natural morphism
\[
     \mathcal{F} \to \varinjlim_{U \ni \infty} f_{U,\ast}(\mathcal{F}_U)
     \]
    is isomorphic to $\mathcal{O}_{\overline{X},\infty}$. By the fundamental exact sequence, we deduce that
$$
H^0(\varinjlim_{U \ni \infty} f_{U,\ast}(1_U))=1, ~~ H^1(\varinjlim_{U \ni \infty} f_{U,\ast}(1_U))=i_{x_C,\ast}(\mathcal{O}_{\overline{X},\infty}),
$$
and the other cohomology groups vanish. Since $\overline{f}_\ast(1_{\overline{X}})\cong 1\oplus 1[-2]$ (this is implied by \Cref{rslt:primitive-comparison-for-smooth-proper-map} and the classical calculation of pro-\'etale $\mathbb{Z}_p$-cohomology on $\overline{X}$), the result now follows from the first isomorphism of the lemma.
\end{proof}

This example suggests the following generalization, which is based on results of Bosco (\cite{bosco2021p}, \cite{bosco2023rational}). Note that in the statement of his results, Bosco assumes that the rigid spaces are paracompact. This is satisfied for partially proper rigid spaces that are countable at infinity. 

\begin{proposition} \label{prop:compactly-supported-coh-is-countable-colimit-compact-times-perfect-complex}
Let $C$ be the completed algebraic closure of a complete discretely valued non-archimedean extension of $\Q_p$ with perfect residue field. Let $f: X \to \mathrm{Spa}(C)$ be a smooth partially proper countable at infinity rigid space. Then $f_!1$ belongs to the full subcategory $\mathcal{C} \subset \DFF(\Spa(C),\Q_p)$ generated under \emph{countable} colimits by tensor products of dualizable objects with basic nuclear solid $\Q_p$-vector spaces.
\end{proposition}
\begin{proof}
We will make use again of the fundamental exact sequence
  \[
    0\to \Q_p\to B_e\to B_{\mathrm{dR}}/B_{\mathrm{dR}}^+\to 0.
  \]
As in the previous proof, let $x_C$ be the point on the Fargues--Fontaine curve of $C^\flat$ corresponding to the untilt $C$, and let $\mathcal{O}(\infty\cdot x_{C}):=\varinjlim_{n\in \N}\mathcal{O}(n\cdot x_C)\in \DFF(\Spa(C),\Q_p)$, $\mathcal{F} \in \DFF(\Spa(C),\Q_p)$ be the objects corresponding to the middle and right terms. Note that $\mathcal{O}(\infty\cdot x_C)\in \mathcal{C}$ (it even belongs to the subcategory generated by countable colimits of perfect complexes). We now make the following claim:
\begin{itemize}
    \item[($*$)] We have $f_!f^\ast(\mathcal{O}(\infty\cdot x_C)) \in \mathcal{C}$.
\end{itemize}
For the proof of this claim, we make use of the notation introduced in the proof of \Cref{prop:description-compactly-supp-cohomology-via-excision}. We have
\[
    f_! f^\ast(\mathcal{O}(\infty\cdot x_C)) = \varinjlim_{n\in \N} f_{n,!}f^\ast_n(\mathcal{O}(\infty\cdot x_C)) = \varinjlim_{n\in \N} g_{n,!} g^\ast_n(\mathcal{O}(\infty\cdot x_C)),
\]
and thus it suffices to show that $g_{n,!} g_n^\ast(\mathcal{O}(\infty\cdot x_C)) \in \mathcal{C}$.
By \Cref{prop:morphism-of-locales} the closed subsets $Y_{n+1}, \overline{V_{n+1}}\subseteq X_{n+2}$ define idempotent algebras $A_{Y_{n+1}}, A_{\overline{V_{n+1}}}\in \DFF(X_{n+2},\Q_p)$. Using \Cref{cor-exicsion-Qp-cohomology} and \cite[Theorem 4.1]{bosco2023rational} we see that $f_{n+2,!}(A_{Y_{n+1}}\otimes f^\ast_{n+2}\mathcal{O}(\infty\cdot x_C))$ is calculated by overconvergent Hyodo-Kato cohomology of $X_{n+1}$, and thus $f_{n+2,!}(A_{Y_{n+1}}\otimes f^\ast_{n+2} \mathcal{O}(\infty\cdot x_C))\in \mathcal{C}$ as overconvergent Hyodo-Kato cohomology of the qcqs rigid-analytic variety $X_{n+1}$ is finite dimensional (\cite[Theorem 3.29]{bosco2023rational}). Similarly, $f_{n+2,!}(A_{\overline{V_{n+1}}}\otimes f^\ast_{n+2} \mathcal{O}(\infty\cdot x_C))$ is calculated by overconvergent Hyodo-Kato cohomology of $V_{n+1}$. As $V_{n+1}$ is as well qcqs, we also get $f_{n+2,!}(A_{\overline{V_{n+1}}}\otimes f^\ast_{n+2}\mathcal{O}(\infty\cdot x_C))\in \mathcal{C}$. By excision, $g_{n,!}g_n^\ast(\mathcal{O}(\infty\cdot x_C))$ is the fiber of the natural morphism
\[
    f_{n+2,!}(A_{Y_{n+1}}\otimes f^\ast_{n+2} \mathcal{O}(\infty\cdot x_C))\to f_{n+2,!}(A_{\overline{V_{n+1}}}\otimes f^\ast_{n+2}\mathcal{O}(\infty\cdot x_C)),
\]
and thus also in $\mathcal{C}$. This finishes the proof of ($*$). We note that the proof even shows that $f_! f^\ast(\mathcal{O}(\infty\cdot x_C))$ belongs to the subcategory generated under countable colimits by perfect complexes. We next prove the following claim:
\begin{itemize}
    \item[($**$)] We have $f_!f^\ast (\mathcal{F}) \in \mathcal{C}$.
\end{itemize}
Arguing like in the proof of ($*$), we see that it is enough to prove that the overconvergent pro-\'etale $\mathbb{B}_{\mathrm{dR}}/\mathbb{B}_{\mathrm{dR}}^+$-cohomology of a smooth affinoid rigid space is in $\mathcal{C}$ (where we slightly abuse notation, following \Cref{rmk:Bosco-nuclear-sheave-examples}).
Passing to a countable colimit, we may argue for the overconvergent $\mathbb{B}^+_{\mathrm{dR}}/t^k$-cohomology instead, and then by a further d\'evissage for the $\mathbb{B}^+_{\mathrm{dR}}/t=\widehat{\mathcal{O}}$-cohomology.
By \cite[Proposition 3.23]{scholze2013perfectoid} this reduces to the sheaves of overconvergent differentials, viewed as a quasi-coherent sheaf $\mathcal{G}$ on the Fargues--Fontaine curve placed at $\infty$.
Now, the $C$-vector space of overconvergent differentials is a basic nuclear $C$-vector space and one can check easily this implies that $\mathcal{G}$ lies in $\mathcal{C}$.
This finishes the proof of ($**$).

By putting together the claims ($*$) and ($**$) together with the distinguished triangle (coming from the fundamental exact sequence)
$$
    f_! 1 \to f_!(f^\ast\mathcal{O}(\infty\cdot x_C)) \to f_!(f^\ast \mathcal{F}),
$$
we deduce that $f_! 1 \in \mathcal C$.
\end{proof}

Now we finally can come back to Poincar\'e duality. Let $f: X \to \mathrm{Spa}(C)$ be a smooth rigid space, of pure dimension $d$, with $C$ as in \Cref{prop:compactly-supported-coh-is-countable-colimit-compact-times-perfect-complex}. By \cref{thm:poincare-duality-smooth-level-ff-curve} we have
$$
    f_\ast 1(d)[2d] = \IHom_{\DFF(X,\Q_p)}(f_! 1, 1).
$$
We now want to translate this statement to a statement about v-sheaves on $\Spa(C)$. The results of this subsection and \cref{sec:an-overc-riem} culminate in the following duality statement:
 
\begin{theorem} \label{thm:poincare-duality-for-general-partially-proper-spaces}
Let $C$ be the completed algebraic closure of a complete discretely valued non-archimedean extension of $\Q_p$ with perfect residue field. Let $f\colon X \to \Spa(C)$ be a smooth partially proper rigid space over $\Spa(C)$ of pure dimension $d \ge 0$. Write $X$ as a filtered colimit of qcqs open subspaces $X_i$, indexed by a filtered partially ordered set $I$, such that if $i<i'$, the morphism $X_i \to X_{i'}$ factors through the adic compactification $Y_i$ of $X_i$. For each $i\in I$ denote $h_i\colon X \backslash Y_i \to \Spa(C)$. Then there is a natural isomorphism
$$ 
    f_{v,\ast}\Q_p(d)[2d] = \IHom(\fib(f_{v,\ast} \Q_p \to \colim_{i \in I} h_{i,v\ast} \Q_p),\Q_p).
$$
Here the $\IHom$ is computed in $\D(\Spa(C)_v, \Q_p)$.
\end{theorem}
\begin{proof}
Let $J$ be the filtered partially ordered set of functors $j\colon \mathbb{N} \to I$ (one such functor is smaller than another one if it so pointwise). For $j\in J$, let $X_j=\colim_{k\in \mathbb{N}} X_{j(k)}$. By passing to the limit over $J$ on both sides of the claimed isomorphism, we can reduce to the case that $X = X_j$ for some $j$, i.e. we can assume that $X$ is countable at infinity and that $I = \N$. Now note that \Cref{rslt:pushforward-commutes-with-sigma-nuclear-Qp-case} and \cref{thm:poincare-duality-smooth-level-ff-curve} imply
$$
    f_{v,\ast} \Q_p(d)][2d] = \sigma_{\Q_p}(f_\ast 1(d)[2d]) = \sigma_{\Q_p} \IHom_{\DFF(X,\Q_p)}(f_! 1, 1).
$$
Applying \Cref{prop:description-compactly-supp-cohomology-via-excision} and using again \Cref{rslt:pushforward-commutes-with-sigma-nuclear-Qp-case}, we see that 
$$
    \sigma_{\Q_p}(f_! 1) = \fib (f_{v,\ast} \Q_p \to \colim_n  h_{n,v\ast} \Q_p)
$$
By \Cref{prop:compactly-supported-coh-is-countable-colimit-compact-times-perfect-complex}, $M := f_! 1$ is in $\mathcal{C}$, i.e. of the form of \Cref{rslt:sigma-Qp-comparison-of-internal-hom}. Then we conclude from this theorem (taking $N=1$) that
$$
    \sigma_{\Q_p} \IHom_{\DFF(X,\Q_p)}(f_! 1, 1) = \IHom_{\D(\Spa(C)_v,\Q_p)}(\sigma_{\Q_p}(f_! 1),\Q_p).
$$
Combining the above identities results in the claim.
\end{proof}

\begin{remarks}
\begin{remarksenum}
    \item \label{relation-colmez-gilles-niziol} \Cref{thm:poincare-duality-for-general-partially-proper-spaces} in particular gives a proof of \cite[Conjecture 1.20]{colmez2023arithmetic}. Recently, Colmez, Gilles and Nizio\l{} have announced a proof of this conjecture as well, see \cite{colmez2024duality}. As far as we understand, their strategy is similar: 1) prove a duality statement on the curve, and 2) use (a version of) $\sigma_{\Q_p}$ to pass to v-sheaves. The difference of the two approaches lies in the proof of 1). Namely, we obtain the duality by abstract means, while they use comparison theorems and duality for Hyodo--Kato/de Rham cohomology.

    \item We stress that our Poincaré duality result at the level of the Fargues--Fontaine curve, i.e. \Cref{thm:poincare-duality-smooth-level-ff-curve}, works for any smooth morphism of analytic adic spaces. It is only in the translation into a duality statement on the v-site that we were forced to restrict to the case of a partially proper rigid variety over a geometric point that is countable at infinity, and had to make use of some non-trivial input from $p$-adic Hodge theory.
\end{remarksenum}
\end{remarks}

\begin{remark}
The hypothesis that $C$ is the completed algebraic closure of a complete discretely valued non-archimedean extension of $\Q_p$ with perfect residue field is here to be able to quote \cite{bosco2023rational}, but should not be necessary (one should replace the comparison with Hyodo-Kato cohomology with motivic arguments to get the desired finiteness). We do not pursue this further.
\end{remark}

\subsection{Examples coming from geometric local Langlands} \label{sec:examples}

While the 6-functor formalism $\DFF(-,\Q_p)$ is the relevant one for questions related to pro-\'etale cohomology of rigid spaces, as the previous subsections illustrated, we believe that, as far as $\Q_p$-coefficients are concerned, the fundamental one is $\D_{(0,\infty)}(-)$. In particular, we hope that it can be useful to the investigation of the $p$-adic aspects of the work of Fargues--Scholze. In this final subsection, we would like to provide a few examples in this direction.

We start with the simplest possible example, $\Spd(\mathbb{F}_p)$. Due to the abstract definition of $\D_{(0,\infty)}(-)$, the following theorem is not obvious, even for $\Spd(\F_p)$. The argument in the proof is due to Felix Zillinger.

\begin{theorem} \label{rslt:compute-D-0-infty-on-Fp-algebra}
Let $R$ be perfect $\F_p$-algebra which is integral over an $\F_p$-algebra of finite type. Then the natural functor
\[
    \D_{\solid}(W(R)[1/p])=\Mod_{W(R)[1/p]}\D_\solid(\Q_p)\to \D_{(0,\infty)}(\Spd(R))
\]
is an equivalence, where $\Spd(R):=\Spa(R,\F_p)^\diamond$.
\end{theorem}
\begin{proof}
Set $S:=\Spd(R)$ and $S':=\Spd(R((t^{1/p^\infty})))$. Then $S'$ is a perfectoid space, and the assumptions on $R$ imply that the natural morphism $S'\to \Spd(\F_p((t^{1/p^\infty})))$ is of finite $\dimtrg$. We can conclude that
\[
    \D_{(0,\infty)}(S')\cong \D_{\hat\solid}(\mathcal{Y}_{(0,\infty),S'}).
\]
Now the critical observation is that
\[
    \mathcal{Y}_{(0,\infty),S'\times_S S'}\cong \mathcal{Y}_{(0,\infty),S'}\times_{\Spa(W(R)[1/\pi],\Z_p)} \mathcal{Y}_{(0,\infty),S'}
\]
in analytic stacks, as one checks similarly to the proof of \Cref{sec:farg-font-curv-1-commutation-with-pullbacks}, i.e. after base change to $\Q_{p,\infty}$ and then on tilts (and similarly for the products $S'\times_S \ldots \times_S S'$), cf.\ as well \cite[Proposition 1.1]{zillinger}.
It is therefore sufficient to see that $\D_{\solid}(-)$ satisfies descent for the morphism $\mathcal{Y}_{(0,\infty),S'}\to \Spa(W(R)[1/\pi],\Z_p)$ (note that $\D_{\hat\solid}(-)=\D_{\solid}(-)$ here by \Cref{sec:defin-d_hats--4-comparison-for-modified-modules}). Picking an untilt $T'$ of $S'$, this reduces to the assertion that $\D_{\solid}(-)$ satisfies descent for the morphism $T'\to \Spa(W(R)[1/\pi],\Z_p)$. This is easy (e.g.\ by producing a section as modules).
\end{proof}

\begin{remark} \label{rmk:result-of-zillinger-bc(O(1))}
The result of \cite{zillinger} is more general, since it also applies to open sub-v-sheaves of v-sheaves attached to perfect schemes. For example, consider $\Spd(\mathbb{F}_p[[t]])$, where $\mathbb{F}_p[[t]]$ is endowed with its $t$-adic topology. It is an open subsheaf of $\Spd(\mathbb{F}_p[[t]]^{\mathrm{triv}})$, where $\mathbb{F}_p[[t]]^{\mathrm{triv}}$ is $\mathbb{F}_p[[t]]$ endowed with the discrete topology (namely, for any affinoid perfectoid space $\Spa(R,R^+)$ of characteristic $p$ with an element $t\in R$, the locus where $t$ is topologically nilpotent is open in $\Spa(R,R^+)$). Moreover, if
$$
    A:=W(\F_p[[t^{1/p^{\infty}}]])=\Z_p[q^{1/p^\infty}]^{\wedge_{(p,q)}}
$$
with the $(p,q)$-adic topology, where $q=[t^\flat]$, then the open subspace of $\mathcal{Y}_{(0,\infty),\Spd(\F_p[[t]]^{\mathrm{triv}})}:=\Spa(A')$ (here $A':=A$ with the $p$-adic topology), which corresponds on the tilt to $\Spd(\F_p[[t]])\times \Spd (\Q_p)$,  is 
$$
    \mathcal{Y}_{(0,\infty),\Spd(\F_p[[t]])}: = \Spa(A) \backslash V(p),
$$
and thus, by the argument of the proof of \Cref{rslt:compute-D-0-infty-on-Fp-algebra}, 
$$
    \D_{(0,\infty)}(\Spd(\mathbb{F}_p[[t]])) = \D_{\solid}(\mathrm{Spa}(A) \backslash V(p)).
$$
As an analytic stack, $\Spa(A) \backslash V(p)$ is an open substack of $\mathrm{AnSpec}(A[1/p])$ obtained by analytically inverting $p$ and thus, it has complementary idempotent algebra given by the perfect bounded Robba ring
$$
    \widetilde{\mathcal{R}}^b = \colim_{r>0} A\langle p/q^{1/r} \rangle [1/q,1/p].
$$
By excision, this for example implies that if $f\colon \Spd(\mathbb{F}_p[[t]]) \to \Spd(\mathbb{F}_p)$ is the structure morphism, then in the 6-functor formalism $\D_{(0,\infty)}(-)$, 
$$
    f_! 1 = \widetilde{\mathcal{R}}^b/A[1/p][-1].
$$
\end{remark}

\begin{remark} \label{D-0-infty-for-profinite-sets}
Let $T$ be a profinite set. One can apply the theorem to $\underline{T}=\mathrm{Spd}(R)$, with $R=C(T,\mathbb{F}_p)$. Since $W(R)=C(T,\Z_p)$, we obtain that
$$
    \D_{(0,\infty)}(T) \isom D_{\solid}(C(T,\Q_p)).
$$
\end{remark}

If $H$ is a locally profinite group, denote by $H^{\mathrm{cont}}$ the group object in analytic stacks obtained as the colimit of $K^{\mathrm{cont}}=\AnSpec(C(K,\Q_p))$ over compact open subgroups $K$ of $H$. As another consequence of \Cref{rslt:compute-D-0-infty-on-Fp-algebra}, we obtain:

\begin{corollary} \label{sec:examples-3-category-on-classifying-stack}
Let $H$ be a locally profinite group. Then
\[
    \D_{(0,\infty)}(\Spd(\F_p)/H) \isom \D(\AnSpec(\Q_{p\solid})/H^{\mathrm{cont}})
\]
\end{corollary}
\begin{proof}
Both sides are obtained via the same descent, using \cref{D-0-infty-for-profinite-sets}. 
\end{proof}

Therefore, $\D_{(0,\infty)}(\Spd(\F_p)/H)$  is a reasonable candidate for a ``derived category of continuous $H$-representations on solid $\Q_p$-vector spaces''. In particular, it contains fully faithfully the category of continuous representations of $H$ on $\Q_p$-Banach spaces, as well as the category of solid locally analytic representations of $H$, when $H$ is a $p$-adic Lie group (\cite[Proposition 6.1.5]{rodrigues2023solid}). 

\begin{remark}
Although most of the discussion of this section makes good sense for the $\Z_p$-linear category $\D_{[0,\infty)}(-)$ (and in particular its mod $p$ variant $\D_{[0,0]}(-)$), the previous statements are the reason we restricted to $\D_{(0,\infty)}(-)$. Indeed, in contrast with \Cref{rslt:compute-D-0-infty-on-Fp-algebra}, the natural functor $\Phi\colon \D_\solid(\F_p)\to \D_{[0,0]}(\Spd \F_p)$ is not an equivalence (and hence one also does not deduce an analogue of \Cref{sec:examples-3-category-on-classifying-stack}). Namely, assume that $\Phi$ is an equivalence, and consider the morphism $f\colon \Spd(\F_p[T],\F_p[T])\to \Spd(\F_p)$. Note that $f$ is qcqs and $p$-bounded, so that by \cref{rslt:Dsha-pushfoward-base-change} the formation of $f_\ast(1)$ commutes with base change. Base changing to $\Spd (C)$ with $C:=\F_p((t^{1/p^\infty}))$ shows that $f_\ast(1)$ is nuclear, and in fact this base change is given by $C\langle T^{1/p^\infty}\rangle$. If $\Phi$ were an equivalence, $f_\ast(1)$ would therefore be nuclear, i.e., discrete and thus a colimit of copies of $\F_p$. But this expresses $C\langle T^{1/p^\infty}\rangle$ as a colimit of relatively discrete solid $C$-modules, which is not possible.
It is conceivable that $\Phi$ is fully faithful, at least when restricting to $\omega_1$-compact objects in $\D_\solid(\F_p)$. We thank Dustin Clausen and Peter Scholze for related discussions.
\end{remark}

Let now $E$ be a finite extension of $\Q_p$\footnote{Since both the geometry and the coefficients are $p$-adic, one should keep in mind which is which. Here, the geometry is over $\mathrm{Spd}(E)$ (in the sense that the moduli stacks appearing such as $\Bun_G$, $\mathrm{Div}^1$, etc. are defined using Fargues--Fontaine curves for the local field $E$), while the coefficients are $\Q_p$ (in the sense of \Cref{rmk:choice-coefficients-Qp}), since the definition of $\D_{(0,\infty)}$ involves Fargues--Fontaine curves for the local field $\Q_p$.} of degree $d$, with residue field $\mathbb{F}_q$, let $G$ be a reductive group over $E$, and let $\Bun_G$ be the small v-stack of $G$-bundles on the Fargues--Fontaine curve relative to $E$. It contains as an open substack the classifying stack 
$$ 
    \Spd(\F_q)/G(E)
$$
which is $\D_{[0,\infty)}$-smooth, by \Cref{prop:classifying-stack-loc-profinite-coh-smooth}. In fact, this holds true for the whole stack $\Bun_G$:
  
\begin{proposition} \label{prop:BunG-is-D0-infty-coh-smooth}
The small v-stack $\Bun_G$ is $\D_{[0,\infty)}$-smooth over $\Spd(\F_q)$ of (cohomological) dimension $\dim \mathrm{Res}_{E/\Q_p}(G)$.
\end{proposition}
\begin{proof}
The claim (including the implicit claim of $!$-ability of $\Bun_G \to \Spd(\F_q)$) can be checked after base change to a suitable perfectoid field $K$, and we can take $K=C^\flat$, with $C$ a completed algebraic closure of $E$. Recall Beauville--Laszlo uniformization, which gives a v-cover
$$
    \bigsqcup_{\mu \in X^\ast(T)} G(E) \backslash Gr_{G,\mu,K} \to \Bun_{G,K}.
$$
This morphism is $!$-able and $\D_{[0,\infty)}$-smooth: this reduces as in the proof of \cite[Proposition IV.1.19]{fargues-scholze-geometrization} to the cohomological smoothness of open Schubert cells, which in turn reduces by the argument of \cite[Proposition VI.2.4]{fargues-scholze-geometrization} to \Cref{rslt:smooth-maps-of-adic-spaces-are-cohom-smooth}. \cite[Proposition V.2.4]{fargues-scholze-geometrization} and \Cref{rslt:smooth-maps-of-adic-spaces-are-cohom-smooth} also prove that for each $\mu$, 
$$
    G(E) \backslash Gr_{G,\mu,K}  \to \Spd(K)/G(E)
$$
is $\D_{[0,\infty)}$-smooth. Moreover, by \Cref{prop:classifying-stack-loc-profinite-coh-smooth}, 
$$
    \Spd(K)/G(E) \to \Spd(K) 
$$
is also $\D_{[0,\infty)}$-smooth. Altogether this shows that $\Bun_G$ admits a cohomologically smooth atlas by a cohomologically smooth small v-stack. Hence by \cite[Lemma~4.5.8(i)]{heyer-mann-6ff} we deduce that $\Bun_G\to \Spd(\F_q)$ is $\D_{[0,\infty)}$-smooth.
The cohomological dimension of $\Bun_G$ (i.e. the degree in which the invertible dualizing complex is concentrated) must agree with the cohomological dimension $\dim \mathrm{Res}_{E/\Q_p} G$ of $\Spd(\F_q)/G_b(E)$ for any $b\in B(G)$ basic because this degree is locally constant on $\Bun_G$ (as can be checked by pullback to strictly totally disconnected perfectoid spaces).
\end{proof}

\begin{remark}
Note that contrary to what happens $\ell$-adically, $\Bun_G$ does not have dimension $0$, but (cohomological) dimension $\dim \mathrm{Res}_{E/\Q_p} G$. In the case $G=\mathrm{GL}_n$, this matches the expected dimension of the moduli stack of $(\varphi,\Gamma)$-modules over the Robba ring (see \cite[Conjecture 5.1.18]{emerton2022introduction}).
\end{remark} 

The small v-stack $\Bun_G$ is endowed with its Harder-Narasimhan filtration, with strata $\Bun_G^b$ indexed by elements $b\in B(G)$. Semi-stable strata are classifying stacks of locally profinite groups (which are the $E$-points of certain inner forms $G_b$ of $G$), for which the category $\D_{(0,\infty)}$ can be described thanks to \Cref{sec:examples-3-category-on-classifying-stack}. Other strata are also classifying stacks, but of more complicated group diamonds. To understand $\D_{(0,\infty)}(\Bun_G)$, one therefore needs to understand $\D_{(0,\infty)}(-)$ for such classifying stacks and how the resulting categories ``glue'' together. It seems interesting to understand both aspects in details. Rather than developing this here, we provide two instructive examples.

\begin{example} \label{ex:beauville-laszlo-unif-and-excision}
If $G=\GL_2$, Beauville--Laszlo uniformization gives a morphism
$$
    f\colon \mathbb{P}_E^{1,\diamond} \to \Bun_2
$$ 
factoring through the natural $\GL_2(E)$-action on the source, with image $\Bun_2^{\leq b_1}$, the union of the stratum $j\colon \Bun_2^{b_0} \to \Bun_2^{\leq b_1}$ corresponding to $\mathcal{E}_{b_0}=\mathcal{O}(1/2)$ and $i\colon \Bun_2^{b_1} \to \Bun_2^{\leq b_1}$ corresponding to $\mathcal{E}_{b_1} = \mathcal{O}\oplus \mathcal{O}(1)$. The inverse image of $\Bun_2^{b_0}$ by $f$ is Drinfeld's plane $\Omega_E^\diamond$, and the inverse image of $\Bun_2^{b_1}$ is $\mathbb{P}^1(E)$. To get an idea for what
$$
    \cofib(j_! j^\ast \to \id)
$$
on $\Bun_2^{\leq b_1}$ looks like, we can describe it after pullback along $f$. Let $j'\colon \Omega_E^\diamond \injto \mathbb{P}_E^{1,\diamond}$ be the open immersion. By proper base-change, it is enough to describe $\cofib(j'_! j^{' \ast} \to \id)$. But we know by \Cref{cor-exicsion-Qp-cohomology} that 
\begin{align*}
    \cofib(j'_! j^{' \ast} \to \id) = \colim_U j_{U,\ast} j_U^\ast,
\end{align*}
the colimit running over all open neighborhoods $j_U\colon U \injto \mathbb{P}_E^1$ of $\mathbb{P}^1(E)$ in $\mathbb{P}_E^1$. Note that actually $f$ factors through the quotient
$$
\mathbb{P}_E^{1,\diamond} \to \mathbb{P}_E^{1,\diamond}/\varphi^\Z,
$$
so for instance we deduce that after pullback $\mathrm{cofib}(j_! 1 \to1)$ is described as the overconvergent pro-\'etale $\Q_p$-cohomology of $\underline{\mathbb{P}^1(E)}$ in $\mathbb{P}_E^1$ (after applying $\sigma_{\solid,\Q_p}$ for $\mathbb{P}_E^{1,\diamond}/\varphi^\Z$). 
\end{example}

\begin{remark}
We warn the reader that Fargues--Scholze's remarkable atlas $(\mathcal{M}_b)_{b\in B(G)}$ (see \cite[\S V.3]{fargues-scholze-geometrization}) does not give a $\D_{(0,\infty)}$-smooth atlas of $\Bun_G$.
\end{remark}

\begin{example} \label{ex:classifying-stack-bc-O(1)}
Assume that $E=\Q_p$, $G = \GL_2$ and consider the stratum $\Bun_2^{b_1}$ of $\Bun_2$ associated to $\mathcal{E}_{b_1}=\mathcal{O}\oplus \mathcal{O}(1)$. This stratum is the classifying stack of the automorphism group $\underline{\Aut}(\mathcal{O}\oplus \mathcal{O}(1))$, which is an extension
$$
    1 \to \mathcal{BC}(\mathcal{O}(1)) \to \underline{\Aut}(\mathcal{O}\oplus \mathcal{O}(1)) \to \Q_p^\times \times \Q_p^\times \to 1.
$$
The quotient is a locally profinite group. The mysterious part is the unipotent part, and hence one can first try to describe
$$
    \D_{(0,\infty)}(\Spd(\F_p)/\mathcal{BC}(\mathcal{O}(1))).
$$
We would like to explain the construction of a ``Fourier--Mukai'' functor 
$$
    \mathcal{F}_{\Spd(\F_p)/\mathcal{BC}(\mathcal{O}(1))}\colon \D_{(0,\infty)}(\Spd(\F_p)/\mathcal{BC}(\mathcal{O}(1))) \to \D_{(0,\infty)}(\Q_p).
$$
First, observe that $\Spd(\F_p)/\mathcal{BC}(\mathcal{O}(1))$ is a $\Q_p$-module in v-stacks (in the terminology of \cite{anschutz2021fourier}, it is a stack in $\Q_p$-vector spaces), hence there is a natural action map
$$
    \alpha\colon \Q_p \times \Spd(\F_p)/\mathcal{BC}(\mathcal{O}(1)) \to \Spd(\F_p)/\mathcal{BC}(\mathcal{O}(1)). 
$$
We have $\mathcal{BC}(\mathcal{O}(1)) = \Spd(\mathbb{F}_p[[t]])$ and therefore, cf. \Cref{rslt:compute-D-0-infty-on-Fp-algebra}, \Cref{rmk:result-of-zillinger-bc(O(1))}, 
$$
    \D_{(0,\infty)}(\Spd(\mathbb{F}_p[[t]])) = \D_{\hat\solid}(\mathcal{Y}_{(0,\infty),\mathcal{BC}(\mathcal{O}(1))})).
$$
Moreover, by \cite[Example 10.1.8]{scholze-berkeley-lectures},
$$
    \mathcal{Y}_{(0,\infty),\mathcal{BC}(\mathcal{O}(1))} = (\widetilde{\mathbb{G}}_m)_\eta
$$
the generic fiber of the universal cover of $\mathbb{G}_m$, represented by the (pre)perfectoid open unit disk (centered at $1$) over $\Q_p$. Thus, arguing as in the proof of \Cref{rslt:v-descent-for-D-0-infty}, we see that
$$
    \D_{(0,\infty)}(\Spd(\F_p)/\mathcal{BC}(\mathcal{O}(1))) = \D_{\hat\solid}(\Spa(\Q_p)/(\widetilde{\mathbb{G}}_m)_\eta).
$$
Let $\mathcal{L}$ be the universal line bundle on $\Spa(\Q_p)/\mathbb{G}_m^{\mathrm{an}}$, and let us also denote by $\mathcal{L}$ the object of $\D_{(0,\infty)}(\Spd(\F_p)/\mathcal{BC}(\mathcal{O}(1)))$ obtained by pullback along the morphism $(\widetilde{\mathbb{G}}_m)_\eta \to \mathbb{G}_m^{\mathrm{an}}$ obtained by composing the first projection with the morphism $\mathbb{G}_{m,\eta} \to \mathbb{G}_m^{\mathrm{an}}$, which on $R$-points corresponds to the inclusion $1+R^{\circ \circ} \to R^\times$. Let also
\begin{align*}
    \pi &\colon \Q_p \times \Spd(\F_p)/\mathcal{BC}(\mathcal{O}(1)) \to \Spd(\F_p)/\mathcal{BC}(\mathcal{O}(1)),\\
    \pi^\vee &\colon \Q_p \times \Spd(\F_p)/\mathcal{BC}(\mathcal{O}(1)) \to \Q_p
\end{align*}
be the two projections. We define: 
$$
    \mathcal{F}_{\Spd(\F_p)/\mathcal{BC}(\mathcal{O}(1))}(-) := \pi^{\vee}_! (\pi^\ast(-) \otimes \alpha^\ast \mathcal{L}).
$$
In fact, the above construction can be done over the base $\Spd(\F_p)/(\Q_p^\times \times \Q_p^\times)$, and gives a $\D_{(0,\infty)}(\Spd(\F_p)/(\Q_p^\times \times \Q_p^\times))$-linear functor
$$
    \mathcal{F}\colon \D_{(0,\infty)}(\Bun_2^{b_1}) \to \D_{(0,\infty)}\left(\Q_p/(\Q_p^\times \times \Q_p^\times)\right).
$$
For an interesting geometric example inspired by \cite{howe2020unipotent}, consider the compact modular curve $\mathcal{X}_\infty$ with infinite level at $p$ and finite unspecified level away from $p$ and the Hodge-Tate period map
$$
    \pi_{\mathrm{HT}}\colon \mathcal{X}_\infty \to \mathbb{P}_{\Q_p}^1.
$$
By properness of $\pi_{\mathrm{HT}}$ and \cite[Theorem 1.3]{zhang2023pel}, the object $\pi_{\mathrm{HT},\ast} (1) \in \D_{(0,\infty)}(\Spd(\mathbb{F}_p)/(\Q_p^\times \times \Q_p^\times))$ descends along the Beauville--Laszlo map to an object of $\D_{(0,\infty)}(\Bun_2)$, supported on $\Bun_2^{\leq b_1}$, in the notations of \Cref{ex:beauville-laszlo-unif-and-excision}. Its pullback $A$ along $i\colon \Bun_2^{b_1} \to \Bun_2^{\leq b_1}$ descends $(\pi_{\mathrm{HT} |_{\mathcal{X}_\infty^{\mathrm{ord}}}})_{\ast} (1)$, since the (closure of the) ordinary locus is the inverse image of $\mathbb{P}^1(\Q_p)$. Pick $C$ a complete algebraically closed non-archimedean extension of $\Q_p$. The base-change 
$$
    \mathcal{F}_C(A) \in \D_{(0,\infty)}\left((\Q_p/(\Q_p^\times \times \Q_p^\times))_C\right)
$$
of $\mathcal{F}(A)$ to $C$ can be pulled back to an object of $\D_{\hat\solid}\left((\Q_p/(\Q_p^\times \times \Q_p^\times))_C\right)$, which one should be able to describe explicitly as in \cite{howe2020unipotent}. 
\end{example}

\begin{remark} \label{rmk:fourier-transform-does-not-work}
Exchanging the roles of $\pi$ and $\pi^\vee$, one can define as well 
$$
    \mathcal{F}_{\Q_p}\colon \D_{(0,\infty)}(\Q_p) \to \D_{(0,\infty)}(\Spd(\F_p)/\mathcal{BC}(\mathcal{O}(1))).
$$
We warn the reader that despite the terminology and notation used, $\mathcal{F}_{\Spd(\F_p)/\mathcal{BC}(\mathcal{O}(1))}$ and $\mathcal{F}_{\Q_p}$ do \emph{not} seem to satisfy the usual (almost) involutivity of a Fourier transform, and thus probably are \emph{not} equivalences of categories. Indeed, denoting $f\colon \mathcal{BC}(\mathcal{O}(1)) \to \Spd(\mathbb{F}_p)$ the structure morphism, such an equivalence would enforce an isomorphism 
$$
    f_! 1 \isom C_c(\Q_p, \Q_p).
$$
We have computed the left-hand side in \Cref{rmk:result-of-zillinger-bc(O(1))}. The right-hand side can also be described using the Amice transform (\cite{colmez2010fonctions}) but the two descriptions do not seem to match up. 
\end{remark}

\Cref{ex:classifying-stack-bc-O(1)} is a particular case of a more general picture which we will briefly discuss in the following remarks.

\begin{remark} \label{rmk:E-instead-of-Q_p}
If $\Q_p$ gets replaced by a finite extension of $\Q_p$, the discussion of \Cref{ex:classifying-stack-bc-O(1)} remains valid, as long as one replaces $\mathcal{BC}(\mathcal{O}(1))$ by $\mathcal{BC}(\mathcal{O}(d))$, with $d=[E:\Q_p]$. Indeed, it is $\mathcal{BC}(\mathcal{O}(d))$ which comes from the formal group $\mathbb{G}_m$. 
\end{remark}

\begin{remark} \label{fourier-transform-general-E-general-bc-spaces}
Even more generally, for $E$ a finite extension of degree $d$ of $\Q_p$ with residue field $\F_q$, and $\lambda \in \Q$, one can construct ``Fourier--Mukai'' functors\footnote{Strictly speaking, the $!$-ability of the necessary maps is not part of \Cref{sec:exampl-da_h-bc-space-shriekable}, but it can be checked to hold here as well (by similar methods and descent from the perfectoid case).}
$$
    \mathcal{F}_{\mathcal{BC}(\mathcal{O}(\lambda))}\colon \D_{(0,\infty)}(\mathcal{BC}(\mathcal{O}(\lambda))) \to \D_{(0,\infty)}(\Spd(\F_q)/\mathcal{BC}(\mathcal{O}(d-\lambda)))
$$
and 
$$
    \mathcal{F}_{\Spd(\F_q)/\mathcal{BC}(\mathcal{O}(\lambda))}\colon \D_{(0,\infty)}(\Spd(\F_q)/\mathcal{BC}(\mathcal{O}(\lambda))) \to \D_{(0,\infty)}(\mathcal{BC}(\mathcal{O}(d-\lambda)))
$$
if $\lambda \in [0,d]$ and\footnote{Note that if $\lambda$, resp. $d-\lambda$, is negative, $\mathcal{BC}(\mathcal{O}(\lambda))$, resp. $\mathcal{BC}(\mathcal{O}(d-\lambda))$, is what we denoted $\mathcal{BC}(\mathcal{O}(\lambda)[1])$, resp. $\mathcal{BC}(\mathcal{O}(d-\lambda)[1])$ in \Cref{sec:the-6ff-formalism-D-0-infty}.}
$$
    \mathcal{F}_{\mathcal{BC}(\mathcal{O}(\lambda))}\colon \D_{(0,\infty)}(\mathcal{BC}(\mathcal{O}(\lambda))) \to \D_{(0,\infty)}(\mathcal{BC}(\mathcal{O}(d-\lambda)))
$$
if $\lambda \notin [0,d]$. Indeed, if for a stack in $E$-vector spaces $\mathcal{G}$ (in the terminology of \cite{anschutz2021fourier}), one defines\footnote{Note that this notation differs from the notation of loc.cit.} 
$$
    \mathcal{G}^{\vee} = \mathcal{H}om(\mathcal{G}, \Spd(\F_q)/\mathcal{BC}(\mathcal{O}(d))),
$$
one has a natural pairing
$$
    \mathcal{G}\times \mathcal{G}^\vee \to \Spd(\F_q)/\mathcal{BC}(\mathcal{O}(d))
$$
and hence as in the example an associated Fourier--Mukai functor. Moreover, for $\mathcal{G}$ any of the examples above on the left, $\mathcal{G}^\vee$ is seen to be what appears on the right, using \cite[Corollary~3.10]{anschutz2021fourier}.
\end{remark}

\begin{remark}
This series of remarks highlights an interesting difference between the $\ell$-adic case and the $p$-adic case. With $\ell$-adic coefficients, the unipotent part of the automorphism groups of $G$-bundles plays a minor role: e.g., for $b\in B(G)$, pullback gives an equivalence $\D_{\mathrm{lis}}(\Bun_G^b,\Q_\ell) \isom \D_{\mathrm{lis}}(\Spd(\F_q)/G_b(E), \Q_\ell)$. This is not the case with $p$-adic coefficients, as illustrated above for $E=\Q_p$, $G=\GL_2$ and $b$ the isocrystal corresponding to $\mathcal{E}_b=\mathcal{O} \oplus \mathcal{O}(1)$. For general $E$, we would expect that pullback gives an equivalence
$$
    \D_{(0,\infty)}(\Spd(\F_q)/\mathcal{BC}(\mathcal{O}(\lambda))) \isom \D_{(0,\infty)}(\Spd(\F_q))
$$
whenever $\lambda \not\in [0,d]$. This is in line with expectations of Emerton--Gee--Hellmann, \cite[\S 7.5.3]{emerton2022introduction}. 
\end{remark}

The above considerations give a partial picture of what the $p$-adic sheaf theory on $\Bun_G$ looks like. On the Galois side, as advocated by the work of Colmez and Emerton--Gee, the objects one would like to see are not representations of the Galois or Weil group, but $(\varphi,\Gamma)$-modules. Recall that if $E$ is a finite extension of $\Q_p$, its (cyclotomic) Robba ring is
$$
    \mathcal{R}_E= \colim_{r \to 1} \lim_{s \to 1} B_E^{[r,s]},
$$
where $B_E^{[r,s]}$ stands for the ring of rigid-analytic functions on the affinoid annulus of radius $[r,s]$ over the maximal unramified extension of $E$ in its cyclotomic extension $E^{\mathrm{cycl}}$. It is endowed with continuous commuting actions of $\varphi$ and $\Gamma=\mathrm{Gal}(E^{\mathrm{cycl}}/E)$. A $(\varphi,\Gamma)$-module over $\mathcal{R}_E$ is then a finite projective $\mathcal{R}_E$-module, endowed with commuting continuous semi-linear actions of $\varphi$ and $\Gamma$, such that the linearization of $\varphi$ is an isomorphism. 

In the work of Fargues-Scholze, $\ell$-adic representations of the Weil group arise as as full subcategory in (the category of dualizable objects of) $\ell$-adic sheaves on (the base change to $\bar{\mathbb{F}}_p$) of $\mathrm{Div}_E^1$. Here, the same phenomenon happens (except that we do not base change to $\bar{\mathbb{F}}_p$!):

\begin{proposition} \label{relation-phi-gamma-modules}
Let $E$ be a finite extension of $\Q_p$. The category of $(\varphi,\Gamma)$-modules over $\mathcal{R}_E$ embeds fully faithfully into the category of dualizable objects of $\D_{(0,\infty)}(\mathrm{Div}_E^1)$. 
\end{proposition}
\begin{proof}
Since
$$
    \mathrm{Div}_E^1 = (\Spd(E^{\mathrm{cycl}})/\Gamma)/\varphi^{\Z}
$$
and $\D_{(0,\infty)}(\Spd(E^{\mathrm{cycl}}))= \D_{\hat\solid}(Y_{E^{\mathrm{cycl},\flat}})$, the category $\D_{(0,\infty)}(\mathrm{Div}_E^1)$ contains fully faithfully the category of $(\varphi,\Gamma)$-equivariant vector bundles on the analytic adic space $Y_{E^{\mathrm{cycl},\flat}}$, which is known by results of Berger and Kedlaya to be equivalent to the category of $(\varphi,\Gamma)$-modules over $\mathcal{R}_E$, see e.g. \cite[Theorem 5.1.5]{emerton2022introduction}. 
\end{proof}

Continuing the above remarks, we next point out that it is also interesting from the perspective of $p$-adic local Langlands to consider $\D_{(0,\infty)}$ for Banach-Colmez spaces rather than their classifying stacks, in relation to \Cref{relation-phi-gamma-modules}. 

\begin{remark} \label{rmk:drinfeld-laumon}
For $E=\Q_p$, as a special case of \Cref{fourier-transform-general-E-general-bc-spaces} we can produce a natural functor
$$
    \mathcal{F}_{\mathcal{BC}(\mathcal{O}(1))}\colon \D_{(0,\infty)}(\mathcal{BC}(\mathcal{O}(1))) \to \D_{(0,\infty)}(\Spd(\F_p)/\Q_p).
$$
Precomposing by !-pushforward along
$$
    j\colon \mathrm{Div}_{\Q_p}^1 = ( \mathcal{BC}(\mathcal{O}(1)) \backslash \{0\} )/\Q_p^\times \injto \mathcal{BC}(\mathcal{O}(1)) /\Q_p^\times
$$
and taking into account $\Q_p^\times$-equivariance, this gives rise to a functor
$$
    \D_{(0,\infty)}(\mathrm{Div}_{\Q_p}^1) \to \D_{(0,\infty)}(\Spd(\F_p)/M(\Q_p)) \isom \D_{\solid}(\AnSpec(\Q_p)/M(\Q_p)^{\mathrm{cont}})
$$
where $M \subset \GL_2$ is the mirabolic subgroup. (This construction is a very special case of the Drinfeld-Laumon construction in the geometric Langlands program, specialized to this setting.) In particular, by \Cref{relation-phi-gamma-modules}, this produces a functor from $(\varphi,\Gamma)$-modules over the Robba ring to continuous representations of the mirabolic subgroup. It can be made explicit, via the same kind of arguments as in \Cref{rmk:result-of-zillinger-bc(O(1))}. We expect this functor to be related to, but not the same as Colmez' functor. 
\end{remark}

\begin{remark}
Note, following \Cref{rmk:E-instead-of-Q_p}, that for $E$ a finite extension of $\Q_p$, one cannot go so directly from (Lubin-Tate) $(\varphi,\Gamma)$-modules  to representations of the mirabolic: one first needs to produce an object in $\D_{(0,\infty)}(\mathrm{Div}_E^d)$ out of a $(\varphi,\Gamma)$-module (in more classical terms, this is related to the passage from Lubin-Tate $(\varphi,\Gamma)$-modules to multivariable $(\varphi,\Gamma)$-modules). Such a construction has been investigated recently for mod $p$ coefficients by Fargues and (partly) ``explains'' the fact that $p$-adic local Langlands is more complicated for a general $p$-adic field $E$ than it is for $\Q_p$. 
\end{remark}

\begin{remark} \label{drawbacks}
In the setting of \Cref{rmk:drinfeld-laumon}, Colmez' construction furnishes, from a $(\varphi,\Gamma)$-module over the Robba ring, not just a continuous representation of the mirabolic, but a locally analytic one. This is not directly visible from the previous geometric construction. In fact, we consider this, together with \Cref{rmk:fourier-transform-does-not-work}, as a drawback of $\D_{(0,\infty)}(-)$. In forthcoming work with Rodr\'iguez Camargo and Scholze, the first two authors aim at formulating a different geometrization of $p$-adic Langlands, which is tailored to capture locally analytic representations and in which these drawbacks disappear.
\end{remark}

Finally, we briefly discuss Hecke functors. We note that the Hecke diagram consists of lpbc maps, if the modification is bounded, hence the Hecke functors are well-defined. Being able to consider these general Hecke functors was a major motivation for developing our formalism in the general context of small v-stacks. Again, rather than studying them in general, we discuss them through an example.

\begin{example} \label{ex:p-adic-jacquet-langlands}
We consider the restriction $\mathrm{JL}$ to the degree $1$ semi-stable stratum of the standard Hecke functor for $G=\mathrm{GL}_2$ for sheaves supported on the neutral stratum ($\mathrm{JL}$ stands for Jacquet--Langlands). Let us make this explicit. Fix $C$ a complete algebraic closure of $E$, and let $\F_q$ be the residue field of $E$, with $q=p^d$.

Let $\mathcal{M}_\infty$ be the small v-stack parametrizing injective $\mathcal{O}$-linear maps from $\mathcal{O}^2$ to $\mathcal{O}(1/2)$ on the Fargues--Fontaine curve associated to $E$. Taking the cofiber of such a map gives a morphism
$$
    \mathcal{M}_\infty \to \mathrm{Div}_E^1
$$
Moreover, $\mathcal{M}_\infty$ is a $D^\times$-torsor over $\Omega_E^\diamond/\varphi^{d\Z}$ and a $\GL_2(E)$-torsor over $X_D^\diamond/\varphi^{d\Z}$, where $D$ is the non-split quaternion algebra over $E$ and $X_D$ its Brauer-Severi variety. Hence we have an isomorphism of small v-stacks
$$
    X_D^\diamond/D^{\times} \isom \Omega_E^\diamond / \mathrm{GL}_2(E).
$$
If $\Pi$ is a continuous $\mathrm{GL}_2(E)$-representation on a solid $\Q_p$-vector space, let $\mathcal{F}_\Pi$ be the object of $\D_{(0,\infty)}(X_D^\diamond/D^\times)$ coming via pullback along the morphism
$$
    X_D^\diamond/D^\times \to \Spd(\mathbb{F}_p)/\mathrm{GL}_2(E) 
$$
induced by the above isomorphism, together with \Cref{sec:examples-3-category-on-classifying-stack}. If $f$ denotes the natural morphism
$$
    f\colon X_D^\diamond/D^{\times} \to \mathrm{Div}_E^1 \times \Spd(\F_p)/D^{\times},
$$
we set 
$$
    \mathrm{JL}(\Pi) := f_\ast \mathcal{F}_\Pi
$$
(note that $f$ is proper). We have a surjective map
$$
    g\colon \Spd(C)/\varphi^{d\Z} \to \mathrm{Div}_E^1 \times \Spd(\F_p)/D^{\times}
$$
(factoring through $\Spd(\mathbb{F}_p)$ on the second factor), and we can describe $g^\ast \mathrm{JL}(\Pi)$ using proper base change: it is given by pushforward of (the pullback of) $\mathcal{F}_\Pi$ along 
$$
X_D^\diamond/\varphi^{d\Z} \to \mathrm{Spd}(C)/\varphi^{d\Z},
$$
i.e. by pushforward in $\DFF(-,\Q_p)$. In other words, this is computing the $\Pi$-isotypic component of the pro-étale $\Q_{p^d}$-cohomology of the Lubin-Tate tower (after applying the functor $\sigma_{\Q_p}$).
\end{example}

\begin{remark} \label{rmk:hecke-preserves-dualizable}
The functor $\mathrm{JL}$ preserves dualizable objects, since it is obtained as pullback along a smooth morphism of rigid spaces followed by pushforward along a proper morphism of rigid spaces in characteristic $0$. As an illustration, take $E=\Q_p$ and $\Pi$ to be the standard representation of $\mathrm{GL}_2$, which we see as a representation of $\mathrm{GL}_2(\Q_p)$. Then, in the notations of the last example, $g^\ast \mathrm{JL}(\Pi)$ is given by the computation of \Cref{sec:pro-etale-cohomology-example-introduction-lubin-tate}, which was seen to be dualizable by a direct computation. In fact, if $G$ is a reductive group over $\Q_p$, then there exists a natural tensor functor
\[
    \Phi\colon \mathrm{Rep}^{\mathrm{alg}}_{\Q_p}G\to \D_{(0,\infty)}(\Bun_G)
\]
from the category of finite dimensional algebraic representations of $G$ towards $\D_{(0,\infty)}(\Bun_G)$. Indeed, let $V\in \mathrm{Rep}^{\mathrm{alg}}_{\Q_p}G$ and pick a perfectoid space $T$ with a morphism to $\Bun_G$, i.e., a $\varphi$-equivariant $G$-torsor $\mathcal{P}$ on the space $\mathcal{Y}_{(0,\infty),T}$. Then $\mathcal{P}\times^{G} V$ defines a vector bundle on $\mathcal{Y}_{(0,\infty),T}$, i.e., an object $\Phi(V)_T\in \D_{(0,\infty)}(T)$. As $T\mapsto \Phi(V)_T$ is compatible with pullback, descent yields the desired object $\Phi(V)\in \D_{(0,\infty)}(\Bun_G)$ and it is clear that $\Phi$ is a tensor functor. It is conceivable that Hecke functors associated with suitable kernels, e.g., to-be-defined Satake kernels, preserve dualizable objects. By applying this to $\Phi(V)$, we obtain an intriguing generalization of \Cref{sec:pro-etale-cohomology-example-introduction-lubin-tate}.
\end{remark} 
  
\begin{remark} \label{rmk:jl-preserves-suave-prim}
The functor $\mathrm{JL}$ also preserves suave objects. The functor defined as $\mathrm{JL}$ but reversing the role of the groups $\mathrm{GL}_2(E)$ and $D^\times$ (i.e. using the dual Hecke functor) preserves prim objects. It seems interesting to investigate further suaveness and primness on classifying stacks of $p$-adic Lie groups, to deduce representation-theoretic consequences of these facts (see e.g. \cite{scholze2018p} or \cite[Th\'eor\`eme 0.5]{colmez2023factorisation}).
\end{remark}

\begin{remark} \label{remark:why-d-0-infty-better}
The discussion of this subsection highlights some drawbacks of $\DFF(-,\Q_p)$ compared to $\D_{(0,\infty)}(-)$:
\begin{remarksenum}
    \item In the case of a classifying stack, if we had used $\DFF(-,\Q_p)$, we would not have obtained continuous $G$-representations on solid $\Q_p$-vector spaces, but such objects endowed with a Frobenius automorphism, which sounds artificial and inconvenient from the perspective of $p$-adic local Langlands. This is to be contrasted with (for example) \cite{mod-p-stacky-6-functors}: in \cite{mod-p-stacky-6-functors}, the $6$-functor formalism $\D_\solid^a(\mathcal{O}^+_{-}/\varpi)$ of the third author is used, which depends on the choice of a pseudo-uniformizer $\varpi$. Hence one cannot plug it in the classifying stack of a $p$-adic Lie group over $\mathrm{Spd}(\mathbb{F}_p)$, but only over a perfectoid base field. In this context, one must add a Frobenius in the picture to see a category of (smooth) representations on $\mathbb{F}_p$-vector spaces. Here, the formalism $\D_{(0,\infty)}(-)$ makes sense absolutely, and in the absolute case one sees representations on $\Q_p$-vector spaces, as desired.

    \item One might be worried that what seems relevant to the geometric realization of $p$-adic local Langlands is $p$-adic (pro)-\'etale $\Q_p$-cohomology of Rapoport-Zink spaces and their generalizations, which, as was exploited in previous subsections, is naturally computed in $\DFF(-,\Q_p)$, not in $\D_{(0,\infty)}(-)$. But, as we just saw in \Cref{ex:p-adic-jacquet-langlands}, this is not an issue: indeed, $\mathrm{Div}^1$ appears naturally when applying Hecke functors, and 
    $$
        \D_{(0,\infty)}(\mathrm{Div}^1) = \DFF(\mathrm{Spd}(\Q_p),\Q_p)
    $$ 
    so we still have the relation to $(\varphi,\Gamma)$-modules over the Robba ring and to pro-\'etale $\Q_p$-cohomology of local Shimura varieties (in other words, even if we do not put the quotient by Frobenius in the sheaf theory, it is forced on us here, by the geometry).

    \item Finally, the construction of the Fourier--Mukai functors discussed above only makes sense in $\D_{(0,\infty)}(-)$, since the kernel does not come from $\DFF(-,\Q_p)$. This is relevant for categorical aspects of $p$-adic local Langlands. 
\end{remarksenum}
\end{remark}

\clearpage
\appendix
\section{Appendix: Supplements on \texorpdfstring{$R$}{R}-linear 6-functor formalisms} \label{sec:abstract-results-6}

In the following we collect some abstract results on 6-functor formalisms. Let us recall from \cite[\S2.1]{heyer-mann-6ff} that a geometric setup $(\mathcal C,E)$ is a pair consisting of a category $\mathcal C$ with a class of morphisms $E$ stable under isomorphisms, compositions and diagonals and such that pullbacks of morphisms in $E$ exist and stay in $E$. We will additionally assume that $\mathcal C$ has all finite limits.

Given a geometric setup $(\mathcal C,E)$ the symmetric monoidal category $\Corr(\mathcal C,E)$ is defined in \cite[Definition~2.2.10]{heyer-mann-6ff}. Most 6-functor formalisms, in particular all the ones considered in this paper, are of the following form.

\begin{definition}
Let $R \in \CAlg(\Pr^L)$ be a presentably symmetric monoidal category and let $(\mathcal C,E)$ be a geometric setup. An \emph{$R$-linear 6-functor formalism} is a lax-symmetric monoidal functor
\[
    \D\colon \Corr(\mathcal C,E)\to \Pr_R^L.
\]
Here $\Pr_R^L$ denotes the category of $R$-linear presentable categories and is equipped with the relative Lurie tensor product $\tensor_R$.
\end{definition}

The above definition is made so that the right adjoints $f_\ast$, $f^!$, $\IHom$ do not necessarily commute with colimits, or are even $R$-linear, as this would impose severe constraints. If we want to highlight that the functors $f^\ast$, $f_\ast$ etc.\ depend on the 6-functor formalism $\D$, we write $f^\ast_\D$, $f_\ast^\D$, etc.

\begin{remark}
By \cite[Lemma~3.2.5]{heyer-mann-6ff} every presentable 6-functor formalism $\D$ is automatically $\D(*)$-linear, where $*$ is the final object in $\mathcal C$.
\end{remark}

The probably most common example is the case that $R = \D(\Z)$ is the derived category of abelian groups, or variants like $\D(\Z/n)$ for some $n\in \Z$. As is visible in this example, it is quite common to base change the coefficients. In abstract terms this amounts to the following construction. 

\begin{lemma} \label{sec:suppl-6-funct-base-change-for-6-functor-formalisms}
Let $R\to R'$ be a morphism in $\CAlg(\Pr^L)$, i.e. a symmetric monoidal functor of presentably symmetric monoidal categories. Let $\D\colon \Corr(\mathcal C,E) \to \Pr^L_R$ be an $R$-linear 6-functor formalism and let $f\colon Y\to X$ be a morphism in $\mathcal C$.
\begin{lemenum}
    \item \label{rslt:base-change-of-6ff-along-R-R'-induces-6ff} The composition
    \[
        \D' := R'\tensor_R \D\colon \Corr(\mathcal C,E)\to \Pr^L_R \xto{R'\otimes_R -} \Pr^L_{R'}
    \]
    is an $R'$-linear 6-functor formalism. 
    
    \item The diagram
    \[\begin{tikzcd}
        {\D'(Y)} & {\D(Y)} \\
        {\D'(X)} & {\D(X)}
        \arrow["{\mathrm{res}}", from=1-1, to=1-2]
        \arrow["{f_\ast^{\D'}}"', from=1-1, to=2-1]
        \arrow["{f_\ast^\D}", from=1-2, to=2-2]
        \arrow["{\mathrm{res}}", from=2-1, to=2-2]
    \end{tikzcd}\]
    commutes naturally, where $\mathrm{res}\colon \D'(-) \to \D(-)$ denotes the right adjoint of the functor $F\colon \D(-) = R\otimes_R \D(-)\to \D'(-)$. Similarly, with $f_\ast$ replaced by $f^!$ if $f \in E$.

    \item If $f_\ast^\D$ commutes with colimits and is $R$-linear\footnote{More precisely, this means that for $A\in R$ and $M\in \D(Y)$ the natural morphism $A\otimes f_\ast^\D(M)\to f_\ast^\D(A\otimes M)$ coming from $R$-linearity of $f^\ast_\D$ is an isomorphism.}, then the diagram
    \[\begin{tikzcd}
        {\D(Y)} & {\D'(Y)} \\
        {\D(X)} & {\D'(X)}
        \arrow["F", from=1-1, to=1-2]
        \arrow["{f_\ast^\D}"', from=1-1, to=2-1]
        \arrow["{f_\ast^{\D'}}", from=1-2, to=2-2]
        \arrow["F", from=2-1, to=2-2]
    \end{tikzcd}\]
    commutes naturally. Similarly, with $f_\ast$ replaced by $f^!$ if $f \in E$.

    \item Let $X\in \mathcal C$ and $M\in \D'(X)$. Then the diagram
    \[\begin{tikzcd}
        {\D'(X)} & {\D(X)} \\
        {\D'(X)} & {\D(X)}
        \arrow["{\mathrm{res}}", from=1-1, to=1-2]
        \arrow["{\IHom(F(M),-)}"', from=1-1, to=2-1]
        \arrow["{\IHom(M,-)}", from=1-2, to=2-2]
        \arrow["{\mathrm{res}}", from=2-1, to=2-2]
    \end{tikzcd}\]
    commutes naturally.

    \item \label{rslt:base-change-of-6ff-preserves-suave-and-prim} If $M \in \D(Y)$ is $f$-suave (resp.\ $f$-prim, resp.\ dualizable, resp.\ invertible) for $\D$, then $F(M) \in \D'(X)$ is $f$-suave (resp.\ $f$-prim, resp.\ dualizable, resp.\ invertible) for $\D'$. Moreover, $F$ preserves the respective ($f$-suave, resp.\ $f$-prim, resp.\ usual) duals.

    \item If $f$ is $\D$-smooth, then $f$ is $\D'$-smooth and $f^!_{\D'}(1) \isom F(f^!_\D(1))$.

    \item If $f$ is $\D$-prim, then $f$ $\D'$-prim and $F$ preserves the prim dual of $1$.
\end{lemenum}
\end{lemma}
\begin{proof}
Part (i) is clear, as $R'\otimes_R - \colon\Pr^L_R\to \Pr^L_{R'}$ is a symmetric monoidal functor. To prove the other claims, note that composing $\D'$ with the (lax symmetric monoidal) forgetful functor $\Pr^L_{R'} \to \Pr^L_R$ produces another $R$-linear 6-functor formalism (which we still denote $\D'$) and the unit of the adjunction between $R' \tensor_R -$ and the forgetful functor induces the morphism $F\colon \D \to \D'$ of $R$-linear 6-functor formalisms. This immediately proves (ii) and (iv) by passing to left adjoints. For (iii) note that by the assumption on $f_\ast^\D$ the adjunction between $f^\ast_\D$ and $f_\ast^\D$ is an adjunction in $\Pr^L_R$ (see \cite[Remark~7.3.2.9]{lurie-higher-algebra}) and is hence preserved by the 2-functor $R' \tensor_R -$.

For the remaining claims we note that by \cite[Proposition~4.2.1(i)]{heyer-mann-6ff} for every $S \in \mathcal C$ the morphism $F\colon \D \to \D'$ induces a 2-functor $\mathcal K_{\D,S} \to \mathcal K_{\D',S}$ which induces $F$ on the morphism categories. This immediately implies the claims about suave and prim objects in (v), from which (vi) and (vii) follow. Finally, for fixed $X \in \mathcal C$ the functor $F\colon \D(X) \to \D'(X)$ is symmetric monoidal (this follows again from the fact that $F$ is a morphism of 6-functor formalisms), which immediately implies that it preserves dualizable and invertible objects.
\end{proof}

\begin{remark}
The $R$-linearity condition in \Cref{sec:suppl-6-funct-base-change-for-6-functor-formalisms}, which amounts to commutation with colimits and projection formulae, might be conditions which can be difficult to guarantee in practice. If $R$ is rigid, then by \cite[Lemma 4.20]{anschuetz_mann_descent_for_solid_on_perfectoids} the right adjoint of an $R$-linear functor is $R$-linear if and only if it commutes with colimits. 
\end{remark}

Some of our constructions can be conceptually viewed as a ``descent from $R'$ to $R$''. Instead of discussing inverse limits of $6$-functor formalisms in general we therefore focus on this case.

\begin{definition} \label{sec:supplements-r-linear-descent-for-r-r-prime}
Let $R\to R'$ be a morphism in $\CAlg(\Pr^L)$. We say that it satisfies 2-descent if the natural functor
\[
    \Pr^L_R \to \varprojlim_{n \in \Delta} \Pr^L_{{R'}^{\otimes_R n}}
\]
is an equivalence.
\end{definition}

More concretely, we see that 2-descent along $R \to R'$ implies that for every $\mathcal{M}\in \Pr^L_R$ the natural functor
\[
    \mathcal{M} \isoto \varprojlim_{n \in \Delta} {R'}^{\otimes_R n}\otimes_R \mathcal{M}
\]
is an equivalence.

\begin{example}
Assume that $R\in \CAlg(\Pr^L)$ is stable. If $R'=\mathrm{Mod}_A(R)$ for some $\mathbb E_\infty$-algebra $A\in R$, then by \cite[Proposition 3.43, Proposition 3.45]{akhil-galois-group-of-stable-homotopy} $R\to R'$ admits $2$-descent if and only if $A$ is a descendable algebra object in $R$. We note that there exist a lot of examples for descendable algebra objects, e.g., $\Z/4\to \Z/2$ is descendable or any countably generated faithfully flat morphism of $\mathbb E_\infty$-rings.
\end{example}

We can use 2-descent along $R \to R'$ in order to deduce properties of an $R$-linear 6-functor formalism from its base-change to $R'$. This observation plays a crucial role in this paper. Before we come to the actual result, we first need a certain functoriality claim:

\begin{lemma} \label{rslt:functor-from-CAlg-to-lax-monoidal-endofun}
Let $\mathcal V$ be a symmetric monoidal category. Then there is a natural functor
\begin{align*}
    \CAlg(\mathcal V) \to \Fun^{\lax,\tensor}(\mathcal V, \mathcal V), \qquad A \mapsto - \tensor A.
\end{align*}
Here $\Fun^{\lax,\tensor}$ denotes the category of lax symmetric monoidal functors.
\end{lemma}
\begin{proof}
We equip the category $\Fun(\mathcal V, \mathcal V)$ with the Day symmetric monoidal structure from \cite[Example~2.2.6.9]{lurie-higher-algebra} (which is in general only an operad, as the necessary colimits may not exist in $\mathcal V$). There is a map of operads $\mathcal V \to \Fun(\mathcal V, \mathcal V)$ given by sending $X \in \mathcal V$ to the endofunctor $- \tensor X\colon \mathcal V \to \mathcal V$. Indeed, by the universal property of Day convolution the desired map of operads is equivalently given a map of operads $\mathcal V \times \mathcal V \to \mathcal V$ sending $(X, Y) \mapsto X \tensor Y$; but this functor exists and is even (strictly) symmetric monoidal. Passing to commutative algebra objects, we obtain a functor
\begin{align*}
    \CAlg(\mathcal V) \to \CAlg(\Fun(\mathcal V, \mathcal V)) = \Fun^{\tensor,\lax}(\mathcal V, \mathcal V),
\end{align*}
where the identity on the right follows again from the universal property of Day convolution (see \cite[Example~2.2.6.9]{lurie-higher-algebra}).
\end{proof}

\begin{proposition} \label{rslt:descent-of-6ff}
Let $R\to R'$ be a morphism in $\CAlg(\Pr^L_{\Sp})$ satisfying 2-descent. Let $\D$ be an $R$-linear 6-functor formalism on a geometric setup $(\mathcal C,E)$, let $f\colon Y \to X$ be a morphism in $E$ and $M \in \D(Y)$. Let $F\colon \D(-) \to \D'(-):=R'\otimes_R \D(-)$ be the natural transformation.
\begin{propenum}
    \item $M$ is $f$-suave for $\D$ (resp.\ $f$-prim for $\D$, resp.\ dualizable, resp.\ invertible) if and only if $F(M)\in \D'(Y)$ has the same property.

    \item $f$ is $\D$-smooth (resp.\ $\D$-prim) if and only if $f$ is $\D'$-smooth (resp.\ $\D'$-prim).
\end{propenum}
\end{proposition}
\begin{proof}
The map $R \to R'$ induces a functor $\Delta \to \CAlg(\Pr^L)$, $n \mapsto R'^{\tensor_R n}$ by the usual \v{C}ech nerve construction. Composing this functor with the one from \cref{rslt:functor-from-CAlg-to-lax-monoidal-endofun} and then pointwise precomposing the result with $\D$ produces a functor from $\Delta$ to $R$-linear 6-functor formalisms which sends $n \in \Delta$ to $\D'_n := R'^{\tensor_R n} \tensor_R \D$. By the functoriality of the category of kernels (more precisely, \cite[Lemma~C.3.3]{heyer-mann-6ff}) we obtain the functor
\begin{align*}
    \Delta \to \Cat_2, \qquad n \mapsto \mathcal K_{\D'_n}.
\end{align*}
In particular there is a natural 2-functor $\mathcal K_D \to \varprojlim_{n\in\Delta} \mathcal K_{D'_n}$ and from 2-descent it follows that this 2-functor is 2-fully faithful: This amounts to the observation that for any $Z \in \mathcal C$ the natural functor
\[
    \D(Z) \isoto \varprojlim_{n \in \Delta} \D'_n(Z)
\]
is an equivalence. Thus by \cite[Proposition~D.2.16]{heyer-mann-6ff} we deduce that a morphism in $\mathcal K_\D$ is left adjoint if and only if its image in each $\mathcal K_{\D'_n}$ is left adjoint, which in turn is equivalent to its image in $\mathcal K_{\D'}$ being left adjoint. The same works relative to $X$, i.e. for $\mathcal K_{\D,X}$ (in the sense of \cite[Definition 4.1.3.(b)]{heyer-mann-6ff}). This immediately implies the claims about suave and prim objects.

It only remains to verify the claim about dualizable and invertible objects in (i). But this immediately follows from the above limit descrition for $\D(Z)$, because dualizable and invertible objects descend along any limit (e.g. by applying \cite[Proposition~D.2.16(ii)]{heyer-mann-6ff} to the associated diagram of 2-categories with single object and endomorphisms given by the symmetric monoidal categories).
\end{proof}

\begin{remark}
The proof of \cref{rslt:descent-of-6ff} might be a bit intimidating, so let us sketch a more direct approach to proving it (albeit less rigorous). In the setup of the result, suppose we are given an object $M \in \D(Y)$ such that $F(M) \in \D'(Y)$ is $f$-suave, where $F\colon \D \to \D'$ is the natural morphism of 6-functor formalisms considered above. We wish to show that $M$ is $f$-suave, which by \cite[Lemma~4.4.5]{heyer-mann-6ff} amounts to showing that the natural map
\begin{align*}
    \pi_1^* \IHom(M, f^! 1) \tensor \pi_2^* M \isoto \IHom(\pi_1^* M, \pi_2^! M)
\end{align*}
is an isomorphism in $\D(Y \times_X Y)$, where $\pi_i\colon Y \times_X Y \to Y$ are the projections. By 2-descent, the functor $F\colon \D(Y \times_X Y) \to \D'(Y \times_X Y)$ is conservative and by assumption on $M$, the above isomorphism holds for $F(M)$ in place of $M$. Thus it is enough to show that the above map for $M$ gets transformed to the analogous map for $F(M)$ under $F$. Since $F$ commutes naturally with pullback and tensor products, the claim then essentially reduces to showing that for all $N \in \D(X)$ the natural map
\begin{align}
    F(\IHom(M, f^! N)) \isoto \IHom(F(M), f^! F(N)) \label{eq:change-of-6ff-compatible-with-IHom}
\end{align}
is an isomorphism (then apply the same also to $\pi_1^* M$, which is $\pi_2$-suave). Now we observe that if $M$ is $f$-suave then \cref{eq:change-of-6ff-compatible-with-IHom} is indeed an isomorphism, which one can deduce by writing $\IHom(F(M), f^!) = \DSuave_f(M) \tensor f^*$ (see \cite[Lemma~4.4.17(i)]{heyer-mann-6ff}) and using \cref{rslt:base-change-of-6ff-preserves-suave-and-prim}. In particular the analogous statement of \cref{eq:change-of-6ff-compatible-with-IHom} is true for $M$ replaced by $F(M)$ and $F$ replaced by the analog for $\D' \to R'^{\tensor_R n} \tensor_R \D$. By 2-descent one can then deduce that \cref{eq:change-of-6ff-compatible-with-IHom} holds as well (see the proof of \cite[Corollary~7.8(ii)]{mann-nuclear-sheaves} for a similar argument).

While the above sketch is more down-to-earth (and in particular avoids 2-categories), it is also less clean than the original proof of \cref{rslt:descent-of-6ff}, because it contains a lot of implicit claims about certain ``natural maps'' being the expected ones. In fact, the proof of \cite[Proposition~D.2.16]{heyer-mann-6ff} (which was used in \cref{rslt:descent-of-6ff}) is essentially the abstract version of the same argument!
\end{remark}

\bibliography{bibliography}
\addcontentsline{toc}{section}{References}

\end{document}